\documentclass[final,reqno]{amsart}
\usepackage[margin=1.5in,bottom=1.25in]{geometry}	

\usepackage{amsmath}	
\usepackage{amssymb}	
\usepackage{amsfonts}	
\usepackage{amsthm}	
\usepackage[foot]{amsaddr}	

\usepackage[centercolon=true]{mathtools}	

\mathtoolsset{%
}

\usepackage[utf8]{inputenc}	
\usepackage[T1]{fontenc}	

\usepackage[
cal=cm,
]
{mathalfa}

\usepackage{dsfont}	

\usepackage[light,scaled=.95]{AlegreyaSans}

\usepackage[proportional,tabular,lining,sf,mono=false]{libertine}

\usepackage{acronym}	
\newcommand{\acdef}[1]{\define{\acl{#1}} \textup{(\acs{#1})}\acused{#1}}	
\newcommand{\acdefp}[1]{\define{\aclp{#1}} \textup{(\acsp{#1})}\acused{#1}}	

\usepackage[labelfont={bf,small},labelsep=colon,font=small]{caption}	

\usepackage{subcaption}	
\usepackage[svgnames]{xcolor}	
\colorlet{MyRed}{Crimson!60!DarkRed}
\colorlet{MyBlue}{DodgerBlue!75!black}
\colorlet{MyGreen}{DarkGreen}
\colorlet{MyViolet}{DarkMagenta}

\colorlet{MyLightBlue}{DodgerBlue!20}
\colorlet{MyLightGreen}{MyGreen!20}

\colorlet{PrimalColor}{MyBlue}
\colorlet{PrimalFill}{MyLightBlue}
\colorlet{DualColor}{MyRed}

\colorlet{AlertColor}{MyRed}	
\colorlet{BadColor}{MyRed}	
\colorlet{GoodColor}{MyGreen}	
\colorlet{LinkColor}{MediumBlue}	
\colorlet{MacroColor}{MyViolet}
\colorlet{RevColor}{MediumBlue}	
\colorlet{RevColor}{Black}  

\usepackage{latexsym}	
\usepackage{fontawesome}	
\usepackage{pifont}	

\usepackage{tikz}	
\usetikzlibrary{calc,patterns}	

\usepackage{array}	
\usepackage{booktabs}	
\usepackage[inline,shortlabels]{enumitem}	
\setlist[1]{topsep=\smallskipamount,itemsep=\smallskipamount,left=\parindent}
\setlist[2]{left=0pt}

\usepackage[kerning=true]{microtype}	

\usepackage{tabto}	
\usepackage{xspace}	

\usepackage[sort&compress]{natbib}	

\bibpunct[, ]{[}{]}{,}{}{,}{,}
\setcitestyle{numbers,square}

\usepackage{hyperref}
\hypersetup{
final,
colorlinks=true,
linktocpage=true,
pdfstartview=FitH,
breaklinks=true,
pdfpagemode=UseNone,
pageanchor=true,
pdfpagemode=UseOutlines,
plainpages=false,
bookmarksnumbered,
bookmarksopen=false,
bookmarksopenlevel=1,
hypertexnames=true,
pdfhighlight=/O,
hyperfootnotes=false,
urlcolor=LinkColor,linkcolor=LinkColor,citecolor=LinkColor,	
pdftitle={},
pdfauthor={},
pdfsubject={},
pdfkeywords={},
pdfcreator={pdfLaTeX},
pdfproducer={LaTeX with hyperref}
}

\usepackage[sort&compress,capitalize,nameinlink]{cleveref}	
\crefname{algo}{Algorithm}{Algorithms}
\crefname{assumption}{Assumption}{Assumptions}
\crefname{case}{Case}{Cases}
\crefname{cond}{Condition}{Conditions}
\crefname{noref}{}{}
\crefname{problem}{Problem}{Problems}

\makeatletter	
\DeclareRobustCommand{\crefnosort}[1]{%
  \begingroup\@cref@sortfalse\cref{#1}\endgroup
}
\makeatother

\usepackage{algorithm}	
\usepackage{algpseudocode}	

\theoremstyle{plain}
\newtheorem{theorem}{Theorem}	
\newtheorem{corollary}{Corollary}	
\newtheorem{lemma}{Lemma}	
\newtheorem{proposition}{Proposition}	

\newtheorem*{theorem*}{Theorem}	
\newtheorem*{corollary*}{Corollary}	

\theoremstyle{definition}
\newtheorem{assumption}{Assumption}	
\newtheorem{definition}{Definition}	
\newtheorem{example}{\raisebox{\depth}{$\blacktriangleright$}~Example}	

\newtheorem*{assumption*}{Assumptions}	
\newtheorem*{definition*}{Definition}	
\newtheorem*{example*}{\raisebox{\depth}{$\blacktriangleright$}~Example}	

\theoremstyle{remark}
\newtheorem{remark}{Remark}	
\newtheorem{notation}{Notation}	

\newtheorem*{remark*}{Remark}	
\newtheorem*{notation*}{Notation}	

\def\endenv{\hfill\raisebox{\depth}{$\blacktriangleleft$}\smallskip}	

\newcounter{proofstep}

\numberwithin{remark}{section}	
\numberwithin{example}{section}	

\usepackage[showdeletions,suppress]{color-edits}	
\newcommand{\draft}[1]{#1}	

\newcommand{\define}[1]{\emph{\draft{#1}}}	
\newcommand{\revise}[1]{{\color{RevColor}#1}}	

\def\beginrev{\color{RevColor}}	
\def\endedit{}	
\def\endedit{\color{black}}	

\newcommand{\newmacro}[2]{\newcommand{#1}{\draft{#2}}}	
\newcommand{\newop}[2]{\DeclareMathOperator{#1}{\draft{#2}}}	
\newcommand{\newoplims}[2]{\DeclareMathOperator*{#1}{\draft{#2}}}	
\newcommand{\newsmartmacro}[2]{
	\NewDocumentCommand{#1}{
		E{_}{{}}
	}{
		\draft{#2_{##1}}
	}
}

\newcommand{\eps}{\varepsilon}	
\newcommand{\wilde}{\widetilde}	

\DeclarePairedDelimiter{\braces}{\{}{\}}	
\DeclarePairedDelimiter{\bracks}{[}{]}	
\DeclarePairedDelimiter{\parens}{(}{)}	
\DeclarePairedDelimiter{\of}{(}{)}	

\DeclarePairedDelimiter{\abs}{\lvert}{\rvert}	
\DeclarePairedDelimiter{\ceil}{\lceil}{\rceil}	
\DeclarePairedDelimiter{\floor}{\lfloor}{\rfloor}	

\DeclarePairedDelimiter{\setof}{\{}{\}}	
\DeclarePairedDelimiterX{\setdef}[2]{\{}{\}}{#1:#2}	
\DeclarePairedDelimiterXPP{\exclude}[1]{\mathopen{}\setminus}{\{}{\}}{}{#1}

\DeclarePairedDelimiterX{\braket}[2]{\langle}{\rangle}{#1,#2}	
\DeclarePairedDelimiterX{\internalInner}[1]{\langle}{\rangle}{#1}
\usepackage{xparse}
\NewDocumentCommand \inner {s m g}{
	\IfBooleanTF{#1}
	{
	\internalInner*{#2\IfValueT{#3}{,#3}}
	}
	{
	\internalInner{#2\IfValueT{#3}{,#3}}
	}
}

\DeclarePairedDelimiter{\norm}{\lVert}{\rVert}	
\DeclarePairedDelimiterXPP{\dnorm}[1]{}{\lVert}{\rVert}{_{\ast}}{#1}	
\DeclarePairedDelimiterXPP{\onenorm}[1]{}{\lVert}{\rVert}{_{1}}{#1}	
\DeclarePairedDelimiterXPP{\twonorm}[1]{}{\lVert}{\rVert}{_{2}}{#1}	
\DeclarePairedDelimiterXPP{\supnorm}[1]{}{\lVert}{\rVert}{_{\infty}}{#1}	

\newcommand{\defeq}{\coloneqq}	

\newcommand{\from}{\colon}	

\newmacro{\nat}{i}	
\newmacro{\natA}{i}	
\newmacro{\natB}{j}	
\newmacro{\natC}{k}	
\newmacro{\nats}{\mathbb{N}}	
\newmacro{\N}{\nats}	

\newmacro{\integer}{a}	
\newmacro{\intA}{a}	
\newmacro{\intB}{b}	
\newmacro{\intC}{c}	
\newmacro{\integers}{\mathbb{Z}}	
\newmacro{\Z}{\integers}	

\newmacro{\rational}{r}	
\newmacro{\ratA}{r}	
\newmacro{\ratB}{s}	
\newmacro{\ratC}{t}	
\newmacro{\rationals}{\mathbb{Q}}	
\newmacro{\Q}{\rationals}	

\newmacro{\real}{x}	
\newmacro{\realA}{x}	
\newmacro{\realB}{y}	
\newmacro{\realC}{z}	
\newmacro{\reals}{\mathbb{R}}	
\newmacro{\R}{\reals}	

\newmacro{\complex}{z}	
\newmacro{\complexA}{z}	
\newmacro{\complexB}{w}	
\newmacro{\complexC}{z}	
\newmacro{\complexes}{\mathbb{C}}	
\newmacro{\C}{\mathbb{C}}	

\newoplims{\argmax}{arg\,max}	
\newoplims{\argmin}{arg\,min}	
\newoplims{\intersect}{\bigcap}	
\newoplims{\union}{\bigcup}	

\newop{\aff}{aff}	
\newop{\bd}{bd}	
\newop{\bigoh}{\mathcal{O}}	
\newop{\card}{card}	
\newop{\cl}{cl}	
\newop{\conv}{conv}	
\newop{\clconv}{\overline{conv}}	
\newop{\crit}{crit}	
\newop{\diag}{diag}	
\newop{\diam}{diam}	
\newop{\dist}{dist}	
\newop{\dom}{dom}	
\newop{\eig}{eig}	
\newop{\ess}{ess}	
\newop{\Hess}{Hess}	
\newop{\ind}{ind}	
\newop{\im}{im}	
\newop{\intr}{int}	
\newop{\Jac}{Jac}	
\newop{\one}{\mathds{1}}	
\newop{\proj}{pr}	
\newop{\prox}{prox}	
\newop{\rank}{rank}	
\newop{\relint}{ri}	
\newop{\sign}{sgn}	
\newop{\supp}{supp}	
\newop{\Sym}{Sym}	
\newop{\tr}{tr}	
\newop{\unif}{unif}	
\newop{\vol}{vol}	

\newcommand{\cf}{cf.\xspace}	
\newcommand{\eg}{e.g.,\xspace}	
\newcommand{\ie}{i.e.,\xspace}	
\newcommand{\viz}{viz.\xspace}	

\newcommand{\textpar}[1]{\textup(#1\textup)}	

\newcommand{\txs}{\textstyle}	
\newcommand{\eqdot}{\,.}	

\newcommand{\alt}[1]{#1'}	
\newcommand{\altalt}[1]{#1''}	

\newmacro{\ball}{\mathbb{B}}	
\newmacro{\clball}{\overline{\mathbb{B}}}	
\newmacro{\sphere}{\mathbb{S}}	

\newmacro{\argdot}{\kern.5pt\boldsymbol{\cdot}\kern.5pt}	
\newmacro{\dd}{\:d}	
\newmacro{\ddt}{\frac{d}{dt}}	
\newmacro{\del}{\partial}	
\newcommand{\contdiff}[1]{\mathcal{C}^{#1}}	

\newcommand{\insum}{\sum\nolimits}	

\newmacro{\const}{c}	
\newmacro{\constA}{a}	
\newmacro{\constB}{b}	
\newmacro{\Const}{C}	

\newmacro{\param}{\theta}	
\newmacro{\params}{\Theta}	

\newmacro{\coef}{\alpha}	
\newmacro{\coefA}{\lambda}	
\newmacro{\coefB}{\mu}	
\newmacro{\coefC}{\nu}	

\newmacro{\expA}{p}	
\newmacro{\expB}{q}	
\newmacro{\expC}{r}	

\newmacro{\precs}{\eps}		
\newmacro{\precsalt}{{\eps'}}	

\newmacro{\asympteq}{\asymp}

\newmacro{\func}{g} 

\newmacro{\realspace}{\R^{\vdim}}	
\newmacro{\vecspace}{\mathcal{X}}	
\newmacro{\dspace}{\R^{\vdim}}	
\newmacro{\subspace}{\mathcal{W}}	

\newmacro{\coord}{i}	
\newmacro{\coordA}{i}	
\newmacro{\coordB}{j}	
\newmacro{\coordC}{k}	
\newmacro{\nCoords}{d}	
\newmacro{\dims}{\nCoords}	
\newmacro{\vdim}{\nCoords}	

\newmacro{\bvec}{e}	
\newmacro{\uvec}{u}	
\newmacro{\bvecs}{\mathcal{E}}	

\newmacro{\point}{x}	
\newmacro{\pointA}{\point}	
\newmacro{\pointB}{\point'}	
\newmacro{\pointalt}{\pointB}
\newmacro{\pointC}{\point''}	
\newmacro{\pointaltalt}{\pointC}
\newmacro{\points}{\mathcal{X}}	
\newmacro{\intpoints}{\relint\points}	

\newmacro{\base}{p}	
\newmacro{\baseA}{q}	
\newmacro{\baseB}{q}	
\newmacro{\baseC}{u}	

\newmacro{\set}{\mathcal{K}}	
\newmacro{\setA}{\set}	
\newmacro{\setB}{\alt\set}	
\newmacro{\setC}{\altalt\set}	

\newmacro{\idx}{i}
\newmacro{\idxalt}{j}
\newmacro{\idxaltalt}{l}
\newmacro{\idxaltaltalt}{k}
\newmacro{\indices}{I}
\newmacro{\indicesalt}{J}

\newmacro{\closed}{\mathcal{C}}	
\newmacro{\cpt}{\mathcal{K}}	
\newmacro{\cptalt}{\alt\cpt}	
\newmacro{\nhd}{\mathcal{U}}	
\newmacro{\nhdalt}{\W}	
\newmacro{\nhdaltalt}{\V}	
\newmacro{\nbd}{\nhd}
\newmacro{\nbdalt}{\nhdalt}
\newmacro{\nbdaltalt}{\nhdaltalt}
\newmacro{\U}{\mathcal{U}}	
\newmacro{\V}{\mathcal{V}}	
\newmacro{\W}{\mathcal{W}}	
\newmacro{\open}{\mathcal{U}}	
\newmacro{\openA}{\mathcal{U}}	
\newmacro{\openB}{\mathcal{V}}	

\newmacro{\mfld}{\mathcal{M}}	
\newmacro{\gmat}{g}	
\newmacro{\gdist}{\dist_{\gmat}}	

\newmacro{\tvec}{z}	
\newmacro{\form}{\omega}	

\newmacro{\radius}{r}
\newmacro{\Radius}{R}
\newmacro{\Radiusalt}{\widetilde{\Radius}}
\newmacro{\radiusalt}{\alt\radius}
\newmacro{\margin}{\delta}
\newmacro{\marginalt}{\alt\margin}
\newmacro{\Margin}{\Delta}
\newmacro{\connectedcomp}{\mathcal{K}}
\newmacro{\plainset}{S} 
\newmacro{\interv}{A} 
\newmacro{\domain}{D}
\newmacro{\bigcpt}{\mathcal{D}}
\newsmartmacro{\bigcptalt}{\alt\bigcpt}
\newsmartmacro{\bigcptaltalt}{\alt\alt\bigcpt}

\newmacro{\cvx}{\mathcal{C}}	
\newmacro{\subd}{\partial}	

\newop{\tspace}{T}	
\newop{\tcone}{TC}	
\newop{\dcone}{\tcone^{\ast}}	
\newop{\ncone}{NC}	
\newop{\pcone}{PC}	
\newop{\hull}{\Delta}	

\newop{\minimize}{minimize}	
\newop{\Opt}{Opt}	
\newop{\Sol}{Sol}	
\newop{\gap}{Gap}	
\newop{\orcl}{\mathsf{G}}	
\newop{\err}{\mathsf{Z}}	

\newmacro{\obj}{f}	
\newmacro{\objalt}{g}	
\newmacro{\objA}{f}	
\newmacro{\objB}{g}	
\newmacro{\sobj}{F}	

\newcommand{\sol}[1][\point]{#1^{\ast}}	

\newmacro{\gvec}{g}	
\newmacro{\gbound}{G}	

\newmacro{\oper}{A}	
\newmacro{\vecfield}{v}	
\newmacro{\vbound}{V}	

\newmacro{\lips}{L}	
\newmacro{\strong}{\alpha}	
\newmacro{\smooth}{\beta}	
\newmacro{\tmplips}{L}	
\newmacro{\tmpbound}{B}	
\newmacro{\growth}{M}	
\newmacro{\regparam}{\lambda}	

\newop{\ex}{\mathbb{E}}	
\newop{\prob}{\mathbb{P}}	
\newop{\probalt}{\mathbb{Q}}	
\newop{\var}{\mathbb{V}}	
\newop{\cov}{cov}	
\newop{\simplex}{\hull}	

\providecommand\given{}	

\DeclarePairedDelimiterXPP{\exof}[1]{\ex}{[}{]}{}{
\renewcommand\given{\nonscript\,\delimsize\vert\nonscript\,\mathopen{}} #1}

\DeclarePairedDelimiterXPP{\exwrt}[2]{\ex_{#1}}{[}{]}{}{
\renewcommand\given{\nonscript\:\delimsize\vert\nonscript\:\mathopen{}} #2}

\DeclarePairedDelimiterXPP{\probof}[1]{\prob}{(}{)}{}{
\renewcommand\given{\nonscript\:\delimsize\vert\nonscript\:\mathopen{}} #1}

\DeclarePairedDelimiterXPP{\probwrt}[2]{\prob_{\!#1}}{(}{)}{}{
\renewcommand\given{\nonscript\:\delimsize\vert\nonscript\:\mathopen{}} #2}

\DeclarePairedDelimiterXPP{\oneof}[1]{\one}{\{}{\}}{}{#1}	

\DeclarePairedDelimiterXPP{\varof}[1]{\var}{[}{]}{}{
\renewcommand\given{\nonscript\,\delimsize\vert\nonscript\,\mathopen{}} #1}

\DeclarePairedDelimiterXPP{\covof}[1]{\cov}{(}{)}{}{
\renewcommand\given{\nonscript\,\delimsize\vert\nonscript\,\mathopen{}} #1}

\newmacro{\event}{E}	
\newmacro{\eventA}{E}	
\newmacro{\eventB}{H}	

\newmacro{\seed}{\theta}	
\newmacro{\seeds}{\Theta}	
\newmacro{\pdist}{P}	
\newmacro{\history}{\mathcal{H}}	

\newmacro{\sample}{\omega}	
\newmacro{\samples}{\Omega}	
\newmacro{\sspace}{\R^{m}}	

\newmacro{\filter}{\mathcal{F}}	
\newmacro{\probspace}{(\samples,\filter,\prob)}	

\newmacro{\mean}{\mu}	
\newmacro{\sdev}{\sigma}	
\newmacro{\variance}{\sdev^{2}}	
\newmacro{\variancealt}{s^2}

\newmacro{\covmat}{\Sigma}
\newmacro{\hessmat}{H}

\newmacro{\rv}{X}
\newmacro{\trv}{X_\Radius}
\newmacro{\gaussian}{\mathcal{N}}

\newmacro{\partition}{Z}

\newmacro{\nSamples}{n}	
\newmacro{\datapoint}{\xi}

\newmacro{\beforestart}{-1}	
\newmacro{\start}{0}	
\newmacro{\afterstart}{1}	
\newmacro{\running}{\start,\afterstart,\dotsc}	

\newmacro{\run}{n}	
\newmacro{\runA}{n}	
\newmacro{\runB}{k}	
\newmacro{\runC}{\ell}	
\newmacro{\nRuns}{N}	
\newmacro{\nRunsalt}{\alt\nRuns}	
\newmacro{\runs}{\mathcal{\nRuns}}	
\newmacro{\runalt}{\runB}	

\newmacro{\tstart}{0}	
\renewcommand{\time}{\draft{t}}	
\newmacro{\timeA}{t}	
\newmacro{\timeB}{s}	
\newmacro{\timealt}{\timeB}	
\newmacro{\timealtalt}{u}
\newmacro{\timeC}{\tau}	
\newmacro{\timeD}{\lambda}	
\newmacro{\horizon}{T}	
\newmacro{\horizonalt}{S}	

\newcommand{\new}[1][\state]{#1^{+}}	

\newmacro{\seq}{a}	
\newmacro{\seqA}{a}	
\newmacro{\seqB}{b}	
\newmacro{\seqC}{c}	

\newmacro{\state}{x}	
\newsmartmacro{\accstate}{x^{\step}}	
\newmacro{\stateA}{x}	
\newmacro{\stateB}{z}	
\newmacro{\statealt}{\stateB}
\newmacro{\stateC}{y}	
\newmacro{\stateD}{p}	
\newmacro{\statealtalt}{\stateC}
\newmacro{\statealtaltalt}{\stateD}

\newmacro{\cstate}{X}
\newmacro{\cstateA}{X}
\newmacro{\cstateB}{Z}
\newmacro{\cstatealt}{\cstateB}

\newmacro{\startingpoint}{\point_\start}	

\newcommand{\init}[1][\state]{\draft{#1}_{\start}}	
\newcommand{\afterinit}[1][\state]{\draft{#1}_{\afterstart}}	

\newcommand{\iter}[1][\state]{\draft{#1}_{\runB}}	

\newcommand{\curr}[1][\state]{\draft{#1_{\run}}}	
\renewcommand{\next}[1][\state]{\draft{#1}_{\run+1}}	

\newmacro{\mat}{M}	
\newmacro{\hmat}{H}	

\newmacro{\ones}{\mathbf{1}}	
\newmacro{\eye}{I}	
\newmacro{\identity}{\eye}	
\newmacro{\zer}{\mathbf{0}}	

\newcommand{\mg}{\succ}	
\newcommand{\mleq}{\preccurlyeq}	

\newmacro{\eigval}{\lambda}	
\newcommand{\orthproj}{P}	

\newop{\Nash}{NE}	
\newop{\CE}{CE}	
\newop{\CCE}{CCE}	
\newop{\NI}{NI}	

\newop{\brep}{br}	
\newop{\reg}{Reg}	
\newop{\preg}{\overline{Reg}}	
\newop{\val}{val}	

\newmacro{\play}{i}	
\newmacro{\playA}{i}	
\newmacro{\playB}{j}	
\newmacro{\playC}{k}	
\newmacro{\nPlayers}{N}	
\newmacro{\players}{\mathcal{\nPlayers}}	

\newmacro{\pure}{\alpha}	
\newmacro{\pureA}{\alpha}	
\newmacro{\pureB}{\beta}	
\newmacro{\pureC}{\gamma}	
\newmacro{\nPures}{A}	
\newmacro{\pures}{\mathcal{\nPures}}	

\newmacro{\strat}{x}	
\newmacro{\stratA}{x}	
\newmacro{\stratB}{\stratA'}	
\newmacro{\stratC}{\stratA''}	
\newmacro{\strats}{\mathcal{X}}	
\newmacro{\intstrats}{\strats^{\circle}}	

\newmacro{\pay}{u}	
\newmacro{\payv}{v}	
\newmacro{\payfield}{v}	
\newmacro{\loss}{\ell}	

\newmacro{\game}{\mathcal{G}}	
\newmacro{\gameFull}{\game(\players,\points,\pay)}	

\newmacro{\fingame}{\Gamma}	
\newmacro{\fingameFull}{\Gamma(\players,\pures,\pay)}	

\newmacro{\minmax}{L}	
\newmacro{\minvar}{\point_{1}}	
\newmacro{\minvarA}{\point_{1}}	
\newmacro{\minvarB}{\minvarA'}	
\newmacro{\minvars}{\points_{1}}	
\newmacro{\maxvar}{\point_{2}}	
\newmacro{\maxvarA}{\point_{2}}	
\newmacro{\maxvarB}{\maxvarA'}	
\newmacro{\maxvars}{\points_{2}}	

\newmacro{\pot}{U}	

\newmacro{\gradientbound}{16 \dims \log 6}	

\newmacro{\hreg}{h}	
\newmacro{\breg}{D}	
\newmacro{\mprox}{P}	
\newmacro{\mirror}{Q}	
\newmacro{\fench}{F}	
\newmacro{\hstr}{K}	
\newmacro{\hrange}{H}	
\newmacro{\proxdom}{\points^{\hreg}}	

\DeclarePairedDelimiterXPP{\bregof}[2]{\breg}{(}{)}{}{#1,#2}	
\DeclarePairedDelimiterXPP{\proxof}[2]{\mprox_{#1}}{(}{)}{}{#2}	

\newmacro{\dpoint}{y}	
\newmacro{\dpointA}{y}	
\newmacro{\dpointB}{\dpointA'}	
\newmacro{\dpointC}{\dpointA''}	
\newmacro{\dpoints}{\mathcal{Y}}	

\newmacro{\dstate}{Y}	
\newmacro{\dvec}{w}	

\newmacro{\zone}{\mathbb{D}}	

\newop{\Eucl}{\Pi}	
\newop{\logit}{softmax}	
\newop{\dkl}{KL}	

\newmacro{\flowmap}{\Theta}	
\DeclarePairedDelimiterXPP{\flowof}[2]{\flowmap_{#1}}{(}{)}{}{#2}	
\DeclarePairedDelimiterXPP{\dotflowof}[2]{\dot\flowmap_{#1}}{(}{)}{}{#2}	

\newmacro{\traj}{x}	
\DeclarePairedDelimiterXPP{\trajof}[1]{\traj}{(}{)}{}{#1}	
\DeclarePairedDelimiterXPP{\difftrajof}[1]{\dot\traj}{(}{)}{}{#1}	

\newcommand{\est}[1]{\hat #1}	

\newmacro{\signal}{\est\gvec}	
\newmacro{\step}{\eta}	
\newmacro{\learn}{\eta}	
\newmacro{\tempinv}{\beta}	
\newmacro{\batch}{B}
\newmacro{\batchidx}{b}

\newmacro{\efftime}{\tau}	

\newmacro{\error}{Z}	

\newmacro{\noise}{\mathsf{U}}	
\newmacro{\snoise}{\xi}	
\newmacro{\noisepar}{\sdev}	
\newmacro{\noisevar}{\variance}	
\newmacro{\aggnoise}{\mathrm{\uppercase\expandafter{\romannumeral1}}}	
\newmacro{\supnoise}{\aggnoise_{\infty}}	
\newmacro{\maxnoise}{\aggnoise^{\ast}}	

\newmacro{\bias}{b}	
\newmacro{\drift}{b}	
\newmacro{\bbound}{B}	
\newmacro{\sbias}{\chi}	
\newmacro{\aggbias}{\mathrm{\uppercase\expandafter{\romannumeral2}}}	
\newmacro{\supbias}{\aggbias_{\infty}}	
\newmacro{\maxbias}{\aggbias^{\ast}}	

\newmacro{\sbound}{M}	
\newmacro{\aggsecond}{\mathrm{\uppercase\expandafter{\romannumeral3}}}	
\newmacro{\supsecond}{\aggsecond_{\infty}}	
\newmacro{\maxsecond}{\aggsecond^{\ast}}	

\newmacro{\mix}{\delta}	
\newmacro{\unitvec}{w}	
\newmacro{\unitvar}{W}	
\newmacro{\perturb}{z}	

\newmacro{\purequery}{\est\pure}	
\newmacro{\query}{\est\state}	
\newmacro{\pivot}{\point}	
\newmacro{\querypoint}{\est\point}	

\newmacro{\vertex}{v}	
\newmacro{\vertexA}{v}	
\newmacro{\vertexB}{w}	
\newmacro{\vertexC}{u}	
\newmacro{\nVertices}{V}	
\newmacro{\vertices}{\mathcal{V}}	

\newmacro{\edge}{e}	
\newmacro{\edgeA}{e}	
\newmacro{\edgeB}{\edgeA'}	
\newmacro{\edgeC}{\edgeA''}	
\newmacro{\nEdges}{E}	
\newmacro{\edges}{\mathcal{\nEdges}}	

\newmacro{\graph}{\mathcal{G}}	
\newmacro{\graphFull}{\graph(\vertices,\edges)}	
\newmacro{\adjmat}{A}	
\newmacro{\wmat}{W}	

\newmacro{\tree}{T}
\newmacro{\treealt}{\alt\tree}
\newmacro{\trees}{\mathcal{T}}
\newmacro{\treesalt}{\wilde\trees}

\newmacro{\mgf}{M}	
\DeclarePairedDelimiterXPP{\mgfof}[2]{\mgf_{#1}}{(}{)}{}{#2}	

\newmacro{\cgf}{K}	
\DeclarePairedDelimiterXPP{\cgfof}[2]{\cgf_{#1}}{(}{)}{}{#2}	

\newmacro{\ham}{\mathcal{H}}	
\DeclarePairedDelimiterXPP{\hamof}[2]{\ham_{#1}}{(}{)}{}{#2}	

\newmacro{\lag}{\mathcal{L}}	
\DeclarePairedDelimiterXPP{\lagof}[2]{\lag_{#1}}{(}{)}{}{#2}	

\newmacro{\mom}{p}
\newmacro{\pos}{q}
\newmacro{\vel}{v}
\newmacro{\velalt}{w}

\newmacro{\hamilt}{\mathcal{H}}
\newmacro{\hamiltalt}{\bar\hamilt}

\newmacro{\lagrangian}{\mathcal{L}}
\newmacro{\lagrangianalt}{\bar\lagrangian}

\newmacro{\curve}{\gamma}	
\DeclarePairedDelimiterXPP{\curveat}[1]{\curve}{(}{)}{}{#1}	
\DeclarePairedDelimiterXPP{\diffcurveat}[1]{\dot\curve}{(}{)}{}{#1}	

\newmacro{\curves}{\Gamma}
\DeclarePairedDelimiterXPP{\curvesat}[3]{\curves_{#1}}{(}{)}{}{#2;#3}	

\newmacro{\contcurves}{\contfuncs}
\DeclarePairedDelimiterXPP{\contcurvesat}[2]{\contcurves_{#1}}{(}{)}{}{#2}	

\newmacro{\lint}{\cstate}	
\DeclarePairedDelimiterXPP{\lintat}[1]{\lint}{(}{)}{}{#1}	

\newmacro{\qpot}{B}	
\newmacro{\qmat}{B}	
\newmacro{\energy}{E}	

\newmacro{\action}{\mathcal{S}}
\newmacro{\act}{\mathcal{S}}
\DeclarePairedDelimiterXPP{\actof}[2]{\act_{#1}}{[}{]}{}{#2}	

\newmacro{\pth}{\curve}
\newmacro{\pthalt}{\varphi}
\newmacro{\pths}{\curves}
\newmacro{\dpth}{\xi}
\newmacro{\dpthalt}{\zeta}
\newmacro{\daction}{\mathcal{A}}
\newmacro{\quasipot}{V}
\newmacro{\dquasipot}{B}
\newmacro{\symdquasipot}{C}
\newmacro{\dquasipotalt}{\widetilde{\dquasipot}}
\newmacro{\invpot}{W}
\newmacro{\dinvpot}{\energy}
\newmacro{\logmgf}{H}
\newmacro{\contfuncs}{\mathcal{C}}
\newmacro{\map}{F}
\newmacro{\rate}{\rho}
\newmacro{\level}{s}

\newmacro{\eqcl}{\mathcal{K}}
\newmacro{\neqcl}{\nComps}
\newmacro{\primvar}{\alpha}
\newmacro{\bdprimvar}{\primvar_\infty}
\newmacro{\bdvar}{\noisepar_{\infty}^{2}}
\newmacro{\exponent}{s}
\newmacro{\bdpot}{\pot_\infty}
\newmacro{\potgbound}{C}

\newmacro{\ratenhd}{\mathcal{N}}

\newmacro{\diffcurr}{\delta \curr}

\newcommand{\exittime}{\tau}
\newcommand{\hittime}{\sigma}

\newmacro{\comp}{\mathcal{K}}
\newmacro{\iComp}{i}
\newmacro{\jComp}{j}
\newmacro{\kComp}{k}
\newmacro{\nComps}{K}

\newmacro{\iGround}{0}
\newmacro{\ground}{\comp_{\iGround}}
\newmacro{\groundstates}{\sol[\indices]}
\newmacro{\asymptstables}{AS}
\newmacro{\nontrivialattract}{NTA}

\addauthor[Pan]{PM}{MediumBlue}

\newmacro{\meas}{\mu}	
\newmacro{\altmeas}{\nu}	
\newmacro{\occmeas}{\meas}	
\newmacro{\borel}{\mathcal{B}}	
\newmacro{\size}{\delta}	
\newmacro{\toler}{\eps}	

\addauthor[Waïss]{WA}{DarkGreen}

\newcommand{\half}[1][1]{{\frac{#1}{2}}}
\renewcommand{\WAdelete}[1]{}

\newmacro{\fn}{f}	
\newmacro{\basin}{\mathcal{B}}	
\newcommand{\invmeas}[1][\infty]{\draft{\meas_{#1}}}	
\newmacro{\measalt}{\widetilde{\mu}}	
\newmacro{\grad}{\nabla}
\newmacro{\errorterm}{E}
\newmacro{\errortermalt}{\widetilde{E}}

\addauthor[Franck]{FI}{Purple}

\addauthor[Jérôme]{JM}{DarkRed}

\begin{document}

 
\title
[The Long-Run Distribution of Stochastic Gradient Descent]	
{What is the Long-Run Distribution of Stochastic\\
Gradient Descent? A Large Deviations Analysis}	

\author
[W.~Azizian]
{Waïss Azizian$^{c,\ast}$}
\address{$^{c}$\,%
Corresponding author.}
\address{$^{\ast}$\,%
Univ. Grenoble Alpes, CNRS, Inria, Grenoble INP, LJK, 38000 Grenoble, France.}
\email[Corresponding author]{waiss.azizian@univ-grenoble-alpes.fr}
\author
[F.~Iutzeler]
{Franck Iutzeler$^{\sharp}$}
\address{$^{\sharp}$\,%
Institut de Mathématiques de Toulouse,  Université de Toulouse,  CNRS, UPS, 31062, Toulouse, France.}
\email{franck.iutzeler@math.univ-toulouse.fr}
\author
[J.~Malick]
{\\Jérôme Malick$^{\ast}$}
\email{jerome.malick@cnrs.fr}
\author
[P.~Mertikopoulos]
{Panayotis Mertikopoulos$^{\diamond}$}
\address{$^{\diamond}$\,%
Univ. Grenoble Alpes, CNRS, Inria, Grenoble INP, LIG, 38000 Grenoble, France.}
\email{panayotis.mertikopoulos@imag.fr}

\thanks{The authors are grateful to Pierre-Louis Cauvin for valuable feedback on an earlier version of this work.}



\newacro{LHS}{left-hand side}
\newacro{RHS}{right-hand side}
\newacro{iid}[i.i.d.]{independent and identically distributed}
\newacro{lsc}[l.s.c.]{lower semi-continuous}
\newacro{usc}[u.s.c.]{upper semi-continuous}
\newacro{rv}[r.v.]{random variable}
\newacro{wp1}[w.p.$1$]{with probability $1$}

\newacro{NE}{Nash equilibrium}
\newacroplural{NE}[NE]{Nash equilibria}

\newacro{GD}{gradient descent}
\newacro{KL}{Kullback\textendash Leibler}
\newacro{KurdL}[K\L]{Kurdyka\textendash\L ojasiewicz}
\newacro{LD}{large deviation}
\newacro{LDP}{large deviation principle}
\newacro{CGF}{cumulant-generating function}
\newacro{LSI}{logarithmic Sobolev inequality}
\newacro{MGF}{moment-generating function}
\newacro{ODE}{ordinary differential equation}
\newacro{SG}{stochastic gradient}
\newacro{SGD}{stochastic gradient descent}
\newacro{SDE}{stochastic differential equation}
\newacro{SME}{stochastic modified equation}
\newacro{SFO}{stochastic first-order oracle}
\newacro{SGO}{stochastic gradient oracle}
\newacro{SGLD}{stochastic gradient Langevin dynamics}
\newacro{SNR}{signal-to-noise ratio}
\newacro{ML}{machine learning}

\newacro{DMFT}{dynamic mean-field theory}

\begin{abstract}
%
%
In this paper, we examine the long-run distribution of \acf{SGD} in general, non-convex problems.
Specifically, we seek to understand which regions of the problem's state space are more likely to be visited by \ac{SGD}, and by how much.
Using an approach based on the theory of large deviations and randomly perturbed dynamical systems, we show that the long-run distribution of \ac{SGD} resembles the Boltzmann\textendash Gibbs distribution of equilibrium thermodynamics with temperature equal to the method's step-size and energy levels determined by the problem's objective and the statistics of the noise.
In particular, we show that, in the long run,
\begin{enumerate*}
[\upshape(\itshape a\hspace*{.5pt}\upshape)]
\item
the problem's critical region is visited exponentially more often than any non-critical region;
\item
the iterates of \ac{SGD} are exponentially concentrated around the problem's minimum energy state (which does not always coincide with the global minimum of the objective);
\item
all other connected components of critical points are visited with frequency that is exponentially proportional to their energy level;
and, finally
\item
any component of local maximizers or saddle points is ``dominated'' by a component of local minimizers which is visited exponentially more often.
\end{enumerate*}
\end{abstract}

\allowdisplaybreaks	
\acresetall	
\acused{iid}
\acused{LHS}
\acused{RHS}

\maketitle

\section{Introduction}
\label{sec:introduction}

Even though \ac{SGD} has been around for more than 70 years \cite{RM51}, it is still the method of choice for training a wide array of modern machine learning architectures \textendash\ from large language models to reinforcement learning and recommender systems.
This phenomenal success is largely owed to the method's simplicity:
given a smooth function $\obj\from\realspace\to\R$ and the associated optimization problem
\begin{equation}
\label{eq:opt}
\tag{Opt}
\min\nolimits_{\point\in\realspace}\;
	\obj(\point)
\end{equation}
the \ac{SGD} algorithm is given by the simple update rule
\begin{equation}
\label{eq:SGD}
\tag{SGD}
\next
	= \curr - \step \curr[\signal]
\end{equation}
where $\step>0$ is the method's step-size and $\curr[\signal]$, $\run=\running$ is a stochastic gradient of $\obj$ at $\curr$.

By virtue of its wide applicability, \eqref{eq:SGD} and its variants have been studied extensively in the literature, for both convex and non-convex objectives.
In the non-convex case (which is the most relevant setting for machine learning), the basic, no-frills guarantees of \eqref{eq:SGD} boil down to bounds of the form $\exof*{\insum_{\runalt=\start}^{\run} \norm{\nabla\obj(\iter)}^{2}} = \bigoh(\sqrt{\run})$ provided that $\step$ has been chosen accordingly \cite{Lan20}.
This guarantee suggests that the sequence $\curr$ eventually spends all but a vanishing fraction of time near regions where $\nabla\obj$ is small,
but it does not answer \emph{where} \eqref{eq:SGD} ultimately settles down.
In particular, the following crucial question remains open:
\vspace{1ex}
\begin{quote}
\centering
\itshape
Which critical points of $\obj$ \textpar{or components thereof} are more likely\\
to be observed in the long run \textendash\ and by how much?
\end{quote}
\vspace{1ex}


\begin{figure}[tbp]
\centering
\footnotesize
\includegraphics[height=38ex]{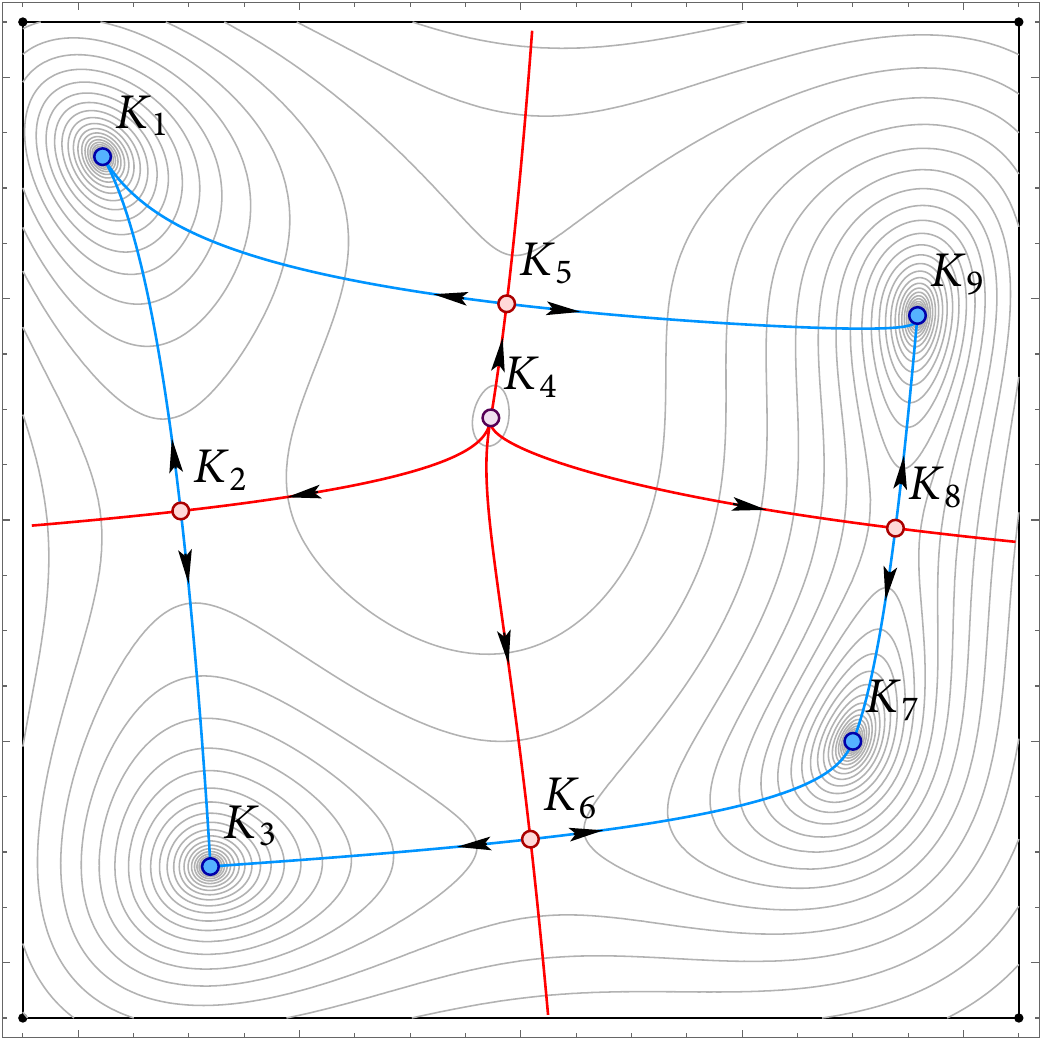}\\[1ex]
\includegraphics[height=42ex]{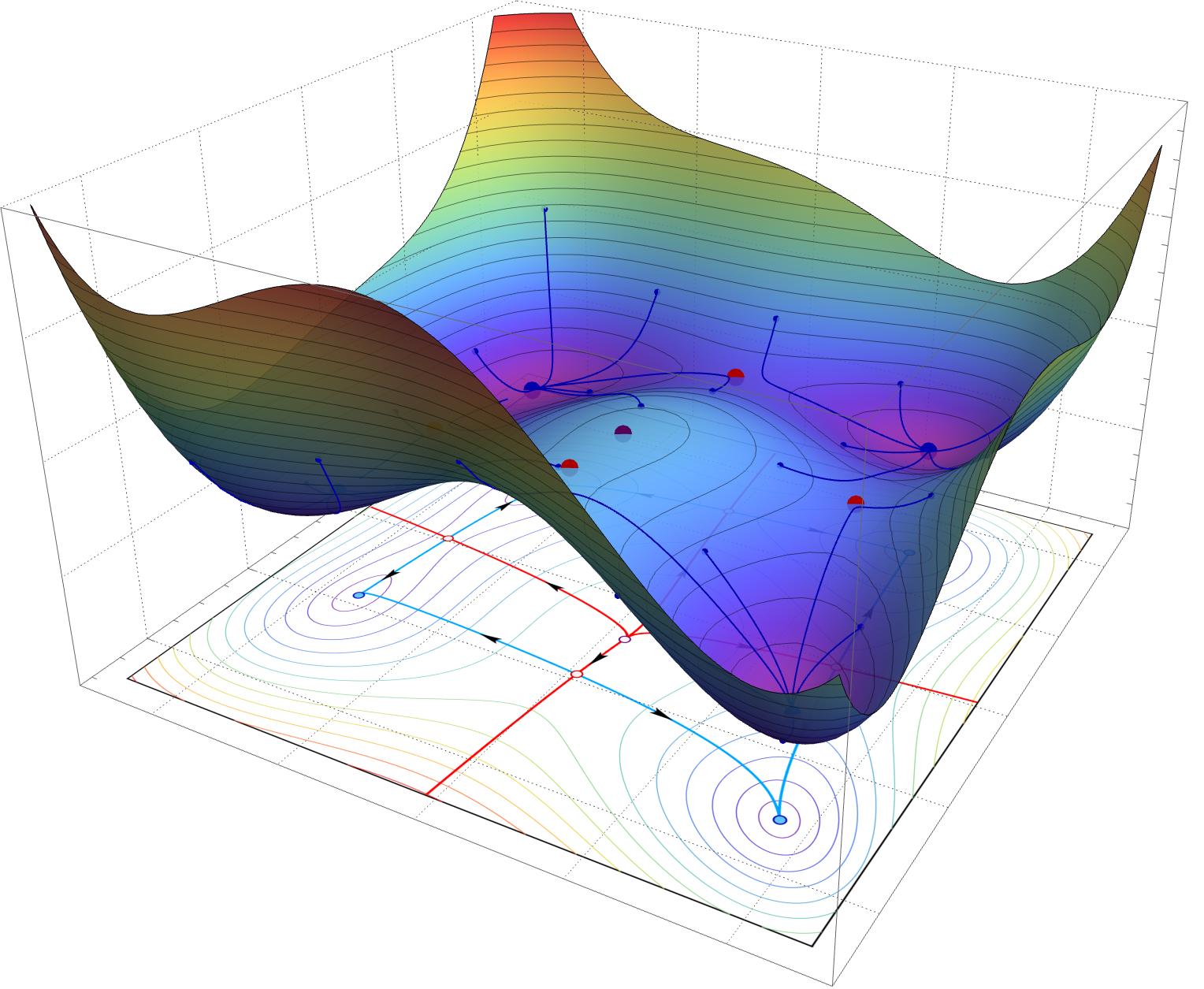}
\hfill
\includegraphics[height=42ex]{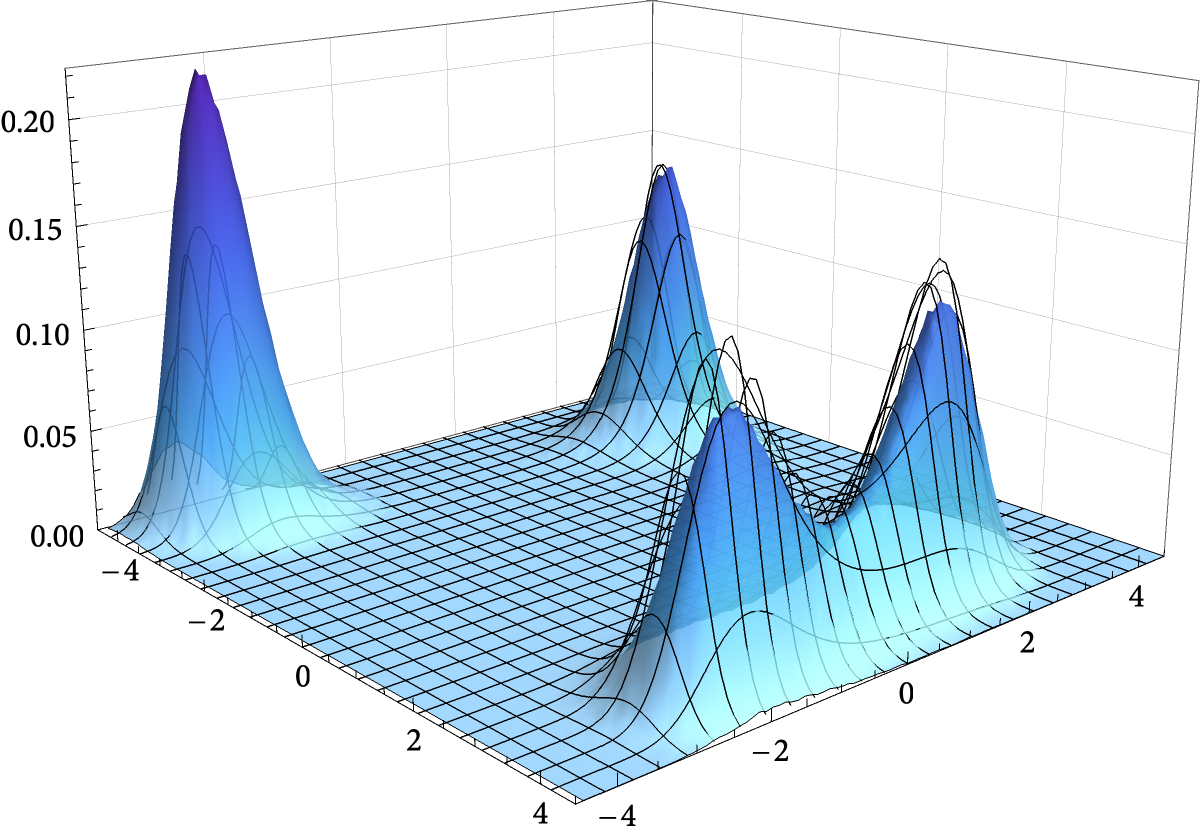}
\caption{Graphical illustration of \crefrange{thm:Gibbs}{thm:ground} for the Himmelblau test function $\obj(\point_{1},\point_{2}) = (\point_{1}^{2} + \point_{2} - 11)^{2} + (\point_{1}+\point_{2}^{2} - 7)^{2}$. 
The figure to the left depicts the loss landscape of $\obj$ with several (deterministic) orbits of the gradient flow of $\obj$ superimposed for visual convenience.
The figure in the middle highlights the $9$ critical components of $\obj$ as well as the ``most likely'' transitions between them:
$\comp_{1}$, $\comp_{3}$, $\comp_{7}$ and $\comp_{9}$ are minimizers (light blue),
$\comp_{5}$ is a global maximum (light purple),
and the rest are saddle points (light red).
The figure to the right illustrates the long-run distribution of $1000$ samples of \eqref{eq:SGD} run with $\step = 0.01$ over a horizon of $2\times 10^{4}$ iterations:
the density landscape represents the observed distribution, while the superimposed wireframe indicates our theoretical prediction.}
\label{fig:benchmark}
\end{figure}


This question is notoriously difficult because the loss landscape of $\obj$ can be exceedingly complicated \textendash\ especially in deep learning problems with hundreds of millions (or even billions) of parameters.
Starting with the negative, $\obj$ may contain a number of spurious saddle points that is exponentially larger than the number of local minima, and the function values associated with worst-case saddle points may be considerably worse than those associated with worst-case local minima \cite{DPGC+14}.
On the flip side, a more positive answer is provided by the literature on the avoidance of saddle-points where, under different assumptions for the method's step-size and the structure of $\obj$, it has been shown that the time spent by $\curr$ in the vicinity of strict saddle points is (vanishingly) small;
for a representative \textendash\ but, by necessity, incomplete \textendash\ list of results, \cf \cite{BH95,JGNK+17,GHJY15,Pem90,MHKC20,HMC21,AMPW22,HKKM23,yangFastConvergenceRandom2020} and references therein.

Now, even though the above justifies the informal mantra that ``\ac{SGD} avoids saddle points'', it does not answer \emph{which} critical regions of $\obj$ are most likely to be observed in the long run, and by how much.
This question has attracted significant interest in deep learning, but the matter remains poorly understood:
on the one hand, some works have shown that, in certain stylized deep net models, most local minimizers are concentrated in an exponentially narrow band of the problem's global minimum \cite{CHMB+15,Kaw16};
on the other hand, empirical studies suggest that, even in this case, the long-run distribution of \eqref{eq:SGD} may not be adequately captured by the shape of the problem's loss function \cite{LPA20}.

\subsection*{Our contributions}

Our goal in this paper is to quantify the long-run distribution of \eqref{eq:SGD} in the most general manner possible.
To do so, we take an approach based on the theory of large deviations \cite{DZ98} and randomly perturbed dynamical systems \cite{FW98,Kif88}, which enables us to estimate the probability of ``rare events'' (such as $\curr$ moving against the gradient flow of $\obj$ for a protracted period of time).
This allows us to characterize the events that occur with high probability and establish the following hierarchy of results (stated formally as \crefrange{thm:Gibbs}{thm:ground} in \cref{sec:analysis}):
\begin{enumerate}
\item
In the long run, the critical region of $\obj$ is visited exponentially more often than any non-critical region.
\item
The iterates of \eqref{eq:SGD} are concentrated with exponentially high probability in the vicinity of a region that minimizes a certain ``energy functional'' which depends on $\obj$ and the statistics of the noise in \eqref{eq:SGD}.
Importantly, the ground state of this functional \emph{does not} necessarily coincide with the global minimum of $\obj$.
\item
Among the remaining connected components of critical points, each component is visited with frequency which is exponentially proportional to its energy, according to the Boltzmann\textendash Gibbs distribution of statistical physics with temperature equal to the method's step-size.
\item
Every connected component of non-minimizing critical points of $\obj$ \textendash\ \ie local maximizers or saddle points \textendash\ is ``dominated'' by a component of local minimizers that is visited exponentially more often.
\end{enumerate}
Finally, we derive an explicit characterization of the invariant measure of \eqref{eq:SGD} under Gaussian noise and other noise models motivated by deep learning considerations.

Taken together, these properties resemble those of a canonical ensemble in statistical physics:
in a sense, each connected component of critical points can be seen as a ``state'' of a statistical ensemble, and the step-size of \eqref{eq:SGD} plays the role of the system's (fixed) temperature, which determines how easy it is to transit from one component to another.
We find this analogy particularly appealing as it provides a way of connecting ideas from equilibrium thermodynamics to the long-run behavior of \eqref{eq:SGD}.

\subsection*{Related work}

The main approaches used in the literature to examine the long-run distribution of \eqref{eq:SGD} hinge on the study of a limiting \ac{SDE}, typically associated with (a version of) the discrete Langevin dynamics or a diffusion approximation of \eqref{eq:SGD}.

Starting with the former, \citet{raginskyNonconvexLearningStochastic2017} examined the law of the \ac{SGLD}, a variant of \eqref{eq:SGD} with injected Gaussian noise of variance $2\step/\tempinv$ for some inverse temperature parameter $\tempinv>0$ (see also \cite{erdogduGlobalNonconvexOptimization2019,liSharpUniformintimeError2022} for some recent follow-ups in this direction).
\citet{raginskyNonconvexLearningStochastic2017} first showed that \ac{SGLD} closely tracks an associated diffusion process over finite time intervals;
the law of this diffusion was then shown to converge to the Gibbs measure $\exp(-\tempinv\obj) / \int\exp(-\tempinv\obj)$ at a geometric rate, fast enough to ensure the convergence of the discrete-time dynamics to the same measure.

\eqref{eq:SGD} can be recovered in the context of \ac{SGLD} by setting the inverse temperature parameter to $\tempinv \propto 1/\step$.
Unfortunately however, the convergence rate of the \ac{SDE} to its invariant distribution is exponential in $\tempinv$, so it is too slow to compensate for the discretization error in this case.
As a result, the bounds between the discrete dynamics and the invariant measure of the limiting diffusion become vacuous in the case of \eqref{eq:SGD} \textendash\ and similar considerations apply to the related work of \citet{majkaNonasymptoticBoundsSampling2019}.%
\footnote{In more detail, the error term in \citet[Eq.~(3.1)]{raginskyNonconvexLearningStochastic2017} can no longer be controlled if $\step$ is small.
Similarly, the constants in the geometric convergence rate guarantees of \citet[Theorems~2.1 and 2.5]{majkaNonasymptoticBoundsSampling2019} would degrade as $\exp(-\Omega(1/\step))$;
as a result, the associated discretization errors would be of the order of $\exp(\Omega(1/\step)))$, which cannot be controlled for small $\step$.}

Another potential approach to studying the long-run distribution of \eqref{eq:SGD} consists of approximating its trajectories via the solutions of a limiting \ac{SDE}.
A key contribution here was provided by the work of \citet{liStochasticModifiedEquations2017,liStochasticModifiedEquations2019} who showed that the tracking error between the iterates of \eqref{eq:SGD} and the solution of a certain \ac{SDE} becomes vanishingly small in the limit $\step\to0$ over any \emph{finite} time interval.
However, in contrast to the Langevin case, the convergence speed of the induced \ac{SME} to its invariant measure degrades exponentially as $\step\to0$ \citep{eberleReflectionCouplingsContraction2016a}, rendering this approach moot for a global description of the invariant measure of \eqref{eq:SGD}.
Though this strategy has been refined, either in the vicinity of global minimizers \citep{blancImplicitRegularizationDeep2020,liWhatHappensSGD2021} or in regions where $\obj$ is locally strongly convex \citep{fengUniforminTimeWeakError2019,liUniformintimeDiffusionApproximation2022}, this diffusion approach still fails to capture the long run behavior of \eqref{eq:SGD} in general non-convex settings.

Nevertheless, the limiting \ac{SDE} still provides valuable insights into certain aspects of the dynamics of \eqref{eq:SGD}.
In particular, a fruitful thread in the literature \citep{huDiffusionApproximationNonconvex2018,xieDiffusionTheoryDeep2021,moriPowerLawEscapeRate,jastrzebskiThreeFactorsInfluencing2018} has sought to estimate the escape rates of the approximating diffusion from local minimizers through this approach.
Interestingly, these works use some elements of the continuous-time Freidlin\textendash Wentzell theory \citep{FW98},
which is also the point of departure of our paper.
That being said, even though these results demonstrate how the structure of the objective function and of the noise locally affect the dynamics of the \ac{SDE} in a basin, they provide no information on the long-run behavior of the (discrete-time) dynamics of \eqref{eq:SGD};
for a more technical discussion, see \cref{app:related}.

One last approach which has gained increased attention in the literature is that of  \citet{dieuleveutBridgingGapConstant2018} and \citet{yuAnalysisConstantStep2020} who study \eqref{eq:SGD} as a discrete-time Markov chain.
This allowed \cite{dieuleveutBridgingGapConstant2018,yuAnalysisConstantStep2020} to derive conditions under which \eqref{eq:SGD} is (geometrically) ergodic and, in this way, to quantify the bias of the invariant measure under global growth conditions, \ie the distance to the global minimum of $\obj$.
Building further on this perspective, \citet{gurbuzbalaban2021heavy,hodgkinson2021multiplicative} and \citet{pavasovic2023approximate} showed that, under general conditions, the asymptotic distribution of the iterates of \eqref{eq:SGD} is heavy-tailed;
however, these results only describe the distribution of \eqref{eq:SGD} near infinity, and they provide no information on which critical regions of $\obj$ are more likely to be observed.
Again, we provide some more details on this in \cref{app:related}.

\subsection*{Our approach and techniques}

The linchpin of our approach is the theory of large deviations of \citet{FW98} for Markov processes, originally developed for diffusion processes in continuous time, and subsequently extended to subsampling in discrete time by \citet{Kif88};
see also \cite{gulinskyLargeDeviationsDiscreteTime1993,deacostaLargeDeviationsMarkov2022}
and
\cite{dupuisStochasticApproximationsLarge1985,dupuisLargeDeviationsAnalysis1988} for applications to stochastic approximation.
However, the starting point of all these works is the study of continuous-time diffusions on closed manifolds;
as far as we are aware, our paper provides the first extension of the theory of \citeauthor{FW98} to discrete-time systems that evolve over unbounded domains, and with a general \textendash\ possibly discrete \textendash\ noise profile.

One of the key challenges that we need to overcome is that most of the potentials introduced in \citep{FW98,Kif88} become drastically less regular in our context;
we remedy this issue by refining the analysis and carefully studying the structure of the attractors of \eqref{eq:SGD}.
This allows us to salvage enough regularity and show that \eqref{eq:SGD} spends most of its time near its attractors (this is achieved by developing suitable tail-bounds for the time spent away from critical points) and, ultimately, to estimate the transition probabilities of \eqref{eq:SGD} between different connected components of critical points.
This involves a series of novel mathematical tools and techniques, which we detail in \cref{app:invmeas}.

\section{Preliminaries and blanket assumptions}
\label{sec:prelims}

\subsection{Standing assumptions}

In this section, we describe our assumptions for the objective function $\obj$ of \eqref{eq:opt} and the black-box oracle providing gradient information for \eqref{eq:SGD}.
We begin with the former, writing throughout $\vecspace \defeq \realspace$ for the domain of $\obj$.

\begin{assumption}
\label{asm:obj}
The objective function $\obj\from\vecspace\to\R$ satisfies the following conditions:
\begin{enumerate}
[left=\parindent,label=\upshape(\itshape\alph*\hspace*{.5pt}\upshape)]
\item
\define{Coercivity:}
\label[noref]{asm:obj-coer}
$\obj(\point)\to\infty$ as $\norm{\point}\to\infty$.
\item
\label[noref]{asm:obj-smooth}
\define{Smoothness:}
$\obj$ is $C^{2}$-differentiable and its gradient is $\smooth$-Lipschitz continuous, that is,
\begin{equation}
\label{eq:LG}
\tag{LG}
\norm{\nabla\obj(\pointB) - \nabla\obj(\pointA)}
	\leq \smooth \norm{\pointB - \pointA}
	\quad
	\text{for all $\pointA,\pointB\in\vecspace$}
	\eqdot
\end{equation}
\item
\label[noref]{asm:obj-crit}
\define{Critical set regularity:}
The critical set
\begin{equation}
\label{eq:crit}
\crit\obj
	\defeq \setdef{\point\in\vecspace}{\nabla\obj(\point) = 0}
\end{equation}
of $\obj$ consists of a finite number of smoothly connected components $\comp_{\iComp}$, $\iComp=1,\dotsc,\nComps$.
\end{enumerate}
\end{assumption}

These requirements are fairly standard in the literature:
\cref{asm:obj}\cref{asm:obj-coer} guarantees that \eqref{eq:opt} admits a solution and rules out infima at infinity (such as $\obj(\point) = \cramped{e^{-\point^{2}}}$);
\cref{asm:obj}\cref{asm:obj-smooth} is a bare-bones regularity requirement for the analysis of gradient methods;
and, finally,
\cref{asm:obj}\cref{asm:obj-crit}
serves to exclude objectives with
anomalous critical sets (\eg exhibiting kinks or other non-smooth features), so it is also quite mild from an operational standpoint.%
\footnote{This last requirement can be replaced by positing for example that $\obj$ is definable in terms of some semi-algebraic / $o$-minimal structure, see \eg \citep{costeINTRODUCTIONOMINIMALGEOMETRY,vandendriesGeometricCategoriesOminimal1996} and \cref{rem:asm_def} in \cref{app:setup}.}

Regarding the gradient input to \eqref{eq:SGD}, we will assume throughout that the optimizer has access to a \acdef{SFO}, that is, a black-box mechanism returning a stochastic estimate of the gradient of $\obj$ at the point of interest.
Formally, when queried at $\point\in\vecspace$, an \ac{SFO} returns a random vector of the form
\begin{equation}
\label{eq:SFO}
\tag{SFO}
\orcl(\point;\sample)
	= \nabla\obj(\point)
		+ \err(\point;\sample)
\end{equation}
where
\begin{enumerate}
[left=\parindent,label=\upshape(\itshape\alph*\hspace*{.5pt}\upshape)]
\item
$\sample$ is a random seed drawn from
a compact subset $\samples$ of $\sspace$ based on some (complete) probability measure $\prob$.%
\footnote{The specific form of $\samples$ is not important;
in practice, random seeds from a target distribution are often generated by inverse transform sampling from $[0,1]^{m}$.}

\item
$\err(\point;\sample)$ is an umbrella error term capturing all sources of noise and randomness in the oracle.
\end{enumerate}
This oracle model is sufficiently flexible to account for all established versions of \eqref{eq:SGD} in the literature, including
minibatch \ac{SGD} (where $\sample$ represents the sampled minibatch),
noisy \acl{GD} (where the optimizer may artificially inject 
noise in the process to enhance 
convergence), 
and Langevin Monte Carlo methods.

With all this in mind, we make the following blanket assumptions for \eqref{eq:SFO}:
\begin{assumption}
\label{asm:noise}
The error term $\err\from\vecspace\times\samples\to\realspace$ of \eqref{eq:SFO} satisfies the following properties:
\begin{enumerate}
[left=\parindent,label=\upshape(\itshape\alph*\hspace*{.5pt}\upshape)]
\item
\label[noref]{asm:noise-zero}
\define{Properness:}
	$\exof{\err(\point;\sample)} = 0$ and $\covof{\err(\point;\sample)} \mg 0$ for all $\point\in\vecspace$.
\item
\label[noref]{asm:noise-growth}
\define{Smooth growth:}
	$\err(\point;\sample)$ is $C^{2}$-differentiable and satisfies the growth condition
\begin{equation}
\label{eq:smooth-growth}
\sup_{\point,\sample} \frac{\norm{\err(\point;\sample)}}{1 + \norm{\point}}
	< \infty
	\eqdot
\end{equation}
\item
\label[noref]{asm:noise-subG}
\define{Sub-Gaussian tails:}
The tails of $\err$ are bounded as
\begin{equation}
\label{eq:subG}
\log\exof*{e^{\inner{\mom, \err(\point;\sample)}}}
	\leq \frac{1}{2} \bdvar \norm{\mom}^{2}
\end{equation}
for some $\sdev_{\infty} > 0$ and all $\mom \in \realspace$.
\end{enumerate}
\end{assumption}

\cref{asm:noise}\cref{asm:noise-zero} is standard in the literature and ensures that the gradient noise in \eqref{eq:SGD} has zero mean and does not vanish identically at any $\point\in\vecspace$;
this requirement in particular plays a crucial role in several incarnations of noisy \acl{GD} that have been proposed to effectively escape saddle points of $\obj$ \cite{Pem90,BH95,JGNK+17,GHJY15}.
\cref{asm:noise}\cref{asm:noise-growth} is a bit more technical but otherwise simply serves to impose a limit on how large the noise may grow as $\norm{\point}\to\infty$.
Finally,
\cref{asm:noise}\cref{asm:noise-subG} is also widely used in the literature:
while not as general as the (possibly fat-tailed) finite variance assumption $\exof{\norm{\err(\point;\sample)}^{2}} \leq \bdvar$, it allows much finer control of the stochastic processes involved, leading in turn to more explicit and readily interpretable results.
We only note here that \cref{asm:noise}\cref{asm:noise-subG} can be relaxed further by allowing the variance proxy $\bdvar$ of $\err(\point;\sample)$ to depend on $\point$, possibly diverging to infinity as $\norm{\point}\to\infty$.
To streamline our presentation,
we defer the general case to the appendix.

Our last blanket requirement is a stability condition ensuring that the \acl{SNR} of \eqref{eq:SFO} does not become too small at infinity.
We formalize this as follows:

\begin{assumption}
\label{asm:SNR}
The \acl{SNR} of $\orcl$ is bounded as
\begin{equation}
\label{eq:SNR}
\liminf_{\norm{\point}\to\infty} \frac{\norm{\nabla\obj(\point)}^{2}}{\bdvar}
	> 16 \log6 \cdot \vdim
	\eqdot
\end{equation}
\end{assumption}

\cref{asm:SNR} is a technical requirement needed to establish a series of concentration bounds later on, and the specific value of the lower bound serves to facilitate some computations later on.
In practice, $\nabla\obj$ is often norm-coercive \textendash\ \ie $\norm{\nabla\obj(\point)} \to \infty$ as $\norm{\point} \to \infty$ \textendash\ so this assumption is quite mild.
Note also that \cref{asm:SNR}, as \cref{asm:noise}\cref{asm:noise-subG}, can be extended to the case where the variance proxy $\bdvar$ of $\orcl$ depends on $\point$, possibly blowing up at infinity;
we postpone the relevant details to \cref{app:subsec:setup}.

Putting together all of the above, the \ac{SGD} algorithm can be written in abstract recursive form as
\begin{equation}
\label{eq:SGD-abstract}
\new
	\gets \state - \step \orcl(\state;\sample)
	\eqdot
\end{equation}
Thus, given a (possibly random) initialization $\init\in\vecspace$ and an \ac{iid} sequence of random seeds $\curr[\sample]\in\samples$, $\run=\running$, the iterates $\curr$ of \eqref{eq:SGD} are obtained by taking $\curr[\signal] \gets \orcl(\curr;\curr[\sample])$ and iterating $\run\gets\run+1$ ad infinitum.
To streamline notation, we will write $\prob_{\!\init}$ for the law of $\curr$ starting at $\init$, and we will refer to it as the law of \eqref{eq:SGD}.

\subsection{Discussion of the assumptions}
\label{subsec:disc-assumptions}

To illustrate the generality of our assumptions, we briefly consider here the example of the regularized empirical risk minimization problem
\begin{equation}
\obj(\point)
    = \frac{1}{\nSamples} \sum_{\idx = 1}^{\nSamples} \loss(\state;\datapoint_{\idx})
    	+ \half[\regparam] \norm{\state}^{2}
\end{equation}
where
$\datapoint_{\idx}$, $\idx = 1,\dotsc,\nSamples$, are the training data of the model,
$\loss(\state;\datapoint)$ represents the loss of the model $\state$ on the data point $\datapoint$,
and
$\regparam>0$ is a regularization parameter.
In this case, we have
\begin{equation}
    \err(\state; \sample) = \grad \loss(\state;\datapoint_{\sample}) - \frac{1}{\nSamples} \sum_{\idx = 1}^{\nSamples} \grad \loss(\state;\datapoint_{\idx})
\end{equation}
where $\sample$ is sampled uniformly at random from $\{1,\dotsc,\nSamples\}$.

If $\loss$ is $C^2$-differentiable, Lipschitz continuous and smooth
\textendash\ see \eg \cite{MHKC20} and references therein \textendash\
\cref{asm:obj,asm:obj-coer,asm:obj-smooth} are satisfied automatically.
The error term $\err(\state; \sample)$ is also uniformly bounded so  \cref{asm:noise,asm:noise-growth,asm:noise-subG} are likewise verified (see \eg \citet[Ex.~2.4]{wainwright2019high}).
Finally, we have $\norm{\nabla\obj(\state)} = \bigoh(\regparam \norm{\state})$, so \cref{asm:SNR} also holds.
Thus, this setting covers two wide classes of examples:
linear models with non-convex losses \citep{erdogduGlobalNonconvexOptimization2019,yuAnalysisConstantStep2020} and smooth neural networks with normalization layers \citep{liReconcilingModernDeep2020}.%
\footnote{Our assumptions can also be linked to the notion of \define{dissipativity}, which is standard in Markov chain and sampling literature, see \eg \citep{erdogduGlobalNonconvexOptimization2019,liSharpUniformintimeError2022,raginskyNonconvexLearningStochastic2017,yuAnalysisConstantStep2020,LM24}.
Considering the gradient oracle obtained by sampling minibatches of size $\batch$, this setting fits into our framework with the relaxed version of \cref{asm:noise} in \cref{app:setup} provided $\batch$ is chosen large enough, see \cref{app:subsec:setup}.}

\section{Analysis and results}
\label{sec:analysis}

\subsection{Invariant measures}

The overarching objective of our paper is to understand the statistics of the limiting behavior of \eqref{eq:SGD}.
To that end, our point of departure will be the \define{mean occupation measure}
of $\curr$, defined here as
\begin{equation}
\label{eq:occmeas}
\curr[\occmeas](\borel)
	= \exof*{\frac{1}{\run} \sum_{\runalt=\start}^{\run-1} \oneof{\iter \in \borel}}
\end{equation}
for every Borel $\borel\subseteq\vecspace$.
In words, $\curr[\occmeas](\borel)$ simply measures the mean fraction of time that $\curr$ has spent in $\borel$ up to time $\run$, so the long-run statistics of \eqref{eq:SGD} can be quantified by the limiting distribution $\lim_{\run\to\infty} \curr[\occmeas]$.

If $\curr$ is ergodic, $\curr[\occmeas]$ converges weakly to some measure $\invmeas$, known as the \define{invariant measure} of $\curr$.%
\footnote{Weak convergence means here that $\lim_{\run\to\infty}\int \varphi \dd \curr[\occmeas]  = \int \varphi \dd \invmeas $ for every bounded continuous function $\varphi \from \vecspace \to \R$.}
Referring to the abstract representation \eqref{eq:SGD-abstract} of \eqref{eq:SGD}, ``invariance'' simply means here that $\invmeas$ satisfies the defining property
\begin{equation}
\label{eq:invariant}
\state\sim\invmeas
	\implies \new\sim\invmeas
\end{equation}
\ie $\invmeas$ is stationary under \eqref{eq:SGD}. Note however that ergodicity generally requires that the noise has a density part \citep{yuAnalysisConstantStep2020}.
More generally, even if $\curr$ is not ergodic, any (weak) limit point
of $\curr[\occmeas]$ must still satisfy the invariance property \eqref{eq:invariant} \citep{doucMarkovChains2018,hernandez-lermaMarkovChainsInvariant2003}, so this will be our principal figure of merit.%
\footnote{Existence always holds in our setting, \cf \cref{lem:inv_meas_exist} in \cref{app:lyapunov}.}

More precisely, our goal will be to quantify the long-run concentration of probability mass near the components $\comp_{\iComp}$ of $\crit\obj$:
in particular, since each $\comp_{\iComp}$ generically has Lebesgue measure zero, we will seek to estimate the probability mass $\invmeas(\nhd_{\iComp})$ where $\invmeas$ is invariant under \eqref{eq:SGD} and $\nhd_{\iComp}$ is a sufficiently small neighborhood of $\comp_{\iComp}$ \textendash\ typically a $\size$-neighborhood of the form $\nhd_{\iComp} \equiv \nhd_{\iComp}(\size) \defeq \setdef{\point\in\vecspace}{\dist(\comp_{\iComp},\point) < \size}$ in the limit $\size\to0$.
Doing this will allow us to determine the probability that \eqref{eq:SGD} is concentrated in the long run near one critical component or another,
as well as the degree of this concentration.

\subsection{A \acl{LDP} for \ac{SGD}}
\label{subsec:ldp}

Our strategy to achieve this will be to estimate the long-run rates of transition between different regions of $\vecspace$ under \eqref{eq:SGD}.
Our first step in this regard will be to establish a \acdef{LDP} for the process $\curr$, $\run=\running$, in the spirit of the general theory of \citet{FW98}.
This will in turn allow us to quantify the probability of ``rare events'' in \eqref{eq:SGD} \textendash\ \eg moving against the gradient flow of $\obj$ for a protracted period of time \textendash\ and it will play a crucial role in the sequel.

Now, since the statistics of \eqref{eq:SGD} are determined by those of \eqref{eq:SFO}, we begin by considering the \acdefp{CGF} of $\err$ and $\orcl$, \viz
\begin{subequations}
\label{eq:CGF}
\begin{align}
\cgfof{\err}{\point,\mom}
	&\defeq \log \exof{\exp(\braket{\mom}{\err(\point;\sample)})}
	\\[\smallskipamount]
\begin{split}
\cgfof{\orcl}{\point,\mom}
	&\defeq \log \exof{\exp(\braket{\mom}{\orcl(\point;\sample)})}
	\\
	&\hspace{2pt}= \cgfof{\err}{\point,\mom} + \braket{\nabla\obj(\point)}{\mom}
\end{split}
\end{align}
\end{subequations}
where $\point\in\vecspace$, $\mom\in\dspace$, and $\braket{\mom}{\vel}$ denotes the standard bilinear pairing between $\mom\in\dspace$ and $\vel\in\vecspace$.
To state our results, we will also require the associated \define{Lagrangians}
\begin{subequations}
\label{eq:Lagrange}
\begin{align}
\lagof{\err}{\point,\vel}
	&\defeq \cgf_{\err}^{\ast}(\point,-\vel)
	\\[\smallskipamount]
\lagof{\orcl}{\point,\vel}
	&\defeq \cgf_{\orcl}^{\ast}(\point,-\vel)
	= \lag_{\err}(\point,\vel+\nabla\obj(\point))
\end{align}
\end{subequations}
where ``$\ast$'' denotes convex conjugation with respect to $\mom$ in $\cgfof{\err}{\point,\mom}$ and $\cgfof{\orcl}{\point,\mom}$.

The importance of the Lagrangian functions \eqref{eq:Lagrange} is that they provide a \acl{LDP} for $\err$ and $\orcl$.
Namely, to leading order in $\run$ (and ignoring boundary effects), we have
\begin{equation}
\label{eq:LDP-point}
\probof*{\frac{1}{\run}\sum_{\runalt=\start}^{\run} \orcl(\point;\curr[\sample]) \in \borel}
	\sim  \exp\of*{-\run\inf_{\vel\in\borel}\lagof{\orcl}{\point,\vel}}
\end{equation}
for every Borel $\borel\in\dspace$ (and likewise for $\err$ and $\lag_{\err}$),
so the long-run statistics of the process $\curr[S] = \sum_{\runalt=\start}^{\run-1} \orcl(\point;\iter[\sample])$ are fully determined by $\lag_{\orcl}$ \cite{DZ98}.
In this regard, $\lag_{\orcl}$ plays the role of a ``rate function'' for $\curr[S]$
and quantifies the rate of occurrence of ``rare events'' in this context \cite{DZ98}.

Going back to \eqref{eq:SGD}, we have $\curr = \init - \step \sum_{\runalt=\start}^{\run-1} \orcl(\iter;\iter[\sample])$, so a promising way to understand the occupation measure of $\curr$ would be to try to derive a \acl{LDP} for $\curr$ starting from \eqref{eq:LDP-point}.
Unfortunately however, in contrast to $\curr[S] = \sum_{\runalt=\start}^{\run-1} \orcl(\point;\iter[\sample])$, this is not possible because \eqref{eq:LDP-point} concerns \ac{iid} samples drawn at a fixed point $\point\in\vecspace$, while the iterates of $\curr$ are highly auto-correlated.
Instead, inspired by the theory of \citet{FW98} for randomly perturbed dynamical systems, we will encode the \emph{entire} trajectory $\curr$ of \eqref{eq:SGD} as a point in some infinite-dimensional space of curves, and we will derive a \acl{LDP} for \eqref{eq:SGD} directly in this space.

To make this idea precise, we first require a continuous-time surrogate for the sequence of iterates of \eqref{eq:SGD}.
Concretely, writing $\curr[\efftime] = \run \step$ for the ``effective time'' that has elapsed up to the $\run$-th iteration of \eqref{eq:SGD}, we define the continuous-time interpolation of $\curr$ as the piecewise affine curve
\begin{equation}
\label{eq:interpolation}
\lintat{\time}
	= \curr + \frac{\time - \curr[\efftime]}{\step} (\next - \curr)
\end{equation}
for all $\run = \running$, and all $\time \in [\curr[\efftime],\next[\efftime]]$.
The resulting curve is continuous by construction, so, for embedding purposes, we will consider the ambient spaces of continuous curves truncated at some finite $\horizon\geq\tstart$:
\begin{subequations}
\begin{align}
\label{eq:contspaces-trunc}
\contcurves_{\horizon}
	&\defeq \contfuncs([\tstart,\horizon],\vecspace)
	\\
\contcurvesat{\horizon}{\point}
	&\defeq \setdef{\curve\in\contcurves_{\horizon}}{\curveat{\tstart} = \point}
	\\
\contcurvesat{\horizon}{\point,\pointalt}
	&\defeq \setdef{\curve\in\contcurves_{\horizon}}{\curveat{\tstart} = \point, \curveat{\horizon} = \pointalt}.
\end{align}
\end{subequations}

With these preliminaries in hand, and in analogy with the Lagrangian formulation of classical mechanics, we define the (normalized) ``\define{action functional}'' of $\lag_{\orcl}$ as
\begin{equation}
\label{eq:action}
\actof{\horizon}{\curve}
	= \int_{\tstart}^{\horizon} \lagof{\orcl}{\curveat{\time},\diffcurveat{\time}} \dd\time
\end{equation}
for all $\curve\in\contcurves_{\horizon}$ and with the convention $\actof{\horizon}{\curve} = \infty$ if $\curve$ is not absolutely continuous.
In a certain sense (to be made precise below), the functional $\actof{\horizon}{\curve}$ is a ``measure of likelihood'' for the curve $\curve$, with lower values indicating higher probabilities.
Accordingly, by leveraging the so-called ``least action principle'' \citep{DZ98,FW98,Kif90}, it is possible to establish the following \acl{LDP} for \eqref{eq:SGD}:

\begin{proposition}
\label{prop:LDP-interpolated}
Fix a time horizon $\horizon>\tstart$,
tolerance margins $\toler,\margin > 0$,
and
an action level $\level>0$.
In addition, write
\begin{equation}
\label{eq:curves-actbound}
\curvesat{\horizon}{\init}{\level}
	\defeq \setdef{\curve\in\contcurvesat{\horizon}{\init}}{\actof{\horizon}{\curve} \leq \level}
\end{equation}
for the space of continuous curves starting at $\init$ and with action at most $\level$.
Then, for all sufficiently small $\step$, we have
\begin{subequations}
\label{eq:LDP-interpolated}
\begin{align}
\begin{split}
\label{eq:LDP-interpolated-close}
\MoveEqLeft
\probwrt*{\init}{\sup_{\tstart\leq\time\leq\horizon} \!\norm{\lintat{\time} - \curveat{\time}} < \margin}
	\\
	&\geq \exp\of*{-\frac{\actof{\horizon}{\curve} + \toler}{\step}}
	\quad
	\text{for all $\curve\in\curvesat{\horizon}{\init}{\level}$}
\end{split}
\shortintertext{and, in addition,}
\begin{split}
\label{eq:LDP-interpolated-far}
\MoveEqLeft
\probwrt*{\init}{\sup_{\tstart\leq\time\leq\horizon} \!\norm{\lintat{\time} - \curveat{\time}} > \margin
	\text{ for all $\curve\in\curvesat{\horizon}{\init}{\level}$}}\!
	\\
	&\leq \exp\of*{-\frac{\level-\toler}{\step}}.
\end{split}
\end{align}
\end{subequations}
\end{proposition}

In words, \cref{prop:LDP-interpolated} states that
\begin{enumerate*}
[\upshape(\itshape\alph*\hspace*{.5pt}\upshape)]
\item
the linear interpolation $\lintat{\time}$ of $\curr$ stays close to low-action trajectories with probability that is exponentially large in their action value;
and
\item
the probability that $\lintat{\time}$ strays far from said trajectories is exponentially small in their action value.
\end{enumerate*}
This is, in fact, the first rung in a hierarchy of \aclp{LDP} that ultimately quantify the probability of rare events in \eqref{eq:SGD};
because these results are fairly technical to set up and prove (and not required for stating our main result), we defer the relevant discussion and proofs to \cref{app:LDP}.

\subsection{Transition costs and the quasi-potential}
\label{subsec:trans-cost}

Now, in view of the characterization \eqref{eq:LDP-interpolated-close} and \eqref{eq:LDP-interpolated-far} of ``rare trajectories'' of \eqref{eq:SGD}, we will seek to derive below an analogous characterization for the ``typical trajectories'' of \eqref{eq:SGD} in terms of $\action$.
To do this, we will associate a certain \define{transition cost} to each pair of components $\comp_{\iComp}$, $\comp_{\jComp}$ of $\crit\obj$, and we will use these costs to quantify how likely it is to observe $\curr$ near a component of critical points of $\obj$.

These costs are defined as follows:
First, following \citet{FW98}, define the \define{quasi-potential} between two points $\pointA,\pointB\in\vecspace$ as
\begin{equation}
\label{eq:qpot-point}
\qpot(\pointA,\pointB)
	\defeq \inf\setdef{\actof{\horizon}{\curve}}{\curve\in\contcurvesat{\horizon}{\pointA,\pointB}, \horizon\in\N}
\end{equation}
and the corresponding quasi-potential between two sets $\setA,\setB\subseteq\vecspace$ as
\begin{equation}
\label{eq:qpot-set}
\qpot(\setA,\setB)
	\defeq \inf\setdef{\qpot(\pointA,\pointB)}{\pointA\in\setA,\pointB\in\setB}.
\end{equation}
By construction, $\qpot(\pointA,\pointB)$ is the action value of the ``most probable'' path from $\pointA$ to $\pointB$, so it can be interpreted as the ``action cost'' of moving from $\pointA$ to $\pointB$.
Accordingly, to capture the difficulty of $\curr$ leaving the vicinity of a given component of $\crit\obj$ and wandering off to another, we will consider the \define{cost matrix}
\begin{equation}
\label{eq:cost-mat}
\qmat_{\iComp\jComp}
	\defeq \qpot(\comp_{\iComp},\comp_{\jComp})
\end{equation}
where $\comp_{\iComp},\comp_{\jComp}$, $\iComp,\jComp=1,\dotsc,\nComps$, are any two components of critical points of $\obj$.

As we mentioned before, the transitions between components of $\crit\obj$ play a crucial role in our analysis because this is where $\curr$ spends most of its time.
To characterize the structure of these transitions more precisely, it will be convenient to encode them in a complete weighted directed graph $\graph = (\vertices,\edges)$, which we call the \define{transition graph} of \eqref{eq:SGD}, and which is defined as follows:
\begin{enumerate}
[left=\parindent,label=\upshape(\itshape\alph*\hspace*{.5pt}\upshape)]
\item
The vertex set of $\graph$ is
$\vertices = \setof{1,\dotsc,\nComps}$,
\ie $\graph$ has one vertex per component of critical points of $\obj$.
\item
The edge set of $\graph$ is $\edges\!=\!\setdef{(\iComp,\jComp)}{\iComp,\jComp=1,\dotsc,\nComps, \iComp\!\neq\!\jComp}$, \ie $\graph$ has an edge per pair of components of $\crit\obj$.
\item
The weight of the directed edge $(\iComp,\jComp) \in \edges$ is $\qmat_{\iComp\jComp}$.
\end{enumerate}
To avoid degenerate cases, we will make the following assumption for the problem's cost matrix:
\begin{assumption}
\label{asm:costs}
$\qmat_{\iComp\jComp} < \infty$ for all $\iComp,\jComp=1,\dotsc,\nComps$.
\end{assumption}

This assumption is purely technical and mainly serves to streamline our presentation and avoid complicated statements involving non-communicating classes of the cost matrix $\qmat_{\iComp\jComp}$;
see \cref{app:estimates_trans_inv} for a more detailed discussion.

The last element we need for the statement of our results is the minimum total cost of reaching a component $\comp_{\iComp}$ of $\crit\obj$ from any starting point.
Since the most likely trajectories of \eqref{eq:SGD} are action minimizers, the ``path of least resistance'' to reach vertex $\iComp$ from vertex $\jComp$ may not follow the edge $(\iComp,\jComp)$ if the cost $\qmat_{\iComp\jComp}$ is too high;
instead, the relevant notion turns out to be
the \define{minimum weight spanning tree} pointing to $\iComp$.%
\footnote{We distinguish here between the notion of an \define{out-tree}
and that of an \define{in-tree}.
In an out-tree, edges point away from the root;
in an in-tree, edges point toward it.}
Formally, writing $\trees_{\iComp}$ for the set of spanning trees of $\graph$ that point to $\iComp$, we define the \emph{energy} of $\comp_{\iComp}$ as%
\begin{equation}
\label{eq:energy}
\energy_{\iComp}
	= \adjustlimits
		\min_{\tree_{\iComp}\in\trees_{\iComp}}
		\sum_{\jComp,\kComp\in\tree_{\iComp}} \qmat_{\jComp\kComp}.
\end{equation}
The terminology ``energy'' is explained below, where we show that, to leading order, the long-run distribution of \eqref{eq:SGD} around $\crit\obj$ follows the Boltzmann\textendash Gibbs distribution for a canonical ensemble with energy levels $\energy_{\iComp}$ at temperature $\step$.

\subsection{The long-run distribution of \ac{SGD}}

With all this in hand, we are finally in a position to state our main results for the statistics of the asymptotic behavior of \eqref{eq:SGD}.
We start by showing that, in the long run, the probability of observing the iterates of \eqref{eq:SGD} near a component of $\crit\obj$ is exponentially proportional to its energy.

\begin{theorem}
\label{thm:Gibbs}
Suppose that $\invmeas$ is invariant under \eqref{eq:SGD},
fix a tolerance level $\toler > 0$,
and
let $\nhd_{\iComp} \equiv \nhd_{\iComp}(\size)$, $\iComp = 1,\dotsc,\nComps$, be $\size$-neighborhoods of the components of $\crit\obj$.
Then, for all sufficiently small $\size,\step > 0$, we have
\begin{equation}
\label{eq:invmeas-Gibbs-log}
\abs*{\step\log\invmeas(\nhd_{\iComp})
	+ \energy_{\iComp} - \min\nolimits_{\jComp} \energy_{\jComp}}
	\leq \toler
\end{equation}
and
\begin{equation}
\label{eq:invmeas-Gibbs-diffs}
\abs*{\step\log\frac{\invmeas(\nhd_{\iComp})}{\invmeas(\nhd_{\jComp})}
	+ \energy_{\iComp} - \energy_{\jComp}}
	\leq \toler.
\end{equation}
More compactly, with notation as above, we have:
\begin{equation}
\label{eq:invmeas-Gibbs}
\invmeas(\nhd_{\iComp})
	\propto \exp\of*{-\frac{\energy_{\iComp} + \bigoh(\toler)}{\step}}.
\end{equation}
\end{theorem}

\Cref{thm:Gibbs} is the formal version of the statement that the long-run distribution of \eqref{eq:SGD} around the components of $\crit\obj$ follows an $\toler$-approximate Boltzmann\textendash Gibbs distribution with energy levels $\energy_{\iComp}$ at temperature $\step$ \citep{LL76}.
However,
since the critical set of $\obj$ includes both minimizing and \emph{non-minimizing} components,%
\footnote{To remove any ambiguity, a component $\comp$ of $\crit\obj$ is called (locally) minimizing if $\comp = \argmin_{\point\in\nhd} \obj(\point)$ for some neighborhood $\nhd$ of $\comp$;
otherwise, we say that $\comp$ is non-minimizing.}
a natural question that arises is whether the non-minimizing components of $\obj$ are selected against under $\invmeas$.
Our next result is a consequence of \cref{thm:Gibbs} and shows that this indeed the case:

\begin{theorem}
\label{thm:unstable}
Suppose that $\invmeas$ is invariant under \eqref{eq:SGD},
and let $\comp$ be a non-minimizing component of $\obj$.
Then, with notation as in \cref{thm:Gibbs}, there exists a minimizing component $\alt\comp$ of $\obj$ and a positive constant $\const \equiv \const(\comp,\alt\comp) > 0$ such that
\begin{equation}
\label{eq:invmeas-unstable}
\frac{\invmeas(\nhd)}{\invmeas(\alt\nhd)}
	\leq \exp\of*{-\frac{\const(\comp,\alt\comp) + \eps}{\step}}
\end{equation}
for all all sufficiently small $\step>0$ and all sufficiently small neighborhoods $\nhd$ and $\alt\nhd$ of $\comp$ and $\alt\comp$ respectively.
In particular, in the limit $\step\to0$, we have $\invmeas(\nhd)\to0$.
\end{theorem}

This avoidance principle is particularly important because it shows that \eqref{eq:SGD} is far less likely to be observed near a non-minimizing components of $\crit\obj$ relative to a minimizing one.
In this regard, \cref{thm:unstable} complements a broad range of avoidance results in the literature \cite{Pem90,GHJY15,JGNK+17,MHKC20,HMC21,HKKM23} without requiring any of the ``strict saddle'' assumptions that are standard in this context.

That being said, \cref{thm:Gibbs,thm:unstable} leave open the possibility that the energy landscape of \eqref{eq:SGD} contains \emph{non-critical} low-energy regions that nonetheless get a significant amount of probability under \eqref{eq:SGD};
put differently, \cref{thm:Gibbs,thm:unstable} do not rule out the eventuality that, in the long run, $\curr$ may still be observed with non-vanishing probability far from the critical region of $\obj$.
Our next result addresses precisely this issue and shows that this probability
is exponentially small.

\begin{theorem}
\label{thm:crit}
Suppose that $\invmeas$ is invariant under \eqref{eq:SGD},
and
let $\nhd \equiv \nhd(\size)$ be a $\size$-neighborhood of $\crit\obj$.
Then there exists a constant $\const \equiv \const_{\size} > 0$ such that, for all sufficiently small $\step > 0$, we have:
\begin{equation}
\label{eq:invmeas-crit}
\invmeas(\nhd)
	\geq 1 - e^{-\const/\step}.
\end{equation}
\end{theorem}

Taken together, \crefrange{thm:Gibbs}{thm:crit} show that, in the long run, the iterates of \eqref{eq:SGD} are exponentially more likely to be observed in the vicinity of $\crit{\obj}$ rather than far from it, and exponentially more likely to be observed near a minimum of $\obj$ rather than a saddle-point (or a local maximizer).

Our next result can be seen as joint consequence of \cref{thm:Gibbs,thm:crit} as it shows that the long-run distribution of \eqref{eq:SGD} is exponentially concentrated around the system's \define{ground state}
\begin{equation}
\label{eq:ground}
\txs
\ground
	= \union_{\iComp\in\argmin_{\jComp} \energy_{\jComp}} \comp_{\iComp}
\end{equation}
that is, the components of $\crit\obj$ with minimal energy.
The precise statement is as follows:

\begin{theorem}
\label{thm:ground}
Suppose that $\invmeas$ is invariant under \eqref{eq:SGD},
and
let $\nhd_{\iGround} \equiv \nhd_{\iGround}(\size)$ be a $\size$-neighborhood of the system's ground state $\ground$.
Then there exists a constant $\const \equiv \const_{\size} > 0$ such that, for all sufficiently small $\step>0$, we have:
\begin{equation}
\label{eq:invmeas-ground}
\invmeas(\nhd_{\iGround})
	\geq 1 - e^{-\const/\step}.
\end{equation}
\end{theorem}

In words, \crefrange{thm:Gibbs}{thm:ground} provide the following quantification of the limiting distribution of $\curr$:
in the long run
\begin{enumerate*}
[\upshape(\itshape a\hspace*{.5pt}\upshape)]
\item
the critical region of $\obj$ is visited exponentially more often than any non-critical region of $\obj$ (\cref{thm:crit});
\item
in particular, the iterates of \eqref{eq:SGD} are exponentially concentrated around the problem's ground state (\cref{thm:ground});
\item
among the mass that remains, every component of $\obj$ gets a fraction that is exponentially proportional to its energy (\cref{thm:Gibbs});
and, finally
\item
every non-minimizing component is ``dominated'' by a minimizing component that is visited exponentially more often (\cref{thm:unstable}).
\end{enumerate*}
Importantly, the problem's energy landscape is shaped by $\obj$, but not $\obj$ alone:
the statistics of \eqref{eq:SFO} play an equally important role, so we may have $\ground \neq \argmin\obj$;
we discuss this issue in detail in \cref{sec:applications}.

The proofs of the above results are quite lengthy and elaborate,
so we defer them to the appendix and only provide below a roadmap describing the overall proof strategy, the main technical challenges encountered, and the way they can be resolved.

\subsection{Outline of the proof}

As discussed in \Cref{subsec:ldp}, the first step of the proof consists in establishing a \acl{LDP} for \eqref{eq:SGD}.
The \ac{LDP} of \cref{prop:LDP-interpolated} for the interpolated process $\lintat{\time}$ is obtained as a consequence of \citep[Chap.\;7]{FW98} in \cref{app:ldp_interpolated}.
With this result in hand, our next step (which we carry out in \cref{app:ldp_discrete})
is to deduce an \ac{LDP} for the ``accelerated'' \ac{SGD} process $\curr[\accstate]$, $\run=\running$, defined here as
\begin{equation}
\label{eq:SGD-acc}
\curr[\accstate] = \state_{\run \floor{1/\step}}
\end{equation}
where $\floor{1/\step}$ denotes the integer part of $1/\step$.
As $\curr[\accstate]$ is a subsampled version of \eqref{eq:SGD}, it essentially shares the same long-run behavior and invariant measures and, in addition, it has a very important feature:
there are sufficiently many time steps between two of its iterates for an \ac{LDP} to hold in the specified subinterval.
[Intuitively, this is because it takes $\bigoh(1/\step)$ steps of \eqref{eq:SGD} to average out random fluctuations due to the noise.]
In view of the above, the rest of our proof focuses on the accelerated sequence $\curr[\accstate]$.

The main thrust of the proof is contained in \Cref{app:invmeas} and consists in adapting the powerful machinery of \cite{FW98,Kif88} to study the limiting behavior of $\curr[\accstate]$.
However, both \cite{FW98,Kif88} study continuous-time diffusion processes on closed manifolds, so there are some key challenges to overcome:
\begin{itemize}
\item
The unconstrained setting renders many elements of \cite{FW98,Kif88} inapplicable.
We remedy this by showing that the time that \eqref{eq:SGD} only spends a negligible amount of time away from $\crit\obj$.
\item
The generality of our assumptions on the noise makes the Lagrangians 
$\lag_{\orcl}$ and $\lag_{\err}$ non-smooth:
more precisely, they must have bounded domains, on which they may fail to be continuous.
As a consequence, most of the objects introduced in \citep{FW98,Kif88} become drastically less regular \textendash\ \eg $\dquasipot$ defined in \eqref{eq:qpot-point} \textendash\ which again renders their results inapplicable.
We remedy this issue by refining the analysis, carefully studying the structure of the attractors, and salvaging enough regularity to proceed.
\end{itemize}
The crux of the proof (\Cref{app:invmeas}) is structured as follows:
\begin{enumerate}
\item
In \cref{app:attractors}, we study the structure of the attractors of \eqref{eq:SGD}, as well as the regularity of the Lagrangians and the quasi-potential near these attractors.
\item
In \Cref{app:lyapunov}, we show that \eqref{eq:SGD} spends most of its time near its attractors by deriving a series of tail-bounds on the time spent away from $\crit\obj$.
These bounds are obtained even for unbounded variance proxy through the construction of a Lyapunov function from $\obj$ and $\bdvar$.
\item
In \cref{app:estimates_trans_inv}, we estimate the transition probabilities of the process between attractors:
if the iterates of \eqref{eq:SGD} are near $\comp_\iComp$, the next component of critical points that they visit is $\comp_\jComp$ with probability
\(
\exp({-\qmat_{\iComp\jComp}/{\step}}).
\)
As such, low-weight paths in the transition graph $\graph$ of \eqref{eq:SGD}
represent the high-probability transitions of \eqref{eq:SGD} between components.
We can then leverage the general theory of \citet{FW98} to obtain \cref{thm:Gibbs}.
\item
Finally, in \cref{app:conv_stab}, we analyze the properties of 
mini\-mi\-zing components to establish \cref{thm:unstable,thm:crit,thm:ground} .
\end{enumerate}

\begin{remark*}
We also note that, since weak limit points of the sequence of occupation measures $\curr[\occmeas]$ of \eqref{eq:SGD} are invariant in the sense of \eqref{eq:invariant}, the results of \cref{thm:Gibbs,thm:unstable,thm:crit,thm:ground} also apply to $\curr[\occmeas]$ if $\run$ is large enough.
We make this observation precise in \cref{app:main_occmeas}.
\hfill
\endenv
\end{remark*}

\section{Examples and applications}
\label{sec:applications}

In this last section, we explore the dependency of the transition costs and the energy levels on the parameters of the problem in certain special cases.

\subsection{Gaussian noise}
\label{sec:gaussian-noise}

We begin with the case of (truncated) Gaussian noise, where the problem's energy levels admit a particularly simple closed form.
To ease notation, we present here the computations in the case of Gaussian noise, and we defer the more intricate case of \emph{truncated} Gaussian noise to \cref{app:subsec:truncated_gaussian_noise}.

To that end, assume that the gradient error $\error(\point, \sample)$ in \eqref{eq:SGD} follows a centered Gaussian distribution with variance $\variance > 0$ for all $\point \in \vecspace$.
The Lagrangian and the action functional of the problem then become:
\begin{align}
\lagof{\orcl}{\point, \vel}
	&= \frac{\norm{\vel + \grad \obj(\point)}^2}{2 \variance}
\shortintertext{and}
\actof{\horizon}{\pth}
	&= \int_0^\horizon \frac{\norm{\diffcurveat{\time} + \grad \obj(\curveat{\time})}^2}{2 \variance} \dd \time
\end{align}
for all $\point,\vel\in\vecspace$ and all $\pth\in\contcurves_{\horizon}$.
This expression shows that $\actof{\horizon}{\pth}$ penalizes the deviation of $\pth$ from the gradient flow of $\obj$:
the closer $\diffcurveat{\time}$ is to $-\grad \obj(\curveat{\time})$, the smaller the action.
Then, for the reverse path $\pthalt(\time) = \pth(\horizon-\time)$, we get
\begin{align}
\actof{\horizon}{\pthalt}
	&= \actof{\horizon}{\pth}
		- 2 \bracks{\obj(\pth_\horizon) -  \obj(\pth_\tstart)}/\variance.
\end{align}
Note that if $\pth$ joins $\comp_\iComp$ to $\comp_\jComp$, then $\pthalt$ joins $\comp_\jComp$ to $\comp_\iComp$, so
\begin{equation}
\label{eq:energy-diff}
\qmat_{\jComp\iComp}
	\leq \qmat_{\iComp\jComp}
		- 2 \parens{\obj_\jComp - \obj_\iComp}/\variance
\end{equation}
where $\obj_\iComp$ denotes the value of $\obj$ on $\comp_{\iComp}$, $\iComp=1,\dotsc,\nComps$.
Exchanging $i$ and $j$, we get equality in \eqref{eq:energy-diff}, so the minimum in the definition \eqref{eq:energy} of $\energy_\iComp$ is reached for the same undirected tree, namely the minimum-weight spanning tree with symmetric weights $\qmat_{\iComp\jComp} + 2\obj_\iComp/\variance$ for all $\iComp,\jComp$.
As such, up to a constant, we get
\begin{equation}
\label{eq:energy-Gauss}
\energy_\idx
	= 2 \obj_\idx/\variance.
\end{equation}
Thus, invoking \cref{thm:Gibbs}, we conclude that the probabiity distribution of \eqref{eq:SGD} over $\crit(\obj)$ is governed by the Boltzmann\textendash Gibbs measure with energy levels given by \eqref{eq:energy-Gauss}

Similarly, in the {\em truncated} Gaussian case, we have:
\begin{proposition}
    \label{prop:truncated_gaussian_main}
    For any $\precs > 0$, if $\error(\point, \sample)$ follows a centered Gaussian distribution with variance $\variance > 0$ conditioned on being in a ball with large enough radius 
    depending on
    $\grad \obj(\point)$ and $\precs$, then, up to a constant,
    \begin{equation}
        \energy_\idx = 2 \obj_\idx/\variance + \bigoh \parens*{\precs}\quad
        \text{ for all }
        \idx  = 1,\dots,\nComps\,.
    \end{equation}
\end{proposition}
\Cref{prop:truncated_gaussian_main} is proven in \cref{app:subsec:truncated_gaussian_noise}, where we also allow $\noisepar$ to depend on $\point$ via $\obj(\point)$.

\subsection{Local dependencies}

 The modeling of the noise is crucial to the understanding of the dynamics of \eqref{eq:SGD}.
 This has been often underlined, especially in the exit-time literature 
\citep{huDiffusionApproximationNonconvex2018,xieDiffusionTheoryDeep2021,moriPowerLawEscapeRate,jastrzebskiThreeFactorsInfluencing2018}.
In particular, \citep{moriPowerLawEscapeRate} showed experimentally that, in deep learning models, the variance of the noise scales linearly with the objective function and also examined the effect of this observation on the exit time from a local minima.

We explain here how this model of the noise influences the invariant measure of \eqref{eq:SGD}.
To that end, following \citet{moriPowerLawEscapeRate}, assume there is a positive-definite matrix
$\sol[\hessmat]$ such that, locally near a  minimizing  $\comp_\iComp$, $\log\obj$ is separable in the eigenbasis of $\sol[\hessmat]$:%
\footnote{Following \citep{moriPowerLawEscapeRate}, $\sol[\hessmat]$ corresponds to the Hessian of $\obj$ at $\comp_\iComp$ for deep learning models.}
\begin{equation}
\log\obj(\point)
	= \sum_{\eigval \in \eig \sol[\hessmat]} \objalt^\eigval(\point_\eigval)
\end{equation}
where $\point_\eigval$ denotes the projection of $\point$ on the eigenspace of $\eigval$.
\begin{lemma}
Suppose that $\error(\point,\sample)$ satisfies
\begin{equation}
\label{eq:local_noise_main}
\cgf_{\err}(\point, \vel)
	\leq \frac{\variance\obj(\point)}{2} \inner{\vel, \sol[\hessmat] \vel}
	\quad
	\text{for all $\point$ near $\comp_{\iComp}$}
	\,.
\end{equation}
Then, for small enough $\margin > 0$ and all $\jComp\neq\iComp$, we have
\begin{equation}
\energy_{\jComp}
	\geq 2\min \setdef*
		{\sum_{\eigval \in \eig \sol[\hessmat]} \!\!\!\!\frac{\objalt^\eigval(\point_\eigval) - \objalt^\eigval_\iComp}{\eigval \variance }}
	{\point,\dist(\point, \comp_\iComp) = \margin}
\label{eq:local_energy_levels_main}
\end{equation}  
where $\objalt^\eigval_\iComp$ is the value of $\objalt^\eigval$ on $\eqcl_\iComp$.
\end{lemma}

This result shows that the energy of each component of $\crit(\obj)$ is lower bounded by the \ac{RHS} of \eqref{eq:local_energy_levels_main}.
This quantity scales as the reciprocal of the eigenvalues of $\sol[\hessmat]$ so, as the minimum becomes flatter (\ie the eigenvalues of $\sol[\hessmat]$ become smaller), the energy levels of all other components become larger:
thus, relative to component $\comp_\iComp$, the other components all become less probable.
Moreover, note that the \ac{RHS} of \eqref{eq:local_energy_levels_main} only scales logarithmically with the value of the objective function around $\comp_\iComp$, \ie the depth of the minimum:
this means that the ``flatness'' of the minimum plays a greater role in the relative probabilities of the components than the depth.

This also shows that deepest minima do not necessarily correspond to the ground state of the problem:
if $\variance$ or the eigenvalues of $\sol[\hessmat]$ is small enough compared to the noise level outside the $\margin$-neighborhood of $\comp_\iComp$, then $\comp_\iComp$ will be the ground state of the system even if it is not the deepest minimum;
we provide a formal proof of this in \cref{app:subsec:local_dependencies}.

\section{Concluding remarks}
\label{sec:discussion}

Our objective was to quantify the long-run distribution of \ac{SGD} in a general, non-convex setting.
As far as we are aware, our paper provides the first description of the invariant measure of \eqref{eq:SGD}, and in particular, its distribution over components of critical points. 
This distribution is governed by energy levels that depend both on the objective function and the statistics of the noise.
An important challenge that remains is to estimate these energy levels in different settings;
this would be a major step towards a better understanding of the generalization properties of \ac{SGD}.
Another is to consider constrained versions of \eqref{eq:SGD}, possibly striving to establish a link with mirror-type \acp{SDE} \cite{MerSta18,MerSta18b};
we defer these questions to future work.

\appendix
\numberwithin{equation}{subsection}	
\numberwithin{lemma}{section}	
\numberwithin{proposition}{section}	
\numberwithin{theorem}{section}	
\numberwithin{corollary}{section}	

\makeatletter	
\newcommand{\thmtag}[1]{	
  \let\oldthetheorem\thetheorem	
  \renewcommand{\thetheorem}{#1}	
  \g@addto@macro\endtheorem{	
    \addtocounter{theorem}{0}	
    \global\let\thetheorem\oldthetheorem}	
  }
\makeatother

\makeatletter	
\newcommand{\asmtag}[1]{	
  \let\oldtheassumption\theassumption	
  \renewcommand{\theassumption}{#1}	
  \g@addto@macro\endassumption{	
    \addtocounter{assumption}{0}	
    \global\let\theassumption\oldtheassumption}	
  }
\makeatother

\section{Further related work}
\label{app:related}

\subsection{Consequences of the diffusion approximation}
The SDE approximation of SGD, introduced by \citet{liStochasticModifiedEquations2017,liStochasticModifiedEquations2019}, has been a fruitful development in the understanding of some aspects of the dynamics of SGD.
For instance, \citet{ziyin2023law} provide explicit descriptions for the invariant measure of the diffusion approximation of SGD for diagonal linear neural networks.
Applications of this SDE approximation also include the study the dynamics of SGD close to manifold of minimizers \citep{blancImplicitRegularizationDeep2020,liWhatHappensSGD2021}.
\Citet{wojtowytsch2023stochastic} study the invariant measure of the diffusion approximation: if the set of global minimizers form a manifold on which the noise vanishes, they show that the invariant measure of the diffusion concentrates on this manifold and moreover provide a description of the limiting measure on this manifold.

Another line of works focuses on the case where the objective function is scale-invariant \citep{liReconcilingModernDeep2020} and how this impacts the convergence of the dynamics of SGD: \citet{wangThreestageEvolutionFast} describes the convergence of the SDE approximation with anisotropic constant noise to the Gibbs measure, while \citet{liFastMixingStochastic} shows that discrete-time SGD dynamics close to a manifold of minimizers enjoy fast convergence to an invariant measure.

Finally, \citet{mignacco2020dynamical,mignacco2022effective,veiga2024stochastic} leverage \ac{DMFT} to study the behavior of the diffusion approximation of SGD.
The \ac{DMFT}, or ``path-integral'' approach, comes from statistical physics and bears a close resemblance to the Freidlin-Wentzell theory of large deviations for SDEs.
However, this methodology, as well as \citet{mignacco2020dynamical,mignacco2022effective,veiga2024stochastic}, is restricted to the continuous-time diffusion and remains at heuristic level.

Let us underline two points comparing these works to ours.
While these works focus on the learning behavior of the algorithm by considering specific statistical models and specific losses, we focus on the optimization aspects and on covering general non-convex objectives.
Secondly, these results are either local or concern the asymptotic distribution of the continuous-time approximation of \eqref{eq:SGD}, which do no provide information of the asymptotic behaviour of the actual discrete-time dynamics.

\subsection{On the heavy-tail character of the asymptotic distribution of \ac{SGD}}
A recent line of work has focused on the heavy-tail character of the asymptotic distribution of \eqref{eq:SGD} \citep{gurbuzbalaban2021heavy,pavasovic2023approximate,hodgkinson2021multiplicative}.
These works show that, under some broad conditions, the stationary distribution of \eqref{eq:SGD} is heavy-tailed:
specifically, 
for some $\alpha>0$, the tails of the stationary distribution $\state_{\infty}$ of the iterates of \eqref{eq:SGD} decays as
\(
\probof{\norm{\state_{\infty}} \geq z}
	= \Theta(z^{-\alpha})
\)
or
\(
\probof{u^{\top} \state_{\infty} \geq z}
	= \Theta(z^{-\alpha})
\)
for all $u\in\realspace$.
As such, these results concern the probability of observing the iterates of SGD at very large distances from the origin.
This is in contrast with our work, which focuses on the distribution of the iterates of SGD near critical regions of the objective function.
These two types of results are thus orthogonal and complementary.

\subsection{\ac{SGD} with a vanishing step-size}

The long-run behavior of \eqref{eq:SGD} is markedly different when the method is run with a vanishing step-size $\curr[\step] > 0$ with $\lim_{\run\to\infty} \curr[\step] = 0$.
This was the original version of \eqref{eq:SGD} as proposed by \citet{RM51} \textendash\ and, in a slightly modified form by \citet{KW52} \textendash\ and, in contrast to the constant step-size case, the vanishing step-size algorithm converges \acl{wp1}.
The first almost sure convergence result of this type was obtained by \citet{Lju77} under the assumption that the iterates of \eqref{eq:SGD} remain bounded.
This boundedness assumption was dropped by \citet{Ben99} and \citet{BT00} who showed that \eqref{eq:SGD} with a vanishing step-size converges to a component of $\crit\obj$ as long as $\curr[\step]$ satisfies the Robbins\textendash Monro summability conditions $\sum_{\run}\curr[\step] = \infty$ and $\sum_{\run} \curr[\step]^{2} < \infty$;
for a series of related results under different assumptions, \cf \cite{ZMBB+20,MHKC20,MHC24} and references therein.
In all these cases, ergodicity of the process is lost (because of the vanishing step-size), so the limiting distribution of \eqref{eq:SGD} depends crucially on its initialization and other non-tail events.
We are not aware of any results quantifying the long-run distribution of \eqref{eq:SGD} with a vanishing step-size.

\section{Setup and Preliminaries}
\label{app:setup}

\subsection{Notation}

Throughout the sequel, we will write
$\inner{\argdot}{\argdot}$ for the standard inner product on $\vecspace \equiv \realspace$ and $\norm{\argdot}$ for the induced (Euclidean) norm.
To lighten notation, we will identify $\vecspace$ with its dual $\dspace \equiv \vecspace^{\ast}$, and we will not formally distinguish between primal and dual vectors (though the distinction should be clear from the context).
We will also write $\ball(\point, \radius)$ (resp.~$\clball(\point, \radius)$) for the open (resp.~closed) ball of radius $\radius$ centered at $\point\in\vecspace$, and we will respectively denote the open and closed $\margin$-neighborhoods of $\plainset\subseteq\in\vecspace$ as
\begin{subequations}
\begin{align}
\nbd_{\margin}(\plainset)
	&\defeq \setdef{\point \in \points}{\dist(\plainset,\point) < \margin}
	\\
\plainset_{\margin}
	\defeq \cl(\nbd_{\margin}(\plainset))
	&=\, \setdef{\point \in \points}{\dist(\plainset,\point) \leq \margin}
\end{align}
\end{subequations}

\newcounter{lastassumption}	
\setcounter{lastassumption}{\theassumption}	
\subsection{Setup and assumptions}
\label{app:subsec:setup}

Before we begin our proof, we revisit and discuss our standing assumptions.
In particular, to extend the range of our results, we provide in the rest of this appendix a weaker version of the blanket assumptions of \cref{sec:prelims}, which we label with an asterisk (``$\ast$'') and which will be in force throughout the appendix.

We begin with our assumptions for the objective function $\obj$ of \eqref{eq:opt}.

\smallbreak
\asmtag{\ref*{asm:obj}$^{\ast}$}
\begin{assumption}[Weaker version of \cref{asm:obj}]
\label{asm:obj-weak}
The objective function $\obj\from\vecspace\to\R$ satisfies the following conditions:
\begin{enumerate}
[left=\parindent,label=\upshape(\itshape\alph*\hspace*{.5pt}\upshape)]
\item
\define{Coercivity:}
\label[noref]{asm:obj-weak-coer}
$\obj(\point)\to\infty$ as $\norm{\point}\to\infty$.
\item
\label[noref]{asm:obj-weak-smooth}
\define{Smoothness:}
$\obj$ is $C^{2}$-differentiable and its gradient is $\smooth$-Lipschitz continuous, that is,
\begin{equation}
\tag{\ref{eq:LG}}
\norm{\nabla\obj(\pointB) - \nabla\obj(\pointA)}
	\leq \smooth \norm{\pointB - \pointA}
	\quad
	\text{for all $\pointA,\pointB\in\vecspace$}
	\eqdot
\end{equation}
\item
\label[noref]{asm:obj-weak-crit}
\define{Critical set regularity:}
The critical set
\begin{equation}
\tag{\ref{eq:crit}}
\crit\obj
	\defeq \setdef{\point\in\vecspace}{\nabla\obj(\point) = 0}
\end{equation}
of $\obj$ consists of a finite number of essentially smoothly connected components $\comp_{\iComp}$, $\iComp=1,\dotsc,\nComps$.
\end{enumerate}
\end{assumption}

The difference between \cref{asm:obj,asm:obj-weak} is that, in the latter, the connected components of $\crit\obj$ are only required to be \emph{essentially} smoothly connected.
Formally, this means that, for any connected component $\comp$ of $\crit\obj$, and for any two points $\point,\pointalt\in\comp$, there exists
\begin{enumerate*}
[\upshape(\itshape a\hspace*{.5pt}\upshape)]
\item
a continuous, almost everywhere differentiable curve $\curve\from[0,1]\to\comp$ such that $\curve(0) = \point$, $\curve(1) = \pointalt$;
and
\item
a partition $0 = \time_{0} < \dotsb < \time_{\run} < \dotsb < \time_{\nRuns} = 1$ of $[0,1]$ such that $\curve$ is integrable on every closed interval of $(\time_{\run-1},\time_{\run})$ for all $\run=1,\dotsc,\nRuns$.
\end{enumerate*}

\begin{remark}
\label{rem:asm_def}
The path-connectedness requirement of \cref{asm:obj}\cref{asm:obj-crit} is satisfied whenever the connected components of $\crit\obj$
are isolated critical points,
smooth manifolds,
or
finite unions of closed manifolds.
More generally, \cref{asm:obj-weak}\cref{asm:obj-weak-crit} is satisfied whenever $\obj$ is definable \textendash\ in which case $\crit\obj$ is also definable, so each component can be connected by piecewise smooth paths \citep{costeINTRODUCTIONOMINIMALGEOMETRY,vandendriesGeometricCategoriesOminimal1996}.
The relaxation provided by \cref{asm:obj-weak}\cref{asm:obj-weak-crit} represents the ``minimal'' set of hypotheses that are required for our analysis to go through.
\endenv
\end{remark}

Moving forward, to align our notation with standard conventions in large deviations theory, it will be more convenient to work with $-\err(\point; \sample)$ instead of $\err(\point; \sample)$ in our proofs.
To make this clear, we restate below \cref{asm:noise} in terms of the noise process
\begin{equation}
\label{eq:noise}
\noise(\point,\sample)
	= -\err(\point; \sample)
\end{equation}
and we make use of this opportunity to relax the definition of the variance proxy of $\noise$.

\smallbreak
\asmtag{\ref*{asm:noise}$^{\ast}$}
\begin{assumption}[Weaker version of \cref{asm:noise}]
\label{asm:noise-weak}
The noise term $\noise\from\vecspace\times\samples\to\realspace$  satisfies the following properties:
\begin{enumerate}
[left=\parindent,label=\upshape(\itshape\alph*\hspace*{.5pt}\upshape)]
\item
\label[noref]{asm:noise-weak-zero}
\define{Properness:}
	$\exof{\noise(\point,\sample)} = 0$ and $\covof{\noise(\point,\sample)} \mg 0$ for all $\point\in\vecspace$.
\item
\label[noref]{asm:noise-weak-growth}
\define{Smooth growth:}
	$\noise(\point,\sample)$ is $C^{2}$-differentiable and satisfies the growth condition
\begin{equation}
\tag{\ref{eq:smooth-growth}}
\sup_{\point,\sample} \frac{\norm{\noise(\point,\sample)}}{1 + \norm{\point}}
	< \infty
	\eqdot
\end{equation}
\item
\label[noref]{asm:noise-weak-subG}
\label{assumption:subgaussian}
\define{Sub-Gaussian tails:}
The \acl{CGF} of $\noise$ is bounded as
\begin{equation}
\tag{\ref{eq:subG}$^{\ast}$}
\log \exof*{e^{\inner{\mom,\noise(\point,\sample)}}}
	\leq \frac{1}{2} \bdvar(\obj(\point)) \, \norm{\mom}^{2}
	\quad
	\text{for all $\mom\in\realspace$},
\end{equation}
where
$\noisepar\from\R\to(0,\infty)$ is continuous,
bounded away from zero ($\inf \bdvar > 0$),
and
grows at infinity as $\bdvar(\obj(\point)) = \Theta(\norm{\point}^{\exponent})$ for some $\exponent\in[0,2]$, \ie
\begin{equation}
0
	< \liminf_{\norm{\point}\to\infty} \frac{\bdvar(\obj(\point))}{\norm{\point}^{\exponent}}
	\leq \limsup_{\norm{\point}\to\infty} \frac{\bdvar(\obj(\point))}{\norm{\point}^{\exponent}}
	< \infty
	\eqdot
\end{equation}

\end{enumerate}
\end{assumption}

\begin{remark}
The difference between \cref{asm:noise,asm:noise-weak} lies in the requirements for the variance proxy $\bdvar$ of the noise in \eqref{eq:SGD}.
Because $\obj$ is coercive, the dependence of $\bdvar$ on $\obj(\point)$ allows for noise processes with an unbounded variance as $\norm{\point}\to\infty$;
the specific functional dependence through $\obj(\point)$ instead of $\point$ is a modeling choice which we make because it greatly simplifies the proof and our calculations.
\endenv
\end{remark}

In this more general setting, we augment \cref{asm:SNR} with a technical condition that is satisfied trivially when $\bdvar$ is constant.

\smallbreak
\asmtag{\ref*{asm:SNR}$^{\ast}$}
\begin{assumption}[Weaker version of \cref{asm:SNR}]
\label{asm:SNR-weak}
\label{assumption:coercivity_noise}
The \acl{SNR} of $\orcl$ is bounded as
\begin{equation}
\label{eq:SNR-weak}
\tag{\ref{eq:SNR}$^{\ast}$}
\liminf_{\norm{\point}\to\infty} \frac{\norm{\nabla\obj(\point)}^{2}}{\bdvar(\obj(\point))}
	> 16 \log6 \cdot \vdim
	\eqdot
\end{equation}
\end{assumption}

\begin{example}[Example of \cref{subsec:disc-assumptions}, redux]
	In \cref{subsec:disc-assumptions}, we discussed the regularized empirical risk minimization problem and mentioned that, under a dissipativity condition and with a large enough batch size, it fits our framework.
We now provide more details.

Consider the objective $\obj$ given by
\begin{equation}
	\obj(\point)
	=
	\frac{1}{\nSamples} 
	\sum_{\idx = 1}^{\nSamples} \loss(\point;\datapoint_{\idx})
	+
	\half[\regparam]
	\norm{\point}^{2}
\end{equation}
where $\loss(\point;\datapoint)$ represents the loss of the model $\point$ on the data point $\datapoint$,
$\datapoint_{\idx}$, $\idx = 1,\dotsc,\nSamples$, are the training data and $\regparam$ is the (positive) regularization parameter.

Let us assume that $\loss$ is non-negative, $C^2$-differentiable, $\smooth$-Lipschitz smooth and that the resulting objective $\obj$ is dissipative: there are $\alpha, \beta > 0$ such that
\begin{equation}
	\inner{\grad \obj(\point)}{\point} \geq \alpha \norm{\point}^{2} - \beta
	\eqdot
\end{equation}
As is usually the case in practice, we consider the \ac{SFO} obtained by sampling mini-batches of size $\batch$.
The noise term $\noise(\point, \sample)$ is then given by
\begin{equation}
	\noise(\point, \sample)
	=
	\frac{1}{\batch}
	\sum_{\batchidx = 1}^{\batch} \grad \loss(\point;\datapoint_{\sample_\batchidx})
	-
	\frac{1}{\nSamples}
	\sum_{\idx = 1}^{\nSamples} \grad \loss(\point;\datapoint_{\idx})
\end{equation}
with $\sample = (\sample_1, \dotsc, \sample_\batch)$ representing $\batch$ indices from $\{1, \dotsc, \nSamples\}$.

All the terms
\begin{equation}\label{eq:app-setup-finite-sum-noise}
	\grad \loss(\point;\datapoint_{\sample_\batchidx}) - \frac{1}{\nSamples} \sum_{\idx = 1}^{\nSamples} \grad \loss(\point;\datapoint_{\idx})
\end{equation}
are uniformly bounded by $\bigoh \parens*{\norm{\point}}$ by smoothness of $\loss$.
In particular, this implies that \cref{asm:noise-weak}\cref{asm:noise-weak-growth} is satisfied.

Moreover, we also obtain that all the terms of the form \cref{eq:app-setup-finite-sum-noise} are $\bigoh \parens*{\norm{\point}^2}$-sub-Gaussian, and therefore, by independence we obtain that $\noise(\point, \sample)$ is $\bigoh \parens*{\frac{1}{\batch} \norm{\point}^2}$-sub-Gaussian.

Since $\loss$ is non-negative, $\obj$ is lower-bounded by $\Omega(\norm{\point}^2)$.
It gives that $\noise(\point, \sample)$ is actually $\bigoh \parens*{\frac{1}{\batch} \obj(\point)}$-sub-Gaussian and therefore \cref{asm:noise-weak}\cref{asm:noise-weak-growth} is satisfied with $\bdvar(t) \propto \frac{t}{\batch}$.

We now show that \cref{asm:SNR-weak} can be satisfied by choosing $\batch$ large enough.
Indeed, by dissipativity we have that
\begin{align}
	\norm{\grad \obj(\point)}^{2}
	&=
	\norm{\grad \obj(\point) - \alpha \point}^{2}
	+
	\alpha^{2} \norm{\point}^{2}
	+ 2 \alpha \inner{\grad \obj(\point) - \alpha \point}{\point}
	\notag\\
	&\geq 
	\alpha^{2} \norm{\point}^{2} - 2 \beta = \Omega(\norm{\point}^{2})\,,
\end{align}
so that
\begin{equation}
	\frac{\norm{\grad \obj(\point)}^{2}}{\bdvar(\obj(\point))}
	\geq
	\Omega \parens*{
	\frac{\norm{\point}^{2}}{{\obj(\point)}/{\batch}}
}
	= \Omega(\batch)
	\eqdot
\end{equation}
The second part of \cref{asm:SNR-weak} is satisfied by non-negativity and smoothness of $\loss$.
\endenv
\end{example}

In this framework, the iterates of \eqref{eq:SGD}, started at $\point \in \vecspace$, are defined by the following recursion:
\begin{equation}
\left\{
\begin{aligned}
	&\init[\point] \in \vecspace\\
	&\next = {\curr - \step \grad \obj(\curr) + \step \curr[\noise]}\,, \quad \text{ where } \curr[\noise] = \noise(\curr, \curr[\sample])
\end{aligned}
\right.
\end{equation}
where $(\curr[\sample])_{\run \geq \start}$ is a sequence of random variables in $\sspace$.
We will denote by $\prob_\point$ the law of the sequence $(\curr[\sample])_{\run \geq \start}$ when the initial point is $\point$ and by $\ex_\point$ the expectation with respect to $\prob_\point$.

\cref{asm:obj-weak,asm:noise-weak} imply the following growth condition, that we assume holds with the same constant for the sake of simplicity.
There is $\growth > 0$ such that, for all $\point \in \vecspace$, $\sample \in \samples$,
\begin{equation}
\label{eq:growth}
				\norm{ \grad \obj(\point)} \leq \growth ( 1 + \norm{\point}) \quad \text{ and } \quad \norm{\noise(\point, \sample)} \leq \growth ( 1 + \norm{\point})
	\eqdot
\end{equation}

\WAdelete{
\begin{assumption}[Blanket assumptions]
	The following holds:
\begin{itemize}
		\item On the objective function: 
\begin{itemize}
				\item $\obj$ is $\contdiff{2}$ and $\smooth(\obj)$-smooth.
				\item  $\inf \obj > - \infty$.
				\item $\setdef{\point \in \vecspace}{\grad \obj(\point) = 0}$ is compact.
\end{itemize}
		\item On the noise:
\begin{itemize}
				\item $(\curr[\sample])_{\run \geq \start}$ are \acl{iid} random variables.
				\item For $\run = \running$, $\curr[\sample]$ belongs to $\samples$, which is a compact subset of $\sspace$.
				\item $\noise$ is $\contdiff{2}$ and $\smooth(\noise)$-smooth.
		\item For every $\point \in \vecspace$, the expectation of $\noise(\point, \sample)$ is zero and its covariance is positive-definite.
\end{itemize}
		\item Growth conditions: there is $\growth > 0$ such that, for all $\point \in \vecspace$, $\sample \in \samples$,
\begin{equation}
				\norm{ \grad \obj(\point)} \leq \growth ( 1 + \norm{\point}) \quad \text{ and } \quad \norm{\noise(\point, \sample)} \leq \growth ( 1 + \norm{\point})
	\eqdot
\end{equation}
\end{itemize}
\end{assumption}
}

We introduce the cumulant generating functions of the noise $\noise(\point, \sample)$ and of the drift $- \grad \obj(\point) + \noise(\point, \sample)$, that we denote by $\hamiltalt$, $\hamilt$ to avoid confusion.
We also define their convex conjugates, that we denote by $\lagrangianalt$, $\lagrangian$.
\begin{definition}[Hamiltonians and Lagrangians]
	Define, for $\point \in \vecspace$, $\vel \in \dspace$,
\begin{equation}
\begin{aligned}
		\hamiltalt(\point, \vel) &= \log \ex \bracks*{\exp \parens*{\inner{\vel, \noise(\point, \sample)}}}\\
	\hamilt(\point, \vel)	&= -\inner{\grad \obj (\point), \vel} + \hamiltalt(\point, \vel)\\
	\lagrangianalt(\point, \vel) &= \hamiltalt(\point, \cdot)^*(\vel)\\
	\lagrangian(\point, \vel) &= \hamilt(\point, \cdot)^*(\vel) = \lagrangianalt(\point, \vel + \grad \obj(\point))
	\eqdot
\label{eq:lagrangian}
\end{aligned}
\end{equation}
\end{definition}
$\lagrangianalt$ and $\lagrangian$ are thus respectively equal to the Lagrangians $\lagof{\err}{\cdot,\cdot}$ and $\lagof{\orcl}{\cdot,\cdot}$.

Finally, \cref{asm:costs} will be discussed in \cref{app:estimates_trans_inv}.

\setcounter{assumption}{\thelastassumption}	
\subsection{Basic properties}

In this section, we derive from our assumptions some basic consequences, which will be useful throughout the proof.

We first state some properties of the Hamiltonian and the Lagrangian, which follow from their definitions.
\begin{lemma}[Properties of $\hamilt$ and $\lagrangian$]
	\label{lem:basic_prop_h_l}
	\leavevmode
\begin{enumerate}
		\item $\hamilt$ is $\contdiff{2}$ and $\hamilt(\point, \cdot)$ is convex for any $\point \in \vecspace$.
		\item $\lagrangian(\point, \cdot)$ is convex for any $\point \in \vecspace$, $\lagrangian$ is \ac{lsc} on $\vecspace \times \dspace$.
        \item For any $\point \in \vecspace$, $\vel \in \dspace$, $\hamilt(\point, \vel) \leq  2 \growth( 1 + \norm{\point})  \revise{\norm{\vel}}$ and 
			$\dom \lagrangian(\point, \cdot) \subset \clball (0, 2\growth(1 + \norm{\point}))$.
		\item For any $\point \in \vecspace$, $\vel \in \dspace$, $\lagrangian(\point, \vel) \geq 0$ and $\lagrangian(\point, \vel) = 0 \iff \vel = \grad \obj(\point)$.
\end{enumerate}
\end{lemma}
\begin{proof}
	For the last point, since $\hamilt(\point, \cdot)$ is convex and differentiable, 
\begin{equation}
		\lagrangian(\point, \vel) = 0 \iff 0 \in \partial_\vel \lagrangian(\point, \vel) \iff \vel \in \partial_{\vel} \hamilt(\point, 0) \iff \vel = \grad_\vel \hamilt(\point, 0)\,,
\end{equation}
	and, since the noise has zero mean, $\grad_\vel \hamilt(\point, 0) = \grad \obj(\point)$.
\end{proof}

\begin{lemma}[Growth of the iterates]
	\label{lem:growth_iterates}
	For any $\init \in \vecspace$, $\step_0 > 0$, for every $\step \in (0, \step_0]$, for any $\nRuns \geq 1$, $\start \leq \run \leq  \nRuns \ceil{\step^{-1}}$,
\begin{equation}
		\norm{\curr} \leq e^{2 \growth (1 + \step_0) \nRuns} (1 + \norm{\init})
	\eqdot
\end{equation}   
\end{lemma}
\begin{proof}
	By the triangular inequality, for any $\run \geq \start$, we have that
\begin{align}
\norm{\next}
	&\leq \norm{\curr} + \step \norm{\grad \obj(\curr)} + \step \norm{\curr[\noise]}
	\notag\\
	&\leq \norm{\curr} + 2 \step \growth (1 + \norm{\curr})
	\notag\\
	&= (1 + 2 \step \growth) \norm{\curr} + 2 \step \growth
\end{align}
where we used the growth condition on $\grad \obj$ and $\noise$.
One can then solve this recursion to show that, for any $\run \geq \start$,
\begin{equation}
\norm{\curr}
	\leq (1 + 2 \step \growth)^{\run} (1 + \norm{\init})
	\eqdot
\end{equation}
The result then follows by noticing that, for $\start \leq  \run \leq  \nRuns \ceil{\step^{-1}}$,
\begin{align}
(1 + 2 \step \growth)^{\run}
	&\leq ( 1 + 2 \step \growth)^{\nRuns \ceil{\step^{-1}}}
	\notag\\
	&\leq e^{2 \growth  \nRuns \step \ceil{\step^{-1}}}
	\notag\\
	&\leq e^{2 \growth  \nRuns (1 + \step)}
\eqdot
\qedhere
\end{align}
\end{proof}

\begin{lemma}
	\label{lem:degrowth_iterates}
	For any $\init \in \vecspace$, for every $\step \in (0, (4 \growth)^{-1}]$, for any $\nRuns \geq 1$, $\start \leq \run \leq  \nRuns \ceil{\step^{-1}}$,
\begin{equation}
\norm{\curr}
	\geq e^{-(4 \growth+1) \nRuns} \norm{\init} - 1
	\eqdot
\end{equation}   
\end{lemma}

\begin{proof}
By the triangular inequality, for any $\run \geq \start$, we have that
\begin{align}
\norm{\next}
	&\geq \norm{\curr} - \step \norm{\grad \obj(\curr)} - \step \norm{\curr[\noise]}
	\notag\\
	&\leq \norm{\curr} - 2 \step \growth (1 + \norm{\curr})
	\notag\\
	&= (1 - 2 \step \growth) \norm{\curr} - 2 \step \growth
\end{align}
where we used the growth condition on $\grad\obj$ and $\noise$.
One can then solve this recursion to obtain that, for any $\run \geq \start$,
\begin{equation}
\norm{\curr}
	\geq (1 - 2 \step \growth)^{\run} \norm{\init} - 1
	\eqdot
\end{equation}
The result then follows by noticing that, for $\start \leq  \run \leq  \nRuns \ceil{\step^{-1}}$,
\begin{align}
(1 - 2 \step \growth)^{\run}
	&\geq ( 1 - 2 \step \growth)^{\nRuns \ceil{\step^{-1}}}
	\notag\\
	&\geq e^{-4 \growth  \nRuns \step \ceil{\step^{-1}}}
	\notag\\
	&\leq e^{-4\growth  \nRuns (1 + \step)}
\end{align}
where we used the fact that $\log (1 - x ) \geq -2x$ on $[0, 1/2]$ since $2 \step \growth \leq 1/2$.
\end{proof}

\section{A \acl{LDP} for \acs{SGD}}
\label{app:LDP}

\subsection{Preliminaries}

The goal of this section is to provide \aclp{LDP} for continuous and discrete trajectories related to our algorithm of interest \eqref{eq:SGD}.
From the sequences $(\curr[\state])_{\run \geq \start}$ and $(\curr[\sample])_{\run \geq \start}$, we define three sequences: one discrete (in lowercase) and two continuous (in uppercase):
\begin{subequations}
\begin{enumerate}
[label={\upshape(\itshape \alph*\hspace*{.5pt}\upshape)}]
\item
The discrete ``rescaled'' trajectory 
\begin{equation}
\label{eq:rescaled}
\accstate_\run
	\defeq \state_{\run \floor{1 / \step}}
	\eqdot
\end{equation}
\item
The continuous ``interpolated'' trajectory, defined  for any $\run \geq \start$, $\time \in \bracks*{\step {\run}, \step \parens*{\run + 1}}$ by
\begin{align}
\label{eq:interpolated}
\cstate_\time
	= \curr +  \parens*{\frac{\time}{\step} - \run} (\next - \curr)
\end{align}
\item 
The continuous ``discretized noise''  trajectory, defined by $\cstatealt_\tstart = \state_\tstart $ and for any $\time>0$
\begin{align}
	 \dot \cstatealt_\time = -\grad \obj (\cstatealt_\time) + \noise(\cstatealt_\time, \sample_{\floor{\time / \step}}) \label{eq:discretizednoise}
\end{align}
\end{enumerate}
\end{subequations}
Note that all three sequences are ``accelerated'' by a factor $1/\step$ compared to the original sequences appearing in \eqref{eq:SGD}.
In this section, we establish a large deviation principle for the discrete rescaled sequence $(\accstate_\run)_{\run \geq \start}$.
To do so, we build upon \citet[Chap.~7]{FW98} to obtain a large deviations principle on $(\cstatealt_\time)_{\time \geq \start}$ and then transfer it to $(\cstate_\time)_{\time \geq \start}$ which enables us to obtain a discrete large deviation principle for $(\accstate_\run)_{\run \geq \start}$.

We equip the space of continuous functions $\contcurves_{\horizon} \defeq \contfuncs([0, \horizon], \vecspace)$ with the distance induced by the uniform norm
\begin{equation}
\dist_{[\tstart,\horizon]}( \pth, \pthalt)
	\defeq \sup_{\time \in [0, \horizon]} \norm{\pth_\time - \pthalt_\time}
	\eqdot
\end{equation}
In order to use it as a proxy later, we will now bound the distance between the (continuous) ``interpolated'' trajectory and the ``discretized noise''  trajectory.
To do so, we will first bound the latter.

\begin{lemma}[Growth of the trajectory]
\label{lem:growth_traj}
For any $\init \in \vecspace$, $\step > 0$,  $\time \geq 0$, we have
\(
\norm{\cstatealt_\time} \leq e^{2 \growth \time } \parens*{\norm{\init} + 2 \growth \time}.
\)
\end{lemma}

\begin{proof}
	Using the definition  of $\cstatealt_\time$ in \eqref{eq:discretizednoise} and the growth condition \eqref{eq:growth} on $\grad\obj$ and $\noise$, we have that
\begin{align}
		\norm{\dot \cstatealt_\time}
		&=
		\norm*{-\grad \obj (\cstatealt_\time) + \noise(\cstatealt_\time, \sample_{\floor{\time / \step}})}
		\leq 2 \growth (1 + \norm{\cstatealt_\time})\eqdot
\end{align}

	Hence, for any $\time \geq 0$,
\begin{equation}
		\norm{\cstatealt_\time} \leq \norm{\init[\statealt]} +  \int_0^\time 2\growth (1 + \norm{\cstatealt_\timealt}) \dd \timealt = \norm{\init[\statealt]} + 2\growth \time +   \int_0^\time 2\growth \norm{\cstatealt_\timealt} \dd \timealt \eqdot
\end{equation}
	Invoking Grönwall's lemma then yields the result.
\end{proof}

The following lemma states that the distance between the ``interpolated'' and ``discretized noise''  trajectories is bounded by a factor proportional to the stepsize $\step$.

\begin{lemma}
	\label{lem:traj_approx}
	Fix $\cpt \subset \vecspace$ compact, $\step_0 > 0$, $\horizon > 0$.
Then, there exists some constant $\const = \const(\cpt, \step_0, \horizon, \obj, \noise, \samples) < + \infty$ such that, for any $\init \in \cpt$, $\step \in (0, \step_0]$,
\begin{equation}
		\dist_{[\tstart,\horizon]}(\cstate, \cstatealt) \leq \const \, \step\eqdot
\end{equation}
\end{lemma}

\begin{proof}
	Before starting, notice that \cref{lem:growth_iterates,lem:growth_traj} imply that there is some compact set $\cptalt$ that depends on $\cpt$, $\horizon$, $\step_0$ and $\growth$ such that, for any $\init \in \cpt$, $\step \in (0, \step_0]$, $\time \in [0, \horizon]$,  $\cstate_\time$ and $\cstatealt_\time$ belong to $\cptalt$.
	In particular, $\noise$ is therefore Lipschitz-continuous on $\cptalt \times \samples$ so that, for any $\sample \in \samples$, $\state \mapsto -\grad \obj(\state) + \noise(\state, \sample)$ is Lipschitz-continuous on $\cptalt$ with constant $\tmplips$ and bounded with constant $\tmpbound$.

	Let us now estimate the derivative of the interpolated trajectory $\cstate_\time$, which is piecewise differentiable by definition (see \cref{eq:interpolated}).
For any $\time \in [0, \horizon]$ such that $\time \in \parens*{\step {\run}, \step \parens*{\run + 1}}$ for some $\run \geq \start$, we have that
\begin{align}
\MoveEqLeft
\norm*{\dot \cstate_\time - \parens*{-\grad \obj(\cstate_\time) + \noise(\cstate_\time, \sample_{\floor{\time / \step}})}}
	\notag\\
\hspace{4em}
	&= \norm*{\frac{\next - \curr}{\step} - \parens*{-\grad \obj(\cstate_\time) + \noise(\cstate_\time, \sample_{\floor{\time / \step}})}}
	\notag\\
	&= \norm*{\parens*{- \grad \obj(\cstate_{\step \run}) + \noise(\cstate_{\step \run}, \sample_{\run})} - \parens*{-\grad \obj(\cstate_\time) + \noise(\cstate_\time, \sample_{\floor{\time / \step}})}}
	\notag\\
	&\leq \tmplips \norm{\cstate_{\step \run} - \cstate_\time}
\end{align}
	where we used the $\tmplips$-Lipschitz-continuity of $\state \mapsto -\grad \obj(\state) + \noise(\state, \sample)$ on $\cptalt$ uniformly in $\sample \in \samples$.
Since this map is also bounded by $\tmpbound$, we have that
\begin{align}
\norm*{\dot \cstate_\time - \parens*{-\grad \obj(\cstate_\time) + \noise(\cstate_\time,\sample_{\floor{\time / \step}})}}
	&\leq \tmplips \norm{\cstate_{\step \run} - \cstate_\time}
	\leq \tmplips \norm{\curr -\next}
	\leq \tmplips \tmpbound \step\eqdot
\end{align}
	
Moreover, the $\tmplips$-Lipschitz-continuity of $\state \mapsto -\grad \obj(\state) + \noise(\state, \sample)$ also gives us that
\begin{align}
\MoveEqLeft
\norm{\dot \cstatealt_\time - \parens*{ -\grad \obj(\cstate_\time) + \noise(\cstate_\time, \sample_{\floor{\time / \step}})}}
	\notag\\
	&= \norm{ -\grad \obj (\cstatealt_\time) + \noise(\cstatealt_\time, \sample_{\floor{\time / \step}}) - \parens*{ -\grad \obj(\cstate_\time) + \noise(\cstate_\time, \sample_{\floor{\time / \step}})}}
	\notag\\
	&\leq \tmplips \norm{\cstatealt_\time - \cstate_\time}\eqdot
\end{align}

	Putting everything together and integrating $\cstate_\time - \cstatealt_\time$ from $0$ to $\time$ yields (since $\cstate_\time$ is absolutely continuous), for any $\time \in [0, \horizon]$,
\begin{align}
		\norm{ \cstate_\time - \cstatealt_\time}
		&\leq  
		\norm{ \cstate_0 - \statealt_0} + \int_0^\time \norm{ \dot \cstate_\timealt - \dot \cstatealt_\timealt} \dd \timealt
		\leq 
		\tmplips \tmpbound \step + \tmplips \int_0^\time \norm{\cstate_\timealt - \cstatealt_\timealt} \dd \timealt
\end{align}
	where we used that $\statealt_0 = \cstate_0 = \init$ by definition of the trajectories.
Finally, Grönwall's lemma then yields the result.
\end{proof}

\subsection{\Acl{LDP} for interpolated trajectories}
\label{app:ldp_interpolated}

From the Lagrangian defined in \eqref{eq:lagrangian}, we define, on $\contcurves_{\horizon} = \contfuncs([0, \horizon], \vecspace)$, the normalized action functional $\action_{[\start,\horizon]}$ as
\begin{equation}\label{eq:actionfunctional}
\action_{[\start,\horizon]}(\pth)
	=  \begin{cases*}
		\int_{0}^{\horizon} \lagrangian(\pth_\time, \dot \pth_\time) \dd\time 
			&\quad
			if $\pth$ is absolutely continuous,
			\\
		\infty
			&\quad
			otherwise,
	\end{cases*}
\end{equation}  
following \citet[Chap.~3.2]{FW98}, as a manner to quantify how ``probable'' a trajectory is.

We first show that the set 
\begin{equation}
\pths_{[\start,\horizon]}^\cpt(\level)
	\defeq \setdef{\pth \in \contfuncs([0, \horizon], \vecspace)}{\pth_0 \in \cpt, \action_{[\start,\horizon]}(\pth) \leq \level}   
\end{equation}
of trajectories with bounded action functional is compact and $\action_{[\start,\horizon]}$ is \acl{lsc}.

\begin{lemma}
\label{lem:pathscompact}
	Fix $\horizon > 0$.
For any $\cpt \subset \vecspace$ compact, $\level \geq 0$, the set $\pths_{[\start,\horizon]}^\cpt(\level)$ is  compact and $\action_{[\start,\horizon]}$ is \ac{lsc} on $\contfuncs([0, \horizon], \vecspace)$.
\end{lemma}

\begin{proof}
	Let us first check that $\action_{[\start,\horizon]}$ is \ac{lsc} on $\contfuncs([0, \horizon], \vecspace)$ by applying \citet[\S9.1.4, Thm.~3]{ioffeTheoryExtremalProblems1979}: $(\time, \point, \vel) \mapsto \lagrangian(\point, \vel)$ is a normal integrand since $(\point, \vel) \mapsto \lagrangian(\point, \vel)$ is \ac{lsc} by construction, quasiregular since $\lagrangian(\point, \cdot)$ is convex for any $\point \in \vecspace$ and satisfies the growth condition because it is non-negative (see \cref{lem:basic_prop_h_l}).
Hence, $\action_{[\start,\horizon]}$ is \ac{lsc} on $\contfuncs([0, \horizon], \vecspace)$.

	The compactness of $\pths_{[\start,\horizon]}^\cpt(\level)$ follows from the idea of proof as \citet[Chap.~7, Lemma~4.2]{FW98} but with the added difficulty that the gradient $\grad \obj$ and the noise $\noise$ are not uniformly bounded.
	Take $\pth \in \pths_{[\start,\horizon]}^\cpt(\level)$.
Since $\action_{[\start,\horizon]}(\pth) \leq \level < +\infty$, it means that $\lagrangian(\pth_\time, \dot \pth_\time) < +\infty$ for almost every $\time$ so that by \cref{lem:basic_prop_h_l}, almost everywhere,
\begin{equation}
\norm{\dot \pth_\time}
	\leq 2 \growth (1 + \norm{\pth_\time})
	\eqdot
\label{eq:derbounded}
\end{equation}
Grönwall's lemma then yields that, for any $\time \in [0, \horizon]$,
\begin{equation}
\norm{\pth_\time}
	\leq e^{2 \growth \horizon} \parens*{\norm{\pth_0} + 2 \growth \horizon}
\end{equation}
which is bounded uniformly in $\pth \in \pths_{[\start,\horizon]}^\cpt(\level)$ since $\pth_0$ is in $\cpt$ compact.
Hence, $\norm{\dot \pth_\time}$ is also uniformly bounded by \cref{eq:derbounded}.
Therefore, the functions in $\pths_{[\start,\horizon]}^\cpt(\level)$ are equicontinuous and uniformly bounded and, $\pths_{[\start,\horizon]}^\cpt(\level)$ is closed by \ac{lsc} of $\action_{[\start,\horizon]}$, so that, by the Arzelà-Ascoli theorem, $\pths_{[\start,\horizon]}^\cpt(\level)$ is compact.
\end{proof}

The following result establishes the fact that the functional $\step^{-1}\action_{[\start,\horizon]}$ is the action functional in $\contfuncs([0, \horizon], \vecspace)$ of the interpolated process $(\cstate_\time)_{\time \in [0, \horizon]}$ of the algorithm started at $ \init$, uniformly with respect to the initial point $\init$ in any compact set $\cpt \subset \vecspace$, as $\step \to 0$.
This enables to build on \citet[Chap.~7, Thm.~4.1$'$]{FW98} to provide a large deviation principle for the interpolated trajectory $(\cstate_\time)_{\time \in [0, \horizon]}$, meaning that for $\step$ small enough, it will be
\begin{enumerate*}
[\upshape(\itshape a\hspace*{.5pt}\upshape)]
\item
close to probable trajectories with a probability exponentially big in their action functional;
and
\item
far from the most probable trajectories with a probability exponentially small in their action value.
\end{enumerate*}

\begin{proposition}
\label{prop:ldp_traj}
Fix $\horizon > 0$.
For any $\level, \margin, \precs > 0$, $\cpt \subset \vecspace$ compact, there exists $\step_0 > 0$ such that, for any $\step \in (0, \step_0]$, for any $\init \in \cpt$, we have that 
\begin{subequations}
\begin{align}
\probwrt*{\init}{\dist_{[\tstart,\horizon]}(\cstate, \pth) < \margin}
	&\geq \exp \parens*{- \frac{\action_{[\start,\horizon]}(\pth) + \precs}{\step}}
	\\
\probwrt*{\init}{\dist_{[\tstart,\horizon]}(\cstate, \pths_{[\start,\horizon]}^{\{\init\}}(\level)) > \margin} 	&\leq \exp \parens*{- \frac{\level - \precs}{\step}}
\end{align}
\end{subequations}
for all $\pth \in \pths_{[\start,\horizon]}^{\{\init\}}(\level)$.
\end{proposition}

\begin{proof}
Our strategy is to apply \citet[Chap.~7, Thm.~4.1$'$]{FW98} to the process $(\cstatealt_\time)_{\time \in [0, \horizon]}$  and use \cref{lem:traj_approx} to transfer the result to $(\cstate_\time)_{\time \in [0, \horizon]}$ (both starting from $\init$).
However, this theorem requires $\drift(\point, \sample) \defeq - \grad \obj(\point) + \noise(\point, \sample)$ to be uniformly bounded as well as its derivative.
To avoid this issue, note that, for a fixed compact set $\cpt$, $\level > 0$, and all $\step \leq  1$, the trajectories of $\pths_{[\start,\horizon]}^\cpt(\level)$ and of the process $(\cstatealt_\time)_{\time \in [0, \horizon]}$ are contained in a compact set $\cptalt$ by the compactness of $\pths_{[\start,\horizon]}^\cpt(\level)$ and \cref{lem:growth_traj}.

Now, consider $\alt\drift \from \vecspace \times \sspace \to \vecspace$ twice differentiable that coincides with $\drift$ on $\cptalt \times \samples$ but that is uniformly bounded along with its derivative.
Define $\alt\hamilt(\point, \vel) = \log \ex \bracks*{\exp \parens*{\inner{\vel, \alt\drift(\point, \sample)}}}$ which is still differentiable and satisfies ``Condition F'' of \citet[Chap.~7.4]{FW98} by \citet[Chap.~7, Lem.~4.3]{FW98} and \ac{iid} assumption.
Hence, \citet[Chap.~7, Thm.~4.1$'$]{FW98} yields that, with $\cpt$ and $\level$ fixed above, for any $\margin, \precs > 0$, there exists $\step_0 > 0$ such that, for any $\step \in (0, \step_0]$, for any $\init \in \cpt$, $\pth \in \pths_{[\start,\horizon]}^{\{\init\}}(\level)$,
\begin{subequations}
\begin{align}
\prob_{\init} \parens*{\dist_{[\tstart,\horizon]}(\cstatealt,\pth) < \margin} &\geq \exp \parens*{- \frac{\action_{[\start,\horizon]}(\pth) + \precs}{\step}}
	\\
\probwrt*{\init}{\dist_{[\tstart,\horizon]}(\cstatealt, \pths_{[\start,\horizon]}^{\{\init\}}(\level)) > \margin}
	&\leq \exp \parens*{- \frac{\level - \precs}{\step}}
	\eqdot
\end{align}
\end{subequations}
To obtain the result for the process $(\cstate_\time)_{\time \in [0, \horizon]}$, fix $\margin > 0$.
By \cref{lem:traj_approx}, there exists $\step_0 > 0$ such that, for any $\step \in (0, \step_0]$, for any initial point $\init \in \cpt$, $\dist_{[\tstart,\horizon]}(\cstate, \cstatealt) < \margin$.
Combining this with the previous result on $\parens*{\cstatealt_\time}_{\time \in [\tstart, \horizon]}$ yields that,  for any $\margin, \precs > 0$, there exists $\step_0 > 0$ such that, for any $\step \in (0, \step_0]$, for any $\init \in \cpt$, $\pth \in \pths_{[\start,\horizon]}^{\{\init\}}(\level)$,
\begin{subequations}
\begin{align}
\prob_{\init} \parens*{\dist_{[\tstart,\horizon]}(\cstate, \pth) < 2\margin}
	&\geq \exp \parens*{- \frac{\action_{[\start,\horizon]}(\pth) + \precs}{\step}}
\shortintertext{and}
\probwrt*{\init}{\dist_{[\tstart,\horizon]}(\cstate, \pths_{[\start,\horizon]}^{\{\init\}}(\level)) > 2\margin}
	&\leq \exp \parens*{- \frac{\level - \precs}{\step}}
\end{align}
\end{subequations}
which concludes the proof.
\end{proof}

\subsection{\Acl{LDP} for discrete trajectories}
\label{app:ldp_discrete}

We now leverage the previous section to show a \ac{LDP} for the discrete rescaled trajectories $(\accstate_{\run })_{\run \geq \start}$ defined in \cref{eq:rescaled}.
To do so, we will use the results of the previous section by considering a finite number of points $ (\cstate_\run)_{\run\geq\start}$ from the continuous interpolation.

For some $\nRuns>0$, we will first equip $\vecspace^{\nRuns}$ with the distance
\begin{equation}
	\dist_{\nRuns}(\dpth, \dpthalt) = \max_{\start \leq \run \leq \nRuns - 1} \norm{\dpth_\run - \dpthalt_\run}
\end{equation}
and bound the difference between the discrete rescaled trajectory and the continuous interpolation.

\begin{lemma}
\label{lem:iterates_approx}
Fix $\cpt \subset \vecspace$ compact, $\step_0 > 0$, $\nRuns \geq 1$.
Then, there exists some constant $\const = \const(\cpt, \step_0, \nRuns, \obj, \noise, \samples) < + \infty$ such that, for any $\init \in \cpt$, $\step \in (0, \step_0]$,
\begin{equation}
\dist_{\nRuns}(\accstate, (\cstate_\run)_{0 \leq \run \leq \nRuns-1})
	= \max_{\start \leq \run \leq \nRuns - 1} \norm{\accstate_\run - \cstate_\run}
	\leq \const \step\eqdot
\end{equation}
\end{lemma}

\begin{proof}
By \cref{lem:growth_iterates}, there is some compact set $\cptalt$ that depends on $\cpt$, $\nRuns$, $\step_0$ and $\growth$ such that, for any $\init \in \cpt$, $\step \in (0, \step_0]$, $\run \leq \ceil{1/\step} \nRuns$, $\state_\run$ belongs to $\cptalt$.
In particular, for almost every $\time \in [0, \nRuns - 1]$, we have, that,
\begin{align}
\dot\cstate_\time
	&= \frac{\state_{\floor{\time / \step} + 1} - \state_{\floor{\time / \step}}}{\step}
	= -\grad\obj(\state_{\floor{\time / \step}}) + \noise(\state_{\floor{\time / \step}}, \sample_{\floor{\time / \step}})
\end{align}
with $\state_{\floor{\time / \step}}$ belonging to $\cptalt$ since $\floor*{\frac{\time}{\step}} \leq \floor*{\frac{\nRuns - 1}{\step}} \leq \ceil*{\frac{\nRuns}{\step}}$.
Hence, the norm of $\dot \cstate_\time$ is bounded by some constant $\tmpbound$ for almost every $\time \in [0, \nRuns - 1]$ uniformly in $\init \in \cpt$, $\step \in (0, \step_0]$ by the growth condition in \cref{asm:noise-weak-growth}.

Therefore, for any $\start \leq \run \leq \nRuns - 1$, since $\accstate_\run = \state_{\run \floor{1/\step}} = \cstate_{\run \step \floor{1/\step}}$, we have that
\begin{equation}
\norm{\accstate_\run - \cstate_\run}
	\leq \tmpbound \run \abs*{ 1 - \step \floor{1 / \step}}
	\leq \tmpbound \nRuns \step
\end{equation}
which concludes the proof.
\end{proof}

Now, for $\nRuns \geq \start$, $\dpth = (\dpth_\start,\dots,\dpth_{\nRuns -1})\in \vecspace^{\nRuns}$, let us define the normalized discrete action functional
\begin{equation}
	\daction_\nRuns(\dpth) \defeq \sum_{\run = \start}^{\nRuns - 2} \rate(\dpth_\run, \dpth_{\run + 1})
\end{equation}
where the cost of moving from one iteration to the next is defined for any $\point, \pointalt \in \vecspace$ from the previous continuous normalized action functional (\cf \cref{eq:actionfunctional}) with horizon $1$ as 
\begin{equation}\label{eq:defrate}
	\rate(\point, \pointalt) \defeq \inf \setdef{\action_{0,1}(\pth)}{\pth \in \contfuncs([0, 1], \vecspace), \pth_0 = \point, \pth_1 = \pointalt}\eqdot
\end{equation}

Again, we show that the set of discrete trajectories with low action functional
\begin{equation}
\pths_{\nRuns}^\cpt(\level)
	\defeq \setdef{\dpth \in \vecspace^{\nRuns}}{\dpth_\start \in \cpt,\daction_\nRuns(\dpth) \leq \level}
\end{equation}
is compact and $\daction_\nRuns$ is \ac{lsc}.

\begin{lemma}
\label{lem:pathscompactdiscrete}
Fix $\nRuns \geq \start$.
For any $\cpt \subset \vecspace$ compact, $\level \geq 0$, the set $ \pths_{\nRuns}^\cpt(\level)$			 is compact and $\daction_\nRuns$ is \ac{lsc} on $\vecspace^{\nRuns}$.
\end{lemma}

\begin{proof}
First, let us show that, for $\level \geq 0$, $\cpt \subset \vecspace$, the set $\setdef{(\point, \pointalt) \in \cpt \times \vecspace}{\rate(\point, \pointalt) \leq \level}$ is compact.
Let $(\point^\runB, \pointalt^\runB)_{\runB \geq \start}$ be a sequence in $\cpt \times \vecspace$ such that $\rate(\point^\runB, \pointalt^\runB) \leq \level$ for all $\runB \geq \start$.
Then, for any $\runB \geq \start$, there exists $\pth^\runB \in \contfuncs([0, 1], \vecspace)$ such that $\pth^\runB_0 = \point^\runB$, $\pth^\runB_1 = \pointalt^\runB$ and $\action_{0,1}(\pth^\runB) \leq \rate(\point^\runB, \pointalt^\runB) + (1 + \runB)^{-1}$.
In particular, for any $\runB \geq \start$, $\pth^\runB$ belongs to $\setdef{\pth \in \contfuncs([0, 1], \vecspace)}{\pth_0 \in \cpt, \action_{0,1}(\pth) \leq \level + 1}$ which is compact by \cref{lem:pathscompact}.
Hence, there exists a subsequence that converges uniformly to some $\pth \in \contfuncs([0, 1], \vecspace)$.
Without loss of generality, assume that $\pth^\runB \to \pth$ as $\runB \to \infty$.
In particular,  $(\point^\runB, \pointalt^\runB)$ converges to $(\pth_0, \pth_1)$ as $\runB \to \infty$ with $\pth_0$ belonging to $\cpt$.
Then, by \ac{lsc} of $\action_{0,1}$, we have that $\action_{0,1}(\pth) \leq \lim\inf_{\runB\to\infty} \rate(\point^\runB, \pointalt^\runB) \leq \level$ so that $\rate(\pth_0, \pth_1) \leq \level$.
Therefore, we have shown that $\setdef{(\point, \pointalt) \in \cpt \times \vecspace}{\rate(\point, \pointalt) \leq \level}$ is compact.

As a consequence $\rate$ is \ac{lsc} on $\vecspace \times \vecspace$: indeed, for any $\level \geq 0$, any convergent sequence of points of $\setdef{(\point, \pointalt) \in \vecspace \times \vecspace}{\rate(\point, \pointalt) \leq \level}$ must be included in $\setdef{(\point, \pointalt) \in \cpt \times \vecspace}{\rate(\point, \pointalt) \leq \level}$ for some $\cpt \subset \vecspace$ compact, which is closed and therefore contains the limit point of any such sequences.
$\daction_\nRuns$ is then immediately \ac{lsc} on $\vecspace^{\nRuns}$.

Finally, the compactness of $\pths_{\nRuns}^\cpt(\level)$ follows by induction on $\nRuns$.
For $\nRuns = 1$, $\pths_1^\cpt(\level) = \setdef{(\point, \pointalt) \in \vecspace \times \vecspace}{\rate(\point, \pointalt) \leq \level}$ which is compact by the previous argument for any compact set $\cpt$.
Now, assume that $\pths_{\nRuns}^\cpt(\level)$ is compact for some $\nRuns \geq 1$.
As a consequence, the set
\begin{equation}
\cptalt
	\defeq \setdef{\point \in \vecspace}{\exists (\dpth_\start, \dots, \dpth_{\nRuns-2}) \in \vecspace^{\nRuns - 1}\,, (\dpth_\start, \dots, \dpth_{\nRuns-2}, \point) \in \pths_{\nRuns}^\cpt(\level)}
\end{equation}is compact as well.
Hence, $\pths_{\nRuns + 1}^\cpt(\level)$ is included in teh product of compact sets $\pths_{\nRuns - 1}^\cpt(\level) \times \pths_1^{\cptalt}(\level)$ and is therefore bounded.
Moreover, $\pths_{\nRuns + 1}^\cpt(\level)$ is closed by \ac{lsc} of $\daction_\nRuns$ and therefore compact.
This concludes the proof by induction and the proof of the lemma.
\end{proof}

We will first provide a discrete analogue of \cref{prop:ldp_traj} for the interpolated trajectory at times $\run=\start,\dots,\nRuns-1$, with $\step^{-1}\daction_\nRuns$ the action functional in $\vecspace^{\nRuns}$ of the process $(\cstate_\run)_{0 \leq \run \leq \nRuns-1}$, uniformly with respect to the starting point $ \init$ in any compact set $\cpt \subset \vecspace$, as $\step \to 0$.

\begin{proposition}
	\label{prop:ldp_iterates}
	Fix $\nRuns \geq \start$.
For any $\level, \margin, \precs > 0$, $\cpt \subset \vecspace$ compact, there exists $\step_0 > 0$ such that, for any $\step \in (0, \step_0]$, for any $\init \in \cpt$, $\dpth \in \pths_{\nRuns}^{\{\init\}}(\level)$, we have that
\begin{subequations}
\begin{align}
\prob_{\init} \parens*{\dist_{\nRuns}((\cstate_\run)_{0 \leq \run \leq \nRuns-1}, \dpth) < \margin}
	&\geq \exp \parens*{- \frac{\daction_\nRuns(\dpth) + \precs}{\step}}
\shortintertext{and}
\probwrt*{\init}{\dist_{\nRuns}((\cstate_\run)_{0 \leq \run \leq \nRuns-1}, \pths_{\nRuns}^{\{\init\}}(\level)) > \margin}
	&\leq \exp\parens*{- \frac{\level - \precs}{\step}}
	\eqdot
\end{align}
\end{subequations}
for all $\dpth \in \pths_{\nRuns}^{\{\init\}}(\level)$
\end{proposition}

\begin{proof}
Invoke the first part of \cref{prop:ldp_traj} with $\horizon \gets \nRuns - 1$ and $\level \gets \level + \precs$.
There exists $\step_0 > 0$ such that, for any $\step \in (0, \step_0]$, for any $\init \in \cpt$, $\pth \in \pths_{0,\nRuns}^{\{\init\}}(\level + \precs)$,
\begin{equation}
\prob_{\init} \parens*{\dist_{0, \nRuns - 1 }(\cstate, \pth) < \margin}
	\geq \exp\parens*{- \frac{\action_{0,\nRuns-1}(\pth) + \precs}{\step}}
	\eqdot
\end{equation}
Take $\step \in (0, \step_0]$, $\init \in \cpt$, $\dpth \in \pths_{\nRuns}^{\{\init\}}(\level)$.
Then, there exists $\pth \in \pths_{0, \nRuns}^{\init}(\level + \precs)$, for any $0 \leq \run \leq \nRuns-1$, $\pth_\run = \dpth_\run$, \cf \eqref{eq:defrate}.
Hence,
\begin{align}
\prob_{\init} \parens*{\dist_{\nRuns}((\cstate_\run)_{0 \leq \run \leq \nRuns-1}, \dpth) < \margin}
	&\geq \prob_{\init} \parens*{\dist_{0, \nRuns - 1 }(\cstate, \pth) < \margin}
	\notag\\
	&\geq \exp \parens*{- \frac{\action_{0,\nRuns-1}(\pth) + \precs}{\step}}
	\notag\\
	&\geq \exp \parens*{- \frac{\daction_\nRuns(\dpth) + 2\precs}{\step}}
\end{align}
which prove the first part of the result.

For the second part, we have similarly from \cref{prop:ldp_traj}  with $\horizon \gets \nRuns - 1$ and $\margin \gets \margin/2$ that for any $\step \in (0, \step_0]$, for any $\init \in \cpt$, 
\begin{equation}
\prob_{\init} \parens*{\dist_{\tstart,\nRuns-1}(\cstate, \pths_{\tstart,\nRuns-1}^{\{\init\}}(\level)) > \margin/2}
	\leq \exp \parens*{- \frac{\level - \precs}{\step}}
	\eqdot
\end{equation}
Now, note that if $\dist_{0, \nRuns - 1}(\cstate, \pths_{0, \nRuns-1}^{\{\init\}}(\level)) < \margin$, then there must exist $\pth \in \pths_{0, \nRuns-1}^{\{\init\}}(\level)$ such that $\dist_{0, \nRuns - 1}(\cstate, \pth) \leq  \margin$.
Consider the discrete path $\dpth \in \vecspace^{\nRuns}$ defined by $\dpth_\run = \pth_\run$ for $\start \leq \run \leq  \nRuns - 1$.
Then, by construction, $\daction_\nRuns(\dpth) \leq \action_{0, \nRuns - 1}(\pth) \leq \level$ and $\dist_{\nRuns}((\cstate_\run)_{\start \leq \run \leq \nRuns-1}, \dpth) \leq  \margin$.
Thus, 
\begin{equation}
\dist_{0, \nRuns - 1}(\cstate, \pths_{0, \nRuns-1}^{\{\init\}}(\level)) < \margin \implies \dist_{\nRuns}((\cstate_\run)_{\start \leq \run \leq \nRuns-1}, \pths_{\nRuns}^{\{\init\}}(\level))
	\leq  \margin
	\eqdot
\end{equation}

Putting all together, we have that 
\begin{align}
\probwrt*{\init}{\dist_{\nRuns}((\cstate_\run)_{0 \leq \run \leq \nRuns-1}, \pths_{\nRuns}^{\{\init\}}(\level)) > \margin}
	&\leq \prob_{\init} \parens*{\dist_{0, \nRuns - 1}(\cstate, \pths_{0, \nRuns-1}^{\{\init\}}(\level) \geq  \margin)}
	\notag\\
	&\leq \prob_{\init} \parens*{\dist_{0, \nRuns - 1}(\cstate, \pths_{0, \nRuns-1}^{\{\init\}}(\level) >  \margin / 2)}
	\notag\\
	&\leq \exp \parens*{- \frac{\level - \precs}{\step}}
\end{align}
which concludes our proof.
\end{proof}

Finally, we end up with a large deviation principle on the discrete rescaled iterates $(\curr[\accstate])_{\start \leq \run \leq \nRuns-1} = (\state_{\run \floor{\step^{-1}}})_{\start \leq \run }$ by leveraging  \cref{lem:iterates_approx}.
In the following result, the functional $\step^{-1}\daction_\nRuns$ is thus the action functional in $\vecspace^{\nRuns}$ of the process $(\curr[\accstate])_{\start \leq \run \leq \nRuns-1} $ uniformly with respect to the starting point $ \init$ in any compact set $\cpt \subset \vecspace$, as $\step \to 0$.

\begin{corollary}
	\label{cor:ldp_iterates}
	Fix $\nRuns \geq \start$.
For any $\level, \margin, \precs > 0$, $\cpt \subset \vecspace$ compact, there exists $\step_0 > 0$ such that, for any $\step \in (0, \step_0]$, for any $\init \in \cpt$, $\dpth \in \pths_{\nRuns}^{\{\init\}}(\level)$, we have that
\begin{subequations}
\begin{align}
\prob_{\init} \parens*{\dist_{\nRuns}(\accstate, \dpth) < \margin}
	&\geq \exp \parens*{- \frac{\daction_\nRuns(\dpth) + \precs}{\step}}
\shortintertext{and}
\prob_{\init} \parens*{\dist_{\nRuns}(\accstate, \pths_{\nRuns}^{\{\init\}}(\level)) > \margin}
	&\leq \exp \parens*{- \frac{\level - \precs}{\step}}
	\eqdot
\end{align}
\end{subequations}
for all $\dpth \in \pths_{\nRuns}^{\{\init\}}(\level)$.
\end{corollary}

\begin{proof}
	Fix $\margin>0$.
Choose $\step_0$ such that both \cref{prop:ldp_iterates} and \cref{lem:iterates_approx} hold with $ \const \step \leq \margin $.
Then for any  $\step \in (0, \step_0]$,  for any $\init \in \cpt$, $\dpth \in \pths_{\nRuns}^{\{\init\}}(\level)$, $\dpth \in \pths_{\nRuns}^{\{\init\}}(\level)$, we have that 
\begin{align}
 \prob_{\init} \parens*{\dist_{\nRuns}((\cstate_\run)_{0 \leq \run \leq \nRuns-1}, \dpth) < \margin} &\geq \exp \parens*{- \frac{\daction_\nRuns(\dpth) + \precs}{\step}}\eqdot
\end{align}
Now, if $\dist_{\nRuns}((\cstate_\run)_{0 \leq \run \leq \nRuns-1}, \dpth) < \margin $ and $\dist_{\nRuns}(\accstate, (\cstate_\run)_{0 \leq \run \leq \nRuns-1}) \leq \margin$, then $\dist_{\nRuns}(\accstate, \dpth) < 2\margin$.
Thus,
\begin{align}
\probwrt*{\init}{\dist_{\nRuns}(\accstate, \dpth) < 2\margin}
	\geq \probwrt*{\init}{\dist_{\nRuns}((\cstate_\run)_{0 \leq \run \leq \nRuns-1}, \dpth)<\margin}
	&\geq \exp \parens*{- \frac{\daction_\nRuns(\dpth) + \precs}{\step}}
	\eqdot
\end{align}
The second part can be obtained similarly.
\end{proof}

To summarize this part on large deviations principles, we state a corollary containing all the results that will be needed in the following sections.

\begin{corollary}
\label{cor:ldp_iterates_full}
Fix $\nRuns \geq \start$.
Then: 
\begin{itemize}
\item
For all $\level > 0$, the set 
\begin{equation}
\pths_{\nRuns}^\cpt(\level)
	\defeq \setdef{\dpth \in \vecspace^{\nRuns}}{\dpth_\start \in \cpt,  \daction_\nRuns(\dpth) \leq \level}
\end{equation}
is compact and $\daction_\nRuns$ is \ac{lsc} on $\vecspace^{\nRuns}$.
\item 
For all $\level,\margin,\precs > 0$, $\cpt \subset \vecspace$ compact, there exists $\step_0 > 0$ such that, for any $\step \in (0, \step_0]$, for all $\init \in \cpt$, $\run \leq \nRuns$, $\dpth \in \pths_{\run}^{\{\init\}}(\level)$, we have that
\begin{subequations}
\begin{align}
\probwrt*{\init}{\dist_{\run}(\accstate, \dpth) < \margin}
	&\geq \exp \parens*{- \frac{\daction_\run(\dpth) + \precs}{\step}}
\shortintertext{and}
\probwrt*{\init}{\dist_{\run}(\accstate, \pths_{\run}^{\{\init\}}(\level)) > \margin}
	&\leq \exp\parens*{- \frac{\level - \precs}{\step}}
	\eqdot
\end{align}
\end{subequations}
\end{itemize}
\end{corollary}

\section{Attractors and limiting measures via large deviations}
\label{app:invmeas}

We now take inspiration from the framework of \citet{Kif88} in order to relate the sets of critical points to the sets where points can move at no cost.
Then, we relate the probability of \ac{SGD} moving to neighborhoods of critical sets to the probability of being close to well-chosen paths, which enables us to use the results of the previous section.
Finally, we build upon these results to provide bounds on the limiting measure of \ac{SGD}.

\subsection{Setup}

We first need to define the gradient flow of $\obj$.

\begin{definition}
Define, for $\point \in \vecspace$, the flow $\flowmap$  of $-\grad\obj$ starting at $\point$, \ie
\begin{equation}
\dotflowof{\time}{\point}
	= - \grad \obj(\flowof{\time}{\point})
	\quad
	\text{with}
	\quad
\flowof{\tstart}{\point}
	= \state
\end{equation}
and let $\map(\state)$  be the value of this flow at time $1$, \ie
\begin{equation}
\label{eq:mapdef}
\map(\point)
	= \flowof{1}{\point}
	\eqdot
\end{equation} 
\end{definition}

\begin{lemma}[Properties of the flow]
	\label{lem:basic_prop_flow}
	$\flowmap$ is well-defined and continous in both time and space, and, for any $\horizon \geq 0$, $\pth \in \contfuncs([0, \horizon], \vecspace)$ such that $\pth_\tstart = \point$,
\begin{equation}
		\action_{0, \horizon}(\pth) = 0 \iff \pth_{\time} = \flowof{\time}{\point} ~~ \text{ for all } \time \in [0, \horizon]\eqdot
\end{equation}
\end{lemma}
\begin{proof}
The well-definition and continuity of $\flowmap$ are a consequence of $\obj$ being twice continuously differentiable and of the global Cauchy-Lipschitz (Picard–Lindelöf) theorem for \ac{ODE}. 
The equivalence follows from the uniqueness of the flow and \cref{lem:basic_prop_h_l} since
\begin{align}
\action_{0,\horizon}(\pth)
	= 0
	&\iff
\lagrangian(\pth_{\time}, \dot \pth_{\time}) = 0 \; \text{almost everywhere}
	\notag\\
	&\iff \dot \pth_{\time} = - \grad \obj(\pth_{\time}) \; \text{almost everywhere}
\end{align}
and thus, by extending $\dot \pth$ by continuity, both $\pth$ and $\flowmap$ satisfy the same \ac{ODE} with the same initial condition and are thus equal for all $\time$ by uniqueness of the solutions.
\end{proof}

The following lemma translates this for $\map$.

\begin{lemma}[Properties of $\map$]
	\label{lem:basic_prop_map}
	   $\map$ is well-defined and continuous and, for any $\point, \pointalt \in \vecspace$,
\begin{equation}
		\rate(\point, \pointalt) = 0 \iff \pointalt = \map(\point)\eqdot
\end{equation}
\end{lemma}

\begin{proof}
The implication $(\Leftarrow)$ is immediate by definition of $\map$ and \cref{lem:basic_prop_flow}.
Now for the reverse, assume that $\rate(\point, \pointalt) = 0$.
\revise{Following the proof of \cref{lem:pathscompactdiscrete}}, there exists $\pth \in \contfuncs([0, 1], \vecspace)$ such that $\pth_{0} = \point$, $\pth_{1} = \pointalt$ and $\action_{0, 1}(\pth) = 0$.
By \cref{lem:basic_prop_flow}, $\pth_{\time} = \flowof{\time}{\point}$ for all $\time \in [0, 1]$ and thus $\pointalt = \flowmap_{1}(\point) = \map(\point)$.
\end{proof}

\subsection{Attractors}
\label{app:attractors}

Let us first formalize the minimum-energy displacement between two points.
\begin{definition}[{\citet[\S1.5]{Kif88}}]
\label{def:dquasipot}
Define, for $\point, \pointalt \in \vecspace$,
\begin{align}
\dquasipot(\point, \pointalt)
	&= \inf\setdef*{\action_{[\tstart,\horizon]}(\pth)}{\pth \in \contfuncs([0, \horizon], \vecspace), \pth_{0} = \point, \pth_\horizon = \pointalt, \horizon \in \N, \horizon \geq 1}
	\notag\\
	&= \inf\setdef*
		{\daction_\nRuns(\dpth)}
		{\dpth \in \points ^ \nRuns\,,\, \dpth_\start = \point\,,\, \dpth_{\nRuns - 1} = \pointalt,\, \nRuns \geq 1}
	\eqdot
\end{align}
\end{definition}

The fact that these two expressions coincide directly come from the definition of $\rate$.

This enables us to define an equivalence relation for the critical points of $\obj$ by grouping points connected by a null-energy path.

\begin{proposition}
   The relation $\sim$ defined for any $\point, \pointalt \in \crit\obj$ as
\begin{equation}
		\point \sim \pointalt \iff \dquasipot(\point, \pointalt) = \dquasipot(\pointalt, \point) = 0\,
\end{equation}
  is an equivalence relation on $\crit\obj$.
\end{proposition}
\begin{proof}
\begin{itemize}
		\item Reflexivity: $\dquasipot(\point, \point) = 0$ by \cref{lem:basic_prop_flow} since the flow started at $\point \in \crit\obj$ is constant.
		\item Symmetry: this follows from the definition of $\sim$.
		\item Transitivity:  for any $\point, \pointalt, \pointaltalt \in \crit\obj$, we have by construction of $\dquasipot$ that
\begin{equation}
				\dquasipot(\point, \pointaltalt) \leq \dquasipot(\point, \pointalt) + \dquasipot(\pointalt, \pointaltalt) \eqdot
\end{equation}
\end{itemize}
Therefore, if $\point \sim \pointalt$ and $\pointalt \sim \pointaltalt$, then $\dquasipot(\point, \pointaltalt) = 0$.
$\dquasipot(\pointaltalt, \point) = 0$ follows with a symmetric argument and thus $\point \sim \pointaltalt$.
\end{proof}

Near critical points of $\obj$, the Lagrangian $\lagrangianalt$ is actually very regular.

\begin{lemma}
	\label{lem:regularity_lagrangian}
	For any $\point \in \vecspace$, there exists $\margin > 0$ such that $\lagrangianalt$ is finite and jointly Lipschitz continuous on $\ball(\point, \margin) \times \ball(0, \margin)$.

	Moreover, the following supremum is finite:
\begin{equation}
\sup \setdef*{\frac{\lagrangian(\pointalt, \vel)}{\norm{\vel}^{2}}}{ \pointalt \in \ball(\point, \margin) \cap \crit\obj\,, \vel \in \ball(0, \margin)} < \infty\eqdot
\end{equation}
\end{lemma}

\begin{proof}
	Take $\point \in \vecspace$.
	We apply the implicit function theorem to the equation 
\begin{equation}
		\grad_\vel \hamiltalt(\pointalt, \velalt) = \vel\,,
		\label{eq:lemma_regularity_lagrangian_implicit_eq}
\end{equation}
	in the variables $(\pointalt, \vel, \velalt) \in \vecspace \times \dspace \times \dspace$.

	We derive that 
\begin{equation}
	\grad_\vel \hamiltalt(\pointalt, \vel) =   \frac{\ex \bracks*{ \noise(\pointalt, \sample) \exp \parens*{\inner{\vel, \noise(\pointalt, \sample)}}}}{\ex \bracks*{\exp \parens*{\inner{\vel, \noise(\pointalt, \sample)}}}} 
\end{equation}
	and thus $(\point, 0, 0)$ is solution of \eqref{eq:lemma_regularity_lagrangian_implicit_eq} since
\begin{equation}
		\grad_\vel \hamiltalt(\point, 0) = \ex \bracks*{ \noise(\point, \sample) } = 0\eqdot
\end{equation}
	Moreover, $\Hess_\vel \hamiltalt(\point, 0) = \ex[\noise(\point, \sample) \noise(\point, \sample)^\top]$ which is positive definite and thus invertible by the blanket assumptions.

	Hence, we can apply the implicit function theorem to get that there exists $\margin > 0$, $\velalt \from \ball(\point, \margin) \times \ball(0, \margin) \to \dspace$ $\contdiff{2}$ such that, for any $\pointalt \in \ball(\point, \margin)$, $\vel \in \ball(0, \margin)$,
\begin{equation}
		\grad_\vel \hamiltalt(\pointalt, \velalt(\pointalt, \vel)) = \vel\eqdot
\end{equation}
	Therefore, for any $\pointalt \in \ball(\point, \margin)$, $\vel \in \ball(0, \margin)$, since $  \lagrangianalt(\pointalt, \vel) = \hamiltalt(\pointalt, \cdot)^*(\vel) $,  we have
\begin{equation}
		\lagrangianalt(\pointalt, \vel) = \inner{\vel, \velalt(\pointalt, \vel)} - \hamiltalt(\pointalt, \velalt(\pointalt, \vel))\,,
		\label{eq:lemma_regularity_lagrangian_explicit_expr}
\end{equation}
	which is finite and $\contdiff{2}$ on $\ball(\point, \margin) \times \ball(0, \margin)$.
		Therefore, $\lagrangianalt$ is actually $\tmplips$-jointly Lipschitz continuous on $\clball(\point, \margin/2) \times \clball(0, \margin/2)$.

  For the second part of the lemma, note that the implicit function theorem also ensures that there is $\V \subset \dspace$ a neighborhood of $0$ such that, 
		for any $\pointalt \in \ball(\point, \margin)$, $\vel \in \ball(0, \margin)$, $\velalt(\pointalt, \vel)$ is the unique solution of \cref{eq:lemma_regularity_lagrangian_implicit_eq} in $\V$.
		But, for any $\pointalt \in \ball(\point, \margin)$ and $\vel = 0$, $\velalt  = 0$ is a solution of \cref{eq:lemma_regularity_lagrangian_implicit_eq} in $\V$ so that necessarily $\velalt(\pointalt, 0) = 0$, and, as a consequence, $\grad_\vel \lagrangianalt(\pointalt, 0) = 0$.

		Hence, for any $\pointalt \in \clball(\point, \margin/2)$, $\vel \in \clball(0, \margin/2)$,
\begin{align}
\lagrangianalt(\pointalt, \vel) 
	&= \lagrangianalt(\pointalt, \vel) - \lagrangianalt(\pointalt, 0) - \inner{\vel, \grad_\vel \lagrangianalt(\pointalt, 0)}
	\notag\\
	&\leq \half \sup_{\clball(\point, \margin/2) \times \clball(0, \margin/2)}\norm*{\Hess_\vel \lagrangianalt} \norm{\vel}^{2}\eqdot
\end{align}
To conclude, it suffices to note that, for any $\pointalt \in \ball(\point, \margin) \cap \crit\obj$, $\grad \obj(\point) = 0$ and therefore $\lagrangian(\pointalt, \cdot)$ and $\lagrangianalt(\pointalt, \cdot)$ coincide.
\end{proof}

\beginrev
\begin{lemma}
\label{lem:continuity_rate}
For any $\point \in \crit(\obj)$, for any $\epsilon > 0$, there exists $\margin > 0$ such that, for any $\pointalt \in \ball(\point, \margin)$, $\rate(\point, \pointalt) \leq \epsilon$ and $\rate(\pointalt, \point) \leq \epsilon$.
\end{lemma}

\begin{proof}
Take $\point \in \vecspace$ such that $\grad \obj(\point) = 0$.
By \cref{lem:regularity_lagrangian}, there exists $\margin > 0$ such that $\lagrangianalt$ is finite and $\tmplips$-jointly Lipschitz continuous on $\ball(\point, \margin) \times \ball(0, \margin)$.
In particular,  for any $\pointalt \in \ball(\point, \margin)$, $\vel \in \ball(0, \margin)$,
\begin{equation}
\lagrangianalt(\pointalt, \vel)
	\leq \tmplips \norm{\vel}
	\eqdot
\end{equation}
For any integer $k \geq 2$, by continuity of $\grad \obj$ (\cref{asm:obj-weak}), there is $0 <\marginalt_k < \margin/k$ such that, for every $\pointalt \in \ball(\point, \marginalt_k)$, $\norm{\grad \obj(\pointalt)} \leq  \margin / k$.
Then, for any $\pointalt \in \ball(\point, \marginalt_k)$, $\vel \in \ball(0, \marginalt_k)$,
\begin{align}
\lagrangian(\pointalt, \vel) &= \lagrangianalt(\pointalt, \vel + \grad \obj(\pointalt))
	\notag\\
	&\leq \tmplips \parens*{\norm{\vel + \grad \obj(\pointalt)}}
	\notag\\
    &\leq \tmplips \parens*{\norm{\vel} + \norm{\grad \obj(\pointalt)}}
    \notag\\
    &\leq \frac{2 \tmplips \margin}{k} \eqdot
\end{align}
Take $\pointalt \in \ball(\point, \marginalt_k)$ and define the path $\pth \in \contfuncs([0, 1], \vecspace)$ by $\pth_{\time} = \point + \time (\pointalt - \point)$.
Since for any $\time \in [0, 1]$, $\pth_{\time} \in \ball(\point, \marginalt_k)$ and $\dot \pth_{\time} = \pointalt - \point \in \ball(0, \marginalt_k)$, we can bound the cost as 
\begin{align}
\rate(\point, \pointalt)
    \leq
\action_{0,1}(\pth)
    = \int_{0}^{1} \lagrangian(\pth_{\time}, \dot \pth_{\time}) \dd \time
    \leq  
    \frac{2 \tmplips \margin}{k} \eqdot
\end{align}
As a result, taking $k$ large enough yields the existence of $\marginalt_k > 0$ such that, for any $\pointalt \in \ball(\point, \marginalt_k)$, $\rate(\point, \pointalt) \leq \epsilon$.
The reverse inequality is obtained by the same argument.
\end{proof}
\endedit

The following lemma relates the Lagrangian to the gradient and our noise structure.

\begin{lemma}
	\label{lem:csq_subgaussian}
		For any $\point \in \vecspace$, $\vel \in \dspace$,
\begin{equation}
		\lagrangian(\point, \vel) \geq \frac{\norm{\vel + \grad \obj(\point)}^{2}}{2\bdvar(\obj(\point))}\eqdot
\end{equation}
\end{lemma}
\begin{proof}
		For any $\point \in \vecspace$, $\vel \in \dspace$, we have that
\begin{align}
			\hamilt(\point, \vel) &= - \inner{\vel, \grad \obj(\point)} + \hamiltalt(\point, \vel)
						\leq - \inner{\vel, \grad \obj(\point)} + \half \bdvar(\obj(\point)) \norm{\vel}^{2}\eqdot
\end{align}
	Taking the conjugate then yields that
\begin{equation}
\lagrangian(\point, \vel) \geq \frac{\norm{\vel + \grad \obj(\point)}^{2}}{2\bdvar(\obj(\point))}
	\eqdot
	\qedhere
\end{equation}
\end{proof}

Now, let us define a potential function $\bdpot$  on $\vecspace$ that uses the minimal displacement energy between two points that will be heavily used in the proofs.

\begin{definition}[Potential]\label{def:bdpot}
	Define, for $\state\in\vecspace$
\begin{align}
	\label{eq:defbdpot}
\bdpot(\state) = 2 \bdprimvar(\obj(\state))
\end{align}
where $\bdprimvar : \revise{\R} \to \R$ is a twice continuously differentiable primitive of $1/\bdvar$.
\end{definition}

\begin{lemma}
	\label{lem:B_is_descent}
	For any $\point, \pointalt \in \vecspace$,
\begin{equation}
    \bdpot(\pointalt) - \bdpot(\point) \leq \dquasipot(\point, \pointalt)
    \eqdot
\end{equation}
\end{lemma}
\begin{proof}
	By \cref{def:dquasipot}, there exists $\horizon \geq 1$, $\pth \in \contfuncs([0, \horizon], \vecspace)$ such that $\pth_{0} = \point$, $\pth_\horizon = \pointalt$ and $\action_{[\tstart,\horizon]}(\pth) \leq \dquasipot(\point, \pointalt) + \precs$.
	Then, by Young's inequality, we have that
\begin{align}
\bdpot(\pointalt) - \bdpot(\point)
	&= 2 \int_{0}^{\horizon} \frac{\inner{\dot \pth_{\time}, \grad \obj(\pth_{\time})}}{\bdvar(\obj(\pth_{\time}))} \dd \time
	\notag\\
    &\revise{\leq \int_{0}^{\horizon} \parens*{
 \frac{\inner{\dot \pth_{\time}, \grad \obj(\pth_{\time})}}{\bdvar(\obj(\pth_{\time}))}
+ \frac{\norm{\grad \obj(\pth_{\time})}^{2} + \norm{\dot \pth_{\time}}^{2}}{2\bdvar(\obj(\pth_{\time}))}} \dd \time
}
    \notag\\
    &= \int_{0}^{\horizon} \frac{\norm{\dot \pth_{\time} + \grad \obj(\pth_{\time})}^{2}}{2\bdvar(\obj(\pth_{\time}))} \dd \time
	\notag\\
	&\leq  \int_{0}^{\horizon}  \lagrangian(\pth_{\time}, \dot \pth_{\time})  \dd \time
	= \action_{[\tstart,\horizon]}(\pth)
\end{align}
where we used \cref{lem:csq_subgaussian} in the last inequality.
Finally, out choice of $\pth$ implies that
\begin{equation}
\bdpot(\pointalt) - \bdpot(\point)
	\leq \dquasipot(\point, \pointalt) +  \precs
\end{equation}
which concludes the proof.
\end{proof}

\begin{lemma}[Equivalence classes are closed]
	\label{lem:equivalence_classes_closed}
	Equivalence classes of $\sim$ are closed in $\vecspace$. As a consequence, equivalence classes are compact.
\end{lemma}
\begin{proof}
Let $\point \in \crit\obj$, and take any sequence $(\pointalt_\runB)_{\runB \geq \start}$ in $\crit\obj$ such that $\point \sim \pointalt_\runB$ for every $\runB \geq \start$ and which converges to some $\pointalt \in \vecspace$.
To show that the equivalence classes are closed, we need to show that $\point \sim \pointalt$; then compacity follow directly since the equivalence classes are subsets of $\crit\obj$ which is compact by assumption.


   Since  $\crit\obj$ is closed, it holds that $\pointalt$ belongs to $\crit\obj$.
	We now show that both $\dquasipot(\point, \pointalt)$ and $ \dquasipot(\pointalt, \point)$ are null.
	We only show that $\dquasipot(\point, \pointalt) = 0$ since the proof for the other equality is symmetric.

    \revise{By \cref{lem:continuity_rate}, for any $\precs>0$ there is a neighborhood of $\pointalt$ on which $\rate(\cdot, \pointalt) \leq \precs$.}
	Take $\runB$ large enough so that $\pointalt_\runB$ belongs to this neighborhood. Since $\point \sim \pointalt_\runB$, there exists $\nRuns \geq 1$, $\dpth \in (\vecspace)^\nRuns$ such that $\dpth_\start = \point$, $\dpth_{\nRuns - 1} = \pointalt_\runB$ and $\daction_\nRuns(\dpth) \leq \precs$.
	Then, the path $\dpthalt \defeq (\dpth_{0}, \dots, \dpth_{\nRuns - 1}, \pointalt) \in (\vecspace)^{\nRuns + 1}$ and satisfies $\dpthalt_\start = \point$, $\dpthalt_{\nRuns} = \pointalt$ and
\begin{equation}
    \daction_{\nRuns + 1}(\dpthalt) \revise{=} \daction_\nRuns(\dpth) + \rate(\pointalt_\runB, \pointalt) \leq 2\precs\eqdot
\end{equation}
	Hence, we have shown that, for any $\precs > 0$, $\dquasipot(\point, \pointalt) \leq 2\precs$ so that $\dquasipot(\point, \pointalt) = 0$.
\end{proof}

\begin{lemma}
	\label{lem:W_nbd}
    \revise{
	For any set $\closed \subset \crit\obj$ and $\state \in \points$, define
    \begin{align}
        \rate(\state, \closed) &\defeq \inf_{\pointalt \in \closed} \rate(\state, \pointalt)\notag\\
        \rate(\closed, \state) &\defeq \inf_{\pointalt \in \closed} \rate(\pointalt, \state)\eqdot
    \end{align}
    Then, for any  $\closed \subset \crit\obj$ and $\radius > 0$, the set 
}
    \begin{equation}
		\nbdalt_\radius(\closed) \defeq \setdef{\state \in \points}{\rate(\state, \closed) < \radius,\, \rate(\closed, \state) < \radius}
\end{equation}
\revise{is a neighborhood of  $\closed$}.
\end{lemma}
\begin{proof}
    \revise{\Cref{lem:continuity_rate} ensures that, for any $\point \in \closed$, there exists $\margin_\point > 0$ such that, for any $\pointalt \in \ball(\point, \margin_\point)$, $\rate(\point, \pointalt) < \radius$ and $\rate(\pointalt, \point) < \radius$. $\V \defeq \bigcup_{\point \in \closed} \ball(\point, \margin_\point)$ is then an open set that contains $\closed$ and is included in $\nbdalt_\radius(\closed)$.}
\end{proof}

This lemma is adapted and significantly expanded from \citet[\S1.5, Lem.~5.2]{Kif88} to handle both the unboundedness of the space and the fact that $\dquasipot$ is neither \ac{lsc} nor \ac{usc}.

\begin{lemma}
	\label{lem:equivalence_classes_dpth_stay_close}
 Let $\cpt$ be an equivalence class of $\sim$. Then, for any $\precs > 0$, there is some $\nRuns \geq 1$ such that, for any $\state, \statealt \in \cpt$,
	there is $\dpth \in (\vecspace)^\nRuns$ such that $\dpth_\start = \state$, $\dpth_{\nRuns - 1} = \statealt$, $\daction_\nRuns(\dpth) < \precs$ and $\max_{\start \leq  \run < \nRuns} d(\dpth_\run, \cpt) < \precs$.
\end{lemma}
\begin{proof}
    \revise{Fix $\precs > 0$.}
	$\cpt$ is made of critical points of $\obj$ so that by \cref{lem:continuity_rate}, 
    \revise{for any $\point \in \cpt$, there exists $\margin_\point > 0$ such that, for any $\pointalt \in \ball(\point, \margin_\point)$, $\rate(\point, \pointalt) < \precs$ and $\rate(\pointalt, \point) < \precs$. 
   }


	\revise{Since the open sets $\ball(\point, \margin_\point)$ for $\point \in \cpt$ form an open cover of $\cpt$ and $\cpt$ is compact (\cref{lem:equivalence_classes_closed}), there exists a finite number of points $\point_\idx \in \cpt$, $\idx \in \indices$ such that
    }
\begin{equation}
    \cpt \subset \bigcup_{\idx \in \indices} \ball(\point_\idx, \revise{\margin_{\point_\idx}})\eqdot
		\label{eq:lemma_nhd_eqcl_covering_cpt}
\end{equation}

	We first show that the result holds for the points $\point_\idx$ before explaining why it actually suffices for the general case.

	Fix $\idx, \idxalt \in \indices$.
	Since $\point_\idx \sim \point_\idxalt$, $\dquasipot(\point_\idx, \point_\idxalt) = 0$ and therefore there exists sequences $\dpth^\runB \in (\vecspace)^{\nRuns_\runB}$ with $\dpth^\runB_\start = \point_\idx$, $\dpth^\runB_{\nRuns_\runB - 1} = \point_\idxalt$ such that $\daction_{\nRuns_\runB}(\dpth^\runB) \to 0$ as $\runB \to \infty$.
	For the sake of contradiction, assume that from some $\runB$, there always exists $0 \leq \run < \nRuns_\runB$ such that $d(\dpth^\runB_\run, \cpt) \geq \precs$. Define $\run_\runB$ as the smallest $\run$ such that it happens, which is necessarily greater or equal to $1$. Note that, by definition, $\dpth^\runB_{\run_\runB - 1}$ satisfies $d(\dpth^\runB_{\run_\runB - 1}, \cpt) \leq  \precs$.

	Define $\cptalt \defeq \setdef{\point \in \vecspace}{d(\point, \cpt) \leq \precs}$ which is compact.
	However, for $\runB$ large enough, $\rate(\dpth^\runB_{\run_\runB - 1}, \dpth^\runB_{\run_\runB}) 
	\leq \daction_{\nRuns_\runB}(\dpth^\runB) \leq 1$ so that $(\dpth^\runB_{\run_\runB-1}, \dpth^\runB_{\run_\runB}) 
	$ belongs to $\pths_{2}^{\cptalt}(1)$, which is compact by \cref{cor:ldp_iterates_full}.
	Therefore, one can extract a subsequence from $(\dpth^\runB_{\run_\runB-1},\dpth^\runB_{\run_\runB},
	)_{\runB \geq 1}$ that converges to some $(\stateA, \stateB
    )\in \vecspace^2$ that satisfies $d(\stateB, \cpt) \geq \precs$. \revise{Without loss of generality, at the cost of replacing the original sequence by the extracted subsequence, we assume that the sequence $(\dpth^\runB_{\run_\runB-1},\dpth^\runB_{\run_\runB})_{\runB \geq 1}$ itself converges to $(\stateA, \stateB)$.}
        By \ac{lsc} of $\rate$, one has that
\begin{align}
\rate(\stateA, \stateB)
	&\leq \liminf_{\runB \to \infty} \rate(\dpth^\runB_{\run_\runB-1}, \dpth^\runB_{\run_\runB})
	\notag\\
	&\leq \liminf_{\runB \to \infty} \daction_{\nRuns_\runB}(\dpth^\runB) = 0\,,
\end{align}
	so that $\rate(\stateA, \stateB) = 0$. 
	We now show that $\stateA = \stateB$ and that it is a critical point.
	By \cref{lem:B_is_descent}, we have that
\begin{subequations}
\begin{align}
\bdpot(\dpth^\runB_{\run_\runB-1}) - \bdpot(\state_\idx)
	&\leq \dquasipot(\state_\idx, \dpth^\runB_{\run_\runB-1})
	\leq \daction_{\nRuns_\runB}(\dpth^\runB)
	\\
\bdpot(\state_\idxalt) - \bdpot(\dpth^\runB_{\run_\runB})
	&\leq \dquasipot(\state_\idxalt, \dpth^\runB_{\run_\runB})
	\leq \daction_{\nRuns_\runB}(\dpth^\runB)
	\,,
\end{align}
\end{subequations}
so, taking the limit $\runB \to \infty$ yields
\begin{subequations}
\begin{align}
\bdpot(\stateA) - \bdpot(\state_\idx)
	&\leq 0
	\\
\bdpot(\state_\idxalt) - \bdpot(\stateB)
	&\leq 0
	\eqdot
\end{align}
\end{subequations}

	However, the fact that $\state_\idx$ and $\state_\idxalt$ are equivalent and \cref{lem:B_is_descent} imply that $\bdpot(\state_\idx) = \bdpot(\state_\idxalt)$ so that $\bdpot(\stateA) \leq  \bdpot(\stateB)$. Since $\bdprimvar$ is increasing, we have that $\obj(\stateA) \leq \obj(\stateB)$.

	But we have that $\rate(\stateA, \stateB)= 0$ so that $\stateB = \flowmap_{1}(\stateA)$. Therefore, if $\grad \obj(\stateA) \neq 0$, we would have
\begin{equation}
		\obj(\stateB) - \obj(\stateA) = - \int_{0}^{1} \norm{\grad \obj(\flowmap_{\time}(\stateA)}^{2} \dd \time < 0\,,
\end{equation}
	which would be a contradiction. Therefore, $\grad \obj(\stateA) = 0$ and $\stateA = \stateB$.
	Since $\stateA$ is a critical point, to show that it belongs to $\cpt$, it suffices to show that $\state_\idx \sim \stateA$.

    Take ${\revise{\precsalt}} > 0$. \revise{By \cref{lem:W_nbd},} $\W_{\revise{\precsalt}}(\setof{\stateA})$ is \revise{a neighborhood} of $\stateA$.
	Since $\dpth^\runB_{\run_\runB}$ converges to $\stateA$, for $\runB$ large enough, $\dpth^\runB_{\run_\runB}$ belongs to $\W_{\revise{\precsalt}}(\setof{\stateA})$.
	$(\dpth^\runB_{0}, \dots, \dpth^\runB_{\run_\runB}, \stateA)$ and $(\stateA, \dpth^\runB_{\run_\runB}, \dots, \dpth^\runB_{\nRuns_\runB - 1})$ are, respectively, paths from $\state_\idx$ to $\stateA$ and from $\stateA$ to $\state_\idxalt$ with action cost of at most $\daction_{\nRuns_\runB}(\dpth^\runB) + {\revise{\precsalt}}$ so that, for $\runB$ large enough, $\dquasipot(\state_\idx, \stateA) \leq 2 {\revise{\precsalt}}$ and $\dquasipot(\stateA, \state_\idxalt) \leq 2 {\revise{\precsalt}}$.
	
	Hence, we have shown that $\dquasipot(\state_\idx, \stateA) = \dquasipot(\stateA, \state_\idxalt) = 0$. 
	 Then, we also have
\begin{equation}
		0 \leq \dquasipot(\stateA, \point_\idx) \leq \dquasipot(\stateA, \point_\idxalt) + \dquasipot(\point_\idxalt, \point_\idx) = 0\,,
\end{equation}
	and therefore $\state \sim \point_\idx$ and thus $\state$ belongs to $\cpt$. This is a contradiction.

	Therefore, there must exist some sequence $\dpth^{\idx,\idxalt} \in (\vecspace)^{\nRuns_{\idx, \idxalt}}$ with $\dpth^{\idx,\idxalt}_\start = \point_\idx$, $\dpth^{\idx,\idxalt}_{\nRuns_{\idx, \idxalt} - 1} = \point_\idxalt$ such that $\daction_{\nRuns_{\idx, \idxalt}}(\dpth^{\idx,\idxalt}) < \precs$ and $\max_{0 \leq \run < \nRuns_{\idx, \idxalt}} d(\dpth^{\idx,\idxalt}_\run, \cpt) < \precs$.

    Finally, for the general class, consider $\state, \statealt \in \cpt$. By \cref{eq:lemma_nhd_eqcl_covering_cpt}, there exists $\idx, \idxalt \in \indices$ such that $\state \in \ball(\point_\idx, \revise{\margin_{\point_\idx}})$, $\statealt \in \ball(\point_\idxalt, \revise{\margin_{\point_\idxalt}})$.
    \revise{
	Consider $\dpth \in (\vecspace)^{\nRuns_{\idx, \idxalt}+2}$ a
	extension of $\dpth^{\idx,\idxalt}$ defined by
\begin{equation}
\dpth_\run
	= \begin{cases*}
		\state
			&\quad
			if $\run = 0$,
		\\
        \dpth^{\idx,\idxalt}_{\run-1}
			&\quad
			if $0 <  \run \leq  \nRuns_{\idx, \idxalt}$,
		\\
		\statealt
			&\quad
			if $\run = \nRuns_{\idx, \idxalt} + 1$,
            \eqdot
	\end{cases*}
\end{equation}
that still satisfies $\max_{0 \leq \run < \nRuns_{\idx, \idxalt}+2} d(\dpth_\run, \cpt) < \precs$.
Then, one has that
\begin{equation}
		\daction_{\nRuns_{\idx, \idxalt}+2}(\dpth) < \daction_{\nRuns_{\idx, \idxalt}}(\dpth^{\idx,\idxalt}) + 2\precs \leq 3 \precs\,,
\end{equation}
	which concludes the proof.
    }
\end{proof}

The following lemma is inspired by \citet[Prop.~3.3.11]{alongiRecurrenceTopology2007}.
\begin{lemma}
	\label{lem:equivalence_classes_connected}
	Equivalence classes are connected.
\end{lemma}

\begin{proof}
	Fix $\cpt$ an equivalence class of $\sim$.

	For the sake of contradiction, assume that there are $\openA$, $\openB$ disjoint open sets of $\vecspace$ such that both $\openA \cap \cpt$ and $\openB \cap \cpt$ are non-empty and $\cpt = (\openA \cup \openB) \cap \cpt$.

	Take $\pointA \in \openA \cap \cpt$ and $\pointB \in \openB \cap \cpt$.
	By \cref{lem:equivalence_classes_dpth_stay_close}, there exists a sequence of paths $\dpth^\runB \in (\vecspace)^{\nRuns_{\runB}}$ with $\dpth^\runB_\start = \pointA$, $\dpth^\runB_{\nRuns_{\runB} - 1} = \pointB$ for $\runB \geq \start$ such that $\daction_{\nRuns_{\runB}}(\dpth^\runB) \to 0$ as $\runB \to \infty$ and $\max_{0 \leq \run < \nRuns_{\runB}} d(\dpth^\runB_\run, \cpt) \to 0$ as $\runB \to \infty$.
	Define $\seqA_\runB$ as the last point of $\dpth^\runB$ that belongs to $\openA$ and $\seqB_\runB$ as the successor of $\seqA_\runB$ in $\dpth^\runB$. Formally, $\seqA_\runB$ and $\seqB_\runB$ are defined by $\seqA_\runB = \dpth^\runB_{\idx_\runB}$ and $\seqB_\runB = \dpth^\runB_{\idx_\runB + 1}$ where $\idx_\runB = \max \setdef{\idx < \nRuns_{\runB} }{\dpth^\runB_\idx \in \openA}$.

	By construction, both $d(\seqA_\runB, \cpt)$ and $d(\seqB_\runB, \cpt)$ go to zero as $\runB \to \infty$. In particular, both sequences lie in $\U_{1}(\cpt)$ from some point onward, which is relatively compact, so they admit convergent subsequences. Without loss of generality, thus assume that $\seqA_\runB \to \seqA$ and $\seqB_\runB \to \seqB$ as $\runB \to \infty$ and $\seqA$, $\seqB$ belong to $\cpt$.

	Since $\seqA_\runB$ belongs to $\openA$ for all $\runB \geq \start$, $\seqA_\runB$ is never in $\openB$ so that $\seqA$ does not belong to $\openB$ either. Since $\seqA$ belongs to $\cpt$, it must thus belong to $\openA$. Similarly, $\seqB$ must belong to $\openB$.

However, by construction, we have that $\rate(\seqA_\runB, \seqB_\runB) \leq \daction_{\nRuns_{\runB}}(\dpth^\runB) \to 0$ as $\runB \to \infty$ so that $\rate(\seqA_\runB, \seqB_\runB)$ converges to $0$ as $\runB \to \infty$ as well. By \ac{lsc} of $\rate$ (\cref{cor:ldp_iterates_full} with $\nRuns=1$), $\rate(\seqA, \seqB) = 0$ so that $\seqB= \map(\seqA)$ by \cref{lem:basic_prop_map}. But since $\seqA$ belongs to $\cpt$, it is a critical point of $\obj$ and therefore $\seqA = \map(\seqA)$. Hence $\seqA = \seqB$ with $\seqA \in \openA$ and $\seqB \in \openB$, which is a contradiction.
\end{proof}

\begin{lemma}
	\label{lem:connected_comp_included_in_single_eqcl}
  Any connected component of $\crit\obj$ is included in a single equivalence class.
\end{lemma}
\begin{proof}
	Let $\connectedcomp$ be a connected component of $\crit\obj$ and fix $\point, \pointalt \in \connectedcomp$. 
	We begin by considering a stronger version of \cref{asm:obj}\cref{asm:obj-weak-crit}, namely that there exists $\pth \in \contfuncs([0, 1], \connectedcomp)$ absolutely continuous such that $\pth_{0} = \point$, $\pth_{1} = \pointalt$, \ie such that $\pth$ is differentiable almost everywhere with $\int_{0}^{1} \norm{\dot \pth_{\time}} \dd \time < \infty$. We show that $\point \sim \pointalt$, \ie that $\dquasipot(\point, \pointalt) = \dquasipot(\pointalt, \point) = 0$.
	
	Let us begin by showing that $\dquasipot(\point, \pointalt) = 0$.

	Since $\crit\obj$ is compact and $\connectedcomp$ is closed as connected component of a closed set, $\connectedcomp$ is compact. By invoking \cref{lem:regularity_lagrangian} at every point of $\connectedcomp$ and extracting a finite covering from the family of balls obtained, we have that, there exists $\margin > 0$, $\tmplips > 0$ such that, for every $\state \in \connectedcomp$, $\vel \in \ball(0, \margin)$,
\begin{equation}
	\lagrangian(\state, \vel) \leq \tmplips {\norm{\vel}^{2}} \eqdot
\end{equation}

	Fix $\precs \in (0, \margin)$ and define, for any $\time \in [0, 1]$,
\begin{equation}
		\timeD_{\time} \defeq \int_{0}^\time \frac{\max(\norm{\dot \pth_{\timealt}}, 1)}{\precs} \dd \timealt \eqdot
\end{equation}
	We have that $\timeD$ is an increasing bijection from $[0, 1]$ to $[0, \timeD_{1}]$ and is absolutely continuous with $\dot \timeD_{\time} = {\max(\norm{\dot \pth_{\time}}, 1)} / {\precs}$ almost everywhere.
	Consider $\timeC \from [0, \timeD_{1}] \to [0, 1]$ the inverse of $\timeD$, which is absolutely continuous with $\dot \timeC_{\time} = {\precs} / {\max(\norm{\dot \pth_{\timeC_{\time}}}, 1)}$ almost everywhere.
	Define $\pthalt \in \contfuncs([0, \timeD_{1}], \vecspace)$ by $\pthalt_{\time} = \pth_{\timeC_{\time}}$ for any $\time \in [0, \timeD_{1}]$.
	Then, $\pthalt$ is absolutely continuous and, for any $\time \in [0, \timeD_{1}]$,
\begin{equation}
		\dot \pthalt_{\time} = \dot \pth_{\timeC_{\time}} \dot \timeC_{\time} = \frac{\precs \dot \pth_{\timeC_{\time}}}{\max(\norm{\dot \pth_{\timeC_{\time}}}, 1)}\,
\end{equation}
   which has norm less than $\precs<\margin$.

	Therefore, we have that,
\begin{align}
\action_{0,\timeD_{1}}(\pthalt)
	&= \int_{0}^{\timeD_{1}} \lagrangian(\pthalt_{\time}, \dot \pthalt_{\time}) \dd \time
	\notag\\ 
	&\leq \tmplips \int_{0}^{\timeD_{1}} \norm{\dot \pthalt_{\time}}^{2} \dd \time
	\notag\\
	&\leq \precs \tmplips \int_{0}^{\timeD_{1}}\norm{\dot \pthalt_{\time}} \dd \time
	\notag\\
	&= \precs \tmplips \int_{0}^{\timeD_{1}} \dot \timeC_{\time}  \norm{\dot \pth_{\timeC_{\time}}} \dd \time
	\notag\\
	&= \precs \tmplips \int_{0}^{1} \norm{\dot \pth_{\time}} \dd \time
\end{align}
where the last equality is obtained by the change of variable $\time \gets \timeC_{\time}$.
Thus, we have shown that, for any $\precs \in (0, \margin)$,
\begin{equation}
\dquasipot(\point, \pointalt)
	\leq \precs \tmplips \int_{0}^{1} \norm{\dot \pth_{\time}} \dd \time
\end{equation}
with $\int_{0}^{1} \norm{\dot \pth_{\time}} \dd \time < \infty$ by construction so that $\dquasipot(\point, \pointalt) = 0$.

Reversing the roles of $\point$ and $\pointalt$ and considering the path $(\pth_{1 - \time})_{\time \in [0, 1]}$ then yields that $\dquasipot(\pointalt, \point) = 0$.
Therefore, we have shown that if there exists $\pth \in \contfuncs([0, 1], \connectedcomp)$ absolutely continuous such that $\pth_{0} = \point$, $\pth_{1} = \pointalt$, then $\point \sim \pointalt$.

We now relax our assumption on the paths from absolute continuity to piecewise absolute continuity (\cref{asm:obj-weak}).
For $\point, \pointalt \in \connectedcomp$, by assumption, there exists $\pth \in \contfuncs([0, 1], \connectedcomp)$ such that $\pth_{0} = \point$, $\pth_{1} = \pointalt$ and such that it is piecewise absolutely continuous: $\pth$ is differentiable almost everywhere and there exists $0 = \time_{0} < \time_{1} < \dots < \time_\nRuns = 1$ such that $\dot \pth$ is integrable on every closed interval of $(\time_\run, \time_{\run + 1})$ for $\run = 0, \dots, \nRuns - 1$.
	Take $0 \leq \run < \nRuns - 1$ and $\time_\run < \timealt < \time < \time_{\run + 1}$.
$\pth$ restricted to $[\timealt, \time]$ is absolutely continuous so that, by the previous case, all the points of $\setdef{\pth_{\timealtalt}}{\timealtalt \in [\timealt, \time]}$ are included in a single equivalence class $\cpt$.
Taking $\timealt \to \time_\run$ and $\time \to \time_{\run + 1}$ yields that $\setdef{\pth_{\timealtalt}}{\timealtalt \in (\time_\run, \time_{\run + 1})}$ is included in $\cpt$.
Moreover, by continuity of $\pth$, $\pth_{\time_\run}$ and $\pth_{\time_{\run + 1}}$ belong to the closure of $\cpt$, which is closed by \cref{lem:equivalence_classes_closed}, so that $\pth_{\time_\run}$ and $\pth_{\time_{\run + 1}}$ belong to $\cpt$ as well.
Therefore, $\pth_{\time_\run} \sim \pth_{\time_{\run + 1}}$.
	By transitivity, we obtain that $\point = \pth_{0} \sim \pth_{\time_{1}} \sim \dots \sim \pth_{\time_{\nRuns - 1}} \sim \pth_{1} = \pointalt$ so that $\point \sim \pointalt$.
\end{proof}

Combining \cref{lem:equivalence_classes_connected,lem:connected_comp_included_in_single_eqcl}, we have shown that  any connected component of $\crit\obj$ is included in a single equivalence class and since they are connected, two distinct connected components of $\crit\obj$ cannot belong to the same equivalent class; hence, we have that they coincide.


\begin{corollary}
	\label{corollary:equivalence_classes_are_connected_components_of_crit}
    \revise{The equivalence classes of $\sim$ are exactly connected components of $\crit\obj$.}
\end{corollary}

We end this section by providing a sufficient condition for $\dquasipot(\point, \pointalt)$ to be finite.
\begin{lemma}
	\label{lem:finite_B}
    Consider $\point, \pointalt \in \vecspace$ and assume that there exists $\horizon > 0$, \revise{$\pth \in \contfuncs([0, \horizon], \vecspace)$ absolutely continuous} such that $\pth_{0} = \point$, $\pth_\horizon = \pointalt$ and such that, for every $\time \in [0, \horizon]$, $\grad \obj(\pth_{\time})$ is in the interior of the closed convex hull of the support of $\noise(\pth_{\time}, \sample)$, \ie
\begin{equation}
		\label{eq:lem:finite_B_condition_grad}
		\grad \obj(\pth_{\time}) \in \intr \clconv \supp \noise(\pth_{\time}, \sample)\eqdot
\end{equation}
	Then, $\dquasipot(\point, \pointalt) < \infty$.
\end{lemma}
\begin{proof}
        First, fix $\time \in [0, \horizon]$ and consider the exponential family of distributions
        \begin{equation}
            \propto e^{\inner{\mom}{\noise(\pth_{\time}, \sample)}} \prob(\revise{d} \sample)
        \end{equation}
        with parameter $\mom \in \dspace$.
        We begin by applying \citet[Thm.~3.6]{brownFundamentalsStatisticalExponential1986}. The family above is indeed minimal in the sense of \citet[\S1.1]{brownFundamentalsStatisticalExponential1986} since the interior of $\clconv \supp \noise(\pth_{\time}, \sample)$ is non-empty by assumption.
    Combined with \cref{eq:lem:finite_B_condition_grad}, \citet[Thm.~3.6]{brownFundamentalsStatisticalExponential1986}
        then implies that, for every $\time \in [0, \horizon]$, $\grad \obj(\pth_{\time})$ belongs to $\grad_\mom \hamiltalt(\pth_{\time}, \vecspace)$ so there exists $\mom_\time \in \dspace$ such that $\grad \obj(\pth_{\time}) = \grad_\mom \hamiltalt(\pth_{\time}, \mom_\time)$.
    As in the proof of \cref{lem:regularity_lagrangian}, we will now invoke the implicit function theorem on the equation	 
\begin{equation}
        \label{eq:lem:finite_B_eq_implicit}
		\grad_\mom\hamiltalt(\point, \mom) = \grad \obj(\point) + \vel\,.
\end{equation}
    For this, need to check that $\Hess_{\mom} \hamiltalt(\pth_\time, \mom_\time)$ is positive definite.
    For notational convenience, denote by $\ex_{\mom_{\time}}$, $\cov_{\mom_{\time}}$ and $\var_{\mom_{\time}}$ the expectation, covariance and variance operators with respect to the reweighted distribution
    \begin{equation}
        \propto e^{\inner{\mom_\time}{\noise(\pth_{\time}, \sample)}} \prob(d \sample)\,.
    \end{equation}
    For any non-zero vector $\uvec \in \dspace$, computing the Hessian of $\hamiltalt$ at $(\pth_{\time}, \mom_{\time})$ yields that
    \begin{align}
        \inner{\uvec}{\Hess_{\mom} \hamiltalt(\pth_{\time}, \mom_{\time}) \uvec}
        &=
        \inner{\uvec}{\cov_{\mom_{\time}}\bracks*{\noise(\pth_{\time}, \sample)} \uvec} \\
        &=
        \var_{\mom_{\time}}\bracks*{\inner{\uvec}{\noise(\pth_{\time}, \sample)}}\\
        &=
        \min_{\coef \in \R}
        \ex_{\mom_{\time}}\bracks*{\parens*{\inner{\uvec}{\noise(\pth_{\time}, \sample)} - \coef}^{2}}\\
        &=
        \min_{\coef \in \R}
        \frac{
            \ex \bracks*{\parens*{\inner{\uvec}{\noise(\pth_{\time}, \sample)} - \coef}^{2} e^{\inner{\mom_{\time}}{\noise(\pth_{\time}, \sample)}} 
        }}{
                \ex \bracks*{e^{\inner{\mom_{\time}}{\noise(\pth_{\time}, \sample)}}}
            }\,.
    \end{align}
    The blanket assumptions  on the noise (\cref{app:subsec:setup}) and Cauchy-Schwarz inequality
    then yield
    \begin{align}
        \inner{\uvec}{\Hess_{\mom} \hamiltalt(\pth_{\time}, \mom_{\time}) \uvec} 
        &\geq
        e^{-2 \norm{\mom_{\time}} \growth \parens*{1 + \norm{\pth_{\time}}}}
        \min_{\coef \in \R} \ex \bracks*{\parens*{\inner{\uvec}{\noise(\pth_{\time}, \sample)} - \coef}^{2}}\\
        &= 
        e^{-2 \norm{\mom_{\time}} \growth \parens*{1 + \norm{\pth_{\time}}}}
        \inner{\uvec}{\revise{\cov}\bracks*{\noise(\pth_{\time}, \sample)} \uvec}\,,
    \end{align}
    which is positive by the positive definiteness of $\revise{\cov}\bracks*{\noise(\pth_{\time}, \sample)}$ (\cref{app:subsec:setup}).
    Hence, $\Hess_{\mom} \hamiltalt(\pth_{\time}, \mom_{\time})$ is positive definite and we can apply the implicit function theorem to \cref{eq:lem:finite_B_eq_implicit} to obtain that there exists $\margin(\pth_{\time}) > 0$, $\mom \from \ball(\pth_{\time}, \margin(\pth_{\time})) \times \ball(0, \margin(\pth_{\time})) \to \dspace$ such that, \revise{for any $\point \in \ball(\pth_{\time}, \margin(\pth_{\time}))$, $\vel \in \ball(0, \margin(\pth_{\time}))$,}
\begin{equation}
		\grad_\mom\hamiltalt(\point, \mom(\point, \vel)) = \grad \obj(\point) + \vel\,,
\end{equation}
	or, equivalently
\begin{equation}
		\grad_\mom\hamilt(\point, \mom(\point, \vel)) = \vel\eqdot
\end{equation}
	Therefore, as in the proof of \cref{lem:regularity_lagrangian}, we obtain that $\lagrangian$ is continuous on $\clball(\pth_{\time}, \margin(\pth_{\time})/2) \times \clball(0, \margin(\pth_{\time})/2)$ and therefore bounded by $M(\pth_{\time}) > 0$.
	Since $\pth$ is continuous, $\setdef*{\pth_{\time}}{\time \in [0, \horizon]}$ is compact and therefore, by extracting a finite covering from
\begin{equation}
		\bigcup_{\time \in [0, \horizon]} \ball(\pth_{\time}, \margin(\pth_{\time})/2)\,,
\end{equation}
we obtain that there exists $\margin > 0$ and $M > 0$ such that, for every $\time \in [0, \horizon]$, $\lagrangian(\pth_{\time}, \cdot)$ is finite and bounded by $M$ on $\ball(\revise{0}, \margin)$.
	Choosing a $\pthalt$ reparametrization of $\pth$, which is $\contdiff{1}$, such that $ \norm{\dot \pthalt_{\time}} < \margin$ for every $\time \in [0, \horizonalt]$, we thus obtain a path 
	such that 
\begin{equation}
		\action_{0, \horizonalt}(\pthalt) = \int_{0}^{\horizonalt} \lagrangian(\pthalt_{\time}, \dot \pthalt_{\time}) \dd \time \leq M \horizonalt < \infty\,,
\end{equation}
	which implies that $\dquasipot(\point, \pointalt) < \infty$.
\end{proof}

\subsection{Lyapunov condition}
\label{app:lyapunov}

\begin{definition}[Stopping times for the accelerated process]
\label{def:stoptimes}
For any set $\plainset \subset \vecspace$, we define the hitting and exit times of $\plainset$:
\begin{subequations}
\begin{align}
\hittime_{\plainset}
	&\defeq \inf \setdef{\run \geq \afterstart}{\curr[\accstate] \in \plainset}
	\\
\exittime_{\plainset}
	&\defeq \inf \setdef{\run \geq \start}{\curr[\accstate] \notin \plainset}
	\eqdot
\end{align}
\end{subequations}
\end{definition}

\WAdelete{
\begin{assumption}[Coercivity relative to noise]
	\leavevmode 
\begin{itemize}
		\item $\obj$ goes to infinity
		\item $\frac{\bdvar(\obj(\point))}{\norm{\state}^{\exponent}}$ is bounded above and below at infinity for $\exponent \in [0, 2]$, \ie
\begin{equation}
				0 < \liminf_{\norm{\state}\to +\infty} \frac{\bdvar(\obj(\point))}{\norm{\state}^{\exponent}}\quad\text{ and }\quad \limsup_{\norm{\state}\to +\infty} \frac{\bdvar(\obj(\point))}{\norm{\state}^{\exponent}} < +\infty\eqdot
\end{equation}
		\item $\frac{\norm{\grad \obj}^2}{\bdvar \circ \obj}$ is large enough at infinity, \ie
\begin{equation}
				\liminf_{\norm{\state} \to +\infty} \frac{\norm{\grad \obj(\point)}^2}{\bdvar(\obj(\point))} > \gradientbound\eqdot
\end{equation}
\end{itemize}
\end{assumption}
}

We will need the following concentration lemma and its corollary.

\begin{lemma}[{Part of the proof of \citep[Th. 1.19]{rigollet2023high}}]
	\label{lem:concentration}
	Let $X$ be a random variable in $\R^\dims$ such that, for all $\vel \in \R^\dims$,
\begin{equation}
		\log \ex \bracks*{\exp \parens*{\inner{\vel, X}}} \leq \frac{\norm{\vel}^2}{2}\eqdot
\end{equation}
	Then, for all $t > 0$, we have
\begin{equation}
		\prob \parens*{\norm{X}^2 \geq t} \leq 6^\dims \exp \parens*{- \frac{t}{8}}\eqdot
\end{equation}
 \end{lemma}

 \begin{corollary}
	 \label{corollary:concentration}
	 In the context of \cref{lem:concentration}, it holds that
	 \begin{equation}
		 \ex \bracks*{\norm{X}^2} \leq \gradientbound\eqdot
	 \end{equation}
 \end{corollary}
 \begin{proof}
	 Since $\norm{X}^2$ is non-negative, its expectation can be written as 
	 \begin{align}
		 \ex \bracks*{\norm{X}^2} &= \int_0^{+\infty} \prob \parens*{\norm{X}^2 > t} \dd t
		 \notag\\
								  &\leq 8 \dims \log 6 + \int_{8 \dims \log 6}^{+\infty} \prob \parens*{\norm{X}^2 > t} \dd t\eqdot
	 \end{align}
	 Invoking \cref{lem:concentration} then yields that
	 \begin{equation}
		 \ex \bracks*{\norm{X}^2} \leq 8 \dims \log 6 + 6^\dims \int_{8 \dims \log 6}^{+\infty} \exp \parens*{- \frac{t}{8}} \dd t \leq  8 \dims \log 6 + 8\,,
	 \end{equation}
	 which concludes the proof since $ 1 \leq \dims \log 6$.
 \end{proof}

 \begin{lemma}[Lyapunov condition]
	 \label{lem:lyapunov_condition}
	 Define  $\bdpot$ as in \cref{lem:B_is_descent}.
	 Then, there exists $\cpt \subset \vecspace$ compact, $\step_0 > 0$, $\const > 0$ such that, for any $\step \leq \step_0$, $\run \geq \start$, if $\curr \notin \cpt$, then, almost surely,
\begin{align}
\bdpot(\next) - \bdpot(\curr)
	&\leq \step \parens*{\frac{\norm{\noise(\curr, \curr[\sample])}^2}{\bdvar(\obj(\curr))} - \frac{\norm{\grad \obj(\curr)}^2}{\bdvar(\obj(\curr))}}
	\notag\\
	&\leq  \step \parens*{\frac{\norm{\noise(\curr, \curr[\sample])}^2}{\bdvar(\obj(\curr))} - (\gradientbound + \const)}\eqdot
\end{align}
\end{lemma}

\begin{proof}
	 By \cref{assumption:coercivity_noise}, there is $\Radius \geq \half$, $ \const > 0$ such that, for any $\point \in \vecspace$ such that $\norm{\point} \geq \Radius$,
	 \begin{equation}
		 \begin{cases}
			 \obj(\point) \geq \const\\
			 \const \leq  \frac{\bdvar(\obj(\point))}{\norm{\point}^{\exponent}}  \leq  \const^{-1}\\
			 \frac{\norm{\grad \obj(\point)}^2}{\bdvar(\obj(\point))} \geq \gradientbound + \const\eqdot
		 \end{cases}
	 \end{equation}
	 Then, define
	 $\cpt \defeq \clball(0, 2\Radius + 1)$.

 By definition, $\bdpot$ is twice continuously differentiable and, its Hessian satisfies, for any $\state \in \vecspace$,
\begin{equation}
		\Hess \bdpot(\state) \mleq \frac{2 \Hess \obj(\point)}{\bdvar(\obj(\point))}
		\mleq \frac{2 \smooth(\obj)} {\bdvar(\obj(\point))} \identity\eqdot
		\label{eq:lemma_est_going_back_to_bigcpt_descent_bdpot}
\end{equation}

	For the sake of clarity, for any $\run \geq \start$, denote by $\diffcurr$ the quantity
\begin{equation}
		\diffcurr = \frac{\next - \curr}{\step}\eqdot
\end{equation}

	For any $\run \geq \start$, we now have
\begin{align}
\bdpot(\next) - \bdpot(\curr)
	&\leq \step \inner{\grad \bdpot(\curr), \diffcurr}
	\notag\\
	&+ \frac{\step^2}{2} \norm{\diffcurr}^2 \sup_{t \in [0, 1]}\frac{2 \smooth(\obj)} {\bdvar(\obj(\curr + t (\next - \curr)))}
	\eqdot
\label{eq:lemma_est_going_back_to_bigcpt_descent_bdpot2}
\end{align}
	
	We first focus on bounding the last term. First note that, by the blanket assumptions,
	$\norm{\next - \curr} \leq 2\step \growth (1 + \norm{\curr})$ so that,
	For any $t \in [0, 1]$, $\step \leq (4 \growth)^{-1}$,
\begin{align}
\norm{\curr + t (\next - \curr)}
	&\geq \norm{\curr} - \norm{\next - \curr}
	\notag\\
	&\geq \norm{\curr} - 2\step \growth (1 + \norm{\curr})
	\notag\\
	&\geq \half (\norm{\curr} - 1)
	\eqdot
\end{align}
If $\curr$ is outside of $\cpt$, we have that $\norm{\curr} \geq 2 \Radius + 1$ and thus $\norm{\curr + t (\next - \curr)} \geq \Radius$.
Moreover, since $\Radius \geq \half$, $\curr$ being outside of $\cpt$ also implies that $\half \norm{\curr} \geq 1$ and so $\norm{\curr + t (\next - \curr)} \geq \tfrac{1}{4} \norm{\curr}$. 
By the definition of $\Radius$, we thus have
\begin{align}
\bdvar(\obj(\curr + t (\next - \curr)))
	&\geq \const \norm{\curr + t (\next - \curr)}^{\exponent}
	\notag\\
	&\geq \const \parens*{\frac{1}{4}}^{\exponent} \norm{\curr}^{\exponent}
	\notag\\
	&\geq \frac{\const^2}{4^{\exponent}} \bdvar(\obj(\curr))
	\eqdot
\end{align}

	Thus, if $\curr \notin \cpt$, \cref{eq:lemma_est_going_back_to_bigcpt_descent_bdpot} yields that
\begin{equation}
	\bdpot(\next) - \bdpot(\curr) \leq \step \inner{\grad \bdpot(\curr), \diffcurr} + \step^2 \frac{4^{\exponent}\smooth(\obj)}{\const^2 \bdvar(\obj(\curr))} \norm{\diffcurr}^2\eqdot
		\label{eq:lemma_est_going_back_to_bigcpt_descent_bdpot_2}
\end{equation}
Now, we can rewrite the inner product as 
\begin{align}
\inner{\grad \bdpot(\curr), \diffcurr}
	&= \frac{2\inner{\grad \obj(\curr), \diffcurr}}{\bdvar(\obj(\curr))}
	\notag\\
	&= \frac{\norm{\diffcurr + \grad \obj(\curr)}^2}{\bdvar(\obj(\curr))} - \frac{\norm{\grad \obj(\curr)}^2}{\bdvar(\obj(\curr))} - \frac{\norm{\diffcurr}^2}{\bdvar(\obj(\curr))}
	\eqdot
\end{align}
Plugging this into \cref{eq:lemma_est_going_back_to_bigcpt_descent_bdpot_2} and assuming that $\step \leq \frac{\const^2}{4^{\exponent} \smooth(\obj)}$, we obtain
\begin{align}
\bdpot(\next) - \bdpot(\curr)
	&\leq \step \parens*{\frac{\norm{\diffcurr + \grad \obj(\curr)}^2}{\bdvar(\obj(\curr))} - \frac{\norm{\grad \obj(\curr)}^2}{\bdvar(\obj(\curr))}}
	\notag\\
	&\leq \step \parens*{\frac{\norm{\noise(\curr, \curr[\sample])}^2}{\bdvar(\obj(\curr))} - \frac{\norm{\grad \obj(\curr)}^2}{\bdvar(\obj(\curr))}}
\end{align}
where we unrolled $\diffcurr$ in the last inequality. 
Using again that $\curr \notin \cpt$, we obtain
\begin{equation}
\bdpot(\next) - \bdpot(\curr)
	\leq  \step \parens*{\frac{\norm{\noise(\curr, \curr[\sample])}^2}{\bdvar(\obj(\curr))} - (\gradientbound + \const)}
\end{equation}
which concludes the proof.
\end{proof}

We will reuse, in later sections, the following fact that we thus state as a lemma: the sequence of iterates of \ac{SGD} is \emph{(weak) Feller} (see \eg \citep[Def.~4.4.2]{hernandez-lermaMarkovChainsInvariant2003}).

\begin{lemma}
	\label{lem:feller}
	The Markov chain $(\curr)_{\run \geq \start}$ is \textpar{weak} Feller.
\end{lemma}

\begin{proof}
since both $\grad \obj$ and $\noise$ are continuous, for any $\func \from \vecspace \to \R$ continuous and bounded, the function.
\begin{equation}
	\point \in \vecspace \longmapsto \ex_{\point} \bracks*{
		\func \parens*{
			\afterinit 
		}
	}
	= \ex_{\point} \bracks*{
		\func \parens*{
			\point -\step \grad \obj(\point) + \step \noise(\point, \sample)
		}
	}
\end{equation}
is still continuous and bounded. Therefore, the Markov chain $(\curr)_{\run \geq \start}$ is weak-Feller.
\end{proof}

 \begin{lemma}
	 \label{lem:inv_meas_exist}
	 There is $\step_0 > 0$ such that, for any $\step \leq \step_0$, there exists an invariant probability measure for $(\curr)_{\run \geq \start}$.
 \end{lemma}
 \begin{proof}
	 We invoke a general result on weak-Feller Markov chains that satisfy a Lyapunov condition, \eg \citet[Thm.~7.2.4]{hernandez-lermaMarkovChainsInvariant2003} or \citet[Thm.~12.3.6]{doucMarkovChains2018}.
	
	 First, \cref{lem:feller} ensures that $(\curr)_{\run \geq \start}$ is weak-Feller.

Moreover, by \cref{lem:lyapunov_condition}, there exists $\cpt \subset \vecspace$ compact, $\step_0 > 0$, $\const > 0$ such that, for any $\step \leq \step_0$, $\init = \point \notin \cpt$,
\begin{equation}
	\bdpot(\afterinit) - \bdpot(\point) \leq  \step \parens*{\frac{\norm{\noise(\point, \init[\sample])}^2}{\bdvar(\obj(\point))} - (\gradientbound + \const)}\eqdot
\end{equation}

Passing to the expectation yields that, for any $\point \notin \cpt$,
\begin{equation}
	\ex_{\point} \bracks*{
	\bdpot(\afterinit)} - {\bdpot(\point)
	} \leq  \step \parens*{\frac{\ex_{\point} \bracks*{\norm{\noise(\point, \init[\sample])}^2}}{\bdvar(\obj(\point))} - (\gradientbound + \const)}\eqdot
\end{equation}
Appplying \cref{corollary:concentration} with $\rv \gets \frac{\noise(\point, \init[\sample])}{\sqrt{\bdvar(\obj(\point))}}$ (the conditions of application are verified from \cref{asm:noise-weak}\cref{assumption:subgaussian}) yields that
\begin{equation}
	\ex_{\point} \bracks*{
	\bdpot(\afterinit)} - {\bdpot(\point)
	} \leq - \step \const\eqdot
\end{equation}

Hence, for any $\point \in \vecspace$, it holds
\begin{equation}
	\ex_{\point} \bracks*{
	\bdpot(\afterinit)} - {\bdpot(\point)
	} \leq - \step \const + \oneof{\point \in \cpt} \parens*{\sup_{\pointalt \in \cpt}\ex_{\pointalt} [ \bdpot(\pointalt)] - \inf_{\vecspace} \bdpot + \step \const}\,,
\end{equation}
with $\bdpot$ which is not identically equal to its minimum since $\obj$ is coercive.

Therefore, we can apply \citet[Thm.~7.2.4]{hernandez-lermaMarkovChainsInvariant2003} with $\bdpot - \inf_{\vecspace} \bdpot$, that guarantees that there exists an invariant probability measure for $(\curr)_{\run \geq \start}$.
 \end{proof}

\begin{lemma}
	\label{lem:est_going_back_to_bigcpt}
	There exists a compact set $\bigcpt \subset \vecspace$, $\step_0 > 0$, such that for any compact set $\bigcptalt \subset \vecspace$ such that $\bigcpt \subset \bigcptalt$, there exists $a, b > 0$,  such that
\begin{equation}
		\forall \step \leq \step_0\,,\, \state \in \bigcptalt \setminus \bigcpt \,,\, \run \geq \start\,,\, \prob_\state \parens*{\hittime_{\bigcpt} > \run} \leq \exp \parens*{- \frac{a \run}{\step} + \frac{b}{\step}}\eqdot
\end{equation}
\end{lemma}

\begin{proof}
	This result is a consequence of \cref{lem:lyapunov_condition}: there exists $\cpt \subset \vecspace$ compact, $\step_0 > 0$, $\const > 0$ such that, for any $\step \leq \step_0$, $\run \geq \start$, if $\curr \notin \cpt$, then, almost surely,
	 \begin{equation}
			   \bdpot(\next) - \bdpot(\curr) \leq  \step \parens*{\frac{\norm{\noise(\curr, \curr[\sample])}^2}{\bdvar(\obj(\curr))} - (\gradientbound + \const)}\eqdot
			   \label{eq:lemma_est_going_back_to_bigcpt_descent_bdpot_prime}
	 \end{equation}
	 First, let us choose $\bigcpt$. There is some $\Radius > 0$ such that $\cpt \subset \ball(0, \Radius)$. Define $\Radiusalt \defeq e^{4 \growth + 1}(\Radius + 1)$ and $\bigcpt \defeq \clball(0, \Radiusalt)$.   
	 With $\step \leq (4 \growth)^{-1}$, if, for some $\run \geq \start$, $\curr$ is not in $\bigcpt$, by \cref{lem:degrowth_iterates}, for any $\runB \leq \ceil{\step^{-1}}$, $\state_{\run + \runB}$ has norm greater or equal to $\Radius$ and thus $\state_{\run + \runB}$ is not in $\cpt$ either.

	Now, fix $\nRuns \geq \start$ and consider the event $\setof*{\hittime_{\bigcpt} > \nRuns}$ with $\init = \point \in \bigcptalt \setminus \bigcpt$. This means that, for any $\start \leq \run \leq \nRuns$, $\accstate_\run$ is outside of $\bigcpt$ and so, for any $\start \leq  \run \leq \nRuns \floor{\step^{-1}}$, $\curr$ is not  in $\cpt$.

	Summing \cref{eq:lemma_est_going_back_to_bigcpt_descent_bdpot_prime} over $\run = \start,\dots,\nRuns \floor{\step^{-1}} - 1$ yields 
\begin{align}
\bdpot(\accstate_\nRuns) - \bdpot(\init)
	&= \bdpot(\state_{\nRuns\floor{\step^{-1}}}) - \bdpot(\init)
	\notag\\
	&\leq \step \sum_{\run = \start}^{\nRuns \floor{\step^{-1}} - 1}
		\parens*{\frac{\norm{\noise(\curr, \sample_\run)}^2}{\bdvar(\obj(\curr))} - (\gradientbound + \const)}
	\eqdot
\end{align}
Define $\Margin \defeq \sup_{\bigcptalt \setminus \bigcpt} \bdpot - \inf_{\vecspace \setminus \bigcpt} \bdpot$  which is finite since $\obj$ is coercive. By definition, we have that $\bdpot(\accstate_\nRuns) -\bdpot(\init) \geq -\Margin$.

	Therefore, on the event $\setof*{\hittime_{\bigcpt} > \nRuns}$, we have
\begin{equation}
		\sum_{\run = \start}^{\nRuns \floor{\step^{-1}} - 1}
		{\frac{\norm{\noise(\curr, \sample_\run)}^2}{\bdvar \circ \obj(\curr)}} \geq \nRuns \floor{\step^{-1}} (\gradientbound + \const) -  \frac{\Margin}{\step}\eqdot
\end{equation}

	Now, on the whole event, consider the random variable $\rv \in \R^{\nRuns \floor{\step^{-1}} \dims}$ defined by
\begin{equation}
		(\rv_{\run \dims}+1, \dots, \rv_{(\run + 1) \dims}) = \frac{\noise(\curr, \sample_\run)}{\sqrt{\bdvar \circ \obj(\curr)}}
	\quad \text{ for } \run = \start,\dots,\nRuns \floor{\step^{-1}} - 1\,,
\end{equation}
	we have that $\probwrt*{\point}{\hittime_{\bigcpt} > \nRuns} \leq \probwrt*{\point}{\norm{\rv}^2 \geq \nRuns \floor{\step^{-1}} (\gradientbound + \const) -  \frac{\Margin}{\step}}$.
	
	Since, for any $\vel \in \dspace$,
\begin{equation}
		\log \exof*{\exp \parens*{\inner*{\vel, \frac{\noise(\curr, \sample_\run)}{\sqrt{\bdvar \circ \obj(\curr)}}}} \given \init, \afterinit, \dots, \curr
		} \leq \frac{\norm{\vel}^2}{2}\,,
\end{equation} 
	the random variable $\rv$ satisfies the assumptions of \cref{lem:concentration} with $\dims \gets \nRuns \floor{\step^{-1}} \dims$.

	First, suppose that $t \defeq \nRuns \floor{\step^{-1}} (\gradientbound + \const) -  \frac{\Margin}{\step}$ is non-negative. Applying  \Cref{lem:concentration} with this $t$ yields that
\begin{align}
\probwrt*{\point}{\hittime_{\bigcpt} > \nRuns}
	&\leq \probwrt*{\point}{\norm{\rv}^2 \geq \nRuns \floor{\step^{-1}} (\gradientbound + \const) -  \frac{\Margin}{\step}}
	\notag\\
	&\leq 6^{\nRuns \floor{\step^{-1}} \dims} \exp \parens*{- \frac{\nRuns \floor{\step^{-1}} (\gradientbound + \const) -  \frac{\Margin}{\step}}{8}}
	\notag\\
	&= \exp \parens*{- \frac{\nRuns \floor{\step^{-1}} (8 \dims \log 6 + \const) -  \frac{\Margin}{\step}}{8}}\eqdot
\end{align}
	If $t$, defined above, is negative, then in particular $\frac{\Margin}{\step} \geq \nRuns \floor{\step^{-1}} \const$ so that this bounds still (trivially) holds.

	Finally, in particular, for $\step \leq 1/2$, $\floor{\step^{-1}} \geq (2 \step)^{-1} $ so we obtain
\begin{equation}
		\probwrt*{\point}{\hittime_{\bigcpt} > \nRuns} \leq \exp \parens*{- \frac{\nRuns (8 \dims \log 6 + \const) }{16 \step} + \frac{ \Margin}{8 \step}}
\end{equation}
and our proof is complete.
\end{proof}

\subsection{Preliminary estimates and lemmas}

We will use the following lemma, which corresponds to \citet[Lem.~5.3]{Kif88}.

\begin{lemma}
	\label{lem:est_stay_in_between}
	Let $\cpt \subset \vecspace$ be a compact set such that $\cpt \cap \crit\obj = \emptyset$.
	Then there exists $\const > 0$, $\nRuns \geq 1$, $\step_0 > 0$,
	such that, for any $\run > \nRuns$, $\state \in \cpt$, $\step \leq \step_0$,
			\begin{equation}
				\prob_\state \parens*{\hittime_{\vecspace\setminus \cpt} > \run} 
				=
				\prob_\state \parens*{\exittime_{\cpt} > \run} 
				\leq  \exp \parens*{- \const (\run - \nRuns) / \step}\eqdot
			\end{equation}
\end{lemma}
\begin{proof}
	The proof is exactly the same as the proof \citet[Lem.~5.3]{Kif88}, which only uses the \ac{lsc} of $\daction_\nRuns$ (\cref{cor:ldp_iterates_full}).
\end{proof}

The following lemma provides a convenient reformulation of the results of \cref{lem:est_going_back_to_bigcpt} and \cref{lem:est_stay_in_between}.
\begin{lemma}
	\label{lem:exponential_tail_bounds}
	\leavevmode
	\begin{itemize}
		\item 
	There exists $\bigcpt \subset \vecspace$ a compact set, $\step_0 > 0$, such that for any $\bigcptalt \subset \vecspace$ compact set such that $\bigcpt \subset \bigcptalt$, there exists $\coef_0, \constA, \constB > 0$ such that,
	\begin{equation}
		\forall \step \leq \step_0,\, \coef \leq  \coef_0,\, \state \in \bigcptalt \setminus \bigcpt\,,\quad
		\ex_{\state} \bracks*{e^{\frac{\coef \hittime_\bigcpt}{\step}}} \leq e^{ \frac{\constA \coef}{\step} +\constB}\eqdot
	\end{equation}
	\item For any $\cpt \subset \vecspace$ compact such that $\cpt \cap \crit\obj = \emptyset$, there exists $\step_0, \coef_0, \constA, \constB  > 0$ such that, 
		\begin{equation}
			\forall \step \leq \step_0,\, \coef \leq  \coef_0,\, \state \in \cpt\,,\quad
			\ex_{\state} \bracks*{e^{\frac{\coef \exittime_\cpt}{\step}}} \leq e^{ \frac{\constA \coef}{\step} + \constB}\eqdot
		\end{equation}
\end{itemize}
\end{lemma}
\begin{proof}
	The proofs of both statements are very similar, so we prove only the second one for notational convenience.
	Fix $\cpt \subset \vecspace$ compact such that $\cpt \cap \crit\obj = \emptyset$.
	By \cref{lem:est_stay_in_between}, there exists $\const > 0$, $\nRuns \geq 1$ such that, for any $\run > \nRuns$, $\state \in \cpt$, $\step \leq \step_0$,
	\begin{equation}
		\prob_\state \parens*{\exittime_{\cpt} > \run} 
		\leq  \exp \parens*{- \const (\run - \nRuns) / \step}\eqdot
		\label{eq:lemma:exponential_tail_bounds:est_stay_in_between}
	\end{equation}

	Let us first bound $\ex_\state \bracks*{\exp \parens*{\frac{\coef (\exittime_\cpt - \nRuns)}{\step}}}$. We have that
	\begin{align}
		\ex_\state \bracks*{\exp \parens*{\frac{\coef (\exittime_\cpt - \nRuns)}{\step}}} 
		&=
		\int_{0}^\infty \prob_\state \parens*{\exp \parens*{\frac{\coef (\exittime_\cpt - \nRuns)}{\step}} > \time} \dd \time
		\notag\\
		&\leq 
	e^{\frac{\coef}{\step}}  + {\int_{e^{\frac{\coef}{\step}}}^\infty \prob_\state \parens*{\exittime_\cpt > \nRuns + \frac{\step \log \time}{\coef}} \dd \time}
	\notag\\
		&\leq 
	e^{\frac{\coef}{\step}}  + {\int_{e^{\frac{\coef}{\step}}}^\infty \prob_\state \parens*{\exittime_\cpt > \nRuns + \floor*{\frac{\step \log \time}{\coef}}} \dd \time}
	\notag\\
		&\leq 
	e^{\frac{\coef}{\step}}  + {\int_{e^{\frac{\coef}{\step}}}^\infty \exp \parens*{- \frac{\const}{\step} \floor*{\frac{\step \log \time}{\coef}}} \dd \time}\,,
	\end{align} 
	where we used \cref{eq:lemma:exponential_tail_bounds:est_stay_in_between} in the last inequality.
	
	Lower bounding $\floor{s}$ by $s - 1$, we obtain that 
	\begin{align}
		\ex_\state \bracks*{\exp \parens*{\frac{\coef (\exittime_\cpt - \nRuns)}{\step}}} 
		&\leq
	e^{\frac{\coef}{\step}}  + {\int_{e^{\frac{\coef}{\step}}}^\infty \exp \parens*{- \frac{\const}{\step} \frac{\step \log \time}{\coef} + \frac{\const}{\step}} \dd \time}
	\notag\\
		&=
		e^{\frac{\coef}{\step}}  + {\int_{e^{\frac{\coef}{\step}}}^\infty e^{\frac{\const}{\step}} \time^{- \frac{\const}{\coef}} \dd \time}\eqdot
	\end{align}
	Performing the change of variable $\timealt \gets e^{-\frac{\coef}{\step}}\time$, we obtain that
	\begin{align}
		\ex_\state \bracks*{\exp \parens*{\frac{\coef (\exittime_\cpt - \nRuns)}{\step}}} 
		&\leq
	e^{\frac{\coef}{\step}} \parens*{1+ {\int_{1}^\infty \timealt^{-\frac{\const}{\coef}}\dd \timealt}}\eqdot
	\end{align}
	When $\coef \leq \const/2$, we obtain that
	\begin{equation}
		\ex_\state \bracks*{\exp \parens*{\frac{\coef (\exittime_\cpt - \nRuns)}{\step}}} 
		\leq
		e^{\frac{\coef}{\step}}\parens*{1 + \int_{1}^\infty \timealt^{-2}\dd \timealt}\,,
	\end{equation}
	and therefore, 
	\begin{equation}
		\ex_\state \bracks*{\exp \parens*{\frac{\coef \exittime_\cpt}{\step}}} 
		\leq
		e^{\frac{(\nRuns + 1)\coef}{\step}}\parens*{1 + \int_{1}^\infty \timealt^{-2}\dd \timealt}\,,
	\end{equation}
	which concludes the proof.
\end{proof}

The following lemma upper-bounds the probability  
of exiting a large neighborhood of the critical points before visiting a smaller one critical points.

\begin{lemma}
	\label{lem:est_not_going_to_bigcpt}

	Consider $\crit\obj \subset \open \subset \bigcpt \subset \vecspace$ with $\open$ an open set and $\bigcpt$ a compact set.
	There exists $\bigcptalt \subset \vecspace$ compact set such that $\bigcpt \subset \bigcptalt$, $\Margin > 0$, $\step_0 > 0$ such that, for any $\step \leq \step_0$, $\state \in \bigcpt$,
	\begin{equation}
		\prob_\state \parens*{\exittime_{\bigcptalt} < \hittime_\open} \leq \exp \parens*{- \frac{\Margin}{\step}}\eqdot
	\end{equation}
\end{lemma}

\begin{proof}

	Define $\bigcptalt \defeq \setdef{\state \in \vecspace}{\obj(\state) \leq \sup_{\bigcpt} \obj + 1}$ and let $\bdpot$ be as in \cref{lem:B_is_descent}.

	
	Since $\bdprimvar$ is (stricly) increasing as its derivative is (stricly) positive by definition, we have 
	\begin{equation}
		\Margin \defeq \bdprimvar \parens*{\sup_\bigcpt \obj + 1} - \bdprimvar \parens*{\sup_\bigcpt \obj} > 0\eqdot
	\end{equation}

	By \cref{lem:est_stay_in_between} applied to $\cpt \gets \bigcptalt_\margin \setminus \open$, there exists $\const > 0$, $\nRuns_0 \geq 1$, $\step_0 > 0$ such that for any $\run > \nRuns_0$, $\state \in \bigcptalt_\margin \setminus \open$, $\step \leq \step_0$,
		\begin{equation}
\prob_\state \parens*{\exittime_{\bigcptalt_\margin \setminus \open} > \run} \leq \exp \parens*{- \const (\run - \nRuns_0) / \step}\eqdot
		\end{equation}
		Defining $\nRuns \defeq \ceil*{
			\frac{\Margin}{\const}
	} + \nRuns_0$, which is greater or equal than 1, we obtain that, for any $\step \leq \step_0$, $\state \in \bigcptalt_\margin \setminus \open$, 
	\begin{equation}
		\prob_\state \parens*{\exittime_{\bigcptalt_\margin \setminus \open} \geq \nRuns} \leq \exp \parens*{- \frac{\Margin}{\step}}\eqdot
		\label{eq:lemma:est_not_going_to_bigcpt:csq_est_stay_in_between}
	\end{equation}
	Note that this inequality actually holds for any $\state \in \bigcptalt_\margin$.

	We now bound $\prob_\state \parens*{\exittime_{\bigcptalt_\margin} > \hittime_\open}$ for $\state \in \bigcpt$ by distinguishing the cases where $\exittime_{\bigcptalt_\margin} < \nRuns$ and $\exittime_{\bigcptalt_\margin} \geq \nRuns$. For any $\state \in \bigcpt$, we have that
	\begin{align}
		\prob_\state \parens*{\exittime_{\bigcptalt_\margin} < \hittime_\open} 
		&\leq 
		\prob_\state \parens*{\exittime_{\bigcptalt_\margin} < \hittime_\open,\,\exittime_{\bigcptalt_\margin} < \nRuns} + \prob_\state \parens*{\exittime_{\bigcptalt_\margin} < \hittime_\open,\, \exittime_{\bigcptalt_\margin} \geq \nRuns}
		\notag\\
		&\leq 
		\prob_\state \parens*{\exittime_{\bigcptalt_\margin} < \nRuns} + \prob_\state \parens*{\exittime_{\bigcptalt_\margin \setminus \open} \geq \nRuns}
		\notag\\
		&\leq 
		\prob_\state \parens*{\exittime_{\bigcptalt_\margin} < \nRuns} + \exp \parens*{- \frac{\Margin}{\step}}
\end{align}
	where we used \cref{eq:lemma:est_not_going_to_bigcpt:csq_est_stay_in_between}.

	We now focus on bounding the first term. For this, we first show that $\exittime_{\bigcptalt_\margin} < \nRuns$ implies that 
	\begin{equation}
        \dist_{\nRuns} \parens*{ \accstate, \pths^{\setof{\state}}_{\nRuns}\parens{\Margin / 2}} > \margin / 2\eqdot
	\end{equation}

	For the sake of contradiction, suppose that this inequality does not hold.
	Therefore, there must exist $\dpth \in \pths^{\setof{\state}}_{\nRuns}\parens{\Margin / 4}$ such that $\dist_{\nRuns} \parens*{ \accstate, \dpth} < \margin$. In particular, there is some $\run < \nRuns$ such that $\dpth_\run \notin \bigcptalt$, so that, by \cref{def:dquasipot},
	\begin{equation}
		\daction_{\nRuns}( \dpth) \geq \daction_{\run+1}( \dpth) \geq \dquasipot(\dpth_\start, \dpth_\run)\eqdot
	\end{equation}

	By \cref{lem:B_is_descent}, we have that
\begin{align}
\dquasipot(\dpth_\start, \dpth_\run) &\geq  {\bdpot(\dpth_\run) - \bdpot(\dpth_\start)}
	\notag\\
                                     &\geq {\bdprimvar\parens*{\inf_{\vecspace \setminus {\bigcptalt}} \obj} - \bdprimvar\parens*{\sup_{\bigcptalt} \obj}}\,,
\end{align}
 since $\bdprimvar$ is increasing. By construction of $\bigcptalt$, we have further that
\begin{equation}
\dquasipot(\dpth_\start, \dpth_\run)
\geq  {\bdprimvar \parens*{\sup_{\bigcpt} \obj + 1} - \bdprimvar \parens*{\sup_{\bigcpt} \obj)}} = \Margin\,,
	\end{equation}
    so that $\daction_\nRuns(\dpth) \geq \Margin$, which is a contradiction with $\dpth \in \pths^{\setof{\state}}_{\nRuns}\parens{\Margin / 2}$.

	Therefore, we have that 
\begin{align}
\probof*{\exittime_{\bigcptalt_\margin} < \nRuns}
    &\leq \probof*{\dist_{\nRuns} \parens*{ \accstate, \pths^{\setof{\state}}_{\nRuns}\parens{\Margin / 2}} > \margin / 2}
	\notag\\
    &\leq \exp \parens*{- \frac{\Margin}{4 \step}}
\end{align}
where we invoked \cref{cor:ldp_iterates_full} with $\margin \gets \margin / 2$, $\level \gets \Margin / 2$, $\precs \gets \Margin / 4$. 
\end{proof}

The following lemma is key to our analysis: it shows that \ac{SGD} spends most of its time near its critical points.

\begin{lemma}
	\label{lem:est_going_back_to_crit}
	Consider $\crit\obj \subset \open \subset \bigcpt \subset \vecspace$ with $\open$ an open set and $\bigcpt$ a compact set.
   Then, there is some $\step_0, \coef_0, \constA, \constB > 0$ such that,
   \begin{equation}
	   \forall \step \leq \step_0,\, \coef \leq  \coef_0,\, \state \in \bigcpt\,,\quad
	   \ex_{\state} \bracks*{e^{\frac{\coef \hittime_{\open}}{\step}}} \leq e^{ \frac{\constA \coef}{\step} + \constB}\eqdot
   \end{equation}
\end{lemma}

\begin{proof}
	\def\smallbigcpt{\bigcpt}
	\def\middlebigcpt{\tilde{\bigcpt}}
	\def\bigbigcpt{\bigcptalt}
	
	Fix $\precs > 0$.
	Without loss of generality, assume that $\smallbigcpt$ is large enough to include the compact set given by the first item of \cref{lem:exponential_tail_bounds} (note that the guarantee of the first item of \cref{lem:exponential_tail_bounds} still holds even if $\smallbigcpt$ is larger).
	
	Apply \cref{lem:est_not_going_to_bigcpt} with $\open \gets \open$, $\bigcpt \gets \smallbigcpt$ and denote by $\middlebigcpt$ the obtained compact and $\step_0, \Margin > 0$ such that, for every $\step \leq \step_0$, $\state \in \smallbigcpt$,
	\begin{equation}
		\prob_\state \parens*{\exittime_{\middlebigcpt} < \hittime_\open} \leq \exp \parens*{- \frac{\Margin}{\step}}\eqdot
	\end{equation}

	Define $\radius \defeq \sup_{\state \in \middlebigcpt} \norm{\state}$ and $\Radius = e^{8 \growth}(1 + \radius)$.
	Assuming that $\step \leq  1$, by \cref{lem:growth_iterates}, for any $\state \in \middlebigcpt$, the next two iterates of $(\accstate_\run)_{\run \geq 0}$ satisfies $\norm{\accstate_{1}} \leq \Radius$ and $\norm{\accstate_{2}} \leq \Radius$.
	Define $\bigbigcpt \defeq \clball(0, \Radius)$.

	We invoke both items of \cref{lem:exponential_tail_bounds} with $\cpt \gets \middlebigcpt \setminus \open$ and denote by $\const > 0$ a constant that satisfies the bounds of both items. In the rest of the proof, we consider $\step$ and $\coef$ smaller than the bounds given by this lemma.

	Our goal is to bound, for any $\nRuns \geq \start$, the quantity,
	\begin{equation}
	s_\nRuns(\coef, \step) \defeq \sup_{\state \in \bigcpt} \ex_{\state} \bracks*{\exp \parens*{\frac{\coef \hittime_{\open}^\nRuns}{\step}}}\,,
		\quad \text{where} \quad
		\hittime_{\open}^\nRuns \defeq \min(\nRuns, \hittime_{\open})\eqdot
	\end{equation}
	Note that, by construction, $s_\nRuns(\coef, \step)$ is finite.

	Take $\state \in \bigcpt$.
	In particular, \cref{lem:exponential_tail_bounds} implies that $\exittime_{\middlebigcpt \setminus \open} < +\infty$ almost surely for all $\state \in \bigcpt$.
	\begin{align}
		\exp \parens*{\frac{\coef \hittime_{\open}^\nRuns}{\step}}
		&=
		\oneof{\accstate_{\exittime_{\middlebigcpt \setminus \open}} \in \open}
		\exp \parens*{\frac{\coef \exittime_{\middlebigcpt \setminus \open}}{\step}}
		\notag\\
		&\qquad
		+ \oneof{\accstate_{\exittime_{\middlebigcpt \setminus \open}} \notin \bigcpt_\margin}
		 \exp \parens*{\frac{\coef \exittime_{\middlebigcpt \setminus \open}}{\step}}
		  \exp \parens*{\frac{\coef \parens*{\hittime_{\open}^\nRuns - 
		  \exittime_{\middlebigcpt \setminus \open}}}{\step}}
		  \notag\\
		&\leq
		\oneof{\accstate_{\exittime_{\middlebigcpt \setminus \open}} \in \open}
		\exp \parens*{\frac{\coef \exittime_{\middlebigcpt \setminus \open}}{\step}}
		\notag\\
		&\qquad
		+ \oneof{\accstate_{\exittime_{\middlebigcpt \setminus \open}} \notin \bigcpt_\margin}
		 \exp \parens*{\frac{\coef \exittime_{\middlebigcpt \setminus \open}}{\step}}
		  \exp \parens*{\frac{\coef \min \parens*{\hittime_{\open} - 
		  \exittime_{\middlebigcpt \setminus \open}, \nRuns}}{\step}}
	\end{align}
	so that we can apply the strong Markov property to the Markov chain $(\curr[\accstate])_{\run \geq \start}$ with stopping time $\exittime_{\middlebigcpt \setminus \open}$ to obtain that
	\begin{align}
		\ex_{\state} \bracks*{
		\exp \parens*{\frac{\coef \hittime_{\open}^\nRuns}{\step}}}
		&\leq 
		\ex_{\state} \bigg[
		\oneof{\accstate_{\exittime_{\middlebigcpt \setminus \open}} \in \open}
		\exp \parens*{\frac{\coef \exittime_{\middlebigcpt \setminus \open}}{\step}}
		\notag\\
		&\qquad
		+ \oneof{\accstate_{\exittime_{\middlebigcpt \setminus \open}} \notin \bigcpt_\margin}
		 \exp \parens*{\frac{\coef \exittime_{\middlebigcpt \setminus \open}}{\step}}
		 \ex_{\accstate_{\exittime_{\middlebigcpt \setminus \open}}}
		 \bracks*{
		  \exp \parens*{\frac{\coef \hittime_{\open}^\nRuns}{\step}}}
		  \bigg]
		  \label{eq:lemma:est_going_back_to_crit:first_strong_markov}
	\end{align}
	We now bound 
	\begin{equation}
		\oneof{\accstate_{\exittime_{\middlebigcpt \setminus \open}} \notin \bigcpt_\margin}
		 \ex_{\accstate_{\exittime_{\middlebigcpt \setminus \open}}}
		 \bracks*{
		  \exp \parens*{\frac{\coef \hittime_{\open}^\nRuns}{\step}}}
		  =
		  \ex_{\accstate_{\exittime_{\middlebigcpt \setminus \open}}}
		 \bracks*{
		 \oneof{\accstate_\start \notin \bigcpt_\margin}		  \exp \parens*{\frac{\coef \hittime_{\open}^\nRuns}{\step}}}
	\end{equation}
	Since $\state$ is in $\bigcpt$, by definition, $\exittime_{\middlebigcpt \setminus \open}$ is at least equal to 1 and $\accstate_{\exittime_{\middlebigcpt \setminus \open} - 1}$ is still in $\middlebigcpt$. By definition of $\bigbigcpt$, $\accstate_{\exittime_{\middlebigcpt \setminus \open}}$ must still be in $\bigbigcpt$.
	Therefore, the guarantee of \cref{lem:exponential_tail_bounds} applies and, since in particular it implies that $\hittime_\smallbigcpt$ is finite almost surely when the chain is started at $\accstate_{\exittime_{\middlebigcpt \setminus \open}}$, we can apply the strong Markov property to obtain that 
\begin{align}
\ex_{\accstate_{\exittime_{\middlebigcpt \setminus \open}}} \bracks*{\oneof{\init \notin \middlebigcpt} \exp \parens*{\frac{\coef \hittime_{\open}^\nRuns}{\step}}}
	&\leq \ex_{\accstate_{\exittime_{\middlebigcpt \setminus \open}}}
		\bracks*{\oneof{\init \notin \middlebigcpt} \exp \parens*{\frac{\coef \hittime_{\bigcpt}}{\step}}
			\ex_{\accstate_{\hittime_{\smallbigcpt}}}\bracks*{
				\exp \parens*{\frac{\coef \hittime_{\open}^\nRuns}{\step}}
			}
		}
	\notag\\
	&\leq s_\nRuns(\coef, \step) \ex_{\accstate_{\exittime_{\middlebigcpt \setminus \open}}} \bracks*{\oneof{\init \notin \middlebigcpt} \exp \parens*{\frac{\coef \hittime_{\smallbigcpt}}{\step}}}
\end{align}
	where, for the second inequality, we used the definition of $s_\nRuns(\coef, \step)$ and the fact that $\state_{\hittime_{\smallbigcpt}}$ is in $\bigcpt$.
	
	Now, to bound the remaining expectation, note that $\init$ does not belong to $\middlebigcpt$ and, \emph{a fortiori}, does not belong to $\bigcpt$. Therefore, $\hittime_{\smallbigcpt}$ depends only on $(\accstate_\run)_{\run \geq \afterstart}$ and the (weak) Markov property implies that
	\begin{align}
		\ex_{\accstate_{\exittime_{\middlebigcpt \setminus \open}}} \bracks*{
			\oneof{\init \notin \middlebigcpt}
				\exp \parens*{\frac{\coef \hittime_{\bigcpt}}{\step}}
		}
	   &\leq
		\ex_{\accstate_{\exittime_{\middlebigcpt \setminus \open}}} \bracks*{
			\oneof{\init \notin \middlebigcpt}
			\ex_{\afterinit}
			\bracks*{
				\exp \parens*{\frac{\coef (1 + \hittime_{\bigcpt})}{\step}}
			}
		}
		\notag\\
		&\leq 
		\exp \parens*{
			\frac{\coef(1 + \const)}{\step}
		}
		\ex_{\accstate_{\exittime_{\middlebigcpt \setminus \open}}} \bracks*{
			\oneof{\init \notin \middlebigcpt}
		}\,,
	\end{align}
	by \cref{lem:exponential_tail_bounds}, since the $\accstate_{\exittime_{\middlebigcpt \setminus \open} - 1}$ is still in $\middlebigcpt$ so that $\afterinit$ of the chain started at $\accstate_{\exittime_{\middlebigcpt \setminus \open}}$ is still in $\bigbigcpt$.

	\Cref{lem:est_not_going_to_bigcpt} then implies that,
  \begin{align}
		\ex_{\accstate_{\exittime_{\middlebigcpt \setminus \open}}} \bracks*{
			\oneof{\init \notin \middlebigcpt}
				\exp \parens*{\frac{\coef \hittime_{\bigbigcpt}}{\step}}
		}
		\leq 
		\exp \parens*{
		  \frac{\coef(1 +\constA) - \Margin}{\step} + \constB
		}
		 \end{align}
		Combining these bounds with \cref{eq:lemma:est_going_back_to_crit:first_strong_markov} then gives
		\begin{align}
			\ex_\state \bracks*{\exp \parens*{\frac{\coef \hittime_{\open}^\nRuns}{\step}}}
			&\leq 
			\parens*{1 + 
			s_\nRuns(\coef, \step)
			 \exp \parens*{
				\frac{\coef(1  + \constA) - \Margin}{\step} + \constB
			}
		}
			\ex_\state \bracks*{
			\exp \parens*{\frac{\coef \exittime_{\middlebigcpt \setminus \open}}{\step}}
			}
			\notag\\
			&\leq 
			\parens*{1 +
			s_\nRuns(\coef, \step)
			 \exp \parens*{
				\frac{\coef(1 + \constA) - \Margin}{\step} + \constB
		}
		 }
			\exp \parens*{
				\frac{\constA \coef}{\step} + \constB
			}\,,
		\end{align}
		by \cref{lem:exponential_tail_bounds}.
		Since this inequality is valid for any $\state \in \smallbigcpt$, we have shown that
		\begin{equation}
			s_\nRuns(\coef, \step) \leq 
			\parens*{e^{\frac{\constA \coef}{\step} + \constB} +
			s_\nRuns(\coef, \step)
			 e^{
				\frac{\coef(1 + 2\constA) - \Margin}{\step} + \constB
				}
			}\eqdot
		\end{equation}
		For $\coef \leq \Margin / (2(1 + 2 \constA))$ and $\step$ small enough,
		\begin{equation}
			e^{\frac{\coef (1 + 2\constA) - \Margin}{\step} + \constB} \leq e^{- \frac{\Margin}{2\step} + \constB} \leq  \half\,,
		\end{equation}
		and we obtain that
		\begin{equation}
		  s_\nRuns(\coef, \step) \leq 2 e^{\frac{\constA \coef}{\step} + \constB}\eqdot   
		\end{equation}
		Taking $\nRuns \to +\infty$ and using Fatou's lemma yields that
		\begin{equation}
			\sup_{\state \in \bigcpt} \ex_{\state} \bracks*{\exp \parens*{\frac{\coef \hittime_{\open}}{\step}}} \leq 2 e^{\frac{\constA \coef}{\step} + \constB}\,,
		\end{equation}
		which concludes the proof.
\end{proof}

\subsection{Estimates of the invariant measure}
\label{app:estimates_trans_inv}

Define, for sets $A, B \subset \vecspace$,
\begin{equation}
	\dquasipot(A, B) \defeq \inf \setdef{\dquasipot(\state, \statealt)}{\state \in A, \statealt \in B}\,.
\end{equation}

Recall that for any $\idx, \idxalt$, we define
	\begin{equation}
		\dquasipot_{\idx, \idxalt} =
		\dquasipot(\eqcl_\idx, \eqcl_\idxalt) = \inf \setdef{\dquasipot(\state, \statealt)}{\state \in \eqcl_\idx, \statealt \in \eqcl_\idxalt}\,.
	\end{equation}

Recall that, to avoid degenerate cases, we assume that \cref{asm:costs} holds, that is, 
\begin{equation}
\qmat_{\iComp\jComp}
	< \infty
	\quad
	\text{for all $\iComp,\jComp=1,\dotsc,\nComps$}.
\end{equation}
With \cref{lem:finite_B}, this assumption is satisfied in particular if the following condition holds: for any $\iComp,\jComp=1,\dotsc,\nComps$, there exists a $\contdiff{1}$ path $\pth$ joining $\comp_{\iComp}$ and $\comp_{\jComp}$ such that, for all $\time$, $\grad \obj(\pth_\time)$ belongs to the interior of the closed convex hull of the support of the noise $\noise(\pth_\time, \sample)$:
\begin{equation}
	\grad \obj(\pth_\time) \in \intr \clconv \supp \noise(\pth_\time, \sample)\,,
\end{equation}
This means there is a sufficient level of noise for the probability of going from $\comp_{\iComp}$ to $\comp_{\jComp}$ with at least one path to be non-zero (though it can be vanishingly small).
Note that it does not constrain the nature of the noise itself --- which can be discrete, continuous or else ---, only the support of its distribution.

Moreover, if \cref{asm:costs} did not hold, the same analysis as in this section could still be carried out. We would consider the components of the graph $\graph$ connected by edges with finite weights proceed with the proof on each of them, and obtain the same results on each of these components. To keep the complexity of the proof reasonable, we will not consider this case here.

\WAdelete{
By combining \cref{lem:W_nbd} with \cref{lem:equivalence_classes_dpth_stay_close}, we immediately obtain the following result.

\begin{lemma}
	\label{lem:W_dpth_stay_close}
	For $\eqcl$ equivalence class, for any $\precs > 0$ small enough, there exists $\W \subset \W_\precs(\eqcl)$ neighborhood of $\eqcl$ such that, for any $\V$ neighborhood of $\eqcl$, there exists $\nRuns \geq 1$ such that,
	\begin{equation}
		\forall \point, \pointalt \in \W,\quad
		\exists \run \leq \nRuns,\,
		\dpth \in (\vecspace)^\run,\quad
		\begin{cases}
			\dpth_\start = \point\,, \dpth_{\run - 1} = \pointalt\\
			\dpth_{\runalt} \in \V \quad\text{ for all }\quad \start < \runalt < \run - 1\\
			\daction_{\run}(\dpth) < \precs
		\end{cases}
	\end{equation}
\end{lemma}
}

We adapt to our context \citet[Lem.~5.4]{Kif88} and simplify it using ideas from \citet[Chap.~6]{FW98}.

\begin{definition}[{\citet[Chap.~6,\S2]{FW98}}]
	For $\idx, \idxalt$,
	\begin{equation}
		\dquasipotalt_{\idx, \idxalt} \defeq
		\inf \setdef*{\daction_{\nRuns}(\dpth)}{\nRuns \geq 1, \dpth \in \vecspace^{\nRuns},\, \dpth_\start \in \eqcl_\idx,\, \dpth_{\nRuns - 1} \in \eqcl_\idxalt,\, \dpth_\run \notin \bigcup_{\idxaltalt \neq \idx,\idxalt} \eqcl_\idxaltalt \text{ for } \run = 1,\dots, \nRuns - 2}\,.
	\end{equation}
\end{definition}

We now defined an important object: the law of the (accelerated) iterated at the first time they reach some set $\nhdaltalt$ (typically a neighborhood of the critical set), following \eg \citet[Chap. 3.4]{doucMarkovChains2018}.

\begin{definition}[{\citet[Prop.~5.3]{Kif88}}]
	For $\nhdaltalt$ open set, with hitting time (of the accelerated sequence) 
$
		\hittime_\nhdaltalt \defeq \inf \setdef{\run \geq \afterstart}{\accstate_\run \in \nhdaltalt}\,,
$ (as in \cref{def:stoptimes}), denote the law of $\accstate_{\hittime_\nhdaltalt}$ started at $\state$ by $\probalt_\nhdaltalt(\state, \cdot)$ and the corresponding $\nRuns$-step transition probability by $\probalt_\nhdaltalt^\nRuns(\state, \cdot)$. 
\end{definition}

In words, $\probalt_\nhdaltalt(\state, \cdot)$  is the distributions of the $\accstate$ started at $\state$ at the first time they reach $\nhdaltalt$. Typically, $\probalt_\nhdaltalt(\state, \nhd)$ is the probability that $\nhd \subset \nhdaltalt$ is reached first among $\nhdaltalt$. Then, $\probalt_\nhdaltalt^\nRuns(\state, \cdot)$ is the distribution of the $\accstate$ started at $\state$ at the $\nRuns$-th time they reach $\nhdaltalt$. 

We first give estimates of the transition probabilities using the $\dquasipotalt_{\idx,\idxalt}$. We will then translate them to $\dquasipot_{\idx,\idxalt}$.

\begin{lemma}
	\label{lem:est_trans_prob_ub}
		For any $\precs > 0$, for any small enough neighborhoods $\nhdaltalt_\idx$ of $\eqcl_\idx$, $\idx = 1, \dots, \neqcl$ , there is some $\step_0 > 0$ such that for all $\idx, \idxalt$, $\state \in \nhdaltalt_{\idx}$, $0 < \step < \step_0$,
		\begin{equation}
			\probalt_{\nhdaltalt}(\state, \nhdaltalt_{\idxalt})
			\leq 
			\exp \parens*{
				- \frac{\dquasipotalt_{\idx, \idxalt}}{\step}
				+ \frac{\precs}{\step}
			}\,.
		\end{equation}
		where we defined $\nhdaltalt \defeq \bigcup_{\idx = 1}^{\neqcl} \nhdaltalt_\idx$.
\end{lemma}

\begin{proof}
	Assume that, without loss of generality, $\precs$ is small enough so that \cref{lem:W_nbd} with $\radius \gets \precs$ can be applied to every $\eqcl_\idx$, $\idx =1, \dots, \neqcl$. Denote by $\W_\idx$, $\idx = 1, \dots, \neqcl$ the corresponding neighborhoods of $\eqcl_\idx$.

    \revise{Since these $\W_\idx$'s are neighborhoods of the $\eqcl_\idx$'s,} there exists $\margin > 0$ such that $\U_{2 \margin} (\eqcl_\idx) \subset \W_\idx$ for all $\idx = 1, \dots, \neqcl$.
	Require then that $\V_\idx$ be contained in $\U_\margin(\eqcl_\idx)$ so that 
	$\U_\margin(\V_\idx) \subset \W_\idx$ for all $\idx = 1, \dots, \neqcl$.
	Moreover, assume that $\margin > 0$ is small enough so that the neighborhoods $\U_\margin(\eqcl_\idx)$, $\idx = 1, \dots, \neqcl$ are pairwise disjoint.

	Define $0 < \marginalt \leq \margin$ such that $\U_\marginalt(\eqcl_\idx)$ is contained in $\V_\idx$ for all $\idx = 1, \dots, \neqcl$.

	Fix $\idx, \idxalt \in \indices$ and consider $\dpth \in (\vecspace)^{\nRuns}$ such that $\dpth_\start \in \U_\marginalt(\V_\idx) \subset \W_\idx$, $\dpth_{\nRuns - 1} \in \U_\marginalt(\V_\idxalt) \subset \W_\idxalt$ and $\dpth_\run \in \U_{\marginalt}\parens*{\vecspace \setminus \bigcup_{\idxaltalt \neq \idx, \idxalt} \V_\idxaltalt}$ for all $\run = 1, \dots, \nRuns - 2$.
	By the choice of $\marginalt$, $\dpth_\run$ cannot be in $\bigcup_{\idxaltalt \neq \idx, \idxalt} \eqcl_\idxaltalt$ for any $\run = 1, \dots, \nRuns - 2$.

	By definition of $\W_\idx$ and $\W_\idxalt$, there are $\point \in \eqcl_\idx$, $\pointalt \in \eqcl_\idxalt$ such that $\rate(\point, \dpth_\start) < \precs$, $\rate(\dpth_{\nRuns - 1}, \pointalt) < \precs$. Therefore, the path $\dpthalt \in (\vecspace)^{\nRuns + 2}$ defined as $\dpthalt = (\point, \dpth_\start, \dpth_1, \dots, \dpth_{\nRuns - 1}, \pointalt)$ satisfies
	\begin{equation}
		\daction_\nRuns(\dpth) \geq \daction_{\nRuns + 2}(\dpthalt) - 2 \precs\,.
	\end{equation}
	and, by definition of $\dquasipotalt_{\idx, \idxalt}$, we thus obtain,
	\begin{equation}
		\daction_\nRuns(\dpth) \geq \daction_{\nRuns+2}(\dpthalt) - 2 \precs \geq \dquasipotalt_{\idx, \idxalt} - 2 \precs\,.
		\label{eq:prop:est_trans_prob:step1:lower_bound_dpth}
	\end{equation}
	
	Fix $\state \in \V_\idx$. Let us now bound the probability
	\begin{equation}
		\probalt_{\nhdaltalt}(\state, \nhdaltalt_{\idxalt}) =
		\prob_\state \parens*{\accstate_{\hittime_\V} \in \V_\idx}\,.
	\end{equation}

	We have, for any $\nRuns \geq \start$,
	\begin{equation}
		\prob_\state \parens*{\accstate_{\hittime_\V} \in \V_\idx} \leq
		\prob_\state \parens*{\accstate_{\hittime_\V} \in \V_\idx,\, \hittime_\V < \nRuns} + \prob_\state \parens*{\hittime_\V \geq \nRuns}\,.
	\end{equation}
	We first bound the second probability using \cref{lem:est_going_back_to_crit} applied to $\open \gets \V_\idx$.
	Take $\nRuns$ such that $\coef_0(\constA - \nRuns) + \step_0 \constB \leq - \dquasipotalt_{\idx, \idxalt}$.
	Then, by Markov's inequality and \cref{lem:est_going_back_to_crit}, it holds that for all $\step \leq \step_0$
	\begin{align}
		\prob_\state \parens*{\hittime_\V \geq \nRuns} 
		&\leq 
		\prob_\state \parens*{\exp \parens*{\frac{\coef_0 \hittime_\V}{\step}} \geq \exp \parens*{\frac{\coef_0 \nRuns}{\step}}}
		\notag\\
		&\leq \exp \parens*{\frac{\coef_0 (\constA - \nRuns)}{\step} + \constB}
		\notag\\
		&\leq \exp \parens*{\frac{- \dquasipotalt_{\idx, \idxalt}}{\step}}\,.
		\label{eq:prop:est_trans_prob:step1:upper_bound_hittime}
	\end{align}
	
	We now bound the term $\prob_\state \parens*{\accstate_{\hittime_\V} \in \V_\idxalt,\, \hittime_\V < \nRuns}$ for this choice of $\nRuns$.

	For this, we show that $\accstate_{\hittime_\V} \in \V_\idxalt$ with $\hittime_\V < \nRuns$ implies that 
	\begin{equation}
	\dist_{\nRuns}\parens*{\accstate, \pths_{\nRuns}^{\setof{\state}}\parens*{\dquasipotalt_{\idx, \idxalt} - 3 \precs}} > \half[\marginalt]\,.
	\end{equation}
	Indeed, on the event $\accstate_{\hittime_\V} \in \V_\idxalt$ with $\hittime_\V < \nRuns$, there is some $\nRunsalt \leq \nRuns$ such that $\hittime_\V = \nRunsalt - 1$.
	If $\dist_{\nRuns}\parens*{\accstate, \pths_{\nRuns}^{\setof{\state}}\parens*{\dquasipotalt_{\idx, \idxalt} - 3 \precs}} > \half[\marginalt]$ did not hold, this would mean that there exists $\dpth \in (\vecspace)^{\nRunsalt}$ such that $\dist_{\nRunsalt}(\accstate, \dpth) < \marginalt$, $\dpth_\start = \state$ and, 
	$\daction_{\nRunsalt}(\dpth) \leq \dquasipotalt_{\idx, \idxalt} - 3 \precs$. In particular, $\dpth$ would also satisfy
	$\dpth_{\nRunsalt - 1} \in \U_{\marginalt}(\V_\idxalt)$, $\dpth_\run \in \U_{\marginalt}(\vecspace \setminus \V)$ for all $\run = 1, \dots, \nRunsalt - 2$. This would be in direct contradiction of \cref{eq:prop:est_trans_prob:step1:lower_bound_dpth}. 

	Therefore, we have that
	\begin{align}
		\prob_\state \parens*{\accstate_{\hittime_\V} \in \V_\idxalt,\, \hittime_\V < \nRuns}
		&\leq 
		\prob_\state \parens*{\dist_{\nRuns}\parens*{\accstate, \pths_{\nRuns}^{\setof{\state}}\parens*{\dquasipotalt_{\idx, \idxalt} - 3 \precs}} > \half[\marginalt]}
		\notag\\
		&\leq 
		\exp \parens*{- \frac{\dquasipotalt_{\idx, \idxalt} - 4 \precs}{\step}}\,,
	\end{align}
	by \cref{cor:ldp_iterates_full}.

	Combining this bound with \cref{eq:prop:est_trans_prob:step1:upper_bound_hittime} yields
	\begin{equation}
		\prob_\state(\accstate_{\hittime_\V} \in \V_\idxalt) \leq \exp \parens*{- \frac{\dquasipotalt_{\idx, \idxalt} - 4 \precs}{\step}} + \exp \parens*{\frac{- \dquasipotalt_{\idx, \idxalt}}{\step}}\,,
	\end{equation}
	which concludes the proof.
\end{proof}

\begin{lemma}
   \label{lem:est_trans_prob_lb}
		For any $\precs > 0$,for any neighborhoods $\nhdaltalt_\idx$ of $\eqcl_\idx$, $\idx = 1, \dots, \neqcl$ small enough, there exists $\nRuns \geq \start$, $\step_0 > 0$ such that for all $\idx, \idxalt$, $\state \in \nhdaltalt_{\idx}$, $0 < \step < \step_0$,
		\begin{equation}
			\probalt_{\nhdaltalt}^\nRuns(\state, \nhdaltalt_{\idxalt})
			\geq
			\exp \parens*{
				- \frac{\dquasipotalt_{\idx, \idxalt}}{\step}
				- \frac{\precs}{\step}
			}\,.
		\end{equation}
	
\end{lemma}

\begin{proof}
	For any $\idx, \idxalt$, there exists $\nRuns_{\idx, \idxalt} \geq 1$, $\dpth^{\idx,\idxalt} \in (\vecspace)^{\nRuns_{\idx,\idxalt}}$ such that $\dpth^{\idx,\idxalt}_\start \in \eqcl_\idx$, $\dpth^{\idx,\idxalt}_{\nRuns_{\idx,\idxalt} - 1} \in \eqcl_\idxalt$, $\dpth^{\idx,\idxalt}_\run \notin \bigcup_{\idxaltalt\neq\idx,\idxalt} \eqcl_\idxaltalt$ for all $\run = 1, \dots, \nRuns_{\idx,\idxalt} - 2$ and $\daction_{\nRuns_{\idx,\idxalt}}(\dpth^{\idx,\idxalt}) \leq \dquasipotalt_{\idx, \idxalt} + \precs$. Define $\margin_{\idx, \idxalt} \defeq \min \setdef*{d(\dpth^{\idx,\idxalt}_\run, \bigcup_{\idxaltalt\neq\idx,\idxalt} \eqcl_\idxaltalt)}{\run = 1, \dots, \nRuns_{\idx,\idxalt} - 2}$ and $\margin \defeq \min_{\idx, \idxalt \in \indices} \margin_{\idx, \idxalt}$. By construction, it holds that $\margin > 0$.

	Require that $\V_\idx$ be contained in $\W_\idx \cap \U_{\margin/2}(\eqcl_\idx)$ for all $\idx = 1, \dots, \neqcl$.
	Now, given such $\V_\idx$ neighborhoods of $\eqcl_\idx$, $\idx = 1, \dots, \neqcl$, there exists $0 < \marginalt \leq \margin / 2$ such that $\U_\marginalt(\eqcl_\idx)$ is contained in $\V_\idx$ for all $\idx = 1, \dots, \neqcl$.

	Apply \cref{lem:equivalence_classes_dpth_stay_close} to $\eqcl_\idx$, $\idx = 1, \dots, \neqcl$ with $\precs \gets \min(\precs, \marginalt/2)$ and denote by $\nRuns_\idx$ the bound on the length of paths obtained.	
	Define
	\begin{equation}
		\nRuns \defeq \max_{\idx \in \indices} \nRuns_\idx  + 1\,.
	\end{equation}

	Fix $\idx, \idxalt \in \indices$ and $\stateA \in \V_\idx$. Since $\V_\idx \subset \W_\idx$, there exists $\stateB \in \eqcl_\idx$ such that $\rate(\stateA, \stateB) < \precs$. Moreover, note that $\rate(\stateB, \stateB) = 0$ since $\stateB$ is a critical point of $\obj$.

	By \cref{lem:equivalence_classes_dpth_stay_close}, there exists $\run \leq \nRuns$, $\dpth \in (\vecspace)^\run$ such that $\dpth_\start = \stateB$, $\dpth_{\nRuns - 1} = \dpth^{\idx, \idxalt}_\start$, $\dpth_\runalt \in \U_{\marginalt/2}(\eqcl_\idx)$ for all $\runalt = 1, \dots, \run - 2$ and $\daction_\run(\dpth) <\precs$.

	Considering the concatenation 
	\begin{equation}
		\dpthalt \defeq \parens*{
		\stateA,
		\underbrace{
		\stateB,
		\stateB,
		\dots,
		\stateB
		}_{\nRuns - \run \text{ times}},
		\dpth_\start,
		\dpth_\afterstart,
		\dots,
		\dpth_{\run - 2},
		\dpth_{\run - 1},
		\dpth^{\idx, \idxalt}_\afterstart,
		\dots,
		\dpth^{\idx, \idxalt}_{\nRuns_{\idx,\idxalt} - 1}
	}
	\end{equation}
	which is a path of length $\nRuns + \nRuns_{\idx,\idxalt}$ made of $\state \in \V_\idx$, then exactly $\nRuns$ points in $\U_{\marginalt/2}(\eqcl_\idx)$ then $\nRuns_{\idx,\idxalt} - 2$ in $\vecspace \setminus \U_{\margin / 2}(\V)$ and $\dpth^{\idx, \idxalt}_{\nRuns_{\idx,\idxalt} - 1} \in \eqcl_\idxalt$. 
	Moreover, by construction, $\daction_{\nRuns + \nRuns_{\idx,\idxalt}}(\dpthalt) \leq \dquasipotalt_{\idx, \idxalt} + 3\precs$.
	Therefore, if 
	\begin{equation}\label{eq:closepath}
	\dist_{\nRuns + \nRuns_{\idx,\idxalt}}\parens*{\accstate, \dpthalt} < \marginalt/2\,,
	\end{equation}
	with $\accstate_\start = \stateA$, then $\accstate_1,\dots,\accstate_\nRuns$ are in $\U_{\marginalt}(\eqcl_\idx) \subset \V_\idx$  and, since $\marginalt \leq \margin / 2$, $\accstate_{\nRuns + 1}, \dots, \accstate_{\nRuns + \nRuns_{\idx,\idxalt} - 2}$ are not in $\U_{\margin / 4}(\V)$, and therefore not in $\V$. Moreover, $\accstate_{\nRuns + \nRuns_{\idx,\idxalt} - 1}$ would be in $\U_{\marginalt/2}(\eqcl_\idxalt) \subset \V_\idxalt$.

	Thus, all the paths $\accstate$ satisfying \eqref{eq:closepath} with $\accstate_\start = \stateA$ are started at $\stateA$ and their $\nRuns$-th point that fall into $\nhdaltalt$ belongs to $\nhdaltalt_{\idxalt}$. 

	Therefore, using the definition of $ \probalt_{\nhdaltalt}^\nRuns$, we have that 
	\begin{align}
		\probalt_{\nhdaltalt}^\nRuns(\state, \nhdaltalt_{\idxalt})
		&\geq
		\prob_\state \parens*{\dist_{\nRuns + \nRuns_{\idx,\idxalt}}\parens*{\accstate, \dpthalt} < \marginalt/2}
		\notag\\
		&\geq
		\exp \parens*{-\frac{\dquasipotalt_{\idx, \idxalt} + 4 \precs}{\step}}\,,
	\end{align}
	by \cref{cor:ldp_iterates_full}.
\end{proof}

\begin{lemma}
	\label{lem:dquasipotalt_to_dquasipot}
	For any $\idx,\idxalt \in \indices$,
	\begin{align}
		\dquasipot_{\idx, \idxalt} 
		&=
		\min
		\setdef*{
			\sum_{\idxaltalt = 0}^{\run-2} \dquasipotalt_{\idx_{\idxaltalt}, \idx_{\idxaltalt + 1}}
		}{
			\idx_0 = \idx,\, \idx_{\run - 1} = \idxalt,\, \idx_{\idxaltalt} \in \indices \text{ for } \idxaltalt = 1, \dots, \run - 2\,, \run \geq 1
		}
		\notag\\
		&=
		\min
		\setdef*{
			\sum_{\idxaltalt = 0}^{\neqcl - 2} \dquasipotalt_{\idx_{\idxaltalt}, \idx_{\idxaltalt + 1}}
		}{
			\idx_0 = \idx,\, \idx_{\neqcl - 1} = \idxalt,\, \idx_{\idxaltalt} \in \indices \text{ for } \idxaltalt = 1, \dots, \neqcl - 2
		}\,.
	\end{align}
\end{lemma}

\begin{proof}
	It suffices to show that
	\begin{equation}
		\dquasipot_{\idx, \idxalt} = 
		\min
		\setdef*{
			\sum_{\idxaltalt = 0}^{\run-2} \dquasipotalt_{\idx_{\idxaltalt}, \idx_{\idxaltalt + 1}}
		}{
			\idx_0 = \idx,\, \idx_{\run - 1} = \idxalt,\, \idx_{\idxaltalt} \in \indices \text{ for } \idxaltalt = 1, \dots, \run - 2\,, \run \geq 1
		}\,.
	\end{equation}
	The statement of the lemma then follows from the fact that $\dquasipotalt_{\idxaltalt,\idxaltalt} = 0$ for all $\idxaltalt \in \indices$ and that shortest paths on graphs can be chosen not to visit the same node twice.

	For the inequality $(\geq)$, notice that any path between $\eqcl_\idx$ and $\eqcl_\idxalt$ can be decomposed into a concatenation of paths between $\eqcl_\idx$ and $\eqcl_{\idx_1}$, $\eqcl_{\idx_1}$ and $\eqcl_{\idx_2}$, \dots, $\eqcl_{\idx_{\run - 1}}$ and $\eqcl_\idxalt$  for some $\idx_1, \dots, \idx_{\run - 1} \in \indices$ that do not enter any other equivalence class in between. Therefore, the inequality $(\geq)$ follows from the definition of $\dquasipot_{\idx, \idxalt}$ and $\dquasipotalt_{\idx_{\idxaltalt}, \idx_{\idxaltalt + 1}}$.

	We now focus on $(\leq)$.

	Fix $\precs$.
	Take $\run \geq 1$, $\idx_0 = \idx$, $\idx_{\run - 1} = \idxalt$, $\idx_{\idxaltalt} \in \indices$ for $\idxaltalt = 1, \dots, \run - 2$.
	There are paths $\dpth^{0},\dots,\dpth^{\run-2}$ of lengths $\nRuns_0,\dots,\nRuns_{\run-2}$ such that $\dpth^{\idxaltalt}_\start \in \eqcl_{\idx_{\idxaltalt}}$, $\dpth^{\idxaltalt}_{\nRuns_{\idxaltalt} - 1} \in \eqcl_{\idx_{\idxaltalt + 1}}$, $\daction_{\nRuns_{\idxaltalt}}(\dpth^{\idxaltalt}) \leq \dquasipotalt_{\idx_{\idxaltalt}, \idx_{\idxaltalt + 1}} + \precs/\run$ for $\idxaltalt = 0, \dots, \run - 2$.

	By \cref{lem:equivalence_classes_dpth_stay_close}, for all $\idxaltalt = 0, \dots, \run - 2$, $\dpth^{\idxaltalt}_{\nRuns_{\idxaltalt} - 1}$ and $\dpth^{\idxaltalt + 1}_\start$ can be connected by path of cost at most $\precs / \run$. Therefore concatenating all these paths yield $\dpthalt$ of length $\nRuns$ with $\dpthalt_\start \in \eqcl_\idx$, $\dpthalt_{\nRuns - 1} \in \eqcl_\idxalt$ and $\daction_\nRuns(\dpthalt) \leq \sum_{\idxaltalt = 0}^{\run - 2} \dquasipotalt_{\idx_{\idxaltalt}, \idx_{\idxaltalt + 1}} + 2\precs$. Since $\daction_\nRuns(\dpthalt) \geq \dquasipot_{\idx, \idxalt}$, we obtain the desired result.
\end{proof}

\begin{notation}
	We will write, for non-decreasing $\fn \from \R \to \R$,
	\begin{equation}
		a \asympteq \fn(b \pm c) \iff \fn(b - c) \leq a \leq \fn(b + c)\,.
	\end{equation}
\end{notation}

\begin{proposition}
	\label{prop:est_trans_prob}
		For any $\precs > 0$ and any small enough neighborhoods $\nhdaltalt_\idx$ of $\eqcl_\idx$, $\idx = 1, \dots, \neqcl$, there exists $\nRuns \geq \start$, $\step_0 > 0$ such that for all $\idx, \idxalt$, $\state \in \nhdaltalt_{\idx}$, $0 < \step < \step_0$,
		\begin{equation}
			\probalt_{\nhdaltalt}^\nRuns(\state, \nhdaltalt_{\idxalt})
			\asympteq
			\exp \parens*{
				- \frac{\dquasipot_{\idx, \idxalt}}{\step}
				\pm \frac{\precs}{\step}
			}\,.
		\end{equation}
	
\end{proposition}
\begin{proof}
Let us first start with $(\geq)$.

Let $\nRuns$ satisfy the conditions of \cref{lem:est_trans_prob_lb} and define $\nRunsalt \defeq (\neqcl-1) \nRuns$.
For any $\idx, \idxalt \in \indices$, by \cref{lem:dquasipotalt_to_dquasipot}, there exist $\idx_0 = \idx$, $\idx_{\neqcl - 1} = \idxalt$, $\idx_{\idxaltalt} \in \indices$ for $\idxaltalt = 1, \dots, \neqcl - 2$ such that
\begin{equation}
	\dquasipot_{\idx, \idxalt} = \sum_{\idxaltalt = 0}^{\neqcl - 2} \dquasipotalt_{\idx_{\idxaltalt}, \idx_{\idxaltalt + 1}}\,.
\end{equation}

Therefore, we have that the probability of reaching $\nhdaltalt_{\idxalt}$ from $\state \in \V_\idx$ in $\nRunsalt = (\neqcl-1) \nRuns$ steps is greater than the probability of reaching sequentially the $\nhdaltalt_{\idx_{\idxaltalt+1}}$ from any point $\state \in \nhdaltalt_{\idx_{\idxaltalt}}$, \ie
\begin{align}
	\inf_{\state \in \V_\idx}
	\probalt_{\nhdaltalt}^{\nRunsalt}(\state, \nhdaltalt_{\idxalt})
	&\geq
	\prod_{\idxaltalt = 0}^{\neqcl - 2}
	\inf_{\state \in \V_{\idx_{\idxaltalt}}}
	\probalt_{\nhdaltalt}^\nRuns(\state, \nhdaltalt_{\idx_{\idxaltalt + 1}})
	\notag\\
	&\geq
	\prod_{\idxaltalt = 0}^{\neqcl - 2}
	\exp \parens*{
		- \frac{\dquasipotalt_{\idx_{\idxaltalt}, \idx_{\idxaltalt + 1}}}{\step}
		- \frac{\precs}{\step}
	} = \exp \parens*{
		- \frac{\dquasipot_{\idx, \idxalt}}{\step}
		- \frac{(\neqcl - 1)\precs}{\step}
	}\,
\end{align}
where we used \cref{lem:est_trans_prob_lb} to get the second inequality.

Now for the reverse inequality $(\leq)$. Take $\idx, \idxalt \in \indices$ and denote $\idx_0 = \idx$, $\idx_{\nRunsalt - 1} = \idxalt$. By \cref{lem:est_trans_prob_ub}, we have that 
\begin{align}
	\sup_{\state \in \V_\idx}
	\probalt_{\nhdaltalt}^{\nRunsalt}(\state, \nhdaltalt_{\idxalt})
	&\leq
	\sum_{\idx_1, \dots, \idx_{\neqcl-2} \in \indices}
	\prod_{\idxaltalt = 0}^{\neqcl - 2}
	\sup_{\state \in \V_{\idx_{\idxaltalt}}}
	\probalt_{\nhdaltalt}^\nRuns(\state, \nhdaltalt_{\idx_{\idxaltalt + 1}})
	\notag\\
	&\leq
	\sum_{\idx_1, \dots, \idx_{\neqcl-2} \in \indices}
	\prod_{\idxaltalt = 0}^{\neqcl - 2}
	\exp \parens*{
		- \frac{\dquasipotalt_{\idx_{\idxaltalt}, \idx_{\idxaltalt + 1}}}{\step}
		+ \frac{\precs}{\step}
	}
	\notag\\
	&\leq 
	\sum_{\idx_1, \dots, \idx_{\neqcl-2} \in \indices}
	\exp \parens*{
		- \frac{\sum_{\idxaltalt = 0}^{\nRunsalt} \dquasipotalt_{\idx_{\idxaltalt}, \idx_{\idxaltalt + 1}}}{\step}
   + \frac{(\neqcl - 1)\precs}{\step}
	}\,.
\end{align}
Using \cref{lem:dquasipotalt_to_dquasipot}, we obtain that 
\begin{align}
\sup_{\state \in \V_\idx} \probalt_{\nhdaltalt}^{\nRunsalt}(\state, \nhdaltalt_{\idxalt})
	&\leq \sum_{\idx_1, \dots, \idx_{\neqcl-2} \in \indices}
	\exp \parens*{
		- \frac{\dquasipot_{\idx, \idxalt}}{\step}
		+ \frac{(\neqcl - 1)\precs}{\step}
	}
	\notag\\
	&= {\neqcl}^{\neqcl - 2}
	\exp \parens*{
		- \frac{\dquasipot_{\idx, \idxalt}}{\step}
		+ \frac{(\neqcl - 1)\precs}{\step}
	}
\end{align}
which concludes the proof.
\end{proof}

Let us now give our first estimates on the invariant measure.

\begin{definition}[{\citep[Chap.~6,\S4]{FW98}}]
	Define, for every $\idx \in \indices$,
	\begin{equation}
		\dinvpot_\idx  = \dinvpot(\cpt_\idx) \defeq \min_{\tree \in \trees_\idx} \sum_{(\idxalt \to \idxaltalt) \in \tree} \dquasipot_{\idxalt, \idxaltalt}\,,
	\end{equation}
	where $\trees_\idx$ denotes the set of trees rooted at $\idx$ in the complete graph on $\indices$.
\end{definition}

\begin{proposition}
	\label{prop:est_induced_invmeas}
	For any $\invmeas[\V]$ invariant probability measure for $\probalt_\V$, the induced chain on $\V$, in the setting of \cref{prop:est_trans_prob}, for any $\idx \in \indices$,
	\begin{equation}
		\invmeas[\V](\V_\idx) 
		\asympteq
		  \exp \parens*{
			  - \frac{\dinvpot(\cpt_\idx)  - \min_{\idxalt \in \indices} \dinvpot(\cpt_\idxalt)}{\step}
			  \pm \frac{\precs}{\step}
		  }\,.
		\end{equation}
\end{proposition}
\begin{proof}
	If $\invmeas[\V]$ is an invariant measure of $\probalt_\V$, then it is an invariant measure of $\probalt_{\nhdaltalt}^\nRuns$ for any $\nRuns$ given by \cref{prop:est_trans_prob}.

	\citet[Chap~.6, Lem.~3.1-3.2]{FW98} combined with \cref{prop:est_trans_prob} then give
\begin{align}
\MoveEqLeft
\exp(- (2\card \indices +2) \precs / \step)
		  \frac{\exp( - \dinvpot(\cpt_\idx) / \step)}
			{\sum_{\idxalt \in \indices} \exp( - \dinvpot(\cpt_\idxalt) / \step)}
			\notag\\
			&\leq 
			\invmeas[\V](\V_\idx)
			\notag\\
			&\leq 
		  \exp((2 \card \indices +2) \precs / \step)
		  \frac{\exp( - \dinvpot(\cpt_\idx) / \step)}
			{\sum_{\idxalt \in \indices} \exp( - \dinvpot(\cpt_\idxalt) / \step)}\,.
		\end{align}
		For $\step$ small enough, it holds that 
		\begin{equation}
			\sum_{\idxalt \in \indices} \exp( - \dinvpot(\cpt_\idxalt) / \step)
			\asympteq
			\exp \parens*{- \min_{\idxalt \in \indices} \dinvpot(\cpt_\idxalt) / \step \pm \precs / \step}\,,
		\end{equation}
		which concludes the proof.
\end{proof}

We now state a result that links invariant measures of $(\curr[\accstate])_\run$ and $\probalt_\V$.
It is a consequence of \citet[Thm.~3.6.5]{doucMarkovChains2018}.

\begin{lemma}
\label{lem:invmeas_from_induced_invmeas}
	There is $\step_0 > 0$ such that, for $0 < \step \leq \step_0$,
	if $(\curr[\accstate])_\run$ has an invariant probability measure $\invmeas$,
	then, for any $\V$ measurable neighborhood of $\crit\obj$,
	we have that $\invmeas(\V) > 0$ and $\invmeas[\V]$, the restriction of ${\invmeas}/{\invmeas(\V)}$ to $\V$ is an invariant measure for the induced chain on $\V$ and, for any measurable set $\event \subset \points$, 
	\begin{equation}
		\frac{\invmeas(\event)}{\invmeas(\V)} =  \int_{\V} \dd \invmeas[\V](\point) \ex_\point \bracks*{
			\sum_{\run = \start}^{\hittime_\V - 1} \oneof{\curr[\accstate] \in \event}
		}\,.
	\end{equation}
\end{lemma}
\begin{proof}
	We invoke the first item of \cref{lem:exponential_tail_bounds}: there exists $\bigcpt \subset \vecspace$ a compact set, $\step_0 > 0$, such that for any $\bigcptalt \subset \vecspace$ compact set such that $\bigcpt \subset \bigcptalt$, there exists $\coef_0  > 0$ such that,
	\begin{equation}
		\forall \step \leq \step_0,\, \state \in \bigcptalt\,,\quad
	\ex_{\state} \bracks*{e^{\frac{\coef_0 \hittime_\bigcpt}{\step}}} < +\infty\,.
	\label{eq:lemma:induced_inv_meas:exponential_tail_bounds}
	\end{equation}
	Without loss of generality, at the potential expense of expanding $\bigcpt$, assume that $\V \subset \bigcpt$.

	By applying  \cref{lem:est_going_back_to_crit} to $\open \gets \V$ and $\bigcpt \gets \bigcpt$, we get that there is some $\step_0,\coef_0, > 0$ such that, for any $\step \leq \step_0$, $\state \in \bigcpt$,
	\begin{equation}
	  \ex_\state \bracks*{e^{\frac{\coef \hittime_\V}{\step}}} < +\infty\,.
	\end{equation}
	In particular, we have that $\prob_{\state}(\hittime_\V < \infty) = 1$ for any $\state \in \bigcpt$, and \emph{a fortiori} for any $\state \in \V$.

	Let us now show that, for any $\state \in \vecspace$, $\prob_{\state}(\hittime_\V < \infty) = 1$.
	Fix $\state \in \vecspace$. By choosing $\bigcptalt$ large enough to contain both $\bigcpt$ and $\point$, 
	\cref{eq:lemma:induced_inv_meas:exponential_tail_bounds} implies that, with $\init = \point$, $\hittime_\bigcpt < \infty$ almost surely.
	Therefore, by the strong Markov property, it holds that,
	
	\begin{equation}
		\prob_{\state}( \hittime_\V < \infty) \geq \inf_{\statealt \in \bigcpt}\prob_{\statealt}( \hittime_\V < \infty) = 1\,.
	\end{equation}

	Therefore the assumptions of \citet[Thm.~3.6.5]{doucMarkovChains2018} are satisfied as well as its first item which yields the result.
\end{proof}

We reach the next proposition, which is the first part of our main result. It is an adaptation of \citet[Thm.~4.1]{FW98} to the discrete time setting.

\begin{proposition}
	\label{prop:est_invmeas}
	For any $\precs > 0$, for any $\V_1, \dots, \V_\neqcl$ measurable neighborhoods of $\eqcl_1, \dots, \eqcl_\neqcl$ small enough, there exists $\step_0 > 0$ such that 
	for any $0 < \step < \step_0$,
	$\invmeas$ invariant probability measure for $(\curr[\accstate])_\run$,
	for any $\idx \in \indices$, 
	\begin{equation}
		\invmeas(\V_\idx) \asympteq \exp \parens*{
			- \frac{\dinvpot(\cpt_\idx) - \min_{\idxalt \in \indices} \dinvpot(\cpt_\idxalt)}{\step}
			\pm \frac{\precs}{\step}
		}\,.
	\end{equation}
\end{proposition}
\begin{proof}
	Let us first provide estimates for the unnormalized measure defined by, for any measurable set $\event$,
	\begin{equation}
		\measalt(\event) \defeq \int_{\event} \dd \invmeas[\V](\point) \ex_\point \bracks*{
			\sum_{\run = \start}^{\hittime_\V - 1} \oneof{\curr[\accstate] \in \event}
		}\,.
	\end{equation}

	By definition of $\hittime_\V$, in the sequence of points $\accstate_\start, \dots, \accstate_{\hittime_\V-1}$, only $\accstate_\start$ can be in $\V$.
	Therefore, for any $\idx, \idxalt \in \indices$, $\state \in \V_\idxalt$,
	\begin{equation}
		\ex_{\state} \bracks*{
			\sum_{\run = \start}^{\hittime_\V - 1} \oneof{\curr[\accstate] \in \V_\idx}
		}
		= 
		\begin{cases}
			1 & \text{if } \idx = \idxalt\,,\\
			0 & \text{otherwise}\,.
		\end{cases}
	\end{equation}

	Thus, we have that,
	\begin{align}
		\measalt(\V_\idx) &= \invmeas[\V](\V_\idx)
		\notag\\
						  &\asympteq \exp \parens*{
			- \frac{\dinvpot(\cpt_\idx) - \min_{\idxalt \in \indices} \dinvpot(\cpt_\idxalt)}{\step}
			\pm \frac{\precs}{\step}
		}\,,
	\end{align}
	by \cref{prop:est_induced_invmeas}.
	It now remains to estimate the normalization constant $\measalt(\vecspace)$.
	On the one hand, we have that
	\begin{align}
		\measalt(\vecspace) &\geq \max_{\idx \in \indices} \measalt(\V_\idx)
		\notag\\
							&\geq \max_{\idx \in \indices} \exp \parens*{
			- \frac{\dinvpot(\cpt_\idx) - \min_{\idxalt \in \indices} \dinvpot(\cpt_\idxalt)}{\step}
			- \frac{\precs}{\step}
			} = \exp \parens*{- \frac{\precs}{\step}}\,.
	\end{align}
	On the other hand,  by \cref{lem:est_going_back_to_crit} applied with $\open \gets \V$ and $\bigcpt \gets \cl{\V}$ (choosing $\V$ small enough so that is bounded), the quantity
	 \begin{equation}
		 \const \defeq \sup \setdef{\ex_{\state}[\hittime_\V]}{\state \in \V,\, 0 < \step \leq \step_0}
	 \end{equation}
	 is finite.
	Therefore, we have that
	\begin{align}
		\measalt(\vecspace) &= \int_{\vecspace} \dd \invmeas[\V](\point) \ex_\point
		\bracks*{
			\hittime_\V
		}
		\notag\\
		&\leq \const \int_{\vecspace} \dd \invmeas[\V](\point) = \const\,,
	\end{align} 
	which, along with choosing $\step_0$ small enough so that $\const \leq \exp(\precs / \step)$, concludes the proof.
\end{proof}

We will need the following lemma to prove the second part of our main result.
\begin{lemma}
	\label{lem:hitting_time_domain}
	For any $\precs > 0$, for any $\V_1, \dots, \V_\neqcl$ measurable neighborhoods of $\eqcl_1, \dots, \eqcl_\neqcl$ small enough, $\domain$ measurable set, $\margin_\domain > 0$,
	there exists $\step_0 > 0$ such that, 
	for any $0 < \step < \step_0$,
	for any $\idx \in \indices$, $\state \in \V_\idx$,
	\begin{equation}
		\prob_\state \parens*{\hittime_{\domain_{-\margin_\domain}} <\hittime_\V} \leq \exp \parens*{- \frac{\dquasipot(\cpt_\idx, \domain) - \precs}{\step}}\,,
	\end{equation}
	where
	\begin{equation}
		\domain_{-\margin_\domain} \defeq \setdef*{\statealt \in \domain}{d(\statealt, \vecspace \setminus \domain) \geq \margin_\domain}\,.
	\end{equation}
\end{lemma}
\begin{proof}
	The requirement on the $\V_\idx$ are the same in the proof of \cref{lem:est_trans_prob_ub} but we restate them here for completeness.

   Assume that, without loss of generality, $\precs$ is small enough so that \cref{lem:W_nbd} with $\radius \gets \precs$ can be applied to every $\eqcl_\idx$, $\idx =1, \dots, \neqcl$. Denote by $\W_\idx$, $\idx = 1, \dots, \neqcl$ the corresponding neighborhoods of $\eqcl_\idx$.
	Since these $\W_\idx$'s are neighborhoods of the $\eqcl_\idx$'s, there exists $\margin  > 0$ such that $\U_{2 \margin} (\eqcl_\idx) \subset \W_\idx$ for all $\idx = 1, \dots, \neqcl$.
	Require then that $\V_\idx$ be contained in $\U_\margin(\eqcl_\idx)$ so that 
	$\U_\margin(\cl \V_\idx) \subset \W_\idx$ for all $\idx = 1, \dots, \neqcl$.
	Moreover, assume that $\margin > 0$ is small enough so that the neighborhoods $\U_\margin(\eqcl_\idx)$, $\idx = 1, \dots, \neqcl$ are pairwise disjoint and that $\margin \leq  \margin_\domain$.
Define $0 < \marginalt \leq \margin$ such that $\U_\marginalt(\eqcl_\idx)$ is contained in $\V_\idx$ for all $\idx = 1, \dots, \neqcl$.

Fix $\idx \in \indices$ and consider $\dpth \in (\vecspace)^\nRuns$ such that $\dpth_\start \in \U_\marginalt(\cl \V_\idx)$, $\dpth_{\nRuns - 1} \in \U_\marginalt(\domain_{-\margin_\domain})$.
By construction, $\dpth_\start \in \U_\margin(\V_\idx) \subset \W_\idx$. Therefore, there exists $\state \in \eqcl_\idx$ such that $\rate(\state, \dpth_\start) < \precs$.
Moreover, since $\marginalt \leq \margin \leq \margin_\domain$, $\U_\marginalt(\domain_{-\margin_\domain}) \subset \domain$ so that $\dpth_{\nRuns - 1} \in \domain$.

Define $\dpthalt \defeq (\state, \dpth_\start, \dots, \dpth_{\nRuns - 1})$ which is a path from $\eqcl_\idx$ to $\domain$ so that
\begin{equation}
	\daction_\nRuns(\dpth) \geq
	\daction_{\nRuns + 1}(\dpthalt)  - \precs
	\geq
	\dquasipot \parens*{\eqcl_\idx, \domain} - \precs\,.
\end{equation}

We now follow the same outline as for the proof of \cref{lem:est_trans_prob_ub}.
Fix $\state \in \V_\idx$. For any $\nRuns \geq \start$, we have that 
\begin{align}
	\prob_\state \parens*{\hittime_{\domain_{-\margin_\domain}} < \hittime_\V}
	&\leq 
	\prob_\state \parens*{\hittime_{\domain_{-\margin_\domain}} < \hittime_\V\,, \hittime_\V < \nRuns} +
	\prob_\state \parens*{\hittime_\V \geq \nRuns}\,.
\end{align}
For some $\nRuns$ large enough, \cref{lem:est_going_back_to_crit} yields
\begin{equation}
	\prob_\state \parens*{\hittime_\V \geq \nRuns} \leq \exp \parens*{- \frac{\dquasipot\parens*{\eqcl_\idx, \domain} - \precs}{\step}}\,.
\end{equation}
We now bound $\prob_\state \parens*{\hittime_{\domain_{-\margin_\domain}} < \hittime_\V\,, \hittime_\V < \nRuns}$ for this choice of $\nRuns$.
For this, it suffices to note that
$\hittime_{\domain_{-\margin_\domain}} < \hittime_\V\,, \hittime_\V < \nRuns$ implies that $\dist_{\nRuns}\parens*{\accstate, \pths_{\nRuns}^{\setof{\state}}\parens*{\dpth - 2 \precs}} > \half[\marginalt]$. Then, applying \cref{cor:ldp_iterates_full} yields
\begin{equation}
	\prob_\state \parens*{\hittime_{\domain_{-\margin_\domain}} < \hittime_\V\,, \hittime_\V < \nRuns} \leq \exp \parens*{- \frac{\dquasipot\parens*{\eqcl_\idx, \domain} - 3\precs}{\step}}\,,
\end{equation}
which concludes the proof.
\end{proof}

We can use this lemma to upper-bound the $\invmeas(\domain)$.
\begin{lemma}
	\label{lem:ub_invmeas_domain}
	 For any $\precs > 0$, 
     \revise{any measurable set $\domain$ with $\bd \domain$  bounded,} there exists $\step_0 > 0$ such that 
	for any $0 < \step < \step_0$,
	$\invmeas$ invariant probability measure for $(\curr[\accstate])_\run$,
	for any $\idx \in \indices$, 
	\begin{equation}
		\invmeas(\domain_{-\margin_\domain})
		\leq 
		\exp \parens*{
			- \frac{
				\min_{\idx \in \indices}
				\braces*{
					\dinvpot(\cpt_\idx)
					+
					\dquasipot\parens*{\eqcl_\idx, \domain} 
				}
				- \min_{\idx \in \indices} \dinvpot(\cpt_\idx)
			}{\step}
			+ \frac{\precs}{\step}
		}\,,
	\end{equation}
	where
	\begin{equation}
		\domain_{-\margin_\domain} \defeq \setdef*{\statealt \in \vecspace}{d(\statealt, \vecspace \setminus \domain) \geq \margin_\domain}\,.
	\end{equation}
\end{lemma}

\begin{proof}
	Using $\V_1, \dots, \V_\neqcl$ measurable neighborhoods of $\eqcl_1, \dots, \eqcl_\neqcl$ small enough given by  \cref{lem:hitting_time_domain}, we provide an estimate for the weight of $\domain$ for the unnormalized measure $\measalt$ defined in the proof of \cref{prop:est_invmeas}, \ie for
	\begin{equation}
		\measalt(\domain_{-\margin_\domain}) \defeq \int_{\V} \dd \invmeas[\V](\point) \ex_\point \bracks*{
			\sum_{\run = \start}^{\hittime_\V - 1} \oneof{\curr[\accstate] \in \domain_{-\margin_\domain}}
		}\,,
	\end{equation}
	and the result will follow from the estimate on the normalization constant $\measalt(\vecspace)$ obtained in the proof of \cref{prop:est_invmeas}.

    \revise{
        Since $\bd \domain$ is bounded, $\bd \domain_{-\margin_\domain}$ is also bounded. 
        Therefore, by \cref{lem:growth_iterates}, there exists $\bigcpt$ compact set such that, almost surely, if $\accstate_{\run-1} \notin \domain_{-\margin_\domain}$ and $\accstate_{\run} \in \domain_{-\margin_\domain}$, then $\accstate_{\run} \in \bigcpt$.
        Without loss of generality, we can assume that $\bigcpt$ contains $\cl \V$ (choosing $\V$ small enough so that it is bounded if necessary).
    }

    By \cref{lem:est_going_back_to_crit} applied with $\open \gets \V$ and \revise{$\bigcpt \gets \bigcpt$}, the quantity
	 \begin{equation}
         \const \defeq \sup \setdef{\ex_{\state}[\hittime_\V]}{\state \in \revise{\bigcpt},\, 0 < \step \leq \step_0}
	 \end{equation}
	 is finite.

	Fix $\idx \in \indices$ and $\state \in \V_\idx$. If $\hittime_{\domain_{-\margin_\domain}} \geq \hittime_\V$, then $\sum_{\run = \start}^{\hittime_\V - 1} \oneof{\curr[\accstate] \in \domain_{-\margin_\domain}}$ would be 0 so, we have that
	\begin{align}
		\ex_\point \bracks*{
			\sum_{\run = \start}^{\hittime_\V - 1} \oneof{\curr[\accstate] \in \domain_{-\margin_\domain}}
		}
	&=
	\ex_\point \bracks*{
		\oneof{\hittime_{\domain_{-\margin_\domain}} < \hittime_\V}
			\sum_{\run = \start}^{\hittime_\V - 1} \oneof{\curr[\accstate] \in \domain_{-\margin_\domain}}
		}\notag\\
	&\leq 
	\ex_\point \bracks*{
		\oneof{\hittime_{\domain_{-\margin_\domain}} < \hittime_\V}
		\hittime_\V
	}\notag\\
	&= 
	\ex_\point \bracks*{
		\oneof{\hittime_{\domain_{-\margin_\domain}} < \hittime_\V}
		\parens*{
			\hittime_{\domain_{-\margin_\domain}} +
			\ex_{\accstate_{\hittime_{\domain_{-\margin_\domain}}}} \bracks*{
			\hittime_\V
	}
	}
}\,,
	\end{align}
	where we used the strong Markov property. Bounding $\hittime_{\domain_{-\margin_\domain}}$ by $\hittime_\V$ on the event $\{\hittime_{\domain_{-\margin_\domain}} < \hittime_\V\}$, we obtain that
	\begin{align}
		\ex_\point \bracks*{
			\sum_{\run = \start}^{\hittime_\V - 1} \oneof{\curr[\accstate] \in \domain_{-\margin_\domain}}
		}
		&\leq 
		\ex_\point \bracks*{
			\oneof{\hittime_{\domain_{-\margin_\domain}} < \hittime_\V}
			\parens*{
				\hittime_\V +
                \ex_{\accstate_{\hittime_{\revise{{\domain_{-\margin_\domain}}}}}} \bracks*{
				\hittime_\V
		}
		}
	}\notag\\
	&\leq 2 \const \prob_\state \parens*{\hittime_{\domain_{-\margin_\domain}} < \hittime_\V}
	\notag\\
	&\leq 2 \const \exp \parens*{- \frac{\dquasipot\parens*{\eqcl_\idx, \domain} - \precs}{\step}}\,,
	\end{align}
    \revise{where we used that $\accstate_{\hittime_{\domain_{-\margin_\domain}}} \in \bigcpt$ almost surely on the event $\{\hittime_{\domain_{-\margin_\domain}} < \hittime_\V\}$ and the definition of $\const$ for the second inequality and} 
	where we invoked \cref{lem:hitting_time_domain} for the last inequality.

	Combining this bound and \cref{prop:est_induced_invmeas}, we obtain that
	\begin{align}
		\measalt(\domain_{-\margin_\domain}) 
		&\leq 
		2 \const \sum_{\idx \in \indices} \invmeas[\V](\V_\idx) \exp \parens*{- \frac{\dquasipot\parens*{\eqcl_\idx, \domain} - \precs}{\step}}
		\notag\\
		&\leq 
		2 \const \card \indices
 \exp \parens*{
			- \frac{
				\min_{\idx \in \indices}
				\braces*{
					\dinvpot(\cpt_\idx)
					+
					\dquasipot\parens*{\eqcl_\idx, \domain} 
				}
				- \min_{\idx \in \indices} \dinvpot(\cpt_\idx)
			}{\step}
			+ \frac{2\precs}{\step}
		}\,,
	\end{align}
	which concludes the proof.
\end{proof}


\subsection{Convergence and stability}
\label{app:conv_stab}

Let us first begin by showing that, for every initial point, the flow of $\obj$ converges to one of the $\eqcl_\idx$.

\begin{lemma}
    \label{lemma:cv_flow}
    For any $\state \in \vecspace$, there is $\idx \in \indices$ such that
    \begin{equation}
        \lim_{\time \to +\infty} d(\flowmap_\time(\state), \eqcl_\idx) = 0.
    \end{equation}
\end{lemma}

\begin{proof}
    Fix $\state \in \vecspace$. For any $\time \geq \tstart$, $\flowmap_\time(\state)$ belongs to $\setdef{\statealt \in \vecspace}{\obj(\statealt) \leq \obj(\state)}$, which is compact by coercivity of $\obj$.
    Therefore, the set of accumulation points of $\parens*{\flowmap_\time(\state)}_{\time \geq \tstart}$ is non-empty, connected and included in $\crit\obj$. Therefore, since the $\eqcl_\idx$, for $\idx \in \indices$, are the connected components of $\crit\obj$, there is $\idx \in \indices$ such that the accumulation points of $\parens*{\flowmap_\time(\state)}_{\time \geq \tstart}$ all belong to $\eqcl_\idx$.

    If $d(\flowmap_\time(\state), \eqcl_\idx)$ did not converge to $0$, there would be a subsequence that would converge to some point out of any $\eqcl_\idx$, which would be a contradiction. Therefore, $d(\flowmap_\time(\state), \eqcl_\idx)$ must converge to $0$.
\end{proof}

\WAdelete{
The next lemma shows how $\dquasipot$ behaves on trajectories of the flow.
\begin{lemma}
    \label{lemma:traj_flow_dquasipot}
    For any $\state \in \vecspace$, if $(\flowmap_\time(\state))_{\time \geq 0}$ converges to $\eqcl_\idx$, then $\dquasipot(\state, \eqcl_\idx) = 0$.
\end{lemma}
\begin{proof}
    Fix $\precs > 0$ and consider $\W_\precs(\eqcl_\idx)$ which is a neighborhood of $\eqcl_\idx$ by \cref{lem:W_nbd}.
    Since $\flowmap_\time(\state)$ converges to $\eqcl_\idx$, there exists $\horizon > 0$, which we can choose integer, such that, for all $\time \geq \horizon$, $\flowmap_\time(\state) \in \W_\precs(\eqcl_\idx)$.
    Therefore, there exists $\statealt \in \eqcl_\idx$ such that $\rate(\flowmap_\horizon(\state), \statealt) < \precs$.
    Consider the discrete path $\dpth \in \vecspace^{\horizon+2}$ defined
    $\dpth_\run = \flowmap_{\run}(\state)$ for all $ 0 \leq \run \leq \horizon$ and $\dpth_{\horizon+1} = \statealt$.
    By \cref{lem:basic_prop_map}, we have that
    \begin{equation}
        \daction_{\horizon + 2}( \dpth ) = \rate(\flowmap_\horizon(\state), \statealt) < \precs\,,
    \end{equation}
    and therefore $\dquasipot(\state, \eqcl_\idx) < \precs$.
\end{proof}
}

Let us restate the definition of minimizing component that we introduced in the main text.
\begin{definition}
    For any $\idx \in \indices$, we say that $\eqcl_\idx$ is a \emph{minimizing component} if there exists $\open$ a neighborhood of $\eqcl_\idx$ such that,
    \begin{equation}
        \argmin_{\point \in \open} \obj(\point) = \eqcl_\idx\,.
    \end{equation}
    $\eqcl_\idx$ is called non-minimizing otherwise.
\end{definition}

We now state a standard definition for asymptotic stability.

\begin{definition}
    A connected component of the critical points $\eqcl_\idx$, for some $\idx \in \indices$, is said to be \emph{asymptotically stable} if there exists $\U$ a neighborhood of $\eqcl_\idx$ such that, for any $\state \in \U$, $\flowmap_\time(\state)$ converges to $\eqcl_\idx$.
\end{definition}

The notions of minimizing component and asymptotic stability are equivalent in our context.

\begin{lemma}
    \label{lemma:minimizing_component_asympt_stable}
    For any $\idx \in \indices$, $\eqcl_\idx$ is a minimizing component if and only if it is asymptotically stable.
\end{lemma}
\begin{proof}
    We start with the direct implication.
    Assume that $\eqcl_\idx$ is a minimizing component. Therefore, there exists  $\margin > 0$ such that
    \begin{equation}
        \argmin_{\point \in \cl\U_\margin(\eqcl_\idx)} \obj(\point) = \eqcl_\idx\,.
    \end{equation}
    Moreover, assume that $\margin$ is small enough so that, $\cl\U_\margin(\eqcl_\idx) \cap \crit\obj = \eqcl_\idx$.

    \revise{ Then, for any $\state \in \U_\margin(\eqcl_\idx) \setminus \eqcl_\idx$, $\obj(\state) > \obj(\eqcl_\idx)$.}
    By continuity of $\obj$, compactness of $\cl\U_\margin(\eqcl_\idx) \setminus \U_{\margin/2}(\eqcl_\idx)$ and definition of minimizing component, we have that
    \begin{equation}
        \min_{\point \in \cl\U_\margin(\eqcl_\idx) \setminus \U_{\margin/2}(\eqcl_\idx)} \obj(\point) > \obj(\eqcl_\idx)\,.
    \end{equation}
    Define $\marginalt \defeq \half \min_{\point \in \cl\U_\margin(\eqcl_\idx) \setminus \U_{\margin/2}(\eqcl_\idx)} \setof*{\obj(\point) - \obj(\eqcl_\idx)}$ and 
    \begin{equation}
        \V \defeq \setdef*{\state \in \U_{\margin/2}(\eqcl_\idx)}{\obj(\state) < \obj(\eqcl_\idx) + \marginalt} \subset \U_\margin(\eqcl_\idx)\,,
    \end{equation}
    which is a neighborhood of $\eqcl_\idx$ by continuity of $\obj$.

    We now show that trajectories of the flow starting in $\V$ converge to $\eqcl_\idx$.
    Take $\state \in \V$. Since $\obj(\flowmap_\time(\state))$ is non-increasing, then $(\flowmap_\time(\state))_\time$ remains in \revise{$\setdef*{\point \in \vecspace}{\obj(\point) < \obj(\eqcl_\idx) + \marginalt}$.}
    \revise{
        $(\flowmap_\time(\state))_\time$ also cannot escape from $\U_{\margin/2}(\eqcl_\idx)$: if it were, by continuity of the flow and of the distance to $\eqcl_\idx$, there would be $\time > 0$ such that $d(\flowmap_\time(\state), \eqcl_\idx) = \margin/2$, which is a contradiction since $\obj(\flowmap_\time(\state)) < \obj(\eqcl_\idx) + \marginalt$. Hence, $(\flowmap_\time(\state))_\time$ remains in $\U_{\margin/2}(\eqcl_\idx)$ and so in $\V$.
    }
    By \cref{lemma:cv_flow}, $(\flowmap_\time(\state))_\time$ converges to some component of the critical points, which must be $\eqcl_\idx$ since $\V$ is disjoint from the other components.

    We now show the converse implication.
    Since $\eqcl_\idx$ is a connected component of $\crit\obj$, $\obj$ is constant on $\eqcl_\idx$. Denote by $\sol[\obj_\idx]$ this value.
    Take $\margin > 0$ small enough such that $\U_\margin(\eqcl_\idx) \cap \crit\obj = \eqcl_\idx$ and such that, for any $\state \in \U_\margin(\eqcl_\idx)$, $(\flowmap_\time(\state))_\time$ converges to $\eqcl_\idx$.
    Take $\state \in \U_\margin(\eqcl_\idx) \setminus \eqcl_\idx$. We show that $\obj(\state) > \sol[\obj_\idx]$.
    For any $\time > 0$, 
    \begin{equation}
        \obj(\flowmap_\time(\state)) - \revise{\obj(\state)}
        = - \int_0^\time \norm{\grad \obj(\flowmap_\timealt(\state))}^2 \dd \timealt\,,
    \end{equation}
    which must (strictly) negative since $\state$ is not a critical point of $\obj$. Since $\obj(\flowmap_\time(\state))$ is non-increasing in $\time$ and lower-bounded --- because $\inf_{\vecspace} \obj > - \infty$ by coercivity of $\obj$ ---, it must converge as $\time \to +\infty$ and its limit satisfies $\lim_{\time \to +\infty} \obj(\flowmap_\time(\state)) < \obj(\state)$.
    Moreover, $(\flowmap_\time(\state))_\time$ converges to $\eqcl_\idx$ so that all its accumulation points belong to $\eqcl_\idx$ and have the same objective value $\sol[\obj_\idx]$. Hence, we have that $\lim_{\time \to +\infty} \obj(\flowmap_\time(\state)) = \sol[\obj_\idx]$ and $\sol[\obj_\idx] < \obj(\state)$.
\end{proof}

In the following, we will thus use the terms minimizing component and asymptotically stable interchangeably.

The next lemma shows that if $\eqcl_\idx$ is not asymptotically stable, then it is unstable in the sense of \citet[Chap.~6,\S4]{FW98}.

\begin{lemma}
    \label{lemma:not_asympt_stable_implies_unstable}
    If $\eqcl_\idx$ is not asymptotically stable, then there exists $\idxalt \in \indices$ such that $\dquasipot_{\idx, \idxalt} = 0$.    
\end{lemma}
\begin{proof}
    If $\eqcl_\idx$ is not asymptotically stable, then, for every $\run \geq 1$, there exists $\state_\run \in \U_{1/\run}( \eqcl_\idx )$ such that $\flowmap_\time(\state_\run)$ does not converge to $\eqcl_\idx$.
    By \cref{lemma:cv_flow}, there exists $\idxalt_\run \in \indices$ such that $(\flowmap_\time(\state_\run))_\time$ converges to $\eqcl_{\idxalt_\run}$.
    Since $\indices$ is finite, there exists $\idxalt \in \indices$ such that $\idxalt_\run = \idxalt$ for infinitely many $\run$.
    Replacing $(\state_\run)_{\run \geq 1}$ by a subsequence, we can assume that $(\flowmap_\time(\state_\run))_\time$ converges to $\eqcl_\idxalt$ for all $\run \geq 1$.

    We now show that $\dquasipot_{\idx, \idxalt} = 0$.

    Fix $\precs > 0$ and consider $\W_\precs(\eqcl_\idx), \W_\precs(\eqcl_\idxalt)$ which are neighborhoods of $\eqcl_\idx, \eqcl_\idxalt$ by \cref{lem:W_nbd}.
    Since $(\state_\run)_{\run \geq 1}$ converges to $\eqcl_\idx$, there exists $\run$ such that $\state_\run \in \W_\precs(\eqcl_\idx)$.
    In turn, since $\flowmap_\time(\state_\run)$ converges to $\eqcl_\idxalt$, there exists $\horizon > 0$, which we can choose integer, such that, $\flowmap_\horizon(\state_\run) \in \W_\precs(\eqcl_\idxalt)$.
    Therefore, there exists $\stateA \in \eqcl_\idx$ and $\stateB \in \eqcl_\idxalt$ such that $\rate(\stateA, \state_\run) < \precs$ and $\rate(\flowmap_\horizon(\state_\run), \stateB) < \precs$.
    Consider the discrete path $\dpth \in \vecspace^{\horizon+3}$ defined by 
    $\dpth_0 = \stateA$,
    $\dpth_\run = \flowmap_{\run-1}(\state)$ for all $ 1 \leq \run \leq \horizon + 1$ and $\dpth_{\horizon+2} = \statealt$.
    By \cref{lem:basic_prop_map}, we have that
    \begin{equation}
        \daction_{\horizon + 3}( \dpth ) = 
        \rate(\stateA, \state_\run) +
        \rate(\flowmap_\horizon(\state), \stateB) < 2\precs\,,
    \end{equation}
    and therefore, since $\dpth_0 \in \eqcl_\idx$ and $\dpth_{\horizon+2} \in \eqcl_\idxalt$, $\dquasipot_{\idx, \idxalt} < 2\precs$.
\end{proof}

The following lemma is now a straightforward adaptation of \citet[Chap.~6, Lemma~4.2]{FW98}.
\begin{lemma}
    \label{lemma:not_asympt_stable_implies_unstable2}
    If $\eqcl_\idx$ is not asymptotically stable, then there exists $\idxalt \in \indices$ such that $\dquasipot_{\idx, \idxalt} = 0$ and such that, for any $\idxaltalt \in \indices$, $\dquasipot_{\idxalt, \idxaltalt} > 0$.
\end{lemma}
\begin{proof}
    For the sake of contradiction, assume that such a $\idxalt$ does not exist.
    Then we build an infinite sequence $\idxalt_0 = \idx, \idxalt_1, \dots$ such that $\dquasipot_{\idxalt_\run, \idxalt_{\run+1}} = 0$ for all $\run \geq 0$.
    By definition of equivalence classes, any $\idxalt$ cannot appear twice in this sequence. But since $\indices$ is finite, this is a contradiction.
\end{proof}

We now reach our final result on unstable equivalence classes: they have negligible weight in the invariant measure.
\begin{lemma}
    \label{lemma:not_asympt_stable_implies_unstable3}
    If $\eqcl_\idx$ is not asymptotically stable, or, equivalently, non-minimizing, then there exists $\idxalt \in \indices$ such that $\eqcl_\idxalt$  is asymptotically stable,  $\dquasipot_{\idx, \idxalt} = 0$ and,
    \begin{equation}
        \dinvpot_\idxalt < \dinvpot_\idx
    \end{equation}
\end{lemma}
\begin{proof}
    By \cref{lemma:not_asympt_stable_implies_unstable2}, there exists $\idxalt \in \indices$ such that $\dquasipot_{\idx, \idxalt} = 0$ and such that, for any $\idxaltalt \in \indices$, $\dquasipot_{\idxalt, \idxaltalt} > 0$.
    \Cref{lemma:not_asympt_stable_implies_unstable} implies in particular that $\eqcl_\idxalt$ must be asymptotically stable

    It remains to show that $\dinvpot_\idxalt < \dinvpot_\idx$.
    Take $\tree \in \trees_\idx$ such that
    \begin{equation}
        \dinvpot_\idx = \sum_{(\idxaltalt \to \idxaltaltalt) \in \tree} \dquasipot_{\idxaltalt,\idxaltaltalt}\,.
    \end{equation}
    $\idxalt$ has an outgoing edge in $\tree$ and denote by $\alt \idxalt$ the other end of that edge.
    Now, consider the tree $\treealt \in \trees_\idxalt$ obtained from $\tree$ by removing the outgoing edge from $\idxalt$ to $\alt \idx$ and adding an edge from $\idx$ to $\idxalt$.
    Then, by definition of $\dinvpot_\idxalt$,
    \begin{align}
        \dinvpot_\idxalt &\leq  \sum_{(\idxaltalt \to \idxaltaltalt) \in \treealt} \dquasipot_{\idxaltalt,\idxaltaltalt} = \sum_{(\idxaltalt \to \idxaltaltalt) \in \tree} \dquasipot_{\idxaltalt,\idxaltaltalt} - \dquasipot_{\idxalt,\alt \idxalt} + \dquasipot_{\idx,\idxalt} \,.
    \end{align}
    But, by definition of $\idxalt$, $\dquasipot_{\idx,\idxalt} = 0$ and $\dquasipot_{\idxalt,\alt \idxalt} > 0$. Therefore, $\dinvpot_\idxalt < \dinvpot_\idx$.
\end{proof}

We now show the second part of our main result: the invariant measure concentrates on the ground states, which are asymptotically stable by \cref{lemma:not_asympt_stable_implies_unstable3}.

For this, we need the following lemma.
\begin{lemma}
    \label{lemma:ground_states_loc_min}
    For any $\idx \in \indices$ such that $\eqcl_\idx$ is minimizing, there exists $\margin > 0$ such that, for any $0 < \marginalt \leq \margin$,
    \begin{equation}
        \dquasipot(\eqcl_\idx, \vecspace \setminus \U_\marginalt(\eqcl_\idx)) > 0\,.
    \end{equation}
\end{lemma}
\begin{proof}
    Since $\eqcl_\idx$ is minimizing, there exists $\margin > 0$ such that, for any $\state \in \U_\margin(\eqcl_\idx)$, $\obj(\state) > \obj_\idx$ where $\obj_\idx$ is the value of $\obj$ on $\eqcl_\idx$.
    Take $\marginalt \leq \margin$, $\open \defeq  \U_\marginalt(\eqcl_\idx)$ and
    \begin{equation}
    \Margin \defeq \min \setdef*{\bdpot(\state) - \bdprimvar(\obj_\idx)}{\state \in \vecspace,\, d(\state, \eqcl_\idx) = \marginalt/2}
    \,.
    \end{equation}
    Then, by the continuity of $\bdpot$ and the fact that $\bdprimvar$ is (strictly) increasing, we have that $\Margin > 0$.
    To conclude the proof of this lemma, we now show that $\dquasipot(\eqcl_\idx, \vecspace \setminus \open) \geq \half[\Margin]$.
    Consider some $\horizon > 0$ and $\pth \in \contfuncs([0, \horizon], \vecspace)$ such that $\pth_0 \in \eqcl_\idx$ and $\pth_\horizon \in \vecspace \setminus \W$.
    By continuity of $\pth$ and $d(\cdot, \eqcl_\idx)$, there exists $\time \in [0, \horizon]$ such that $d(\pth_\time, \eqcl_\idx) = \marginalt/2$.
    By the same computation as in \cref{lem:B_is_descent}, we have that
    \begin{align}
        \Margin 
        &\leq \bdpot(\pth_\time) - \bdpot(\pth_0)
        \notag\\
        &= 2 \int_0^\time \frac{
            \inner{\dot \pth_\timealt}{\grad \obj(\pth_\timealt)}
        }{
            \bdvar(\obj(\pth_\timealt))
        }\notag\\
        &\leq 2 \action_{[0, \time]}(\pth)
        \notag\\
        &\leq 2 \action_{[0, \horizon]}(\pth)\,.
    \end{align}
    Since this is valid for any $\pth$, we obtain that $\dquasipot(\eqcl_\idx, \vecspace \setminus \open) \geq \half[\Margin]$.
\end{proof}

The next proposition shows that the invariant measure concentrates exponentially on states that are asymptotically stable (and contain the ground states).

\begin{proposition}
    \label{prop:inv_meas_ground_states}
    Consider $\indicesalt \subset \indices$ such that, for all $\idx \in \indicesalt$, $\eqcl_\idx$ is minimizing and, such that, $\indicesalt$ contains $\argmin_{\idx \in \indices} \dinvpot_\idx$.
    Consider $\V_\idx$ small enough neighborhoods of $\eqcl_\idx$ for $\idx \in \indicesalt$.
    Then, there exists $\const > 0$, $\step_0 > 0$ such that, for any $\step \leq \step_0$, for any $\invmeas$ invariant measure of $(\curr[\accstate])_{\run \geq \start}$,
    \begin{equation}
        \invmeas \parens*{\vecspace \setminus \bigcup_{\idx \in \indicesalt} \V_\idx} \leq e^{-\frac{\const}{\step}}\,.
    \end{equation}
\end{proposition}
\begin{proof}
    Take $\margin > 0$ small enough so that, for any $\idx \in \indicesalt$,
    $\U_\margin(\eqcl_\idx) \subset \V_\idx$ and 
    \begin{equation}
        d(\eqcl_\idx, \vecspace \setminus \U_\margin(\eqcl_\idx)) > 0\,.
    \end{equation}
    This is possible 
    by \cref{lemma:ground_states_loc_min}.

    Define 
    \begin{equation}
        \domain \defeq \vecspace \setminus \bigcup_{\idx \in \indicesalt} \U_{\margin/2}(\eqcl_\idx)\,.
    \end{equation}
    With the notations of \cref{lem:ub_invmeas_domain}, we show that
    \begin{equation}
        \vecspace \setminus \bigcup_{\idx \in \indicesalt} \V_\idx \subset \domain_{-\margin/2}\,.
    \end{equation}
    Indeed, take $\state \in \vecspace \setminus \bigcup_{\idx \in \indicesalt} \V_\idx$.
    Then, for any $\idx \in \indicesalt$, it holds that $d(\state, \eqcl_\idx) \geq \margin$ and so, we have
    \begin{equation}
        d(\state, \vecspace \setminus \domain)
        = d \parens*{\state, \vecspace \setminus \bigcup_{\idx \in \indicesalt} \U_{\margin/2}(\eqcl_\idx)}
        \geq \frac{\margin}{2}\,.
    \end{equation}
    Hence $\state \in \domain_{-\margin/2}$ and, it suffices to bound $\invmeas(\domain_{-\margin/2})$ to show the result.

    Moreover, for any $\idx \in \indicesalt$,
    \begin{equation}
        \dquasipot(\eqcl_\idx, \domain) \geq \dquasipot(\eqcl_\idx, \vecspace \setminus \U_{\margin/2}(\eqcl_\idx)) > 0\,,
    \end{equation}
    so that the quantity
    \begin{equation}
        \const \defeq \min \parens*{
            \min_{\idx \notin \indicesalt} \dinvpot_\idx - \min_{\idx \in \indices} \dinvpot_\idx,
            \min_{\idx \in \indicesalt} \dquasipot(\eqcl_\idx, \domain)
        }
    \end{equation}
     is positive.

     Apply \cref{lem:ub_invmeas_domain} with $\margin_\domain \gets \margin/2$, $\precs \gets \const / 2$ to get that, for any $\step \leq \step_0$,
     \begin{equation}
         \invmeas(\domain_{-\margin/2}) \leq 
             \exp \parens*{
            - \frac{
                \min_{\idx \in \indices}
                \braces*{
                    \dinvpot(\eqcl_\idx)
                    +
                    \dquasipot\parens*{\eqcl_\idx, \domain} 
                }
                - \min_{\idxalt \in \indices} \dinvpot(\eqcl_\idxalt)
            }{\step}
            + \frac{\const}{2\step}
        }\,.
     \end{equation}
     But, for any $\idx \in \indices$, the exponent can be estimated as 
     \begin{equation}
            \braces*{
                    \dinvpot(\eqcl_\idx)
                    +
                    \dquasipot\parens*{\eqcl_\idx, \domain} 
                }
                - \min_{\idxalt \in \indices} \dinvpot(\eqcl_\idx)
            \geq
            \begin{cases}
             \dinvpot(\eqcl_\idx) - \min_{\idxalt \in \indices} \dinvpot(\eqcl_\idxalt) & \text{if } \idx \notin \indicesalt\\
            \dquasipot(\eqcl_\idx, \domain) & \text{if } \idx \in \indicesalt\,,
            \end{cases}
     \end{equation}
     which is always positive, even in the first case, since $\argmin_{\idxalt \in \indices} \dinvpot(\eqcl_\idxalt) \subset \indicesalt$.
     Therefore, it holds that
       \begin{equation}
                \min_{\idx \in \indices}
                \braces*{
                    \dinvpot(\eqcl_\idx)
                    +
                    \dquasipot\parens*{\eqcl_\idx, \domain} 
                }
                - \min_{\idxalt \in \indices} \dinvpot(\eqcl_\idxalt)
            - \precs
            \geq \const - \precs  = \frac{\const}{2}\,,
       \end{equation}
       which concludes the proof.
\end{proof}

\subsection{Main results}
\label{app:main}
In this section, we restate the main results \cref{thm:Gibbs,thm:unstable,thm:crit,thm:ground}
and provide their proofs. They are now mostly corollaries of the results of \cref{app:estimates_trans_inv,app:conv_stab}.
\begin{theorem}
Suppose that $\invmeas$ is invariant under \eqref{eq:SGD},
fix a tolerance level $\toler > 0$,
and
let $\nhd_{\iComp} \equiv \nhd_{\iComp}(\size)$, $\iComp = 1,\dotsc,\nComps$, be $\size$-neighborhoods of the components of $\crit\obj$.
Then, for all sufficiently small $\size,\step > 0$, we have
\begin{equation}
\abs*{\step\log\invmeas(\nhd_{\iComp})
	+ \energy_{\iComp} - \min\nolimits_{\jComp} \energy_{\jComp}}
	\leq \toler
\end{equation}
and
\begin{equation}
\abs*{\step\log\frac{\invmeas(\nhd_{\iComp})}{\invmeas(\nhd_{\jComp})}
	+ \energy_{\iComp} - \energy_{\jComp}}
	\leq \toler.
\end{equation}
More compactly, with notation as above, we have:
\begin{equation}
\invmeas(\nhd_{\iComp})
	\propto \exp\of*{-\frac{\energy_{\iComp} + \bigoh(\toler)}{\step}}.
\end{equation}
\end{theorem}
\begin{proof}
	Note that if $\invmeas$ is invariant for \eqref{eq:SGD}, it is \emph{a fortiori} invariant for the accelerated process $(\curr[\accstate])_{\run \geq \start}$.
	This result is then a direct consequence of \cref{prop:est_invmeas} (\cref{app:estimates_trans_inv}).
\end{proof}

\begin{theorem}
Suppose that $\invmeas$ is invariant under \eqref{eq:SGD},
and let $\comp$ be a non-minimizing component of $\obj$.
Then, with notation as in \cref{thm:Gibbs}, there exists a minimizing component $\alt\comp$ of $\obj$ and a positive constant $\const \equiv \const(\comp,\alt\comp) > 0$ such that
\begin{equation}
\frac{\invmeas(\nhd)}{\invmeas(\alt\nhd)}
	\leq \exp\of*{-\frac{\const(\comp,\alt\comp) + \eps}{\step}}
\end{equation}
for all all sufficiently small $\step>0$ and all sufficiently small neighborhoods $\nhd$ and $\alt\nhd$ of $\comp$ and $\alt\comp$ respectively.
In particular, in the limit $\step\to0$, we have $\invmeas(\nhd)\to0$.
\end{theorem}

\begin{proof}
	Let $\eqcl = \eqcl_\idx$ be a non-minimizing component. By \cref{lemma:not_asympt_stable_implies_unstable3}, there exists $\idxalt \in \indices$ such that $\energy_\idxalt < \energy_\idx$. The statement then follows from \cref{thm:Gibbs}.
\end{proof}
\begin{theorem}
Suppose that $\invmeas$ is invariant under \eqref{eq:SGD},
fix a tolerance level $\size>0$,
and
let $\nhd \equiv \nhd(\size)$ be a $\size$-neighborhood of $\crit\obj$.
Then there exists a constant $\const \equiv \const_{\size} > 0$ such that, for all sufficiently small $\step > 0$, we have:
\begin{equation}
\invmeas(\nhd)
	\geq 1 - e^{-\const/\step}.
\end{equation}
\end{theorem}
\begin{proof}
	We actually show a slighly stronger result. Define $\indicesalt \defeq \setdef*{\idx \in \indices}{\eqcl_\idx \text{ is minimizing}}$. 
	We prove that there exists $\open$ neighborhood of $\bigcup_{\idx \in \indicesalt} \eqcl_\idx$, a constant $\const > 0$, such that, for all $\step$ small enough,
\begin{equation}
\invmeas \parens*{\vecspace \setminus \bigcup_{\idx \in \indicesalt} \V_\idx} \leq e^{-\frac{\const}{\step}}
	\eqdot
\end{equation}
This is then a consequence of \cref{prop:inv_meas_ground_states} (\cref{app:conv_stab}) with $\indicesalt \gets \indicesalt$. Note $\indicesalt$ contain the ground states since they are minimizing by \cref{lemma:not_asympt_stable_implies_unstable3}.
\end{proof}
\begin{theorem}
Suppose that $\invmeas$ is invariant under \eqref{eq:SGD},
fix a tolerance level $\size>0$,
and
let $\nhd_{\iGround} \equiv \nhd_{\iGround}(\size)$ be a $\size$-neighborhood of the system's ground state $\ground$.
Then there exists a constant $\const \equiv \const_{\size} > 0$ such that, for all sufficiently small $\step>0$, we have:
\begin{equation}
\invmeas(\nhd_{\iGround})
	\geq 1 - e^{-\const/\step}.
\end{equation}
\end{theorem}

\begin{proof}
	It suffices to apply \cref{prop:inv_meas_ground_states} (\cref{app:conv_stab}) with $\indicesalt$ as the set of ground states $\argmin_{\idx \in \indices} \energy_\idx$.
	They are necessarily minimizing by \cref{lemma:not_asympt_stable_implies_unstable3}.
\end{proof}

\subsection{Extension: the mean occupation measures}
\label{app:main_occmeas}

We now state and prove analogue of  \cref{thm:Gibbs,thm:unstable,thm:crit,thm:ground} for the mean occupation measure $\curr[\occmeas]$.

For this we will need a strengthened version of \cref{asm:SNR-weak}, \viz

\smallbreak
\asmtag{\ref*{asm:SNR}$^{\ast\ast}$}
\begin{assumption}
[Variant of \cref{asm:SNR-weak}]
\label{asm:SNR-weak-occmeas}
\label{assumption:coercivity_noise}
The \acl{SNR} of $\orcl$ satisfies
\begin{equation}
\label{eq:SNR-weak-occmeas}
\tag{\ref*{eq:SNR}$^{\ast\ast}$}
\frac{\norm{\grad\obj(\point)}^{2}}{\bdvar(\obj(\point))}
	\to\infty
	\quad
	\text{as}
	\quad
\norm{\point}
	\to \infty
	\eqdot
\end{equation}
\end{assumption}

In this section, we will posit that \cref{asm:SNR-weak-occmeas} holds in addition to
\crefnosort{asm:obj-weak,,asm:noise-weak,,asm:costs}.
We then have the following series of results for the occupation measure $\curr[\occmeas]$ of \eqref{eq:SGD}.

\smallbreak
\thmtag{\ref*{thm:Gibbs}$^{\ast}$}
\begin{theorem}
[Occupation variant of \cref{thm:Gibbs}]
\label{thm:Gibbs-occmeas}
Fix a tolerance level $\toler > 0$,
and
let $\nhd_{\iComp} \equiv \nhd_{\iComp}(\size)$, $\iComp = 1,\dotsc,\nComps$, be $\size$-neighborhoods of the components of $\crit\obj$.
Then, for all sufficiently small $\size,\step > 0$ and large enough $\run$, we have
\begin{equation}
\abs*{\step\log\curr[\occmeas](\nhd_{\iComp})
	+ \energy_{\iComp} - \min\nolimits_{\jComp} \energy_{\jComp}}
	\leq \toler
\end{equation}
and
\begin{equation}
\abs*{\step\log\frac{\curr[\occmeas](\nhd_{\iComp})}{\curr[\occmeas](\nhd_{\jComp})}
	+ \energy_{\iComp} - \energy_{\jComp}}
	\leq \toler.
\end{equation}
More compactly, with notation as above, we have:
\begin{equation}
\curr[\occmeas](\nhd_{\iComp})
	\propto \exp\of*{-\frac{\energy_{\iComp} + \bigoh(\toler)}{\step}}.
\end{equation}
\end{theorem}

\smallbreak
\thmtag{\ref*{thm:unstable}$^{\ast}$}
\begin{theorem}
[Occupation variant of \cref{thm:unstable}]
\label{thm:unstable-occmeas}
Let $\comp$ be a non-minimizing component of $\obj$.
Then, with notation as in \cref{thm:Gibbs}, there exists a minimizing component $\alt\comp$ of $\obj$ and a positive constant $\const \equiv \const(\comp,\alt\comp) > 0$ such that
\begin{equation}
	\frac{\curr[\occmeas](\nhd)}{\curr[\occmeas](\alt\nhd)}
	\leq \exp\of*{-\frac{\const(\comp,\alt\comp) + \eps}{\step}}
\end{equation}
for all all sufficiently small $\step>0$, $\run$ large enough and all sufficiently small neighborhoods $\nhd$ and $\alt\nhd$ of $\comp$ and $\alt\comp$ respectively.
\end{theorem}

\smallbreak
\thmtag{\ref*{thm:crit}$^{\ast}$}
\begin{theorem}
[Occupation variant of \cref{thm:crit}]
\label{thm:crit-occmeas}
Fix a tolerance level $\size>0$,
and
let $\nhd \equiv \nhd(\size)$ be a $\size$-neighborhood of $\crit\obj$.
Then there exists a constant $\const \equiv \const_{\size} > 0$ such that, for all sufficiently small $\step > 0$ and large enough $\run$, we have:
\begin{equation}
\curr[\occmeas](\nhd)
	\geq 1 - e^{-\const/\step}.
\end{equation}
\end{theorem}

\smallbreak
\thmtag{\ref*{thm:ground}$^{\ast}$}
\begin{theorem}
[Occupation variant of \cref{thm:ground}]
\label{thm:ground-occmeas}
Fix a tolerance level $\size>0$,
and
let $\nhd_{\iGround} \equiv \nhd_{\iGround}(\size)$ be a $\size$-neighborhood of the system's ground state $\ground$.
Then there exists a constant $\const \equiv \const_{\size} > 0$ such that, for all sufficiently small $\step>0$ and large enough $\run$, we have:
\begin{equation}
\curr[\occmeas](\nhd_{\iGround})
	\geq 1 - e^{-\const/\step}.
\end{equation}
\end{theorem}

We begin with a preliminary lemma which shows that, under \cref{asm:SNR-weak-occmeas}, the sequence of mean occupation measure $(\curr[\occmeas])_{\run \geq \start}$ is tight (see \eg \citet[Chap.~23]{kallenbergFoundationsModernProbability2021}).

\begin{lemma}
	\label{lemma:tight_occmeas}
	The sequence of mean occupation measures $(\curr[\occmeas])_{\run \geq \start}$ is tight.
\end{lemma}

The proof of this lemma first follows the proof of \cref{lem:inv_meas_exist} and then relies on the same reasoning as the proof of \citet[Thm.~12.3.3]{doucMarkovChains2018}.

\begin{proof}
	By \cref{lem:lyapunov_condition}, there exists $\cpt \subset \vecspace$ compact, $\step_0 > 0$, $\const > 0$ such that, for any $\step \leq \step_0$, $\init = \point \notin \cpt$,
\begin{equation}
	\bdpot(\afterinit) - \bdpot(\point) \leq  \step \parens*{\frac{\norm{\noise(\point, \init[\sample])}^2}{\bdvar(\obj(\point))} - \frac{\norm{\grad \obj(\point)}^2}{\bdvar(\obj(\point))}}\,.
\end{equation}

Passing to the expectation yields that, for any $\point \notin \cpt$,
\begin{equation}
	\ex_{\point} \bracks*{
	\bdpot(\afterinit)} - {\bdpot(\point)
	} \leq  \step \parens*{\frac{\ex_{\point} \bracks*{\norm{\noise(\point, \init[\sample])}^2}}{\bdvar(\obj(\point))} - \frac{\norm{\grad \obj(\point)}^2}{\bdvar(\obj(\point))}}\,.
\end{equation}
Appplying \cref{corollary:concentration} with $\rv \gets \frac{\noise(\point, \init[\sample])}{\sqrt{\bdvar(\obj(\point))}}$ (the conditions of application are verified from \cref{asm:noise-weak}\cref{assumption:subgaussian}) yields that
\begin{equation}
	\ex_{\point} \bracks*{
	\bdpot(\afterinit)} - {\bdpot(\point)
	} \leq - \step \frac{\norm{\grad \obj(\point)}^2}{\bdvar(\obj(\point))} + \step \times \gradientbound\,.
\end{equation}

Hence, for any $\point \in \vecspace$, we have
\begin{align}
\ex_{\point} \bracks*{
	\bdpot(\afterinit)} 
	 - \bdpot(\point)
	 &\leq 
	 \oneof{\point \in \cpt} \parens*{\sup_{\pointalt \in \cpt}\ex_{\pointalt} [ \bdpot(\pointalt)] - \inf_{\vecspace} \bdpot}
	 \notag\\
	 &- \oneof{\point \notin \cpt}
	\parens*{
		\step \frac{\norm{\grad \obj(\point)}^2}{\bdvar(\obj(\point))} - \step \times \gradientbound
	}\,,
\end{align}
or, after rearranging,
\begin{align}
\ex_{\point} \bracks*{\bdpot(\afterinit)} 
	 &- \oneof{\point \in \cpt} \parens*{\sup_{\pointalt \in \cpt}\ex_{\pointalt} [ \bdpot(\pointalt)] - \inf_{\vecspace} \bdpot}
	 \notag\\
	 &+ \oneof{\point \notin \cpt}
	\parens*{
		\step \frac{\norm{\grad \obj(\point)}^2}{\bdvar(\obj(\point))} - \step \times \gradientbound
	}
	 \leq  \bdpot(\point)
	\eqdot
\end{align}
Since the function 
\begin{align}
\point
	\mapsto
	&- \oneof{\point \in \cpt} \parens*{\sup_{\pointalt \in \cpt}\ex_{\pointalt} [ \bdpot(\pointalt)] - \inf_{\vecspace} \bdpot}
	\notag\\
	&+ \oneof{\point \notin \cpt} \parens*{\step \frac{\norm{\grad \obj(\point)}^2}{\bdvar(\obj(\point))} - \step \times \gradientbound}
\end{align}
is measurable, lower-bounded and goes to infinity as $\norm{\point} \to \infty$ by \cref{asm:SNR-weak-occmeas}, one can then apply the same computations as in the proof of \citet[Thm.~12.3.3]{doucMarkovChains2018} to obtain that the sequence of occupation measures $(\curr[\occmeas])_{\run \geq \start}$ is tight.
\end{proof}

We now prove \cref{thm:Gibbs-occmeas} by adapting the proof of \cref{thm:Gibbs}.
Since the process is exactly the same for \cref{thm:unstable-occmeas,thm:crit-occmeas,thm:ground-occmeas}, we omit their proofs.
\WAcomment{While we rely on it heavily, it is not made clear that $\nhd_i$ are open neighborhoods.}
\begin{proof}[Proof of \cref{thm:Gibbs-occmeas}]
We show that, for sufficiently small $\size, \step > 0$, for any $\iComp \in \indices$, we have that
\begin{align}
\exp\of*{-\frac{\energy_{\iComp} - \min\nolimits_{\jComp} \energy_{\jComp} + \toler}{\step}}
	&\leq \liminf_{\run \to \infty}\curr[\occmeas](\nhd_{\iComp})
	\notag\\
	&\leq \limsup_{\run \to \infty} \curr[\occmeas](\nhd_{\iComp})
	\notag\\
	&\leq \exp\of*{-\frac{\energy_{\iComp} - \min\nolimits_{\jComp} \energy_{\jComp} - \toler}{\step}}
\label{eq:proof_occmeas_Gibbs_target}
\end{align}
	and the results in the statement will follow with $2 \toler$ in place of $\toler$.

	We apply \cref{prop:est_invmeas} (\cref{app:estimates_trans_inv}): take $\size$ small enough so 
	\cref{prop:est_invmeas} can be applied with both the neighborhoods $\nhd_{1}, \dots, \nhd_{\nComps}$ and $\cl \nhd_1, \dots, \cl \nhd_{\nComps}$.
	One then obtain $\step_0 > 0$ such that, for all $0 < \step < \step_0$ and any $\invmeas$ invariant probability measure for $(\curr[\accstate])_\run$,
	for any $\idx \in \indices$, 
	\begin{equation}
	\label{eq:proof_occmeas_Gibbs_est_invmeas}
	\begin{aligned}
		\invmeas(\nhd_\idx) &\geq \exp \parens*{
			- \frac{\dinvpot(\cpt_\idx) - \min_{\idxalt \in \indices} \dinvpot(\cpt_\idxalt)}{\step}
			- \frac{\precs}{\step}
		}\\
		\invmeas(\cl \nhd_\idx) &\leq \exp \parens*{
			- \frac{\dinvpot(\cpt_\idx) - \min_{\idxalt \in \indices} \dinvpot(\cpt_\idxalt)}{\step}
			+ \frac{\precs}{\step}
		}\eqdot
	\end{aligned}
	\end{equation}
	
	We now prove that \cref{eq:proof_occmeas_Gibbs_target} holds.
	Fix $\idx \in \indices$. By \cref{lemma:tight_occmeas}, the sequence of mean occupation measures $(\curr[\occmeas])_{\run \geq \start}$ is tight, so that, by Prohorov theorem \citep[Thm.~23.2]{kallenbergFoundationsModernProbability2021}, it is sequentially compact for the weak topology, or, in other terms, for the convergence in distribution.
	Therefore, $(\curr[\occmeas])_{\run \geq \start}$ admits a weak accumulation point which is a probability distribution and that we denote by $\altmeas$.
	  Applying Portmanteau theorem \citep[Thm.~5.25]{kallenbergFoundationsModernProbability2021} to the open set $\nhd_\idx$ and the closed set $\cl \nhd_\idx$ yields that
	  \begin{equation}
		  \altmeas(\nhd_\idx)   
		  \leq 
		  \liminf_{\run \to \infty}
		  \curr[\occmeas](\nhd_\idx)
		  \leq 
		  \limsup_{\run \to \infty}
		  \curr[\occmeas](\cl \nhd_\idx)
		  \leq 
		  \altmeas(\cl \nhd_\idx)\eqdot
		\label{eq:proof_occmeas_Gibbs_portmanteau}
	  \end{equation}

	  Since $(\curr)_{\run \geq \start}$, the sequence of iterates of \ac{SGD}, is (weak) Feller by \cref{lem:feller}, $\altmeas$ is actually invariant for $(\curr)_{\run \geq \start}$, by, \eg \citet[Prop.~12.3.1]{doucMarkovChains2018}, and
\emph{a fortiori} invariant for the accelerated process $(\curr[\accstate])_{\run \geq \start}$.
	  Combining \cref{eq:proof_occmeas_Gibbs_est_invmeas} with \cref{eq:proof_occmeas_Gibbs_portmanteau} gives the result \cref{eq:proof_occmeas_Gibbs_target}.
\end{proof}

\section{Potential for the invariant measure}
\label{app:potential}

\subsection{Gaussian noise}
\label{app:subsec:gaussian_noise}
Though it does not formally fit into our setting, let us first begin with the case where the noise is Gaussian. Since it is unbounded, our assumptions are not satisfied and our theorems describing the invariant measure do not apply. However, all the objects we consider are still well-defined, and, in that case, it is possible to compute the $\dinvpot_\idx$ explicitly. Moreover, this section serves as a blueprint for the truncated Gaussian case of the next section.

Assume that, for every $\state \in \vecspace$, $\noise(\state, \sample)$ follows a centered Gaussian distribution with covariance $\variance(\obj(\state)) \identity$ for some continuous function $\variance : \R \to (0, + \infty)$.

Akin to $\bdpot$ in \cref{app:invmeas}, a key role is played by the function $\pot: \vecspace \to \R$ defined by
\begin{equation}
\pot(\state) \defeq 2 \primvar(\obj(\state)) \quad \text{ with } \quad \primvar \from \R \to \R \quad \text{ a primitive of } \quad 1/{\variance}\,.
\end{equation}

Since the noise is Gaussian, the Lagrangian and Hamiltonian have explicit expressions: for every $\state, \mom, \vel \in \vecspace$,
\begin{subequations}
\begin{align}
	\hamilt(\state, \mom) &= - \inner{\grad \obj(\state)}{\mom} + \half \variance(\obj(\state)) \norm{\mom}^2\\
	\lagrangian(\state, \vel) &= \frac{\norm{\vel + \grad \obj(\state)}}{2 \variance(\obj(\state))}\,.
\end{align}
\end{subequations}
This expression of $\lagrangian$ make it clear that the action function penalizes the deviation of a path from the flow: for a path $\pth \in \contfuncs([0, \horizon])$ for some $\horizon > 0$,
\begin{equation}
	\action_\horizon(\pth) = \int_{0}^{\horizon} \frac{\norm{\dot \pth_\time + \grad \obj(\pth_\time)}^2}{2 \variance(\obj(\pth_\time))} \dd \time\,.
\end{equation}

The computation of the $\dinvpot_\idx$ relies on the following observation.
Take a path $\pth \in \contfuncs([0, \horizon])$ for some $\horizon > 0$ and consider $\pthalt$ defined by $\pthalt_\time = \pth_{\horizon -\time}$ for $\time \in [0, \horizon]$. Then, the action cost of $\pthalt$ is given by,
\begin{align}
	\action_\horizon(\pthalt) &= \int_{0}^{\horizon} \frac{\norm{-\dot \pth_\time + \grad \obj(\pth_\time)}^2}{2 \variance(\obj(\pth_\time))} \dd \time
	\notag\\
							  &= \int_{0}^{\horizon} \frac{\norm{\dot \pth_\time + \grad \obj(\pth_\time)}^2}{2 \variance(\obj(\pth_\time))} \dd \time 
							  - \int_{0}^{\horizon} \frac{2 \inner{\dot \pth_\time, \grad \obj(\pth_\time)}}{ \variance(\obj(\pth_\time))} \dd \time
	\notag\\
	&= \action_\horizon(\pth) - \int_{0}^{\horizon} \inner{\dot \pth_\time, \grad \pot(\pth_\time)} \dd \time\,,
\end{align}
since $\grad \pot(\state) = 2 \primvar'(\obj(\state)) \grad \obj(\state)$.
Therefore, we get that
\begin{equation}
							\action_\horizon(\pthalt)
							  = \action_\horizon(\pth) -( \pot(\pth_\horizon) -  \pot(\pth_\tstart))\,.
\end{equation}

Take $\idx, \idxalt \in \indices$. This equality then translates to a relation between $\dquasipot_{\idx, \idxalt}$ and $\dquasipot_{\idxalt, \idx}$: considering $\pth \in \contfuncs([0, \horizon])$ such that $\pth_{0} \in \cpt_\idx$, $\pth_\horizon \in \cpt_\idxalt$ and taking the infimum over all such paths, we get that
\begin{equation}
	\dquasipot_{\idxalt, \idx}  \leq  \dquasipot_{\idx, \idxalt} + (\pot_\idxalt -  \pot_\idx)\,.
\end{equation}
where, since $\obj$ is constant on $\cpt_\idx$ and $\cpt_\idxalt$, we denote by $\pot_\idx$ and $\pot_\idxalt$ the values of $\pot$ on $\cpt_\idx$ and $\cpt_\idxalt$ respectively.

Reversing the roles of $\idx$ and $\idxalt$ and applying the same argument shows that this inequality is an equality:
\begin{equation}
	\dquasipot_{\idxalt, \idx} + \pot_\idxalt = \dquasipot_{\idx, \idxalt} + \pot_\idx\,.
\end{equation}
Denote by $\symdquasipot_{\idx, \idxalt}$ this common value. Crucially, $\symdquasipot_{\idx, \idxalt}$ is symmetric in $\idx$ and $\idxalt$.

Consider now $\sol[\tree]$ a minimum weight spanning tree on the complete but now undirected graph on $\indices$ with weights $(\symdquasipot_{\idx, \idxalt})_{\idx, \idxalt}$. We show that the minima in the $\dinvpot_\idx$ are attained $\sol[\tree]$, or more precisely, a directed version of it.

Fix $\idx$ and consider $\tree \in \trees_\idx$ a spanning tree rooted at $\idx$.
We have that, since any node $\idxalt$ is the origin of exactly one edge in $\tree$,
\begin{align}
	\sum_{(\idxalt \to \idxaltalt) \in \tree} \dquasipot_{\idxalt, \idxaltalt} + \sum_{\idxalt \in \indices} \pot_\idxalt
	&= \sum_{(\idxalt \to \idxaltalt) \in \tree} \parens*{\dquasipot_{\idxalt, \idxaltalt} + \pot_\idxalt} + \pot_\idx
	\notag\\
	&= \sum_{(\idxalt \to \idxaltalt) \in \tree} \symdquasipot_{\idxalt, \idxaltalt} + \pot_\idx \,.
\end{align}
But by definition of $\sol[\tree]$, this sum is at least greater than $\sum_{(\idxalt \leftrightarrow \idxaltalt) \in \sol[\tree]} \symdquasipot_{\idxalt, \idxaltalt}$ so that we have
\begin{equation}
	\sum_{(\idxalt \to \idxaltalt) \in \tree} \dquasipot_{\idxalt, \idxaltalt} + \sum_{\idxalt \in \indices} \pot_\idxalt \geq \sum_{(\idxalt \leftrightarrow \idxaltalt) \in \sol[\tree]} \symdquasipot_{\idxalt, \idxaltalt} + \pot_\idx\,,
\end{equation}
and taking $\tree_\idx$ an oriented version of $\sol[\tree]$ rooted at $\idx$, the equality is attained.
Therefore, we have that
\begin{equation}
    \dinvpot_\idx = \sum_{(\idxalt \to \idxaltalt) \in \tree_\idx} \dquasipot_{\idxalt, \idxaltalt} =
    \sum_{(\idxalt \leftrightarrow \idxaltalt) \in \sol[\tree]} \symdquasipot_{\idxalt, \idxaltalt} - \sum_{\idxalt \in \indices} \pot_\idxalt + \pot_\idx\,,
\end{equation}
or, in short, $\dinvpot_\idx = \pot_\idx + \const$ where $\const > 0$ is independent of $\idx$.

Therefore, the mass distribution over critical points is governed by a Gibbs measure with potential $\pot$.

Let us now mention two particular cases.
\begin{itemize}
	\item If $\variance$ is constant, then $\pot = \frac{2 \obj}{\variance}$.
	\item It $\variance$ is linear, \ie of the form $\variance(\obj(\state)) = \variance_1 (\obj(\state) + \variance_{0})$, then $\pot = \frac{2}{\variance_1} \log (\obj + \variance_{0})$.
\end{itemize}

\subsection{Truncated Gaussian noise}
\label{app:subsec:truncated_gaussian_noise}

To fit into our theoretical framework, we consider truncated Gaussian noise instead. The general outline of the proof but with added steps to handle the error due to the truncation. In particular, one must show that, without loss of generality, we can only conider paths whose derivative has the same norm as the gradient of $\obj$. This is done with \citet[Chap.~4, Lem.~3.1]{FW98} that we adapt to our setting.

Assume that $\noise(\state, \sample)$ follows a centered Gaussian distribution with covariance $\variance(\obj(\state)) \identity$ conditioned on being in $\ball(0, \Radius(\state))$ for some $\Radius(\state) > 0$.

As in \cref{def:bdpot}, we define $\pot(\state) = 2 \primvar(\obj(\state))$ with $\primvar' = \frac{1}{\variance}$ and denote by $\pot_\idx$ the value taken by $\pot$ on $\eqcl_\idx$.

Consider some $0 < \precsalt \leq \half$ and assume that

\begin{equation}
	\sup_{\state \in \vecspace} 2^{\dims + 4}(\dims + 1)e^{-\tfrac{\Radius^2}{16 \variance}} \leq \precsalt\,,
\end{equation}
so that the error term $2 \errorterm(\variance(\obj(\state)), \Radius(\state))$ in \cref{lem:truncated_gaussian:lagrangian} is bounded by $\precsalt$.

Moreover, assume that, for any $\state \in \vecspace$
\begin{equation}
	\norm{\grad \obj(\state)} \leq \frac{\Radius(\state)}{8} 
\end{equation}
\begin{lemma}
	\label{lemma:pth_resc}
	Consider $\pth \in \contfuncs([0, \horizon])$.
	Then, there exists $\widetilde \pth \in \contfuncs([0, \horizonalt])$ a reparametrization of $\pth$ such that, for any $\time \in [0, \horizonalt]$,
	\begin{equation}
		\norm{\dot{\widetilde \pth}_\timealt} = \norm{\grad \obj(\widetilde \pth_\timealt)}\,.
	\end{equation}
	and
	\begin{equation}
	\action_\horizon(\pth) \geq \int_{0}^{\horizonalt} \frac{\norm{\dot{\widetilde \pth}_\timealt + \grad \obj(\widetilde \pth_\timealt)}^2}{2 (1 + \precsalt)\variance(\obj(\widetilde \pth_\timealt))} \dd \timealt\,.
	\end{equation}
\end{lemma}
\begin{proof}
	\def\resc{\widetilde{\pth}}
	By the proof \citet[Chap.~4, Lem.~3.1]{FW98}, there exists $\time(\timealt)$ change of time such that,  with $\resc_\timealt = \pth_{\time(\timealt)}$, $\norm{\dot \resc_\timealt} = \norm{\grad \obj(\resc_\timealt)}$.

	We have that 
	\begin{equation}
		\action_\horizon(\pth) = \int_{0}^{\time^{-1}(\horizon)} \dot \time(\timealt) \lagrangian(\resc_\timealt, (\dot \time(\timealt))^{-1}\dot \resc_\timealt) \dd \timealt\,,
	\end{equation}
	so it suffices to bound $\lagrangian(\resc_\timealt, \dot \resc_\timealt)$ from below: by definition, we have
	\begin{align}
		\lagrangian(\resc_\timealt, (\dot \time(\timealt))^{-1}\dot \resc_\timealt) 
		&\geq 
		\sup \setdef*{
			\inner{\mom, (\dot \time(\timealt))^{-1}\dot \resc_{\revise\timealt} + \grad \obj(\resc_\timealt)}
			- \hamiltalt(\resc_\timealt, \mom)
		}{
			\norm{\mom} \leq \frac{\Radius(\resc_\timealt)}{2 \variance(\obj(\resc_\timealt))}
		}
		\notag\\
		&\geq
		\sup \setdef*{
			\inner{\mom, (\dot \time(\timealt))^{-1}\dot \resc_{\revise\timealt} + \grad \obj(\resc_\timealt)}
			- (1 + \precsalt) \half[\variance(\obj(\resc_\timealt))] \norm{\mom}^2
		}{
		   \norm{\mom} \leq \frac{\Radius(\resc_\timealt)}{2 \variance(\obj(\resc_\timealt))}
		}\,,
	\end{align}
	by \cref{lem:truncated_gaussian:lagrangian}.
    Applying \cref{lemma:almost_sq_norm_opt_rescale} with \revise{$\vel \gets (\dot \time(\timealt))^{-1}\dot \resc_{\time}$, $\velalt \gets \grad \obj(\resc_\timealt)$ and $\lambda \gets (\dot \time(\timealt))^{-1}$},
	now exactly yields, for almost all $\timealt$,
	\begin{align}
		\lagrangian(\resc_\timealt, (\dot \time(\timealt))^{-1}\dot \resc_\timealt) 
		&\geq
        \revise{(\dot \time(\timealt))^{-1}}
		\sup \setdef*{
			\inner{\mom, \dot \resc_{\revise\timealt} + \grad \obj(\resc_\timealt)}
			- (1 + \precsalt) \half[\variance(\obj(\resc_\timealt))] \norm{\mom}^2
		}{
			\norm{\mom} \leq \frac{\Radius(\resc_\timealt)}{2 \variance(\obj(\resc_\timealt))}
		}
		\notag\\
        &= \revise{(\dot \time(\timealt))^{-1}}
		\frac{\norm{\dot \resc_\timealt + \grad \obj(\resc_\timealt)}^2}{2 (1 + \precsalt)\variance(\obj(\resc_\timealt))}\,,
	\end{align}
	since $\norm{\dot \resc_\timealt + \grad \obj(\resc_\timealt)} \leq \frac{\Radius(\resc_\timealt)}{2}$.
\end{proof}

\begin{lemma}
	With this setting, for any $\idx \in \indices$
	\begin{equation}
		\dinvpot_\idx =
    \sum_{(\idxalt \leftrightarrow \idxaltalt) \in \sol[\tree]} \symdquasipot_{\idxalt, \idxaltalt} - \sum_{\idxalt \in \indices} \pot_\idxalt + \pot_\idx\,,
	\end{equation}
\end{lemma}

\begin{proof}

	Consider $\pth \in \contfuncs([0, \horizon])$ such that $\pth_{0} \in \cpt_\idx$, $\pth_\horizon \in \cpt_\idxalt$ and $\action_\horizon(\pth) < + \infty$.
	Then, by the previous lemma \cref{lemma:pth_resc}, there exists $\widetilde \pth \in \contfuncs([0, \horizonalt])$ a reparametrization of $\pth$ such that, for any $\time \in [0, \horizonalt]$,
	\begin{align}
		\action_\horizon(\pth)&\geq \int_{0}^{\horizonalt} \frac{\norm{\dot{\widetilde \pth}_\timealt + \grad \obj(\widetilde \pth_\timealt)}^2}{2 (1 + \precsalt)\variance(\obj(\widetilde \pth_\timealt))} \dd \timealt
		\notag\\
		&= \int_{0}^{\horizonalt} \frac{\norm{- \dot{\widetilde \pth}_\timealt  + \grad \obj(\widetilde \pth_\timealt)}^2}{2 (1 + \precsalt)\variance(\obj(\widetilde \pth_\timealt))} \dd \timealt + \frac{\pot(\pth_\horizon) - \pot(\pth_{0})}{1 + \precsalt}\,,
	\end{align}
	where we performed the same computations as above in \cref{app:subsec:gaussian_noise}.
	Considering the path $(\widetilde{\pth}_{\horizonalt - \timealt})_{\timealt \in [0, \horizonalt]}$ and invoking the upper-bound on the Lagrangian from \cref{lem:truncated_gaussian:lagrangian} with $- \dot{\widetilde \pth}_\timealt + \grad \obj(\widetilde \pth_\timealt)$ which still has norm less than $\Radius(\widetilde \pth_\timealt) / 4$, we get that
	\begin{equation}
		\action_\horizon(\pth) \geq \frac{1 - \precsalt}{1 + \precsalt}\dquasipot_{\idxalt, \idx} + \frac{\pot(\pth_\horizon) - \pot(\pth_{0})}{1 + \precsalt}\,.
	\end{equation}
	The result now follows from the same computations as in \cref{app:subsec:gaussian_noise}.
\end{proof}

\subsection{Local dependencies}
\label{app:subsec:local_dependencies}

Under local assumptions similar to \citet{moriPowerLawEscapeRate}, we demonstrate how the modelling of the noise influences the invariant measure.
\begin{lemma}
	\label{lem:loc_dep}
	Consider $\comp_\iComp$ minimizing, 
	$\variance \from \R \to (0, +\infty)$ continuous, and take $\primvar \from \R \to \R$ such that $\primvar' = \frac{1}{\variance}$.
	Assume that $\sol[\hessmat]$ is a positive definite matrix such that, locally near $\eqcl_\iComp$, it holds that:
\begin{equation}
	\primvar(\obj(\state)) = \sum_{\eigval \in \eig \sol[\hessmat]} \objalt^\eigval(\state_\eigval)\,,
\end{equation}
where $\state_\eigval$ denotes the orthogonal projection of $\state$ on the eigenspace of the eigenvalue $\eigval$ and where $\objalt^\eigval \from \vecspace \to \R$ is continuously differentiable.
	Define the potential $\pot \from \vecspace \to \R$ by
	\begin{equation}
		\pot(\state) = \sum_{\eigval \in \eig \sol[\hessmat]} \frac{2 \objalt^\eigval(\state_\eigval)}{\eigval}\,.
	\end{equation}
	If we have the anistropic subGaussian bound: for $\point$ in a neighborhood of $\eqcl_\iComp$, $\mom \in \clball(0, \norm{\pot(\state)})$,
	\begin{equation}
		\hamiltalt(\point, \mom) \leq \frac{\variance(\obj(\state))}{2} \inner{\mom, \sol[\hessmat] \mom}\,,
	\end{equation}
	then there is $\margin > 0$ such that, for all $\jComp \neq \iComp$, any $0 < \marginalt \leq \margin$,
	\begin{equation}
		\energy_\jComp \geq  \min \setdef*{
			\pot(\state)  - \pot_\idx
		}{
			\state,
			d(\state, \comp_\iComp) = \marginalt
		} > 0 \,,
		\label{eq:local_energy_levels}
	\end{equation}  
	where $\pot_\idx$ is the value of $\pot$ on $\comp_\idx$.

	Moreover, there exists $\Radius > \margin$ such that, if there exists $\const > 0$, $\variancealt > 0$, such that, for all $\state \in \clball(0, \Radius) \setminus \U_\margin(\eqcl_\iComp)$, $\vel \in \ball(\grad \obj(\state), \const)$,
	\begin{equation}
		\lagrangianalt(\state, \vel) \leq \frac{\norm{\vel}}{2 \variancealt}\,,
	\end{equation}
	then there exists $\Const > 0$ that depends only on $\radius, \Radius, \const$, $\obj$ restricted to $\vecspace \setminus \U_{\radius}(\eqcl_\iComp)$ 
	such that
	\begin{equation}
		\energy_\idx \leq  \frac{\Const}{\variancealt}\,.
	\end{equation}
\end{lemma}

The assumptions on the Hamiltonian and the Lagrangian roughly say that the noise share some similarities with Gaussian distributions with the prescribed variances. In particular, they are satisfied in the (truncated) Gaussian case, as shown in \cref{lem:truncated_gaussian:lagrangian}.

Moreover, a takeaway of this lemma is that, if $\variance$ or the eigenvalues of $\sol[\hessmat]$ are small enough, then $\eqcl_\idx$ must be the ground state even if it may not be the global minimum of $\pot$.
This lemma is general enough to handle non-constant variance: a notable example is when $\variance(\obj(\state))$ is linear in $\obj(\state)$ and where the resulting potential $\pot$ then depends logarithmically on the value $\obj$.

\begin{proof}
	Denote $\orthproj_\eigval \in \R^{d \times d}$ the orthogonal projection on the eigenspace of $\sol[\hessmat]$ associated with the eigenvalue $\eigval$. $\pot$ can thus be rewritten as 
	\begin{equation}.
		\pot(\state) = \sum_{\eigval \in \eig \sol[\hessmat]} \frac{2 \objalt^\eigval(\orthproj_\eigval \state)}{\eigval}\,,
	\end{equation}
	so that its gradient is given by
	\begin{equation}
		\grad \pot(\state) = \sum_{\eigval \in \eig \sol[\hessmat]} \frac{2 \orthproj_\eigval \grad \objalt^\eigval(\orthproj_\eigval \state)}{\eigval} \,.
	\end{equation}
	In particular, we obtain that for $\state$ close enough to $\eqcl_\iComp$, using the orthogonality of the projections,
	\begin{align}
	   \hamiltalt(\state, \grad \pot(\state)) 
	   &\leq 
	\frac{\variance(\obj(\state))}{2} \inner{\grad \pot(\state), \sol[\hessmat] \grad \pot(\state)}
	\notag\\
	   &= \frac{\variance(\obj(\state))}{2} \sum_{\eigval \in \eig \sol[\hessmat]}  \frac{4\norm{\orthproj_\eigval \grad \objalt^\eigval(\orthproj_\eigval \state)}^2}{\eigval}
	   \notag\\
	   &= \frac{\variance(\obj(\state))}{2} \inner*{\grad \pot(\state), 2 \grad (\primvar \circ \obj)(\state)}
	   \notag\\
	   &=  \inner*{\grad \pot(\state), \grad \obj(\state)}\,.
	\end{align}
	Therefore, for $\state$ close enough to $\eqcl_\iComp$, we have that
	\begin{align}
		\label{eq:lemma:anis_bd_hamilt}
		\hamilt(\state, \grad \pot(\state))
		&= -\inner*{\grad \pot(\state), \grad \obj(\state)} + \hamiltalt(\state, \grad \pot(\state))
		\leq 0\,.
	\end{align}
	For $\state$ close to $\eqcl_\idx$, let us compute $\inner{\grad \pot(\state)}{\grad \obj(\state)}$:
	\begin{align}
		\inner{\grad \pot(\state)}{\grad \obj(\state)} &= \variance(\obj(\state)) \inner{\grad \pot(\state)}{\grad (\primvar \circ \obj)(\state)}
		\notag\\
													   &= \variance(\obj(\state)) \sum_{\eigval \in \eig \sol[\hessmat]} \frac{2 \norm{\orthproj_\eigval \grad \objalt^\eigval(\orthproj_\eigval \state}}{\eigval}\,,
	\end{align}
	which is (stricly) positive in a small neighborhood of $\eqcl_\idx$ (excluding $\eqcl_\idx$ itself). Therefore, $\pot$ is decreasing on trajectories of the flow in this neighborhood. With the same proof as in \cref{lemma:minimizing_component_asympt_stable}, we deduce that there exists $\margin > 0$ such that $\argmin_{\state \in \U_\margin(\eqcl_\idx)} \pot(\state) = \eqcl_\idx$.
	Moreover, take $\margin > 0$ small enough such that \cref{eq:lemma:anis_bd_hamilt} holds on $\U_\margin(\eqcl_\idx)$, $\U_\margin(\eqcl_\idx) \cap \crit\obj = \eqcl_\idx$ and the trajectories of the flow started in $\U_\margin(\eqcl_\idx)$ stay converge to $\eqcl_\idx$.
	We now proceed as in the proof of \cref{lemma:ground_states_loc_min}.
	Take $\marginalt \leq \margin/2$, $\open \defeq  \U_\marginalt(\eqcl_\idx)$ and $\Margin \defeq \min \setdef*{\pot(\state) - \pot_\idx}{\state \in \vecspace,\, d(\state, \eqcl_\idx) = \marginalt}$, which is positive by definition.
	Fix $\jComp \neq \iComp$: we show that $\qpot_{\idx, \jComp} \geq \Margin$.
	Consider some $\horizon > 0$ and $\pth \in \contfuncs([0, \horizon], \vecspace)$ such that $\pth_{0} \in \eqcl_\idx$ and $\pth_\horizon \in \eqcl_\idxalt$. By definition of $\margin$ and by continuity of $\pth$ and $d(\cdot, \eqcl_\idx)$, there exists $\time \in [0, \horizon]$ such that $d(\pth_\time, \eqcl_\idx) = \marginalt$.
	Therefore, we have that
	\begin{align}
		\action_{0, \horizon}(\pth)
		&\geq
		\action_{0, \time}(\pth)
		= \int_{0}^\time \lagrangian(\pth_\timealt, \dot \pth_\timealt) \dd \timealt\,,
	\end{align}
	and therefore, by definition of $\lagrangian$ as the conjugate of $\hamilt$, we obtain that
	\begin{align}
		\action_{0, \horizon}(\pth)
		&\geq
		\int_{0}^\time
		\inner*{\dot \pth_\timealt, \grad \pot(\pth_\timealt)} - \hamiltalt(\pth_\timealt, \grad \pot(\pth_\timealt)) \dd \timealt
		\geq
		\int_{0}^\time 
		\inner*{\dot \pth_\timealt, \grad \pot(\pth_\timealt)}\,,
	\end{align}
	where we used \cref{eq:lemma:anis_bd_hamilt} in the last inequality.
	Thus, we get
	\begin{align}
		\action_{0, \horizon}(\pth)
		&\geq \pot(\pth_\time) - \pot(\pth_{0})
		\geq \Margin\,,
	\end{align}
	with any $\marginalt \leq \margin/2$ (and therefore $\margin$ in the statement corresponds to $\margin / 2$ here).

	Finally, it remains to show that, any $\idxaltalt \neq \idx$, $\energy_{\idxaltalt} \geq \Margin$.
	Consider any tree rooted at $\idxaltalt$: it must have an edge of the form $\idx \to \idxalt$ for some $\idxalt \in \indices$ and therefore the sum of its weights will be at least $\dquasipot_{\idx, \idxalt} \geq \Margin$. Hence, since it holds for any such tree, we have that $\energy_{\idxaltalt} \geq \Margin$.

	We now prove the second part of the lemma. $\margin$ was chosen small enough so that it is possible to find $\Radius > \margin$ such that $\ball(0, \Radius) \setminus \U_{\margin}(\eqcl_\idx)$ contains all the $\eqcl_\idxalt$ for $\idxalt \neq \idx$ and not $\eqcl_\idx$.
	The assumption on the Lagrangian implies that, for any $\state \in \ball(0, \Radius) \setminus \U_{\margin}(\eqcl_\idx)$, $\vel \in \ball(0, \const)$,
	\begin{equation}
		\lagrangian(\state, \vel) \leq \frac{\norm{\vel + \grad \obj(\state)}^2}{2 \variancealt}\,.
	\end{equation}
	Fix $\idxalt \neq \idx$ and take $\pth \in \contdiff{1}([0, \horizon], \vecspace)$ such that $\pth_{0} \in \eqcl_\idxalt$, $\state \defeq \pth_\horizon \in \U_{\margin}(\eqcl_\idx)$ and $\pth$ remains in $\clball(0, \Radius) \setminus \U_{\margin}(\eqcl_\idx)$.
	Without loss of generality, at the expense of replacing $\pth$ by a reparametrization, we can assume that $\dot \pth_\time$ in $\ball(0, \const)$ for all $\time \in [0, \horizon]$.
	Then, we have that
	\begin{align}
		\action_{0, \horizon}(\pth)
		&\leq 
		\int_{0}^{\horizon} \lagrangian(\pth_\timealt, \dot \pth_\timealt) \dd \timealt
		\leq 
		\int_{0}^{\horizon} \frac{\norm{\dot \pth_\timealt + \grad \obj(\pth_\timealt)}^2}{2 \variancealt} \dd \timealt\,,
	\end{align} 
	which depends only on the value of $\grad \obj$ outside of $\U_{\margin/2}(\eqcl_\idx)$.
	But, by the same reasoning as in \cref{lemma:not_asympt_stable_implies_unstable2}, since the flow started at $\state$ converges to $\eqcl_\idx$, we have that $\dquasipot(\setof{\state}, \eqcl_\idx) = 0$. Therefore, we have that,
	\begin{equation}
		\dquasipot_{\idxalt, \idx} \leq   \int_{0}^{\horizon} \frac{\norm{\dot \pth_\timealt + \grad \obj(\pth_\timealt)}^2}{2 \variancealt} \dd \timealt\,,
	\end{equation}
	and, taking the maximum of such quantities over all $\idxalt \neq \idx$, we obtain $\Const > 0$ such, for any $\idxalt \neq \idx$,
	\begin{equation}
		\qpot_{\idxalt, \idx} \leq \frac{\Const}{\variancealt}\,.
	\end{equation}
	To conclude on the value of $\energy_\idx$, we consider the tree rooted at $\idx$ made of all the edges $(\idxalt, \idx)$ for $\idxalt \neq \idx$. It has weight at most $(\neqcl - 1) \Const / \variancealt$ and therefore, we have that $\energy_\idx \leq (\neqcl - 1) \Const / \variancealt$.
\end{proof}

\section{Auxiliary results}
\label{app:aux}
\subsection{Truncated Gaussian distribution}
\label{app:truncated_gaussian}

\begin{lemma}
	\label{lem:truncated_gaussian:hamilt}
	Consider $\rv \sim \gaussian(0, \covmat)$ a multivariate Gaussian distribution $\covmat \in \R^{\dims \times \dims}$ positive definite.
	For $\Radius > 0$, define the truncated Gaussian \ac{rv} $\trv$ by conditioning $\rv$ to the ball $\clball(0, \Radius)$.
	Define
	\begin{equation}
		\hamiltalt(\mom) \defeq \log \ex \bracks*{
			e^{\inner{\mom}{\trv}}
		}
	\end{equation}
	and
	\begin{equation}
		\errortermalt(\covmat, \Radiusalt) \defeq 
e^{-\tfrac{\Radiusalt^2}{4 \norm{\covmat}}}(\tr \covmat + \norm{\covmat}) 2^{\dims + 3}
	\end{equation}
   Then, for $\Radiusalt > 0$ such that
	\begin{equation}
		\Radiusalt \geq \sqrt{\norm{\covmat} (2 \dims + 4) \log 2}\,,
	\end{equation}
	it holds that, for any $\mom \in \vecspace$ such that $\norm{\covmat \mom} \leq \Radius - \Radiusalt$,
	\begin{subequations}
   \begin{align}
	   \abs*{
		   \hamiltalt(\mom) - \tfrac{1}{2} \inner{\mom, \covmat \mom}
	   }
	   &\leq 
	   \tfrac{1}{2}\errortermalt(\covmat, \Radiusalt) \norm{\mom}^2\\
	   \norm*{
		   \grad \hamiltalt(\mom) - \covmat \mom
	   }
	   &\leq 
	   \errortermalt(\covmat, \Radiusalt) \norm{\mom}\\
	   \norm*{
		   \Hess \hamiltalt(\mom) - \covmat
	   }
	   &\leq
	   \errortermalt(\covmat, \Radiusalt)\eqdot
   \end{align}
   \end{subequations}
\end{lemma}
\begin{proof}
	\def\dinner{\cdot}
	\def\thalf{\tfrac{1}{2}}
	Define, for $\plainset$ a measurable set,
	\begin{equation}
		\partition(\plainset) = \int_\plainset e^{-\thalf {\point \dinner \covmat^{-1} \point }} \dd \point\,,
	\end{equation}
	and, for convenience, $\cpt \defeq \clball(0, \Radius)$.
	For notational convenience, in this proof, we will denote the inner product between two vectors $\point$ and $\mom$ with a simple dot $\point \dinner \mom$.
	We have that
	\begin{align}
		\ex \bracks*{
			e^{\trv \dinner \mom}
		}
		&= \frac{e^{ \thalf \mom \dinner \covmat \mom}}{\partition(\cpt)} \int_\cpt e^{- \thalf (\point - \covmat \mom) \dinner \covmat^{-1} (\point - \covmat \mom)} \dd \point
		\notag\\
		&= \frac{e^{ \thalf \mom \dinner \covmat \mom}}{\partition(\cpt)} \int_{\cpt - \covmat \mom} e^{- \thalf \pointalt \dinner \covmat^{-1} \pointalt} \dd \pointalt
		\notag\\
		&= e^{ \thalf \mom \dinner \covmat \mom} \frac{\partition(\cpt - \covmat \mom)}{\partition(\cpt)}\,,
	\end{align}
	where we performed the change of variable $\pointalt \gets \point - \covmat \mom$.

	Define
	\begin{equation}
		\fn(\mom) \defeq \partition(\cpt - \covmat \mom) = \int_\cpt e^{-\thalf (\point - \covmat \mom) \dinner \covmat^{-1} (\point - \covmat \mom)} \dd \point\eqdot
	\end{equation} so that
	\begin{equation}
		\hamiltalt(\mom) 
		=
		\log \ex \bracks*{
			e^{\trv \dinner \mom}
		}
		= \half \mom \dinner \covmat \mom
		+ \log \frac{\fn(\mom)}{\fn(0)}\eqdot
	\end{equation}
	Therefore it suffices to bound $\log \frac{\fn(\mom)}{\fn(0)}$ and its derivatives.

	Differentiating yields
	\begin{subequations}
	\begin{align}
		\grad \fn(\mom)
		&=
		\int_\cpt (\point - \covmat \mom) e^{-\thalf (\point - \covmat \mom) \dinner \covmat^{-1} (\point - \covmat \mom)} \dd \point
		\\
		\Hess \fn(\mom)
		&=
		\int_\cpt (\point - \covmat \mom) (\point - \covmat \mom)^\top e^{-\thalf (\point - \covmat \mom) \dinner \covmat^{-1} (\point - \covmat \mom)} \dd \point
		- \covmat \fn(\mom)\eqdot
	\end{align}
\end{subequations}
Note that, by symmetry of $\clball(0, \Radius)$, $\grad \fn(0) = 0$.

	We first bound $\Hess \fn(\mom)$.
	Performing the change of variable $\pointalt \gets \point - \covmat \mom$ again gives us
	\begin{align}
		\Hess \fn(\mom)
		&=
		\int_{\cpt - \covmat \mom} \pointalt \pointalt^\top e^{-\thalf \pointalt \dinner \covmat^{-1} \pointalt} \dd \pointalt
		- \covmat \fn(\mom)
		\notag\\
		&=
		\int_{\cpt - \covmat \mom} (\pointalt \pointalt^\top -\covmat) e^{-\thalf \pointalt \dinner \covmat^{-1} \pointalt} \dd \pointalt
		\notag\\
		&=
		- \int_{\vecspace \setminus (\cpt - \covmat \mom)} (\pointalt \pointalt^\top - \covmat) e^{-\thalf \pointalt \dinner \covmat^{-1} \pointalt} \dd \pointalt\,,
		\end{align}
		where we used that $\int_{\vecspace}( \pointalt \pointalt^\top - \covmat) e^{-\thalf \pointalt \dinner \covmat^{-1} \pointalt} \dd \pointalt = 0$.

		By definition of $\Radiusalt$, $\cpt - \covmat \mom$ contains $ \clball(0, \Radiusalt)$.
		We now bound the operator norm of $\Hess \fn(\mom)$:
		\begin{align}
			\norm*{
				\Hess \fn(\mom)
			}
			&\leq 
			\int_{\vecspace \setminus (\cpt - \covmat \mom)}  (\norm{\pointalt}^2 + \norm{\covmat}) e^{-\thalf \pointalt \dinner \covmat^{-1} \pointalt} \dd \pointalt
			\notag\\
			&\leq
		\int_{\vecspace \setminus \clball(0, \Radiusalt)} (\norm{\pointalt}^2 + \norm{\covmat}) e^{-\thalf \pointalt \dinner \covmat^{-1} \pointalt} \dd \pointalt
		\notag\\
			&\leq
			e^{-\tfrac{\Radiusalt^2}{4 \norm{\covmat}}}
			\int_{\vecspace} (\norm{\pointalt}^2 + \norm{\covmat}) e^{-\tfrac{1}{4} \pointalt \dinner \covmat^{-1} \pointalt} \dd \pointalt
			\notag\\
			&=
			e^{-\tfrac{\Radiusalt^2}{4 \norm{\covmat}}}
			(\tr \covmat + \norm{\covmat}) (4 \pi)^{\dims/2} \det(\covmat)^{1/2}\,,
			\label{eq:lemma:truncated_gaussian:Hess}
		\end{align}

	We now bound $\grad \fn(\mom)$. The change of variable $\pointalt \gets \point - \covmat \mom$ yields
	\begin{align}
		\grad \fn(\mom)
		&=
		\int_{\cpt - \covmat \mom} \pointalt e^{-\thalf \pointalt \dinner \covmat^{-1} \pointalt} \dd \pointalt
		=
		\int_{\vecspace \setminus (\cpt - \covmat \mom)} \pointalt e^{-\thalf \pointalt \dinner \covmat^{-1} \pointalt} \dd \pointalt\,,
	\end{align}
	where we used that $\int_{\vecspace} \pointalt e^{-\thalf \pointalt \dinner \covmat^{-1} \pointalt} \dd \pointalt = 0$.

	Therefore, similar computations as above yield that
	\begin{align}
		\norm{\grad \fn(\mom)}
		&\leq
		\int_{\vecspace \setminus (\cpt - \covmat \mom)} \norm{\pointalt} e^{-\thalf \pointalt \dinner \covmat^{-1} \pointalt} \dd \pointalt
		\notag\\
		&\leq
		\int_{\vecspace \setminus \clball(0, \Radiusalt)} \norm{\pointalt} e^{-\thalf \pointalt \dinner \covmat^{-1} \pointalt} \dd \pointalt
		\notag\\
		&\leq 
		e^{-\tfrac{\Radiusalt^2}{4 \norm{\covmat}}} \int_{\vecspace} \norm{\pointalt} e^{-\tfrac{1}{4} \pointalt \dinner \covmat^{-1} \pointalt} \dd \pointalt
		\notag\\
		&\leq 
		e^{-\tfrac{\Radiusalt^2}{4 \norm{\covmat}}} \sqrt{\tr \covmat} (4 \pi)^{\dims/2} \det(\covmat)^{1/2}\eqdot
		\label{eq:lemma:truncated_gaussian:grad}
	\end{align}

	Let us now lower-bound $\fn(\mom)$.
	In the same fashion as above, we have that
		\begin{align}
			\fn(\mom)
			&= \int_{\cpt - \covmat \mom} e^{-\thalf \pointalt \dinner \covmat^{-1} \pointalt} \dd \pointalt
			\notag\\
			&\geq \int_{\clball(0, \Radiusalt)} e^{-\thalf \pointalt \dinner \covmat^{-1} \pointalt} \dd \pointalt
			\notag\\
			&\geq (2 \pi)^{\dims/2} \det(\covmat)^{1/2} \parens*{
			1 - 2^{\dims/2} e^{-\tfrac{\Radiusalt^2}{4 \norm{\covmat}}}}\,,
		\end{align}
		so that, if $\Radiusalt\geq\sqrt{\norm{\covmat}(2 \dims + 4) \log 2}$,
		it holds that
		\begin{equation}
			\fn(\mom)
			\geq
			\thalf (2 \pi)^{\dims/2} \det(\covmat)^{1/2}
			\label{eq:lemma:truncated_gaussian:fn}
		\end{equation}

		The Hessian of $\log \fn / \fn(0)$ is then given by
		\begin{equation}
			\Hess \log \frac{\fn(\mom)}{\fn(0)}
			= \frac{\Hess \fn(\mom)}{\fn(\mom)} - \frac{\grad \fn(\mom) \grad \fn(\mom)^\top}{\fn(\mom)^2}\eqdot
		\end{equation}
		\cref{eq:lemma:truncated_gaussian:Hess,eq:lemma:truncated_gaussian:grad,eq:lemma:truncated_gaussian:fn} combined yield that
		\begin{align}
			\norm*{
				\Hess \log \frac{\fn(\mom)}{\fn(0)}
			}
			&\leq
			e^{-\tfrac{\Radiusalt^2}{4 \norm{\covmat}}}(\tr \covmat + \norm{\covmat}) 2^{\dims/2 + 1} + \parens*{
				e^{-\tfrac{\Radiusalt^2}{4 \norm{\covmat}}} \sqrt{\tr \covmat} 2^{\dims/2 + 1}
			}^{2}
			\notag\\
			&\leq 
			e^{-\tfrac{\Radiusalt^2}{4 \norm{\covmat}}}(\tr \covmat + \norm{\covmat}) 2^{\dims + 3} = 
			\errortermalt(\covmat, \Radiusalt)\eqdot
		\end{align}

		Taylor-Lagrange inequality now yields the full result since $\log \fn(0)/\fn(0) = 0$ and $\grad \log \fn(0) = 0$.
	\end{proof}

	\begin{lemma}
		\label{lem:truncated_gaussian:lagrangian}
		Consider $\rv \sim \gaussian(0, \variance \identity)$ a multivariate Gaussian distribution with $\variance > 0$. For $\Radius > 0$, define the truncated Gaussian \ac{rv} $\trv$ by conditioning $\rv$ to the ball $\clball(0, \Radius)$.
		Define
		\begin{subequations}
		\begin{align}
			\hamiltalt(\mom) &\defeq \log \ex \bracks*{
				e^{\inner{\mom}{\trv}}
			}\\
			\lagrangianalt(\mom) &\defeq \hamiltalt^*(\mom)\,,
		\end{align}
		\end{subequations}
		and 
		\begin{equation}
			\errorterm(\variance, \Radius) \defeq 
			e^{-\tfrac{\Radius^2}{16 \variance}}2^{\dims + 3}(\dims + 1)
		\end{equation}
		and assume that $\Radius > 0$ satisfies
		\begin{equation}
			\Radius \geq 4 \sdev \sqrt{(\dims + 3) \log 2 + \log(\dims + 1)}\eqdot
		\end{equation}
		Then, for any $\mom \in \vecspace$ such that $\norm{\mom} \leq \frac{\Radius}{2 \variance}$, $\vel \in \vecspace$ such that $\norm{\vel} \leq \frac{\Radius}{4}$, it holds that
		\begin{subequations}
		\begin{align}
			\parens*{1 - \errorterm(\variance, \Radius)} \frac{\variance \norm{\mom}^2}{2}
			\leq &\hamiltalt(\mom) \leq 
			\parens*{1 + \errorterm(\variance, \Radius)} \frac{\variance \norm{\mom}^2}{2}\\
			\parens*{1 - 2\errorterm(\variance, \Radius)} \frac{\norm{\vel}^2}{2 \variance}
			\leq &\lagrangianalt(\vel) \leq
			\parens*{1 + 2\errorterm(\variance, \Radius)} \frac{\norm{\vel}^2}{2 \variance}\eqdot
		\end{align}
		\end{subequations}
	\end{lemma}

	\begin{proof}
		First, let us show that, for any $\radius > 0$, $\grad \hamiltalt \parens*{\clball(0, \radius)}$ is a closed ball centered at 0.
		Since the distribution of $\trv$ is invariant by rotation, this set can be rewritten as
		\begin{equation}
			\grad \hamiltalt \parens*{\clball(0, \radius)} = \setdef*{\vel \in \vecspace}{\norm{\vel} \in \setdef*{\norm*{\grad \hamiltalt(\mom)}}{\mom \in \clball(0, \radius)}}\eqdot
		\end{equation}
        But, by connectedness of $\clball(0, \radius)$ and continuity of $\grad \hamiltalt$, $\setdef*{\norm*{\grad \hamiltalt(\mom)}}{\mom \in \clball(0, \radius)}$ is an interval. Since $\grad \hamiltalt(0) = 0$, it is either \revise{$\left[0, \sup_{\clball(0, \radius)}\norm*{\grad \hamiltalt}\right]$ or $\left[0, \sup_{\clball(0, \radius)}\norm*{\grad \hamiltalt}\right)$}. $\grad \hamiltalt  \parens*{\clball(0, \radius)}$ being  compact and therefore closed, it must be the latter.
		Hence, we have shown that
		\begin{equation}
		\grad \hamiltalt \parens*{\clball(0, \radius)} = \clball \parens*{0, \sup_{\clball(0, \radius)}\norm*{\grad \hamiltalt}}\eqdot
		\end{equation}

   We apply \cref{lem:truncated_gaussian:hamilt} with $\Radiusalt \gets \half[\Radius]$. Note that our choice of $\Radius$ implies that $\Radiusalt$ satisfies the condition of \cref{lem:truncated_gaussian:hamilt} and, moreover, that
   $\errorterm(\variance, \Radius) \leq \half$. \Cref{lem:truncated_gaussian:hamilt} directly implies the bound on $\hamiltalt$.

   Consider $\mom \in \vecspace$ such that $\norm{\mom} \leq \frac{\Radius}{2 \variance}$.
   Then, by \cref{lem:truncated_gaussian:hamilt}, we have that
		\begin{align}
			\norm*{\grad \hamiltalt \parens*{\mom}}
			&\geq \variance \norm{\mom} - e^{-\tfrac{\Radius^2}{16 \variance}}2^{\dims + 3}
			\parens*{
			\tr\parens*{\variance \identity} + \norm*{\variance \identity}} \norm{\mom}
			\notag\\
			&= \variance \norm{\mom} (1 - \errorterm(\variance, \Radius))
			\notag\\
			&\geq \frac{\variance \norm{\mom}}{2}\,,
		\end{align}
		where we used that $\errorterm(\variance, \Radius) \leq \half$.
		Therefore, we obtain that $\sup \setdef*{\norm*{\grad \hamiltalt(\mom)}}{\mom \in \clball \parens*{0, \frac{\Radius}{2 \variance}}} \geq \frac{\Radius}{4}$ so that $\grad \hamiltalt \parens*{\clball(0, \frac{\Radius}{2 \variance})}$ contains $\clball(0, \frac{\Radius}{4})$.

		Take $\vel \in \clball(0, \frac{\Radius}{4})$ which therefore belongs to $\grad \hamiltalt \parens*{\clball(0, \frac{\Radius}{2 \variance})}$.
		Therefore, $\lagrangianalt(\vel)$ can be rewritten as 
		\begin{equation}
			\lagrangianalt(\vel) = \sup_{\mom \in \clball(0, \frac{\Radius}{2 \variance})} \inner{\vel}{\mom} - \hamiltalt(\mom)\eqdot
		\end{equation}
		Using \cref{lem:truncated_gaussian:hamilt} again, we obtain that
		\begin{align}
			\sup_{\mom \in \clball(0, \frac{\Radius}{2 \variance})} \inner{\vel}{\mom} -  \frac{\variance}{2}(1 + \errorterm(\variance, \Radius)) \norm{\mom}^2
			\leq \lagrangianalt(\vel) \leq \sup_{\mom \in \clball(0, \frac{\Radius}{2 \variance})} \inner{\vel}{\mom} -  \frac{\variance}{2}(1 - \errorterm(\variance, \Radius)) \norm{\mom}^2\,,
		\end{align}
		so that, since $\errorterm(\variance, \Radius) \leq \half$, we obtain that
		\begin{equation}
			\frac{\norm{\vel}^2} {2 \variance (1 + \errorterm(\variance, \Radius))} \leq \lagrangianalt(\vel) \leq \frac{\norm{\vel}^2} {2 \variance (1 - \errorterm(\variance, \Radius))}\eqdot
		\end{equation}
		Since, for $x \in [0, 1/2]$, both $\frac{1}{1 +x } \geq 1 - 2 x$ and $\frac{1}{1 - x} \leq 1 + 2 x$ hold,
		we get
		\begin{equation}
		\frac{\norm{\vel}^2} {2 \variance} \left(1 - 2\errorterm(\variance, \Radius)\right) \leq \lagrangianalt(\vel) \leq \frac{\norm{\vel}^2} {2 \variance}  \left(1 + 2\errorterm(\variance, \Radius)\right)		\end{equation}
		which concludes the proof.
	  \end{proof}

	  We will require the following technical lemma.
	  \begin{lemma}\label{lemma:almost_sq_norm_opt_rescale}
		  Consider $\vel, \velalt \in \vecspace$ such that $0 < \norm{\velalt} \leq \tfrac{\mu \Radius}{2}$ for some $\Radius, \mu > 0$.
		  Define, 
		  \begin{equation}
			  \fn(u) = \sup_{\mom \in \vecspace: \norm{\mom} \leq \Radius} \inner{\mom, u} - \half[\mu] \norm{\mom}^2\,,
		  \end{equation}
		  then, with $\lambda = \frac{\norm{\vel}}{\norm{\velalt}}$,
		  \begin{equation}
		  \lambda \fn \parens*{\frac{\vel}{\lambda} + \velalt} \leq \fn(\vel + \velalt)\eqdot
		  \end{equation}
	  \end{lemma}
	  \begin{proof}
		  Define $\mom \defeq \frac{1}{\mu} \parens*{\frac{\vel}{\lambda} + \velalt}$ which has norm at most $\Radius$ since $\norm{\velalt} \leq \frac{\Radius}{2 \mu}$.
		  Then, we have that
		  \begin{align}
		  \lambda \fn \parens*{\frac{\vel}{\lambda} + \velalt} - \fn(\vel + \velalt)
			  &\leq 
			  \lambda \parens*{\inner*{\mom, \frac{\vel}{\lambda} + \velalt } - \half[\mu] \norm{\mom}^2} - \parens*{\inner{\mom, \vel + \velalt} - \half[\mu] \norm{\mom}^2}
			  \notag\\
			  &= (\lambda - 1)\parens*{\inner{\mom, \velalt} -\half[\mu] \norm{\mom}^2}
			  \notag\\
			  &= (\lambda - 1) \times \frac{1}{\mu} \parens*{\frac{\inner{\vel, \velalt}}{\lambda} + \norm{\velalt}^2 - \half \parens*{2 \norm{\velalt}^2 + 2 \frac{\inner{\vel, \velalt}}{\lambda}}}
			  \notag\\
			  &= 0\,,
		  \end{align}
		  and our proof is complete.
	  \end{proof}

\section*{Acknowledgments}
\begingroup
\small
%
%
This research was supported in part by 
the French National Research Agency (ANR) in the framework of
the PEPR IA FOUNDRY project (ANR-23-PEIA-0003),
the ``Investissements d'avenir'' program (ANR-15-IDEX-02),
the LabEx PERSYVAL (ANR-11-LABX-0025-01),
MIAI@Grenoble Alpes (ANR-19-P3IA-0003).
PM is also a member of the Archimedes Research Unit/Athena RC, and was partially supported by project MIS 5154714 of the National Recovery and Resilience Plan Greece 2.0 funded by the European Union under the NextGenerationEU Program.
\endgroup

\bibliographystyle{icml2024}
\bibliography{bibtex/IEEEabrv,bibtex/PaperBib}

\begin{thebibliography}{71}
\providecommand{\natexlab}[1]{#1}
\providecommand{\url}[1]{\texttt{#1}}
\expandafter\ifx\csname urlstyle\endcsname\relax
  \providecommand{\doi}[1]{doi: #1}\else
  \providecommand{\doi}{doi: \begingroup \urlstyle{rm}\Url}\fi

\bibitem[Alongi \& Nelson(2007)Alongi and Nelson]{alongiRecurrenceTopology2007}
Alongi, J.~M. and Nelson, G.~S.
\newblock \emph{Recurrence and Topology}, volume~85 of \emph{Graduate Studies
  in Mathematics}.
\newblock American Mathematical Society, 2007.

\bibitem[Antonakopoulos et~al.(2022)Antonakopoulos, Mertikopoulos, Piliouras,
  and Wang]{AMPW22}
Antonakopoulos, K., Mertikopoulos, P., Piliouras, G., and Wang, X.
\newblock {AdaGrad} avoids saddle points.
\newblock In \emph{ICML '22: Proceedings of the 39th International Conference
  on Machine Learning}, 2022.

\bibitem[Bena{\"\i}m(1999)]{Ben99}
Bena{\"\i}m, M.
\newblock Dynamics of stochastic approximation algorithms.
\newblock In Az{\'e}ma, J., {\'E}mery, M., Ledoux, M., and Yor, M. (eds.),
  \emph{S{\'e}minaire de Probabilit{\'e}s XXXIII}, volume 1709 of \emph{Lecture
  Notes in Mathematics}, pp.\  1--68. Springer Berlin Heidelberg, 1999.

\bibitem[Bena{\"\i}m \& Hirsch(1995)Bena{\"\i}m and Hirsch]{BH95}
Bena{\"\i}m, M. and Hirsch, M.~W.
\newblock Dynamics of {Morse}-{Smale} urn processes.
\newblock \emph{Ergodic Theory and Dynamical Systems}, 15\penalty0
  (6):\penalty0 1005--1030, December 1995.

\bibitem[Bertsekas \& Tsitsiklis(2000)Bertsekas and Tsitsiklis]{BT00}
Bertsekas, D.~P. and Tsitsiklis, J.~N.
\newblock Gradient convergence in gradient methods with errors.
\newblock \emph{SIAM Journal on Optimization}, 10\penalty0 (3):\penalty0
  627--642, 2000.

\bibitem[Blanc et~al.(2020)Blanc, Gupta, Valiant, and
  Valiant]{blancImplicitRegularizationDeep2020}
Blanc, G., Gupta, N., Valiant, G., and Valiant, P.
\newblock Implicit regularization for deep neural networks driven by an
  {O}rnstein-{U}hlenbeck like process.
\newblock In \emph{COLT '20: Proceedings of the 33rd Annual Conference on
  Learning Theory}, 2020.

\bibitem[Brown(1986)]{brownFundamentalsStatisticalExponential1986}
Brown, L.~D.
\newblock \emph{Fundamentals of Statistical Exponential Families with
  Applications in Statistical Decision Theory}, volume~9 of \emph{Lecture
  Notes-Monograph Series}.
\newblock Institute of Mathematical Statistics, 1986.

\bibitem[Choromanska et~al.(2015)Choromanska, Henaff, Mathieu, {Ben Arous}, and
  LeCun]{CHMB+15}
Choromanska, A., Henaff, M., Mathieu, M., {Ben Arous}, G., and LeCun, Y.
\newblock The loss surfaces of multilayer networks.
\newblock In \emph{AISTATS '15: Proceedings of the 18th International
  Conference on Artificial Intelligence and Statistics}, 2015.

\bibitem[Coste(1999)]{costeINTRODUCTIONOMINIMALGEOMETRY}
Coste, M.
\newblock An introduction to $o$-minimal geometry, 1999.

\bibitem[Dauphin et~al.(2014)Dauphin, Pascanu, Gulcehre, Cho, Ganguli, and
  Bengio]{DPGC+14}
Dauphin, Y.~N., Pascanu, R., Gulcehre, C., Cho, K., Ganguli, S., and Bengio, Y.
\newblock Identifying and attacking the saddle point problem in
  high-dimensional non-convex optimization.
\newblock In \emph{NIPS '14: Proceedings of the 27th International Conference
  on Neural Information Processing Systems}, 2014.

\bibitem[de~Acosta(2022)]{deacostaLargeDeviationsMarkov2022}
de~Acosta, A.~D.
\newblock \emph{Large Deviations for Markov Chains}.
\newblock Cambridge Tracts in Mathematics. Cambridge University Press,
  Cambridge, UK, 2022.

\bibitem[Dembo \& Zeitouni(1998)Dembo and Zeitouni]{DZ98}
Dembo, A. and Zeitouni, O.
\newblock \emph{Large Deviations Techniques and Applications}.
\newblock Springer-Verlag, Berlin, 1998.

\bibitem[Dieuleveut et~al.(2020)Dieuleveut, Durmus, and
  Bach]{dieuleveutBridgingGapConstant2018}
Dieuleveut, A., Durmus, A., and Bach, F.
\newblock Bridging the gap between constant step size stochastic gradient
  descent and {Markov} chains.
\newblock \emph{The Annals of Statistics}, 48\penalty0 (3):\penalty0
  1348--1382, June 2020.

\bibitem[Douc et~al.(2018)Douc, Moulines, Priouret, and
  Soulier]{doucMarkovChains2018}
Douc, R., Moulines, E., Priouret, P., and Soulier, P.
\newblock \emph{Markov Chains}.
\newblock Springer Series in Operations Research and Financial Engineering.
  Springer, 2018.

\bibitem[Dupuis(1988)]{dupuisLargeDeviationsAnalysis1988}
Dupuis, P.
\newblock Large deviations analysis of some recursive algorithms with state
  dependent noise.
\newblock \emph{The Annals of Probability}, 16\penalty0 (4):\penalty0
  1509--1536, October 1988.

\bibitem[Dupuis \& Kushner(1985)Dupuis and
  Kushner]{dupuisStochasticApproximationsLarge1985}
Dupuis, P. and Kushner, H.~J.
\newblock Stochastic approximations via large deviations: {Asymptotic}
  properties.
\newblock \emph{SIAM Journal on Control and Optimization}, 23\penalty0
  (5):\penalty0 675--696, September 1985.

\bibitem[Eberle(2016)]{eberleReflectionCouplingsContraction2016a}
Eberle, A.
\newblock Reflection couplings and contraction rates for diffusions.
\newblock \emph{Probability Theory and Related Fields}, 166\penalty0
  (3-4):\penalty0 851--886, December 2016.

\bibitem[Erdogdu et~al.(2018)Erdogdu, Mackey, and
  Shamir]{erdogduGlobalNonconvexOptimization2019}
Erdogdu, M.~A., Mackey, L.~W., and Shamir, O.
\newblock Global non-convex optimization with discretized diffusions.
\newblock In \emph{NeurIPS '18: Proceedings of the 32nd International
  Conference of Neural Information Processing Systems}, 2018.

\bibitem[Feng et~al.(2020)Feng, Gao, Li, Liu, and
  Lu]{fengUniforminTimeWeakError2019}
Feng, Y., Gao, T., Li, L., Liu, J.-G., and Lu, Y.
\newblock Uniform-in-time weak error analysis for stochastic gradient descent
  algorithms via diffusion approximation.
\newblock \emph{Communications in Mathematical Sciences}, 18\penalty0
  (1):\penalty0 163--188, 2020.

\bibitem[Freidlin \& Wentzell(1998)Freidlin and Wentzell]{FW98}
Freidlin, M.~I. and Wentzell, A.~D.
\newblock \emph{Random Perturbations of Dynamical Systems}.
\newblock Springer, 1 edition, 1998.

\bibitem[Ge et~al.(2015)Ge, Huang, Jin, and Yuan]{GHJY15}
Ge, R., Huang, F., Jin, C., and Yuan, Y.
\newblock Escaping from saddle points \textendash\ {Online} stochastic gradient
  for tensor decomposition.
\newblock In \emph{COLT '15: Proceedings of the 28th Annual Conference on
  Learning Theory}, 2015.

\bibitem[Gulinsky \& Veretennikov(1993)Gulinsky and
  Veretennikov]{gulinskyLargeDeviationsDiscreteTime1993}
Gulinsky, O.~V. and Veretennikov, A.~Y.
\newblock \emph{Large Deviations for Discrete-Time Processes with Averaging}.
\newblock De Gruyter, 1993.

\bibitem[G{\"u}rb{\"u}zbalaban et~al.(2021)G{\"u}rb{\"u}zbalaban, {\c S}im{\c
  s}ekli, and Zhu]{gurbuzbalaban2021heavy}
G{\"u}rb{\"u}zbalaban, M., {\c S}im{\c s}ekli, U., and Zhu, L.
\newblock The heavy-tail phenomenon in {SGD}.
\newblock In \emph{ICML '21: Proceedings of the 38th International Conference
  on Machine Learning}, 2021.

\bibitem[Hern{\'a}ndez-Lerma \& Lasserre(2003)Hern{\'a}ndez-Lerma and
  Lasserre]{hernandez-lermaMarkovChainsInvariant2003}
Hern{\'a}ndez-Lerma, O. and Lasserre, J.-B.
\newblock \emph{Markov Chains and Invariant Probabilities}.
\newblock {Birkh{\"a}user}, {Basel}, 2003.

\bibitem[Hodgkinson \& Mahoney(2021)Hodgkinson and
  Mahoney]{hodgkinson2021multiplicative}
Hodgkinson, L. and Mahoney, M.
\newblock Multiplicative noise and heavy tails in stochastic optimization.
\newblock In \emph{ICML '21: Proceedings of the 38th International Conference
  on Machine Learning}, 2021.

\bibitem[Hsieh et~al.(2021)Hsieh, Mertikopoulos, and Cevher]{HMC21}
Hsieh, Y.-P., Mertikopoulos, P., and Cevher, V.
\newblock The limits of min-max optimization algorithms: {Convergence} to
  spurious non-critical sets.
\newblock In \emph{ICML '21: Proceedings of the 38th International Conference
  on Machine Learning}, 2021.

\bibitem[Hsieh et~al.(2023)Hsieh, Karimi, Krause, and Mertikopoulos]{HKKM23}
Hsieh, Y.-P., Karimi, M.~R., Krause, A., and Mertikopoulos, P.
\newblock Riemannian stochastic optimization methods avoid strict saddle
  points.
\newblock In \emph{NeurIPS '23: Proceedings of the 37th International
  Conference on Neural Information Processing Systems}, 2023.

\bibitem[Hu et~al.(2019)Hu, Li, Li, and
  Liu]{huDiffusionApproximationNonconvex2018}
Hu, W., Li, C.~J., Li, L., and Liu, J.-G.
\newblock On the diffusion approximation of nonconvex stochastic gradient
  descent.
\newblock \emph{Annals of Mathematical Sciences and Applications}, 4\penalty0
  (1):\penalty0 3--32, 2019.

\bibitem[Ioffe et~al.(1979)Ioffe, Tikhomirov, and
  Makowski]{ioffeTheoryExtremalProblems1979}
Ioffe, A.~D., Tikhomirov, V.~M., and Makowski, K.
\newblock \emph{Theory of Extremal Problems}.
\newblock North Holland, Amsterdam, NL, 1979.

\bibitem[Jastrz{\k e}bski et~al.(2017)Jastrz{\k e}bski, Kenton, Arpit, Ballas,
  Fischer, Bengio, and Storkey]{jastrzebskiThreeFactorsInfluencing2018}
Jastrz{\k e}bski, S., Kenton, Z., Arpit, D., Ballas, N., Fischer, A., Bengio,
  Y., and Storkey, A.
\newblock Three factors influencing minima in {SGD}.
\newblock \url{https://arxiv.org/abs/1711.04623}, 2017.

\bibitem[Jin et~al.(2017)Jin, Ge, Netrapalli, Kakade, and Jordan]{JGNK+17}
Jin, C., Ge, R., Netrapalli, P., Kakade, S.~M., and Jordan, M.~I.
\newblock How to escape saddle points efficiently.
\newblock In \emph{ICML '17: Proceedings of the 34th International Conference
  on Machine Learning}, 2017.

\bibitem[Kallenberg(2021)]{kallenbergFoundationsModernProbability2021}
Kallenberg, O.
\newblock \emph{Foundations of Modern Probability}, volume~99 of
  \emph{Probability Theory and Stochastic Modelling}.
\newblock Springer, 2021.

\bibitem[Kawaguchi(2016)]{Kaw16}
Kawaguchi, K.
\newblock Deep learning without poor local minima.
\newblock In \emph{NIPS '16: Proceedings of the 30th International Conference
  on Neural Information Processing Systems}, 2016.

\bibitem[Kiefer \& Wolfowitz(1952)Kiefer and Wolfowitz]{KW52}
Kiefer, J. and Wolfowitz, J.
\newblock Stochastic estimation of the maximum of a regression function.
\newblock \emph{The Annals of Mathematical Statistics}, 23\penalty0
  (3):\penalty0 462--466, 1952.

\bibitem[Kifer(1988)]{Kif88}
Kifer, Y.
\newblock \emph{Random Perturbations of Dynamical Systems}.
\newblock Birkh{\"a}user, Boston, MA, 1988.

\bibitem[Kifer(1990)]{Kif90}
Kifer, Y.
\newblock A discrete-time version of the {Wentzell}-{Freidlin} theory.
\newblock \emph{Annals of Probability}, 18\penalty0 (4):\penalty0 1676--1692,
  October 1990.

\bibitem[Lan(2020)]{Lan20}
Lan, G.
\newblock \emph{First-order and Stochastic Optimization Methods for Machine
  Learning}.
\newblock Springer, 2020.

\bibitem[Landau \& Lifshitz(1976)Landau and Lifshitz]{LL76}
Landau, L.~D. and Lifshitz, E.~M.
\newblock Statistical physics.
\newblock In \emph{Course of Theoretical Physics}, volume~5. Pergamon Press,
  Oxford, 1976.

\bibitem[Li \& Wang(2022{\natexlab{a}})Li and
  Wang]{liSharpUniformintimeError2022}
Li, L. and Wang, Y.
\newblock A sharp uniform-in-time error estimate for stochastic gradient
  {Langevin} dynamics.
\newblock \url{https://arxiv.org/abs/2207.09304}, 2022{\natexlab{a}}.

\bibitem[Li \& Wang(2022{\natexlab{b}})Li and
  Wang]{liUniformintimeDiffusionApproximation2022}
Li, L. and Wang, Y.
\newblock On uniform-in-time diffusion approximation for stochastic gradient
  descent.
\newblock \url{https://arxiv.org/abs/2207.04922}, 2022{\natexlab{b}}.

\bibitem[Li et~al.(2017)Li, Tai, and E]{liStochasticModifiedEquations2017}
Li, Q., Tai, C., and E, W.
\newblock Stochastic modified equations and adaptive stochastic gradient
  algorithms.
\newblock In \emph{ICML '17: Proceedings of the 34th International Conference
  on Machine Learning}, 2017.

\bibitem[Li et~al.(2019)Li, Tai, and E]{liStochasticModifiedEquations2019}
Li, Q., Tai, C., and E, W.
\newblock Stochastic modified equations and dynamics of stochastic gradient
  algorithms {I}: {Mathematical} foundations.
\newblock \emph{Journal of Machine Learning Research}, 20\penalty0
  (40):\penalty0 1--47, 2019.

\bibitem[Li et~al.(2020)Li, Lyu, and Arora]{liReconcilingModernDeep2020}
Li, Z., Lyu, K., and Arora, S.
\newblock Reconciling modern deep learning with traditional optimization
  analyses: {The} intrinsic learning rate.
\newblock In \emph{NeurIPS '20: Proceedings of the 34th International
  Conference on Neural Information Processing Systems}, 2020.

\bibitem[Li et~al.(2021)Li, Wang, and Arora]{liWhatHappensSGD2021}
Li, Z., Wang, T., and Arora, S.
\newblock What happens after {SGD} reaches zero loss? --{A} mathematical
  framework.
\newblock In \emph{ICLR '21: Proceedings of the 2021 International Conference
  on Learning Representations}, 2021.

\bibitem[Li et~al.(2022)Li, Wang, and Yu]{liFastMixingStochastic}
Li, Z., Wang, T., and Yu, D.
\newblock Fast mixing of stochastic gradient descent with normalization and
  weight decay.
\newblock In \emph{NeurIPS '22: Proceedings of the 36th International
  Conference on Neural Information Processing Systems}, 2022.

\bibitem[Liu et~al.(2020)Liu, Papailiopoulos, and Achlioptas]{LPA20}
Liu, S., Papailiopoulos, D., and Achlioptas, D.
\newblock Bad global minima exist and {SGD} can reach them.
\newblock In \emph{NeurIPS '20: Proceedings of the 34th International
  Conference on Neural Information Processing Systems}, 2020.

\bibitem[Ljung(1977)]{Lju77}
Ljung, L.
\newblock Analysis of recursive stochastic algorithms.
\newblock \emph{{IEEE} Trans. Autom. Control}, 22\penalty0 (4):\penalty0
  551--575, August 1977.

\bibitem[Lu et~al.(2021)Lu, Balasubramanian, Volgushev, and
  Erdogdu]{yuAnalysisConstantStep2020}
Lu, Y., Balasubramanian, K., Volgushev, S., and Erdogdu, M.~A.
\newblock An analysis of constant step size {SGD} in the non-convex regime:
  {Asymptotic} normality and bias.
\newblock In \emph{NeurIPS '21: Proceedings of the 35th International
  Conference on Neural Information Processing Systems}, 2021.

\bibitem[Lytras \& Mertikopoulos(2024)Lytras and Mertikopoulos]{LM24}
Lytras, I. and Mertikopoulos, P.
\newblock Tamed {Langevin} sampling under weaker conditions.
\newblock \url{https://arxiv.org/abs/2405.17693}, 2024.

\bibitem[Majka et~al.(2020)Majka, Mijatovi{\'c}, and
  Szpruch]{majkaNonasymptoticBoundsSampling2019}
Majka, M.~B., Mijatovi{\'c}, A., and Szpruch, {\L}.
\newblock Nonasymptotic bounds for sampling algorithms without log-concavity.
\newblock \emph{The Annals of Applied Probability}, 30\penalty0 (4):\penalty0
  1534--1581, 2020.

\bibitem[Mertikopoulos \& Staudigl(2018{\natexlab{a}})Mertikopoulos and
  Staudigl]{MerSta18}
Mertikopoulos, P. and Staudigl, M.
\newblock On the convergence of gradient-like flows with noisy gradient input.
\newblock \emph{SIAM Journal on Optimization}, 28\penalty0 (1):\penalty0
  163--197, January 2018{\natexlab{a}}.

\bibitem[Mertikopoulos \& Staudigl(2018{\natexlab{b}})Mertikopoulos and
  Staudigl]{MerSta18b}
Mertikopoulos, P. and Staudigl, M.
\newblock Stochastic mirror descent dynamics and their convergence in monotone
  variational inequalities.
\newblock \emph{Journal of Optimization Theory and Applications}, 179\penalty0
  (3):\penalty0 838--867, December 2018{\natexlab{b}}.

\bibitem[Mertikopoulos et~al.(2020)Mertikopoulos, Hallak, Kavis, and
  Cevher]{MHKC20}
Mertikopoulos, P., Hallak, N., Kavis, A., and Cevher, V.
\newblock On the almost sure convergence of stochastic gradient descent in
  non-convex problems.
\newblock In \emph{NeurIPS '20: Proceedings of the 34th International
  Conference on Neural Information Processing Systems}, 2020.

\bibitem[Mertikopoulos et~al.(2024)Mertikopoulos, Hsieh, and Cevher]{MHC24}
Mertikopoulos, P., Hsieh, Y.-P., and Cevher, V.
\newblock A unified stochastic approximation framework for learning in games.
\newblock \emph{Mathematical Programming}, 203:\penalty0 559--609, January
  2024.

\bibitem[Mignacco \& Urbani(2022)Mignacco and Urbani]{mignacco2022effective}
Mignacco, F. and Urbani, P.
\newblock The effective noise of stochastic gradient descent.
\newblock \emph{Journal of Statistical Mechanics: Theory and Experiment},
  2022\penalty0 (8):\penalty0 083405, 2022.

\bibitem[Mignacco et~al.(2020)Mignacco, Krzakala, Urbani, and
  Zdeborov{\'a}]{mignacco2020dynamical}
Mignacco, F., Krzakala, F., Urbani, P., and Zdeborov{\'a}, L.
\newblock Dynamical mean-field theory for stochastic gradient descent in
  {Gaussian} mixture classification.
\newblock In \emph{NeurIPS '20: Proceedings of the 34th International
  Conference on Neural Information Processing Systems}, 2020.

\bibitem[Mori et~al.(2022)Mori, Ziyin, Liu, and Ueda]{moriPowerLawEscapeRate}
Mori, T., Ziyin, L., Liu, K., and Ueda, M.
\newblock Power-law escape rate of {SGD}.
\newblock In \emph{ICML '22: Proceedings of the 39th International Conference
  on Machine Learning}, 2022.

\bibitem[Pavasovic et~al.(2023)Pavasovic, Durmus, and {\c S}im{\c
  s}ekli]{pavasovic2023approximate}
Pavasovic, K.~L., Durmus, A., and {\c S}im{\c s}ekli, U.
\newblock Approximate heavy tails in offline (multi-pass) stochastic gradient
  descent.
\newblock In \emph{NeurIPS '23: Proceedings of the 37th International
  Conference on Neural Information Processing Systems}, 2023.

\bibitem[Pemantle(1990)]{Pem90}
Pemantle, R.
\newblock Nonconvergence to unstable points in urn models and stochastic
  aproximations.
\newblock \emph{Annals of Probability}, 18\penalty0 (2):\penalty0 698--712,
  April 1990.

\bibitem[Raginsky et~al.(2017)Raginsky, Rakhlin, and
  Telgarsky]{raginskyNonconvexLearningStochastic2017}
Raginsky, M., Rakhlin, A., and Telgarsky, M.
\newblock Non-convex learning via stochastic gradient {Langevin} dynamics: {A}
  nonasymptotic analysis.
\newblock In \emph{COLT '17: Proceedings of the 30th Annual Conference on
  Learning Theory}, 2017.

\bibitem[Rigollet \& H{\"u}tter(2023)Rigollet and H{\"u}tter]{rigollet2023high}
Rigollet, P. and H{\"u}tter, J.-C.
\newblock High-dimensional statistics.
\newblock \url{https://arxiv.org/abs/2310.19244}, 2023.

\bibitem[Robbins \& Monro(1951)Robbins and Monro]{RM51}
Robbins, H. and Monro, S.
\newblock A stochastic approximation method.
\newblock \emph{Annals of Mathematical Statistics}, 22:\penalty0 400--407,
  1951.

\bibitem[van~den Dries \& Miller(1996)van~den Dries and
  Miller]{vandendriesGeometricCategoriesOminimal1996}
van~den Dries, L. and Miller, C.
\newblock Geometric categories and $o$-minimal structures.
\newblock \emph{Duke Mathematical Journal}, 84\penalty0 (2), August 1996.

\bibitem[Veiga et~al.(2024)Veiga, Remizova, and Macris]{veiga2024stochastic}
Veiga, R., Remizova, A., and Macris, N.
\newblock Stochastic gradient flow dynamics of test risk and its exact solution
  for weak features.
\newblock \url{https://arxiv.org/abs/2402.07626}, 2024.

\bibitem[Wainwright(2019)]{wainwright2019high}
Wainwright, M.~J.
\newblock \emph{High-dimensional statistics: A non-asymptotic viewpoint},
  volume~48.
\newblock Cambridge University Press, 2019.

\bibitem[Wang \& Wang(2022)Wang and Wang]{wangThreestageEvolutionFast}
Wang, Y. and Wang, Z.
\newblock Three-stage evolution and fast equilibrium for {SGD} with
  non-degerate critical points.
\newblock In \emph{ICML '22: Proceedings of the 39th International Conference
  on Machine Learning}, 2022.

\bibitem[Wojtowytsch(2023)]{wojtowytsch2023stochastic}
Wojtowytsch, S.
\newblock Stochastic gradient descent with noise of machine learning type.
  {Part} {I}: {Discrete} time analysis.
\newblock \emph{Journal of Nonlinear Science}, 33\penalty0 (3):\penalty0 45,
  2023.

\bibitem[Xie et~al.(2021)Xie, Sato, and Sugiyama]{xieDiffusionTheoryDeep2021}
Xie, Z., Sato, I., and Sugiyama, M.
\newblock A diffusion theory for deep learning dynamics: {Stochastic} gradient
  descent exponentially favors flat minima.
\newblock In \emph{ICLR '21: Proceedings of the 2021 International Conference
  on Learning Representations}, 2021.

\bibitem[Yang et~al.(2021)Yang, Hu, and Li]{yangFastConvergenceRandom2020}
Yang, J., Hu, W., and Li, C.~J.
\newblock On the fast convergence of random perturbations of the gradient flow.
\newblock \emph{Asymptotic Analysis}, 122\penalty0 (3-4):\penalty0 371--393,
  2021.

\bibitem[Zhou et~al.(2020)Zhou, Mertikopoulos, Bambos, Boyd, and
  Glynn]{ZMBB+20}
Zhou, Z., Mertikopoulos, P., Bambos, N., Boyd, S.~P., and Glynn, P.~W.
\newblock On the convergence of mirror descent beyond stochastic convex
  programming.
\newblock \emph{SIAM Journal on Optimization}, 30\penalty0 (1):\penalty0
  687--716, 2020.

\bibitem[Ziyin et~al.(2023)Ziyin, Li, and Ueda]{ziyin2023law}
Ziyin, L., Li, H., and Ueda, M.
\newblock Law of balance and stationary distribution of stochastic gradient
  descent.
\newblock \url{https://arxiv.org/abs/2308.06671}, 2023.

\end{thebibliography}

\end{document}